\documentclass[10pt,a4paper,titling]{article}
\usepackage[T1]{fontenc}
\usepackage[utf8]{inputenc}
\usepackage[english]{babel}

\usepackage{amsfonts}
\usepackage{amsmath}
\usepackage{amssymb}
\usepackage{amsthm}
\usepackage{mathrsfs}
\usepackage{mathtools}
\usepackage{dsfont}
\usepackage{tikz-cd}
\usepackage{enumitem}
\usepackage[perpage,bottom]{footmisc}
\usepackage{caption}
\usepackage{cancel}
\usepackage{textcomp}
\usepackage{xcolor}
\usepackage{minibox}
\usepackage{imakeidx}

\makeindex[columns=2, title=Index, intoc]
\usepackage[linktoc=all]{hyperref}
\usepackage{fancyhdr}
\usepackage{graphicx}
\graphicspath{ {./Figures/} }
\newtheorem{Thm}{Theorem}[section]
\newtheorem{Lem}[Thm]{Lemma}
\newtheorem{Pro}[Thm]{Proposition}
\newtheorem{Cor}[Thm]{Corollary}
\newtheorem{Con}[Thm]{Conjecture}
\theoremstyle{definition}
\newtheorem{Def}[Thm]{Definition}
\newtheorem{Ex}[Thm]{Example}
\newtheorem{Rem}[Thm]{Remark}

\begin{document}
	
\begin{titlepage}
	\begin{center}
	
	\LARGE \textbf{Towards Homological Mirror Symmetry}\\[1.62cm] 
			
	\Large Alessandro Imparato\footnote{Department of Mathematics, ETH Zürich, Autumn 2020--Winter 2021. This work was last updated on December 22, 2022.} \\[3.14cm]

	\begin{abstract}
		This is an expository article on the A-side of Kontsevich's Homological Mirror Symmetry Conjecture. We give first a self-contained study of $A_\infty$-categories and their homological algebra, and later restrict to Fukaya categories, with particular emphasis on the basics of the underlying Floer theory, and the geometric features therein. Finally, we place the theory in the context of mirror symmetry, building towards its main predictions. 
	\end{abstract}

	\end{center}

	\vspace*{7.5cm}
	
	\noindent This article is the author's Master Thesis, supervised by Prof. Dr. William John Merry and submitted for the degree \textit{Master of Science ETH in Mathematics}. It should be accessible to graduate students.\\
	
	
\end{titlepage}

\newpage
\pagenumbering{roman}
\thispagestyle{plain}

\tableofcontents

\newpage

\setcounter{page}{1}

\pagenumbering{arabic}

\section*{Introduction}\markboth{INTRODUCTION}{INTRODUCTION}
\addcontentsline{toc}{section}{Introduction}
\thispagestyle{plain}

\vspace*{1cm}

At its core, Homological Mirror Symmetry expresses a relation between symplectic and algebraic geometry through category theoretic structures. This sentence alone tells us a great deal about the intrinsic beauty of this subject, and the difficulties its ambitious goal presents. 

The conjecture guiding the overall research was first stated by the Russian-born mathematician Maxim Kontsevich in his famous talk (transcripted in \cite{[Kon94]}) at the International Congress of Mathematicians of 1994, held in Zürich, and builds upon the numerical predictions of Mirror Symmetry, a mathematical branch evolved from string theory. The latter claims the existence of a duality between pairs of ``mirror'' Calabi--Yau manifolds, which are some specific kind of complex symplectic manifolds. Physically speaking, these geometrical objects possess the necessary extra dimensions along which strings vibrations are theorised to occur, and their duality translates into a congruence of the physical laws therein described. 

Kontsevich gave the Mirror Symmetry Conjecture a categorical character, reformulating it through the language of derived categories, a topic of great interest to homological algebra. Roughly speaking, the claimed correspondence now selects only part of the underlying Calabi--Yau structures: namely, given a pair $(X,X^\dagger)$ of mirror Calabi--Yau manifolds, there is a duality between the symplectic structure of $X$ (the ``A-side'') --- encoded in its Fukaya category $\mathscr{F}(X)$ --- and the algebraic geometric structure of $X^\dagger$ (the ``B-side'') --- encoded in its category of coherent sheaves $\mathsf{Coh}(X^\dagger)$ --- and vice versa, the B-side of $X$ reflects the A-side of $X^\dagger$. Explicitly, it is conjectured that there are two equivalences of triangulated categories,
\[
\mathsf{D}^\pi(\mathscr{F}(X))\cong\mathsf{D^b}(\mathsf{Coh}(X^\dagger))\qquad\text{and}\qquad \mathsf{D}^\pi(\mathscr{F}(X^\dagger))\cong\mathsf{D^b}(\mathsf{Coh}(X))\,.
\] 

Our goal is to plausibly explain the exact meaning of such expressions. Actually, the landscape is so rich we will only be able to discuss the A-side, merely providing some hints of the ``bigger picture'' described above. Moreover, that offered is just a polished version of a much more profound and hence unavoidably insidious theory; we will often work in significantly simplified settings, especially when dealing with Floer theory in the second part, encouraging, whenever sensible, a comprehensive look rather than a single in-depth analysis. Hopefully, such an approach will convey the essentiality of each pure mathematical branch homological mirror symmetry borrows from. After all, this synergy between distinct areas of mathematics is perhaps the most appealing virtue offered. \\    

The work consists of two parts: a general though self-contained treatment of the abstract algebraic notion of $A_\infty$-categories, and a subsequent specialization to a particular instance thereof, the Fukaya category. Specifically, we start our journey in Chapter 1 with $A_\infty$-algebras, the prototypes of $A_\infty$-categories, which we properly introduce in Chapter 2 together with the necessary toolkit. In Chapter 3, we apply this new technology to forge the intermediate notion of bounded derived category $\mathsf{D^b}(\mathcal{A})$ associated to an $A_\infty$-category $\mathcal{A}$. In Chapter 4, we enrich it to the split-closed derived category $\mathsf{D}^\pi(\mathcal{A})$ appearing in the homological mirror conjecture. Next, in Chapter 5, we move our focus to symplectic manifolds and study Lagrangian Floer homology, integral to the definition of Fukaya category $\mathscr{F}(X)$, which we provide in Chapter 6 along with the description of its $A_\infty$-structure and special geometric features. Finally, in Chapter 7, we give a survey of the theory leading to the Homological Mirror Symmetry Conjecture, briefly touching on its physical motivation and origins in mirror symmetry, and aiming at a qualitative understanding of the B-side as well. \\

\noindent\textbf{Acknowledgments.} I would like to thank my supervisor, Will Merry, for carefully reading through the various iterations of this thesis and, more importantly, always seconding my ideas on how to best develop it, until its publication.

\newpage

\part{Homological Algebra}
\markboth{I\quad HOMOLOGICAL ALGEBRA}{I\quad HOMOLOGICAL ALGEBRA}
\thispagestyle{plain}

\vspace*{2cm}

The first incarnation of $A_\infty$-categories were studied by Stasheff in his \textit{Homotopy associativity of H-spaces}, dating 1963 (see \cite{[Sta63]}). In an attempt to find useful tools for the study of topological spaces with group-like features, he developed the notion of $A_\infty$-algebras: these are graded vector spaces (properly, algebras) which are associative only up to a system of higher order homotopies fulfilling certain relations (whence the notation ``$A_\infty$''); as such, they can be regarded as an enriched version of differential graded algebras. Packing together this data, one speaks of $A_\infty$-structure. The earliest example studied by Stasheff were loop spaces, instances of the so-called $A_\infty$-spaces, which result from endowing topological spaces with suitable $A_\infty$-structures.

Soon after their discovery, $A_\infty$-algebras became objects of interest of many other illustrious mathematicians of that period, Kadeishvili among these (particularly in \cite{[Kad80]}), though always within the topological context, specifically homotopy theory. This tendence didn't change until the early 90s, when $A_\infty$-structures started to exhibit their relevance towards geometry and mathematical physics through works of the various Getzler--Jones (\cite{[GJ90]}), Fukaya (\cite{[Fuk93]}) and Stasheff himself. In particular, Fukaya's take sparked Kontsevich's Homological Mirror Symmetry Conjecture of 1994 (\cite{[Kon94]}), which heavily relies on the $A_\infty$-structure of Fukaya categories in the symplectic A-side. 

Specifically, $A_\infty$-categories took the stage. Thinkable as $A_\infty$-algebras with ``several objects'', $A_\infty$-categories are variations of the classical notion of categories. Albeit similar in aspect, since consisting of objects and morphisms --- the latters collected in graded vector spaces (or modules), the hom-spaces --- they are not categories in general, because they lack units and violate associativity. Like for $A_\infty$-algebras, associativity is supplanted by more complicated equations involving higher order ``composition maps'' which generalize the aforementioned homotopies, providing the final piece of data. 

The goal of this first part is to give a thorough introduction to the algebraic machinery of $A_\infty$-categories, gradually building towards the item $\mathsf{D}^\pi(\mathcal{A})$ appearing on the A-side of the Homological Mirror Symmetry Conjecture. We do this under the slogan: ``the larger the structure, the easier its description''. The precise plan is articulated as follows:
\begin{itemize}[leftmargin=0.5cm]
	\item In Chapter 1, we introduce $A_\infty$-algebras and $A_\infty$-modules in order to acquire some familiarity with the type of algebraic structures in the ensuing chapters. 
	
	\item In Chapter 2, we present the core ingredients of the subject without taking a breath. We define $A_\infty$-categories and their cohomological counterparts, $A_\infty$-functors between $A_\infty$-categories, themselves related through pre-natural transformations. After a brief technical digression in Hochschild cohomology, we give the notion of homotopic $A_\infty$-functors and see how they take part in the fabrication of $A_\infty$-categories. Then we move one step up, presenting $A_\infty$-modules over $A_\infty$-categories, which offer a way of studying the latters indirectly.
	\newline The basic vocabulary being set, we discuss the matter of units, with special focus on cohomologically unital $A_\infty$-categories and quasi-equivalences thereof, in fact the ``best calibrated'' sort of $A_\infty$-functors appearing in most propositions. Finally, we define Yoneda embeddings, which are special $A_\infty$-functors fundamental in bridging $A_\infty$-categories to their respective category of $A_\infty$-modules.
	
	\item In Chapter 3, we discuss the idea of triangulated categories, a major topic in homological algebra, first axiomatized by Jean-Louis Verdier (in \cite{[Ver96]}). After a few preliminaries, we highlight some fundamental operations and constructions for $A_\infty$-modules, particularly an abstract parent of the well-known mapping cone from topology and some special diagrams called exact triangles.
	\newline Shortly afterwards, we will see these are in fact part of the program to extend triangulation to the $A_\infty$-world, ultimately resulting in the $A_\infty$-category $Tw\mathcal{A}$ of twisted complexes. The latter is then applied to produce a couple criteria of triangularity for $A_\infty$-categories, finally proving that their cohomological categories are triangulated in the classical sense. This culminates in the definition of bounded derived category $\mathsf{D^b}(\mathcal{A})$.
	
	\item In Chapter 4, we enlarge our construction to account for direct summands forming objects of $\mathsf{D^b}(\mathcal{A})$. We explore the classical notion of split images and split-closure, and later translate it to the abstract realm of $A_\infty$-categories. We use the so obtained technology to finally produce the split-closed derived category $\mathsf{D}^\pi(\mathcal{A})$. At the end, we also discuss qualitatively the twisting procedure, which links the algebra of Fukaya categories to their subtended geometry.       
\end{itemize}   

References for Chapter 1 are \cite{[Kel01]} and \cite{[Kel06]}. Chapter 2 to 4 are based on part I of the excellent book \cite{[Sei08]} by Paul Seidel, currently one of the major experts in the area. If not otherwise explicited, every notion must be traced back to it. Also, at the beginning of each chapter we dedicate a section to the necessary prerequisites from category theory and homological algebra; especially when it comes to derived categories and functors, we will refer to \cite{[GM03]}. However, the advanced reader might skip these preliminaries so to preserve a smoother reading experience.

\newpage
\pagestyle{fancy}

\section{$A_{\infty}$-algebras}
\thispagestyle{plain}

\subsection{Example of non-associativity: $A_\infty$-spaces}\label{ch1.0}

Let us begin our study of $A_\infty$-algebras from a topological example which delivers most clearly the concept of non-associativity. Let $(X,x_0)$ be a pointed topological space, $\Omega X \coloneqq \{ f: [0,1]\rightarrow X\mid  f \text{ continuous, } f(0)=f(1)=x_0\}$ the space of loops in $X$ based at $x_0\in X$. We can define a composition map by concatenating loops: $m_2: \Omega X \times \Omega X\rightarrow \Omega X,\; (f_1,f_2) \mapsto (f_1\ast f_2)$, where
\[
(f_1\ast f_2)(t) \coloneqq  
\begin{cases}
f_1(2t) & $if $ 0\leq t \leq 1/2 \\ 
f_2(2t-1) & $if $ 1/2 \leq t \leq 1
\end{cases}\;, 
\]
so that the resulting loop ``spends half the time'' along $f_1$, then half along $f_2$. Obviously, $m_2$ is not associative, since $m_2(m_2(f_1,f_2),f_3) \neq m_2(f_1,m_2(f_2,f_3))$ for a third $f_3 \in\Omega X$ (in the former concatenation only a ``quarter of time'' is spent along $f_1$, against half in the latter). 

However, we can modify the time distribution via a homotopy $m_3: [0,1] \times (\Omega X)^3 \rightarrow \Omega X$ which continuously deforms the concatenation $(f_1\ast f_2)\ast f_3$ into $f_1\ast (f_2\ast f_3)$. The same $m_3$ can be used to move between the five possible concatenations of four loops: $(f_1\ast(f_2\ast f_3))\ast f_4$, $((f_1\ast f_2)\ast f_3)\ast f_4$, $(f_1\ast f_2)\ast(f_3\ast f_4)$, $f_1\ast(f_2\ast(f_3 \ast f_4))$ and $f_1\ast((f_2\ast f_3)\ast f_4)$. Drawing schematically these configurations as vertices of a graph whose sides represent the path of $m_3$, one ends up with the boundary of a pentagon. Call it $K_4$. 

In turn, this can be extended to generic $(n-2)$-dimensional polytopes $K_n$ by iteratively defining homotopies $m_n: K_n \times (\Omega X)^n\rightarrow \Omega X$ for $n>3$ (we set $K_2 \coloneqq \{*\}$ and $K_3 \coloneqq [0,1]$; see Figure \ref{associahedra}). 

\begin{figure}[htp]
	\centering
	\includegraphics[width=0.9\textwidth]{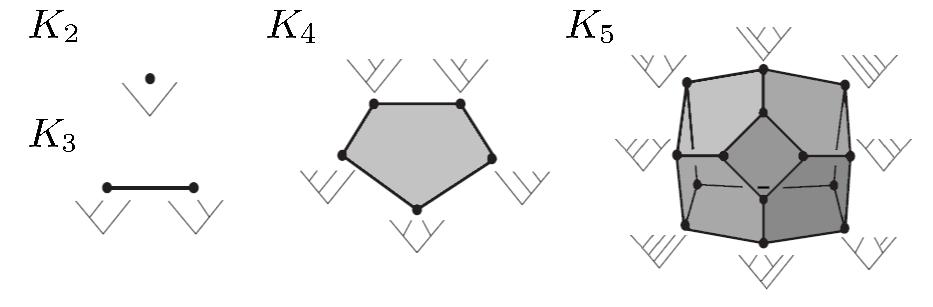}		  
	\caption{The first four associahedra $K_2, K_3, K_4, K_5$. Each concatenation at the vertices can be represented by a binary tree: for example, $K_2$ consists of a single point $(f_1\ast f_2)\in\Omega X$, whose binary tree is trivial; $K_3$ has two vertices, $((f_1\ast f_2)\ast f_3)$ and $(f_1\ast(f_2\ast f_3))$ respectively; the concatenations forming $K_4$ are listed above. On the other hand, higher dimensional cells parametrize trees with more nodal ramifications: for example, the 1-cell of $K_3$ is labeled by a trifurcated tree, encoding $f_1\ast f_2\ast f_3$. \newline
	[Source: \texttt{http://pseudomonad.blogspot.com/2010/08/m-theory-lesson-345.html}]}
	\label{associahedra}
\end{figure}

These $K_n$ are called \textbf{associahedra}\index{associahedron}, first constructed by Stasheff in \cite{[Sta63]}. He included them in his definition of \textit{$A_\infty$-spaces}: topological spaces $Y$ with a collection of compatible maps $m_n: K_n \times Y^n\rightarrow Y$, for $n\geq 2$. As we have just seen, the loop space $\Omega X$ is a prototypical example of $A_\infty$-space. Moreover, for any such $Y$, the singular cochain complex $C^{\bullet}(Y)$ serves as our first case of an $A_\infty$-algebra. Let us see why.

\subsection{$A_\infty$-algebras and morphisms}\label{ch1.1}

Throughout this thesis, we will work over a fixed ground field $\mathbb{K}$ and assume that gradings are taken over $\mathbb{Z}$, unless otherwise stated. We start with a basic review of graded vector spaces and complexes (cf. \cite[section 2.1]{[Kel06]}).

\begin{Def}\label{gradedvectorspace}		
	Let $V$ be a $\mathbb{Z}$-graded $\mathbb{K}$-vector space, thus of the form $V = \bigoplus_{p\in\mathbb{Z}}V^p$. Its \textit{$n$-fold shift} $V[n]$, for $n\in\mathbb{Z}$, is the graded vector space with all indices shifted $n$ times downwards, so that $(V[n])^p = V^{p+n}$ for all $p\in\mathbb{Z}$. We write $SV\coloneqq V[1]$ and call it the \textit{suspension} of $V$.
	
	Given another graded vector space $V'$, a morphism of graded vector spaces $f:V\rightarrow V'[d]$ is a homogeneous graded $\mathbb{K}$-linear map of degree $d = |f|\in\mathbb{Z}$ (so that a homogeneous element $v\in V$ of degree $p = |v|$ will be mapped to a homogeneous element $f(v)\in V'$ of degree $p+d$; for practicality, we will often suppress information about the degree of morphisms). 
	
	Graded vector spaces and their morphisms form the \textbf{category $\mathsf{GradVect}$}\index{category!of graded vector spaces}.
\end{Def}

When considering the tensor product $f\otimes g:V\otimes W\rightarrow V'\otimes W'$ of two morphisms of graded vector spaces $f:V\rightarrow V'$ and $g:W\rightarrow W'$, the output is subject to the \textit{Koszul sign rule}:
\[
(f\otimes g)(v\otimes w) = (-1)^{|g||v|}f(v)\otimes g(w)\;,
\]
for all homogeneous elements $v\in V, w\in W$.

\begin{Def}\label{complex}							
	A graded vector space $V$ is a (\textit{cochain}) \textit{complex} when endowed with a differential $d_V:V\rightarrow V[1]$, a homogeneous morphism of degree 1 such that $d_V\circ d_V = 0$. We define $d_{SV} \coloneqq -d_{V}$, while
	\[
	d(f) \coloneqq  d_{V'}\circ f -(-1)^{|f|}f\circ d_V
	\]
	for a morphism of graded vector spaces $f:V\rightarrow V'$. The latter becomes a morphism of complexes (a cochain map) if and only if $d(f) = 0$. 
	
	Complexes and their morphisms form the \textbf{category $\mathsf{Ch}$}\index{category!of complexes}. We can therefore study its cohomological counterpart $H(\mathsf{Ch})$, whose objects are graded vector spaces $H(V)= \bigoplus_{p\in\mathbb{Z}} H^p(V)$ and whose morphisms are the induced graded morphisms (indeed $H(\mathsf{Ch})\cong\mathsf{GradVect}$ as categories). 
	
	Finally, we recall that two morphisms of complexes $f$ and $f'$ are homotopic if and only if there exists a morphism of graded vector spaces $h$ such that $f - f' = d(h)$. Homological algebra then tells us that the induced maps in cohomology are the same: $H^p(f) = H^p(f')$ for all $p\in\mathbb{Z}$.    
\end{Def}

We proceed to define $A_\infty$-algebras, also known as \textit{strongly homotopy associative algebras}, for reasons that will become clear in Remark \ref{Ainfalgrem} below. The main reference is \cite[sections 3--4]{[Kel01]}. 

\begin{Def}\label{Ainfalgebra}							
	An \textbf{$A_\infty$-algebra}\index{aaa@$A_\infty$-algebra} over $\mathbb{K}$ is a $\mathbb{Z}$-graded $\mathbb{K}$-vector space
	\[
	A = \bigoplus_{p\in\mathbb{Z}}A^p
	\] 
	endowed with graded morphisms $m_n: A^{\otimes n}\rightarrow A[2-n]$, for $n\geq 1$, satisfying the equations
	\begin{equation}\label{Ainfalgasseqs}
	\sum_{r+s+t=n}\mkern-12mu(-1)^{r+st}m_{r+1+t}(\textup{id}_{A}^{ \otimes r}\otimes m_s\otimes {\textup{id}_{A}^{\otimes t}}) = 0\;,
	\end{equation}	
	as maps $A^{\otimes n}\rightarrow A$, for all $n\geq 1$.
	
	$A$ is \textit{strictly unital} if endowed with some $1_A\in A^0$ such that $m_1(1_A) = 0,\; m_2(1_A, a) = a = m_2(a, 1_A)$ for all $a\in A$ and $m_i(...,1_A,...) = 0$ if $i>2$. (If it exists, $1_A$ is unique; we call it \textit{strict unit}.)
	
	An \textit{$A_\infty$-subalgebra} $B$ of $A$ is a graded subspace such that all morphisms $m_n$ restrict to $B$.  
\end{Def}

Here we won't concern ourselves too much about the signs appearing in (\ref{Ainfalgasseqs}), since explicit inputs will modify it according to the Koszul rule from Definition \ref{gradedvectorspace}. We will be more precise when discussing $A_\infty$-categories, where all formulae are presented in this fashion.

\begin{Rem}\label{Ainfalgsusp}							
	If one wishes to make the degree of each morphism $m_n$ of Definition \ref{Ainfalgebra} independent from $n$, there is a canonical way to do it using suspensions: define a bijective correspondence between the $m_n$'s and the graded maps $b_n: (SA)^{\otimes n}\rightarrow SA$ by imposing commutativity of the diagram
	\[
	\begin{tikzcd}
		A^{\otimes n}\arrow[r, "m_n"]\arrow[d, "s^{\otimes n}"'] & A\arrow[d, "s"] \\
		(SA)^{\otimes n}\arrow[r, "b_n"'] & SA
	\end{tikzcd}
	\]
	for all $n\geq 1$. Here $s: A\rightarrow SA,\; a\in A^p \mapsto a\in (SA)^{p-1}$ has degree $-1$, so that each $b_n$ has $|b_n| = -(-1)\cdot n + (2-n) + (-1) = 1$. In particular, $b_1 = -m_1$ and $b_2(a_1,a_2) = (-1)^{|a_1|}m_2(a_1,a_2)$. One can prove that the $m_n$'s give an $A_\infty$-algebra structure on $A$ if and only if
	\[
	\sum_{r+s+t=n}\mkern-12mu b_{r+1+t}(\textup{id}_{A}^{ \otimes r}\otimes b_s\otimes {\textup{id}_{A}^{\otimes t}}) = 0
	\]
	for all $n\geq 1$, where we note that the signs in \eqref{Ainfalgasseqs} have vanished.
	
	Concretely, we will almost never work with suspensions, but one should account for them when trying to put together the various sign conventions.
\end{Rem}

\begin{Rem}\label{Ainfalgrem}						
	The set of equations \eqref{Ainfalgasseqs} yields the following consequences:
	\begin{itemize}[leftmargin=0.5cm]
		\item $m_1 m_1 = 0$, so that $(A,m_1)$ is a cochain complex (then, by definition, so is $(SA,-m_1)$).
		\item $m_1 m_2 = m_2 (m_1\otimes \textup{id}_A + \textup{id}_A\otimes m_1)$ makes the differential $m_1$ a derivation with respect to the multiplication\footnote{We refer to it as ``multiplication'' because, in the particular situation when $A^p = 0$ for all $p\neq 0$, $A=A^0$ has $m_2$ as only non-vanishing map, of degree 0, hence thought of as multiplication for the ordinary associative algebra $A$.} $m_2$ (that is, compatible through the graded Leibniz rule --- as one would expect). Since $(m_1\otimes \textup{id}_A + \textup{id}_A\otimes m_1)$ is the differential of $A^{\otimes 2}$, $m_2$ is a morphism of complexes.
		\item $m_2(\textup{id}_A\otimes m_2 - m_2\otimes \textup{id}_A) = m_1 m_3 + m_3(m_1\otimes \textup{id}_A\otimes \textup{id}_A + \textup{id}_A\otimes m_1\otimes \textup{id}_A$ $+ \textup{id}_A\otimes \textup{id}_A\otimes m_1) = d(m_3)$ tells us that $m_2$ is associative \textit{only up to homotopy}! (whence the terminology ``strongly homotopy associative'' algebra). Therefore, the corresponding cohomology $H(A)$ is an associative graded algebra with respect to $H(m_2)$ (its unit element, if it exists, is induced by the \textit{cohomological unit} of $A$).
		\item if $m_n = 0$ for all $n>2$, then $(A,m_1,m_2)$ turns into an associative\footnote{Here $d(m_3) = 0$ vacuously.} differential $\mathbb{Z}$-graded algebra; conversely, any such algebra can be seen as particular instance of an $A_\infty$-algebra.   
	\end{itemize}
\end{Rem}

The second bullet point exposes the nature of $A_\infty$-algebras as algebras which are associative up to higher homotopies (and possibly lack identities). Also, the third bullet point already foreshadows the advantage of working at cohomological level, one of our guiding principles in the theory to come.

\begin{Def}\label{Ainfalgmorph}				
	Given two $A_\infty$-algebras $A$ and $B$, with morphisms $m_{n}^{A}$ respectively $m_{n}^{B}$, a \textbf{morphism of $A_\infty$-algebras}\index{morphism!of $A_\infty$-algebras} $f:A\rightarrow B$ is a sequence of graded maps $f_n:A^{\otimes n}\rightarrow B[1-n]$, for $n\geq 1$, fulfilling the equations
	\begin{equation}\label{Ainfalgmorpheqs}
	\sum_{r+s+t=n}\mkern-12mu(-1)^{r+st}f_{r+1+t}(\textup{id}_{A}^{\otimes r}\otimes m_{s}^{A}\otimes {\textup{id}_{A}^{\otimes t}}) = \mkern-24mu\sum_{\substack{1\leq r\leq n \\ s_1+...+s_r = n}}\mkern-24mu(-1)^u m_{r}^{B}(f_{s_1}\otimes ...\otimes f_{s_r})\;
	\end{equation}
	for all $n\geq 1$, where $u =\coloneqq \sum_{i=1}^{r-1}(r-i)(s_i-1)$.
	
	Composition with a second morphism of $A_\infty$-algebras $g:B\rightarrow C$ is specified by the collection of graded maps
	\[
	(g\circ f)_n \coloneqq  \mkern-24mu\sum_{\substack{1\leq r\leq n \\ s_1+...+s_r = n}}\mkern-24mu(-1)^u g_r\circ (f_{s_1}\otimes ...\otimes f_{s_r})
	\]
	(plugging them into \eqref{Ainfalgmorpheqs} shows that indeed $g\circ f:A\rightarrow C$ is a morphism of $A_\infty$-algebras).
	
	For $A=B$, it is easy to see that $f$ with $f_1 \coloneqq \textup{id}_A$ and $f_n \coloneqq 0$ otherwise corresponds to the identity morphism on $A$.
\end{Def}

In particular, the first two equalities of \eqref{Ainfalgmorpheqs}, $f_1 m_1^A = m_1^B f_1$ and $f_1 m_2^A - m_2^B(f_1\otimes f_1) = m_1^B f_2 + f_2(m_1^A\otimes \textup{id}_A + \textup{id}_A\otimes m_1^A) = d(f_2)$, tell us that $f_1$ is a morphism of complexes commuting with $m_2$ only up to homotopy. Again, the signs in \eqref{Ainfalgmorpheqs} can be dispensed with upon use of suspensions (see Remark \ref{Ainfalgsusp}). Finally, we present a few definitions that will reappear in their $A_\infty$-categorical formulation.

\begin{Def}							
	A morphism of $A_\infty$-algebras $f:A\rightarrow B$ is called a \textit{quasi-isomorphism} if $f_1$ is a quasi-isomorphism, that is, if it induces an algebra isomorphism $H(f_1): H(A)\rightarrow H(B)$.
	
	A morphism of strictly unital $A_\infty$-algebras $f:(A,1_A)\rightarrow (B,1_B)$ is itself called \textit{strictly unital} if $f_1(1_A) = 1_B$ and $f_i(...,1_A,...) = 0$ for $i>1$. 
\end{Def}

\subsection{$A_{\infty}$-modules and morphisms}\label{ch1.2}

\begin{Def}\label{Ainfalgmod}							
	Let $A$ be an $A_\infty$-algebra, equipped with morphisms $m_{n}^{A}$ as in Definition \ref{Ainfalgebra}. An \textbf{$A_\infty$-module over $A$}\index{aaa@$A_\infty$-module!over an $A_\infty$-algebra} is a $\mathbb{Z}$-graded $\mathbb{K}$-vector space 
	\[
	M = \bigoplus_{p\in\mathbb{Z}}M^p
	\]
	endowed with maps $m_{n}^{M}: M\otimes A^{\otimes n-1}\rightarrow M[2-n]$, for $n\geq 1$, fulfilling equations similar to (\ref{Ainfalgasseqs}):
	\begin{equation}\label{Ainfalgmodasseq}
	\begin{split}
	\sum_{r+s+t=n}\mkern-12mu(-1)^{r+st}m_{r+1+t}^{M}(\textup{id}_{M}\otimes \textup{id}_{A}^{\otimes r-1}\otimes m_{s}^{A}\otimes \textup{id}_{A}^{\otimes t}) &= 0\qquad \text{if }r>0,	\\
	\sum_{s+t=n}\mkern-6mu(-1)^{st}m_{1+t}^{M}(m_{s}^{M}\otimes \textup{id}_{A}^{\otimes t}) &= 0\qquad \text{if }r=0.
	\end{split}
	\end{equation}
\end{Def}

\begin{Def}\label{Ainfmodmorph}				
	Let $M$ and $N$ be $A_\infty$-modules over an $A_\infty$-algebra $A$, with morphisms $m_{n}^{M}$ respectively $m_{n}^{N}$. A \textbf{morphism of $A_\infty$-modules}\index{morphism!of $A_\infty$-modules} $f:M\rightarrow N$ is a sequence of graded maps $f_n:M\otimes A^{\otimes n-1}\rightarrow N[1-n]$, for $n\geq 1$, fulfilling the equations
	\begin{equation}\label{Ainfmodmorpheqs}
	\sum_{r+s+t=n}\mkern-12mu(-1)^{r+st}f_{r+1+t}(\textup{id}_A^{\otimes r}\otimes m_{s}^M\otimes \textup{id}_{A}^{\otimes t}) = \mkern-6mu\sum_{s+t = n}\mkern-6mu(-1)^{(s+1)t}m_{1+t}^{N}(f_s\otimes \textup{id}_{A}^{\otimes t})\;
	\end{equation}
	for all $n\geq 1$, where a distinction between the cases $r>0$ and $r=0$ as in last definition must be made.
	
	Composition with a second morphism of $A_\infty$-modules $g:N\rightarrow P$ is specified by a collection of graded maps
	\[
	(g\circ f)_n \coloneqq  \mkern-6mu\sum_{s+t=n}\mkern-6mu(-1)^{(s-1)t} f_{1+t}\circ (g_s\otimes \textup{id}_{A}^{\otimes t})
	\]
	(which can be shown to satisfy equations \eqref{Ainfmodmorpheqs}). The definitions of identity morphism and quasi-isomorphism are identical to the $A_\infty$-algebra case.
	
	$A_\infty$-modules together with morphisms of $A_\infty$-modules form a category, which we denote by $\mathsf{C}_\infty (A)$.
\end{Def}

Now that we have had a little taste of the kind of algebraic structures that Part I deals with, let us move to the more general notion of \textit{$A_\infty$-categories}.

\newpage

\section{$A_{\infty}$-categories}
\thispagestyle{plain}

\subsection{Category theoretic preliminaries}\label{ch2.1}

Before studying $A_\infty$-categories, we summarize here some basic definitions and results from category theory which will appear frequently and play a key role in the proofs to come. Recall that we are working over some fixed ground field $\mathbb{K}$.

\begin{Def}											
	A category $\mathsf{C}$ is said to be:
	\begin{itemize}[leftmargin=0.5cm]
		\item \textit{linear} and \textit{graded} if all its Hom-sets are graded $\mathbb{K}$-vector spaces and composition of compatible morphisms is a bilinear operation over $\mathbb{K}$;
		\item \textit{differential $\mathbb{Z}$-graded} (\textit{dg}) if each Hom-set is a $\mathbb{Z}$-graded (cochain) complex, so that composition of morphisms is a cochain map and the differential of each identity is zero (example: the category $\mathsf{Ch}$ from Definition \ref{complex});
		\item \textit{small} if both obj($\mathsf{C}$) and the collection of all Hom-sets are sets;
		\item \textit{non-unital} if possibly lacking some identity morphisms (if so, $\mathsf{C}$ is not an actual category).		
	\end{itemize}
\end{Def}

\begin{Def}											
	Let $\mathsf{C}$ and $\mathsf{D}$ be categories. A functor $\mathsf{F}:\mathsf{C}\rightarrow\mathsf{D}$ is:
	\begin{itemize}[leftmargin=0.5cm]
		\item an \textit{isomorphism} if there exists a functor $\mathsf{G}:\mathsf{D}\rightarrow\mathsf{C}$ such that $\mathsf{G}\circ\mathsf{F} = \mathsf{Id_C}$ and $\mathsf{F}\circ\mathsf{G} = \mathsf{Id_D}$ as functors; in this case $\mathsf{C}$ and $\mathsf{D}$ are called \textit{isomorphic};
		\item an \textit{equivalence} if there exists a functor $\mathsf{G}:\mathsf{D}\rightarrow\mathsf{C}$ (sometimes called \textit{quasi-inverse}) and natural isomorphisms $\Phi: \mathsf{Id_C}\rightarrow\mathsf{G}\circ\mathsf{F}$ and $\Psi: \mathsf{F}\circ\mathsf{G}\rightarrow \mathsf{Id_D}$ (meaning that each $\Phi_X\in \textup{Hom}_\mathsf{C}\big(X,\mathsf{G}(\mathsf{F}(X))\big)$ for $X\in \textup{obj}(\mathsf{C})$ and each $\Psi_Y\in \textup{Hom}_\mathsf{D}(\mathsf{F}(\mathsf{G}(Y)),Y)$ for $Y\in \textup{obj}(\mathsf{D})$ is an isomorphism of objects); then $\mathsf{C}$ and $\mathsf{D}$ are called \textit{equivalent}; 
		\item \textit{full} and \textit{faithful} if $\mathsf{F}_{X,Y}:\textup{Hom}_\mathsf{C}(X,Y)\rightarrow\textup{Hom}_\mathsf{D}(\mathsf{F}(X),\mathsf{F}(Y))$ is surjective respectively injective for all $X, Y \in$ obj$(\mathsf{C})$ (\textit{fully faithful} if both);
		\item \textit{essentially surjective} if for all $Y\in$ obj($\mathsf{D}$) there are an $X\in$ obj($\mathsf{C}$) and an isomorphism of objects $f\in$ Hom$_\mathsf{D}(Y,\mathsf{F}(X))$; 
		\item an \textit{embedding} if faithful and injective on objects (that is, $\mathsf{F}(X_1)=\mathsf{F}(X_2)\in\textup{obj}(\mathsf{D})$ implies $X_1=X_2\in\textup{obj}(\mathsf{C})$);
		\item \textit{linear} and \textit{graded} if each $\mathsf{F}_{X,Y}$ as above is a $\mathbb{K}$-linear, graded morphism between graded vector spaces (assuming $\mathsf{C}, \mathsf{D}$ are linear, graded);
		\item \textit{\textup(non-\textup)unital} if (not) carrying all identity morphisms to identity morphisms. 
	\end{itemize} 
\end{Def}

The following results are an easy consequence of above definitions (for example, their proofs can be found in \cite{[GM03]}).

\begin{Lem}\label{aboutfunctors}				
	A functor $\mathsf{F}:\mathsf{C}\rightarrow\mathsf{D}$ is:
	\begin{enumerate}[leftmargin=0.6cm] 
		\renewcommand{\theenumi}{\roman{enumi}}
		\item an isomorphism if and only if it is bijective on objects and \textup{Hom}-sets,
		\item an equivalence if and only if it is full, faithful and essentially surjective.  
	\end{enumerate}
	In particular, isomorphisms are equivalences, thus full and faithful.
\end{Lem}
\vspace*{0.15cm}

\noindent\minibox[frame]{Henceforth, we will assume that all our ($A_\infty$-)categories are linear and small,\\ unless otherwise stated.}

\subsection{$A_{\infty}$-categories and $A_{\infty}$-functors}\label{ch2.2}

\begin{Def}\label{Ainfcat}								
	A \textbf{non-unital $A_\infty$-category}\index{aaa@$A_\infty$-category} $\mathcal{A}$ consists of the following data:
	\begin{itemize}[leftmargin=0.5cm]
		\renewcommand{\labelitemi}{\textendash}
		\item a set of objects obj($\mathcal{A}$),
		\item for any pair $X_0, X_1\in$ obj($\mathcal{A}$), a $\mathbb{Z}$-graded $\mathbb{K}$-vector space 
		\[
		\textup{hom}_\mathcal{A}(X_0,X_1) = \bigoplus_{p\in\mathbb{Z}}\textup{hom}_\mathcal{A}^{p}(X_0,X_1)\;,
		\]
		whose elements we still call morphisms (homogeneous of degree $p$ when belonging to $\textup{hom}_\mathcal{A}^{p}(X_0,X_1)$),
		\item (multi)linear graded maps
		\begin{equation}\label{compomap}
		\mu_{\mathcal{A}}^{d}:\textup{hom}_\mathcal{A}(X_{d-1},X_d)\otimes...\otimes\textup{hom}_\mathcal{A}(X_0,X_1)\rightarrow\textup{hom}_\mathcal{A}(X_0,X_d)[2-d]
		\end{equation}
		for every $d\geq 1$ and $X_0,...,X_d\in\textup{obj}(\mathcal{A})$, which we refer to as \textbf{composition maps}\index{composition maps}. These must satisfy the quadratic \textit{$A_\infty$-associativity equations}
		\begin{equation}\label{Ainfcatasseqs}
		\mkern-30mu\sum_{\substack{1\leq m\leq d \\ \quad 0\leq n\leq d-m}}\mkern-24mu(-1)^{\dagger_n}\mu_{\mathcal{A}}^{d-m+1}(a_d,...,a_{n+m+1},\underbrace{\mu_{\mathcal{A}}^{m}(a_{n+m},...,a_{n+1})}_{\in \textup{hom}_{\mathcal{A}}(X_n,X_{n+m})},a_n,...,a_1) = 0
		\end{equation}
		in $\textup{hom}_\mathcal{A}(X_0,X_d)$, where $a_i\!\in\!\textup{hom}_{\mathcal{A}}(X_{i-1},X_i)$ and $\dagger_n \!\coloneqq\! (\sum_{i=1}^{n}|a_i|)-n$.\footnote{$\dagger_n$, along with other variables regulating the signs of our expressions, will be henceforth used as standard notation, without further comment.}\break For $d=1,2$, the first two such equations read $\mu_\mathcal{A}^1(\mu_\mathcal{A}^1(a_1)) = 0$ and\break $\mu_\mathcal{A}^1(\mu_\mathcal{A}^2(a_2,a_1)) = (-1)^{|a_1|}\mu_\mathcal{A}^2(\mu_\mathcal{A}^1(a_2),a_1) -\mu_\mathcal{A}^2(a_2,\mu_\mathcal{A}^1(a_1))$.
	\end{itemize}
	An \textbf{$A_\infty$-subcategory}\index{aaa@$A_\infty$-subcategory} $\mathcal{B}\subset\mathcal{A}$ is an $A_\infty$-category such that $\textup{obj}(\mathcal{B})\subset\textup{obj}(\mathcal{A})$, $\textup{hom}_\mathcal{B}(X_0,X_1)\subset\textup{hom}_\mathcal{A}(X_0,X_1)$ are graded subspaces for all $X_0,X_1\in\textup{obj}(\mathcal{B})$, and the composition maps $\mu_\mathcal{B}^d$ have target spaces within $\mathcal{B}$.  
\end{Def}

The labeling system in \eqref{compomap} and \eqref{Ainfcatasseqs} can be somewhat misleading: it is an attempt to reconcile various standard conventions. For consistency's sake, we will be faithful, here and in the future, to the notation adopted by Seidel in \cite{[Sei08]}.

\begin{Rem}\label{Ainfcatrem}				
	Consider the case where $\text{obj}(\mathcal{A})\coloneqq\{X\}$. Then the only possible graded vector space $A\coloneqq  \textup{hom}_\mathcal{A}(X,X)$ endowed with graded maps $m_n \coloneqq  \mu_{\mathcal{A}}^n$ is indeed an $A_\infty$-algebra, since the equations (\ref{Ainfalgasseqs}) precisely coincide with those in (\ref{Ainfcatasseqs}) (possibly up to a sign).
	
	Moreover, an $A_\infty$-category $\mathcal{A}$ with $\mu_{\mathcal{A}}^d = 0$ if $d>2$ is a differential graded category, and vice versa (for example, the dg category $\mathsf{Ch}$ can be considered an $A_\infty$-category). 
\end{Rem}

The $A_\infty$-associativity equations yield the exact same conclusions we made in Remark \ref{Ainfalgrem} for $A_\infty$-algebras: non-unital $A_\infty$-categories are not categories in general, since they may lack identity morphisms and each composition map $\mu_{\mathcal{A}}^2$ is not necessarily associative. 

However, $\mu_{\mathcal{A}}^1\mu_{\mathcal{A}}^1 = 0$, so that each morphism space $\textup{hom}_\mathcal{A}(X_0,X_1)$ is a cochain complex with differential $\mu_{\mathcal{A}}^1:\textup{hom}_\mathcal{A}(X_0,X_1)\rightarrow\textup{hom}_\mathcal{A}(X_0,X_1)[1]$. Hence, it makes sense to talk about cocycles (any $a\in\textup{hom}_\mathcal{A}(X_0,X_1)$ such that $\mu_\mathcal{A}^1(a)=0$), coboundaries (any $a=\mu_\mathcal{A}^1(b)$ for some $b\in\textup{hom}_\mathcal{A}(X_0,X_1)$) and cohomology:

\begin{Def}\label{cohomcategory}			
	Let $\mathcal{A}$ be a non-unital $A_\infty$-category. Its \textbf{cohomological category}\index{cohomological category} $H(\mathcal{A})$ is defined as follows: $\textup{obj}(H(\mathcal{A})) \coloneqq  \text{obj}(\mathcal{A})$, each pair $X_0, X_1\in \textup{obj}(\mathcal{A})$ yields a graded cohomology group 
	\begin{equation}
		\text{Hom}_{H(\mathcal{A})}(X_0,X_1) \coloneqq  H(\textup{hom}_\mathcal{A}(X_0,X_1)) = \bigoplus_{p\in\mathbb{Z}}H^p(\textup{hom}_\mathcal{A}(X_0,X_1))\,,
	\end{equation}
	and the ``usual'' associative (!) composition map is\footnote{Observe the choice of notation: whenever we speak of cohomological categories, we adopt the capital ``H'' in ``Hom'' (against ``hom'' for $A_\infty$-categories) and the composition symbol ``$\circ$''.} 
	\begin{align}\label{cohomcatcompo}
	\circ:\; &\textup{Hom}_{H(\mathcal{A})}(X_1,X_2)\times\textup{Hom}_{H(\mathcal{A})}(X_0,X_1)\rightarrow\textup{Hom}_{H(\mathcal{A})}(X_0,X_2),  \\
	&(\langle a_2\rangle,\langle a_1\rangle) \mapsto \langle a_2\rangle\circ\langle a_1\rangle \coloneqq  (-1)^{|a_1|}\langle\mu_{\mathcal{A}}^2(a_2,a_1)\rangle\;.\nonumber
	\end{align}
	This is a non-unital ordinary linear graded category.
	
	We designate by $H^0(\mathcal{A}) \subset H(\mathcal{A})$ the \textbf{zeroth cohomological subcategory}\index{zeroth cohomological subcategory}, having $\textup{obj}(H^0(\mathcal{A}))\! \coloneqq\!  \textup{obj}(H(\mathcal{A}))$ and $\text{Hom}_{H^0(\mathcal{A})}(X_0,X_1)\! \coloneqq\! H^0(\textup{hom}_\mathcal{A}(X_0,X_1))$.
\end{Def}

\begin{Def}												
	A \textbf{non-unital $A_\infty$-functor}\index{aaa@$A_\infty$-functor} $\mathcal{F}:\mathcal{A}\rightarrow\mathcal{B}$ between two non-unital $A_\infty$-categories $\mathcal{A}$ and $\mathcal{B}$ is a map $\mathcal{F}:\textup{obj}(\mathcal{A})\rightarrow\textup{obj}(\mathcal{B})$ together with a sequence of (multi)linear graded maps
	\begin{equation}\label{Ainffunc}
	\mkern-9mu\mathcal{F}^d:\textup{hom}_\mathcal{A}(X_{d-1},X_d)\otimes...\otimes\textup{hom}_\mathcal{A}(X_0,X_1)\rightarrow\textup{hom}_\mathcal{B}(\mathcal{F}(X_0),\mathcal{F}(X_d))[1-d]
	\end{equation}
	for every $d\geq 1$ and $X_0,...,X_d\in$ obj($\mathcal{A}$), fulfilling the equations
	\begin{align}\label{Ainffunceqs}
	&\sum_{\substack{1\leq r\leq d \\ s_1+...+s_r = d}}\mkern-24mu\mu_{\mathcal{B}}^r(\mathcal{F}^{s_r}(a_d,...,a_{d-s_r+1}),...,\mathcal{F}^{s_1}(a_{s_1},...,a_1)) =	\\ 
	&\;\sum_{\substack{1\leq m\leq d \\ 0\leq n\leq d-m}}\mkern-18mu(-1)^{\dagger_n}\mathcal{F}^{d-m+1}(a_d,...,a_{n+m+1},\mu_{\mathcal{A}}^{m}(a_{n+m},...,a_{n+1}),a_n,...,a_1)\;, \nonumber
	\end{align}
	in $\textup{hom}_\mathcal{B}(\mathcal{F}(X_0),\mathcal{F}(X_d))$, where $a_i\in\textup{hom}_{\mathcal{A}}(X_{i-1},X_i)$. 
	
	Given another non-unital $A_\infty$-functor $\mathcal{G}:\mathcal{B}\rightarrow\mathcal{C}$, the (strictly associative) composition $\mathcal{G}\circ\mathcal{F}:\mathcal{A}\rightarrow\mathcal{C}$ is the $A_\infty$-functor specified by
	\begin{equation}
	(\mathcal{G}\circ\mathcal{F})^d(a_d,...,a_1) \coloneqq  \mkern-18mu\sum_{\substack{1\leq r\leq d \\ s_1+...+s_r = d}}\mkern-18mu\mathcal{G}^r(\mathcal{F}^{s_r}(a_d,...,a_{d-s_r+1}),...,\mathcal{F}^{s_1}(a_{s_1},...,a_1)) 
	\end{equation}
	in $\textup{hom}_\mathcal{C}\big(\mathcal{G}(\mathcal{F}(X_0)),\mathcal{G}(\mathcal{F}(X_d))\big)$ (one can show that these do satisfy the defining equations \eqref{Ainffunceqs}). In particular, $(\mathcal{G}\circ\mathcal{F})^1(a_1) = \mathcal{G}^1(\mathcal{F}^1(a_1))$ and $(\mathcal{G}\circ\mathcal{F})^2(a_2,a_1)=\mathcal{G}^1(\mathcal{F}^2(a_2,a_1))+\mathcal{G}^2(\mathcal{F}^1(a_2),\mathcal{F}^1(a_1))$.
	
	If $\mathcal{A}=\mathcal{B}$ and $\mathcal{A}$ is unital, the identity $A_\infty$-functor $\mathcal{I}d_\mathcal{A}$ is given by $\mathcal{I}d_\mathcal{A}(X)\coloneqq X$,\; $\mathcal{I}d_\mathcal{A}^1 \coloneqq \textup{id}_{\textup{hom}_\mathcal{A}(X,X)}$ for $X\in\textup{obj}(\mathcal{A})$, and $\mathcal{I}d_\mathcal{A}^k \coloneqq 0$ for all $k>1$.
\end{Def}

As expected, an $A_\infty$-functor between one-object $A_\infty$-categories is just an $A_\infty$-algebra morphism. 

From \eqref{Ainffunceqs} we deduce that the respective first order maps commute: for any $a_1\in\textup{hom}_\mathcal{A}(X_0,X_1)$ holds $\mu_{\mathcal{B}}^1(\mathcal{F}^1(a_1))\!=\!\mathcal{F}^1(\mu_{\mathcal{A}}^1(a_1))$. The case $d=2$ produces instead $\mu_\mathcal{B}^1(\mathcal{F}^2(a_2,a_1))+\mu_\mathcal{B}^2(\mathcal{F}^1(a_2),\mathcal{F}^1(a_1)) = \mathcal{F}^1(\mu_\mathcal{A}^2(a_2,a_1)) + \mathcal{F}^2(a_2,\mathcal{A}^1(a_1))$ $+(-1)^{|a_1|-1}\mathcal{F}^2(\mu_\mathcal{A}^1(a_2),$ $a_1)$. These equations are compatible with what we obtained from Definition \ref{Ainfalgmorph} (again, up to some signs).

\begin{Def}\label{cohomfunctor}				
	A non-unital $A_\infty$-functor $\mathcal{F}:\mathcal{A}\rightarrow\mathcal{B}$ induces a \textbf{cohomological functor}\index{cohomological functor} $H(\mathcal{F}): H(\mathcal{A})\rightarrow H(\mathcal{B})$ which maps $X\in\textup{obj}(H(\mathcal{A}))\mapsto \mathcal{F}(X)\in\textup{obj}(H(\mathcal{B}))$ and 
	\begin{equation}\label{cohomAinffunc}
	\langle a_1\rangle\in\textup{Hom}_{H(\mathcal{A})}(X_0,X_1) \longmapsto \langle \mathcal{F}^1(a_1)\rangle\in\textup{Hom}_{H(\mathcal{B})}(\mathcal{F}(X_0),\mathcal{F}(X_1))\;,
	\end{equation}
	for each $X_0,X_1\in\textup{obj}(H(\mathcal{A}))$. This is an ordinary non-unital linear graded functor. We call $\mathcal{F}$:
	\begin{itemize}[leftmargin=0.5cm]
		\item a \textbf{quasi-isomorphism}\index{quasi-isomorphism} if $H(\mathcal{F})$ is an isomorphism;
		\item \textbf{cohomologically full and faithful}\index{aaa@$A_\infty$-functor!cohomologically full and faithful} if $H(\mathcal{F})$ is full and faithful;
		\item \textbf{cohomologically essentially surjective}\index{aaa@$A_\infty$-functor!cohomologically essentially surjective} if $H(\mathcal{F})$ is essentially surjective. 
	\end{itemize}
	Note that $H(\mathcal{F}): H(\mathcal{A})\rightarrow H(\mathcal{B})$ itself reduces to $H^0(\mathcal{F}): H^0(\mathcal{A})\rightarrow H^0(\mathcal{B})$ on the zeroth cohomological categories.
\end{Def}

As an immediate consequence of Lemma \ref{aboutfunctors}, quasi-isomorphisms are cohomologically full and faithful $A_\infty$-functors. 

There is a particular instance of $A_\infty$-functors which will become useful in simplifying the setting of several proofs. The definition exploits the following (non-trivial) fact: given an $A_\infty$-category $\mathcal{A}$ and a sequence of graded maps $(\Phi^0,\Phi^1,...)$ as in \eqref{Ainffunc} such that $\Phi^1:\textup{hom}_\mathcal{A}(X_0,X_1)\rightarrow\textup{hom}_\mathcal{A}(X_0,X_1)$ is a linear automorphism, recursive solving of the equations \eqref{Ainffunceqs} with respect to $\mu_\mathcal{B}$ produces a unique alternative $A_\infty$-structure $\tilde{\mathcal{A}}$ for $\mathcal{A}$, given by $\textup{obj}(\tilde{\mathcal{A}}) \coloneqq \textup{obj}(\mathcal{A})$, $\textup{hom}_{\tilde{\mathcal{A}}}(X_0,X_1) \coloneqq \textup{hom}_\mathcal{A}(X_0,X_1)$ and $\mu_{\tilde{\mathcal{A}}}\coloneqq\mu_\mathcal{B}$. Then the multilinear graded maps $\Phi^d$ fulfill \eqref{Ainffunceqs}, by construction.

\begin{Def}\label{formaldiffeo}					
	Let $\mathcal{A}$ be an $A_\infty$-category, with alternative $A_\infty$-structure $\tilde{\mathcal{A}}$ constructed from a sequence of graded maps $(\Phi^0,\Phi^1,...)$ as above. The associated non-unital $A_\infty$-functor $\Phi:\mathcal{A}\rightarrow\Phi_*\mathcal{A} \coloneqq  \tilde{\mathcal{A}}$ is called \textbf{formal diffeomorphism}\index{formal diffeomorphism}. By equation \eqref{cohomAinffunc} and choice of $\Phi^1$, it is in particular a quasi-isomorphism (hence cohomologically full and faithful).   
\end{Def}

\subsection{Pre-natural transformations and compositions}\label{ch2.3}

In standard category theory we can look at maps between functors, the natural transformations. There is of course an adaptation also for $A_\infty$-functors.

\begin{Def}									
	Given two non-unital $A_\infty$-categories $\mathcal{A}$ and $\mathcal{B}$, we define $\mathcal{Q}= nu\text{-}\!fun(\mathcal{A},\mathcal{B})$ to be the non-unital $A_\infty$-category specified by:
	\begin{itemize}[leftmargin=0.5cm]
		\renewcommand{\labelitemi}{\textendash}
		\item a set obj($\mathcal{Q}$) consisting of non-unital $A_\infty$-functors $\mathcal{F}:\mathcal{A}\rightarrow\mathcal{B}$;
		\item for any pair $\mathcal{F}_0,\mathcal{F}_1\in$ obj($\mathcal{Q}$), morphisms whose homogeneous parts $T\in\textup{hom}_\mathcal{Q}^g(\mathcal{F}_0,\mathcal{F}_1)$ are sequences $(T^0,T^1, ...)$, where each $T^d$ for $d\geq 1$ is a collection of (multi)linear graded maps
		\begin{equation}\label{Ainffunccatmorph}
		\mkern-9mu\textup{hom}_\mathcal{A}(X_{d-1},X_d)\otimes...\otimes\textup{hom}_\mathcal{A}(X_0,X_1)\!\rightarrow\!\textup{hom}_\mathcal{B}(\mathcal{F}_0(X_0),\mathcal{F}_1(X_d))[g-d],
		\end{equation}
		where $X_0,...,X_d\in\textup{obj}(\mathcal{A})$. Instead, $T^0$ must be regarded as a collection $\{T_{X}^0\equiv T^0(X)\in\textup{hom}_\mathcal{B}^g(\mathcal{F}_0(X),\mathcal{F}_1(X))\mid X\in\textup{obj}(\mathcal{A})\}$ (in accordance with the standard definition of natural transformation). $T$ is called \textbf{pre-natural transformation}\index{pre-natural transformation} of degree $|T| = g$ from $\mathcal{F}_0$ to $\mathcal{F}_1$;
		\item composition maps $\mu_{\mathcal{Q}}^d$ like those of \eqref{compomap}, so that for $T_k\in\textup{hom}_\mathcal{Q}(\mathcal{F}_{k-1},\mathcal{F}_k)$ the output $S\coloneqq \mu_{\mathcal{Q}}^d(T_d,...,T_1)$ $\in\textup{hom}_\mathcal{Q}(\mathcal{F}_0,\mathcal{F}_d)[2-d]$ is again a sequence $(S^0,S^1,...)$ as in \eqref{Ainffunccatmorph}. The first two cases $d=1,2$ are:\footnote{Note: by definition, in \eqref{shit1} only $s_i$ can be zero, while in \eqref{shit2} both $s_i$ and $s_j$ can.}
		\begin{align}\label{shit1}
		\mkern-12mu\mu_{\mathcal{Q}}^1&(T_1)^h(a_h,...,a_1) \coloneqq\mkern-18mu\sum_{1\leq r\leq h+1}\mkern-24mu\sum_{\substack{1\leq i\leq r\\\quad  s_1+...+s_r=h}}\mkern-30mu(-1)^{\ddagger_i^1}\mu_{\mathcal{B}}^r\Big(\mathcal{F}_{1}^{s_r}(a_h,...,a_{h-s_r+1}),...,\nonumber
		\\ &\mathcal{F}_1^{s_{i+1}}(a_{h-...-s_{i+2}},...,a_{s_1+...+s_i+1}),T_{1}^{s_i}(a_{s_1+...+s_i},...,a_{s_1+...+s_{i-1}+1}),\nonumber\\ &\mathcal{F}_0^{s_{i-1}}(a_{s_1+...+s_{i-1}},...,a_{s_1+...+s_{i-2}+1}),...,\mathcal{F}_0^{s_1}(a_{s_1},...,a_1)\Big)\nonumber \\
		&-\mkern-30mu\sum_{\substack{1\leq m\leq h \\ \quad 0\leq n\leq h-m}}\mkern-24mu(-1)^{\dagger_{n}+|T_1|-1}T_{1}^{h-m+1}(a_h,...,\mu_{\mathcal{A}}^{m}(a_{n+m},...,a_{n+1}),...,a_1)  	
		\end{align}
		in $\textup{hom}_\mathcal{B}(\mathcal{F}_0(X_0),\mathcal{F}_1(X_h))$, and  
		\begin{align}\label{shit2}
		\mu_{\mathcal{Q}}^2&(T_2,T_1)^h(a_h,...,a_1) \coloneqq\mkern-18mu\sum_{1\leq r\leq h+2}\mkern-12mu\sum_{\substack{1\leq i<j\leq r\\\quad  s_1+...+s_r=h}}\mkern-30mu(-1)^{\ddagger_i^1+\ddagger_j^2}\mu_{\mathcal{B}}^r\Big(\mathcal{F}_{2}^{s_r}(a_h,...,a_{h-s_r+1}),\nonumber \\
		&...,\mathcal{F}_2^{s_{j+1}}(...),T_{2}^{s_j}(a_{s_1+...+s_j},...,a_{s_1+...+s_{j-1}+1}),\mathcal{F}_{1}^{s_{j-1}}(...),\nonumber\\
		&...,\mathcal{F}_{1}^{s_{i+1}}(...),T_{1}^{s_i}(a_{s_1+...+s_i},...,a_{s_1+...+s_{i-1}+1}),\mathcal{F}_{0}^{s_{i-1}}(...),\nonumber\\ &...,\mathcal{F}_{0}^{s_1}(a_{s_1},...,a_1)\Big)
		\end{align}
		in $\textup{hom}_\mathcal{B}(\mathcal{F}_0(X_0),\mathcal{F}_2(X_h))$, where $h\geq 1$, $a_k\in\textup{hom}_\mathcal{A}(X_{k-1},X_k)$ and $\ddagger_i^j \coloneqq (|T_j|-1)\cdot\sum_{k=1}^{s_1+...+s_{i-1}}(|a_k|-1)$.
	\end{itemize}
\end{Def}

Higher order formulae for $\mu_{\mathcal{Q}}^d(T_d,...,T_1)$ can be found for example in Fukaya's work (\cite[Definition 10.15]{[Fuk99]}). Actual computations are best tackled with help of some programming. Here, we limit ourselves to give the first two terms of $\eqref{shit1}$: 
\begin{equation}\label{cutie0}
\mu_{\mathcal{Q}}^1(T)^0 = \{\mu_{\mathcal{Q}}^1(T)_{X}^0 = \mu_{\mathcal{B}}^1(T_{X}^0)\in\textup{hom}_\mathcal{B}(\mathcal{F}_0(X),\mathcal{F}_1(X))\mid  X\in\textup{obj}(\mathcal{A})\}\,,
\end{equation}
\vspace*{-0.75cm}
\begin{align}\label{cutie1}
\mkern-46mu\mu_{\mathcal{Q}}^1(T)^1(a_1) =\; & \mu_{\mathcal{B}}^1(T^1(a_1))+(-1)^{|T|-1}T^1(\mu_{\mathcal{A}}^1(a_1))+\mu_{\mathcal{B}}^2(\mathcal{F}_{1}^1(a_1),T_{X_0}^0)\nonumber \\
&\!+(-1)^{(|T|-1)(|a_1|-1)}\mu_{\mathcal{B}}^2(T_{X_1}^0,\mathcal{F}_{0}^1(a_1))\;.
\end{align}

\begin{Def}										
	A pre-natural transformation which is closed with respect to the boundary operator of $\textup{hom}_\mathcal{Q}(\mathcal{F}_0,\mathcal{F}_1)$, that is, any $T\in\textup{hom}_\mathcal{Q}(\mathcal{F}_0,\mathcal{F}_1)$ with $\mu_{\mathcal{Q}}^1(T)^d = 0$ for all $d\geq 0$, is called an ($A_\infty$-)\textbf{natural transformation}\index{natural transformation}.
\end{Def}  

Consider a natural transformation $T\in\textup{hom}_\mathcal{Q}(\mathcal{F}_0,\mathcal{F}_1)$. By \eqref{cutie1} and \eqref{cohomcatcompo}, for all $a_1\in\textup{hom}_\mathcal{A}(X_0,X_1)$ the classes $\langle T_{X}^0\rangle$ satisfy the naturality condition $\langle T_{X_1}^0\rangle\circ\langle \mathcal{F}_{0}^1(a_1)\rangle $ $= (-1)^{|a_1|\cdot|T|}\langle \mathcal{F}_{1}^1(a_1)\rangle\circ\langle T_{X_0}^0\rangle \in \textup{Hom}_{H(\mathcal{B})}(\mathcal{F}_0(X_0),\mathcal{F}_1(X_1))$. Therefore, $H(T)\equiv H(T^0)\coloneqq \{\langle T_{X}^0\rangle\in \textup{Hom}_{H(\mathcal{B})}(\mathcal{F}_0(X),\mathcal{F}_1(X))\mid X\in\textup{obj}(\mathcal{A})\}$ is a ``usual'' natural transformation, and we may define:

\begin{Def}\label{Nu-Fun}										
	Let $\mathcal{Q}= nu\text{-}\!fun(\mathcal{A},\mathcal{B})$. Then $\mathcal{H(Q)}=Nu\text{-}\!Fun(H(\mathcal{A}),H(\mathcal{B}))$ denotes the (non-unital) category whose objects are non-unital cohomological functors $H(\mathcal{F}):H(\mathcal{A})\rightarrow H(\mathcal{B})$ and whose morphisms are their natural trans-\break formations $H(T)\!\in\!\textup{Hom}_{\mathcal{H(Q)}}(H(\mathcal{F}),$ $H(\mathcal{G}))$, defined as in the last paragraph. We obtain the non-unital functor:
	\begin{equation}\label{Hfunctor}
	\mathcal{H}: H(\mathcal{Q})\longrightarrow \mathcal{H(Q)},\; 
	\begin{cases}
	\mathcal{F}\in\textup{obj}(\mathcal{Q})\longmapsto H(\mathcal{F}) \\
	\langle T\rangle\in\textup{Hom}_{H(\mathcal{Q})}(\mathcal{F},\mathcal{G})\longmapsto H(T)
	\end{cases}
	\end{equation} 
\end{Def}

Two useful operations are left and right composition with some fixed $A_\infty$-functor.

\begin{Def}									
	Let $\mathcal{G}:\mathcal{A}\rightarrow\mathcal{B}$ be a non-unital $A_\infty$-functor and $\mathcal{C}$ any non-unital $A_\infty$-category. We define the non-unital $A_\infty$-functors of:
	\begin{itemize}[leftmargin=0.5cm]
		\item \textbf{left composition}\index{aaa@$A_\infty$-functor!left composition} $\mathcal{L}_\mathcal{G}: nu\text{-}\!fun(\mathcal{C},\mathcal{A})\rightarrow nu\text{-}\!fun(\mathcal{C},\mathcal{B})$, given by $\mathcal{F}\mapsto\mathcal{G}\circ\mathcal{F}$ and $\mathcal{L}_\mathcal{G}^1:\textup{hom}_{nu\text{-}\!fun(\mathcal{C},\mathcal{A})}(\mathcal{F}_0,\mathcal{F}_1)\rightarrow\textup{hom}_{nu\text{-}\!fun(\mathcal{C},\mathcal{B})}(\mathcal{G}\circ\mathcal{F}_0,\mathcal{G}\circ\mathcal{F}_1)$ specified by
		\begin{align}\label{lshit1}
		\mkern-12mu\mathcal{L}_\mathcal{G}^1&(T_1)^d(a_d,...,a_1) \coloneqq\mkern-18mu\sum_{1\leq r\leq d+1}\mkern-24mu\sum_{\substack{1\leq i\leq r\\\quad s_1+...+s_r=d}}\mkern-30mu(-1)^{\ddagger_i^1}\mathcal{G}^r\Big(\mathcal{F}_{1}^{s_r}(a_d,...,a_{d-s_r+1}),...,\nonumber \\ &\mathcal{F}_1^{s_{i+1}}(a_{d-...-s_{i+2}},...,a_{s_1+...+s_i+1}),T_{1}^{s_i}(a_{s_1+...+s_i},...,a_{s_1+...+s_{i-1}+1}),\nonumber\\ &\mathcal{F}_0^{s_{i-1}}(a_{s_1+...+s_{i-1}},...,a_{s_1+...+s_{i-2}+1}),...,\mathcal{F}_0^{s_1}(a_{s_1},...,a_1)\Big)
		\end{align}
		in $\textup{hom}_\mathcal{B}\big(\mathcal{G}(\mathcal{F}_0(X_0)),\mathcal{G}(\mathcal{F}_1(X_d))\big)$, and with complicated higher order terms $\mathcal{L}_\mathcal{G}^d$;
		\item \textbf{right composition}\index{aaa@$A_\infty$-functor!right composition} $\mathcal{R}_\mathcal{G}: nu\text{-}\!fun(\mathcal{B},\mathcal{C})\rightarrow nu\text{-}\!fun(\mathcal{A},\mathcal{C})$, given by $\mathcal{F}\mapsto\mathcal{F}\circ\mathcal{G}$ and $\mathcal{R}_\mathcal{G}^1:\textup{hom}_{nu\text{-}\!fun(\mathcal{B},\mathcal{C})}(\mathcal{F}_0,\mathcal{F}_1)\rightarrow\textup{hom}_{nu\text{-}\!fun(\mathcal{A},\mathcal{C})}(\mathcal{F}_0\circ\mathcal{G},\mathcal{F}_1\circ\mathcal{G})$ specified by
		\begin{equation}\label{rshit1}
		\mkern-14mu\mathcal{R}_\mathcal{G}^1(T_1)^d(a_d,...,a_1) \coloneqq\mkern-24mu\sum_{\substack{0\leq r\leq d\\ s_1+...+s_r=d}}\mkern-24mu T^r\Big(\mathcal{G}^{s_r}(a_d,...,a_{d-s_r+1}),...,
		\mathcal{G}^{s_1}(a_{s_1},...,a_1)\Big)
		\end{equation}
		in $\textup{hom}_\mathcal{C}(\mathcal{F}_0(\mathcal{G}(X_0)),\mathcal{F}_1(\mathcal{G}(X_d))$, and $\mathcal{R}_\mathcal{G}^d \coloneqq 0$ for $d>1$. 
	\end{itemize}
\end{Def}

Again, it is worth calculating the lower order terms: $\mathcal{L}_\mathcal{G}^1(T)_{X}^0 = \mathcal{G}^1(T_{X}^0)$ and $\mathcal{L}_\mathcal{G}^1(T)^1(a_1) \!=\!  \mathcal{G}^1(T^1(a_1))+\mathcal{G}^2(\mathcal{F}_{1}^1(a_1),T_{X_0}^0)+(-1)^{(|T|-1)(|a_1|-1)}\mathcal{G}^2(T_{X_1}^0,\mathcal{F}_{0}^1(a_1))$,
respectively $\mathcal{R}_\mathcal{G}^1(T)_{X}^0 = T_{X}^0$ and $\mathcal{R}_\mathcal{G}^1(T)^1(a_1) = T^1(\mathcal{G}^1(a_1))$.

For compatible $A_\infty$-functors, left and right compositions themselves can be composed: $\mathcal{R}_{\mathcal{G}_1}\circ\mathcal{R}_{\mathcal{G}_2}=\mathcal{R}_{\mathcal{G}_1\circ\mathcal{G}_2}$, $\mathcal{L}_{\mathcal{G}_1}\circ\mathcal{L}_{\mathcal{G}_2}=\mathcal{L}_{\mathcal{G}_1\circ\mathcal{G}_2}$ and $\mathcal{L}_{\mathcal{G}_1}\circ\mathcal{R}_{\mathcal{G}_2}=\mathcal{R}_{\mathcal{G}_2}\circ\mathcal{L}_{\mathcal{G}_1}$.

\subsection{Interlude: Hochschild cohomology}\label{ch2.4}

Hochschild cohomology is a useful tool for proving many key results by comparison arguments. It is prominently used in deformation theory. However, a full treatment would take us too far afield and hence we limit ourselves to give the relevant definitions and main applications. A more rigorous take can be found in \cite[subsection III.7.3]{[GM03]} and \cite[section (1f)]{[Sei08]}.

\begin{Def}\label{Hochcohom}				
	Let $\mathcal{A}$ be an $A_\infty$-category whose cohomological category $H(\mathcal{A})$ is unital. We define the \textbf{Hochschild cohomology}\index{Hochschild cohomology} group of $\mathcal{A}$ to be the graded morphism space $HH(\mathcal{A}) \coloneqq H(\textup{hom}_{nu\text{-}\!fun(\mathcal{A},\mathcal{A})}(\mathcal{I}d_\mathcal{A},\mathcal{I}d_\mathcal{A}))$. It can be computed from the cochain complex $(CC(\mathcal{A}),d)$ whose cochains $h = (h^0,h^1,...)$ $\in CC^r(\mathcal{A})$ are sequences of (multi)linear graded maps
	\[
	h^d:\textup{hom}_\mathcal{A}(X_d,X_{d+1})\otimes...\otimes\textup{hom}_\mathcal{A}(X_1,X_2)\rightarrow\textup{hom}_\mathcal{A}(X_1,X_{d+1})[r-d]
	\]
	for each tuple $(X_1,...,X_{d+1})\subset\textup{obj}(\mathcal{A})$, with differential acting on them as:
	\begin{align*}
	(dh)^d(a_d,...,a_1) \coloneqq &\mkern-18mu\sum_{\quad 0\leq i+j<d+1}\mkern-18mu(-1)^{(r+1)\dagger_i}\mu_\mathcal{A}^{d+1-j}(a_d,...,h^j(a_{i+j},...,a_{i+1}),...,a_1)\,+ \\
	&\mkern-18mu\sum_{\quad 0 < i+j \leq d+1}\mkern-18mu(-1)^{r+1+\dagger_i}h^{d+1-j}(a_d,...,\mu_\mathcal{A}^j(a_{i+j},...,a_{i+1}),...,a_1)\;.
	\end{align*} 
\end{Def}

In fact, one can be more explicit about the definition of Hochschild complex in the bigraded case. Consider now two arbitrary non-unital $A_\infty$-categories $\mathcal{A}$, $\mathcal{B}$, and $\mathcal{Q}=nu\text{-}\!fun(\mathcal{A},\mathcal{B})$. Given two $\mathcal{G}_0$, $\mathcal{G}_1\in\textup{obj}(\mathcal{Q})$, the vector space $\textup{hom}_\mathcal{Q}(\mathcal{G}_0,\mathcal{G}_1)$ can be decreasingly filtrated by allowing for the $r$-th term only those pre-natural transformations $T$ with $T^0 = ...= T^{r-1}=0$. The spectral sequence associated to this filtration starts with
\begin{align}\label{spectral}
E_{1}^{rs} = \mkern-12mu\prod_{X_0,...,X_r}\mkern-12mu \textup{Hom}_\mathbb{K}^s\big(&\textup{Hom}_{H(\mathcal{A})}(X_{r-1},X_r)\otimes...\otimes\textup{Hom}_{H(\mathcal{A})}(X_0,X_1),\nonumber \\[-2.5ex] &\textup{Hom}_{H(\mathcal{B})}(\mathcal{G}_0(X_0),\mathcal{G}_1(X_r))\big).
\end{align}
Equation \eqref{shit1} for the boundary operators in $\mathcal{Q}$ then furnishes the following differential $d\equiv d^{rs}$ on $E_{1}^{rs}$:
\begin{align*}
(dT)(a_{r+1},...,a_1) =\; &(-1)^{(r+s)(|a_1|-1)+1}T(a_{r+1},...,a_2)H(\mathcal{G}_0)(a_1)\, + \\
&(-1)^{r+s+\dagger_r}H(\mathcal{G}_1)(a_{r+1})T(a_r,...,a_1)\, + 	\\
&\mkern-28mu\sum_{\quad 0\leq n\leq r-1}\mkern-18mu(-1)^{r+s-1+\dagger_{n+1}}T(a_{r+1},...,a_{n+2}a_{n+1},...,a_1)
\end{align*}
(for $T$ and the $a_i$'s cohomological maps). Then we obtain a bigraded complex $CC^{r+s}(H(\mathcal{A}),H(\mathcal{B}))^s\coloneqq  E_{1}^{rs}$, inducing the (bigraded) Hochschild cohomology $HH^{r+s}(H(\mathcal{A}),H(\mathcal{B}))^s\coloneqq  E_{2}^{rs}$ (in particular, we have $HH^s(H(\mathcal{A}),H(\mathcal{B}))^s=\textup{Hom}_{\mathcal{H}(\mathcal{Q})}^s(H(\mathcal{G}_0),H(\mathcal{G}_1))$).

For our purposes, here are the principal results whose proofs exploit Hochschild cohomology:

\begin{Lem}\label{1.6}										
	Let $\mathcal{Q}=nu\text{-}\!fun(\mathcal{A},\mathcal{B})$, $\mathcal{G}_1, \mathcal{G}_2, \mathcal{G}_3 \!\in\!\textup{obj}(\mathcal{Q})$ and $T\!\in\!\textup{hom}_\mathcal{Q}(\mathcal{G}_1,\mathcal{G}_2)$ be a natural transformation. Suppose that:
	\begin{itemize}[leftmargin=0.5cm]
		\item $\textup{Hom}_{H(\mathcal{B})}(\mathcal{G}_2(X),\mathcal{G}_3(X))\rightarrow\textup{Hom}_{H(\mathcal{B})}(\mathcal{G}_1(X),\mathcal{G}_3(X)),\;\langle c\rangle\mapsto\langle c\rangle\circ\langle T_{X}^0\rangle$ is an isomorphism for all $X\in\textup{obj}(\mathcal{A})$. $\!\!$Then $\textup{Hom}_{H(\mathcal{Q})}(\mathcal{G}_2,\mathcal{G}_3)\!\rightarrow\!\textup{Hom}_{H(\mathcal{Q})}(\mathcal{G}_1,\mathcal{G}_3),$ $\langle S\rangle\mapsto\langle S\rangle\circ\langle T\rangle$ is an isomorphism too.
		\item $\textup{Hom}_{H(\mathcal{B})}(\mathcal{G}_3(X),\mathcal{G}_1(X))\rightarrow\textup{Hom}_{H(\mathcal{B})}(\mathcal{G}_3(X),\mathcal{G}_2(X)),\;\langle c\rangle\mapsto\langle T_{X}^0\rangle\circ\langle c\rangle$ is an isomorphism for all $X\in\textup{obj}(\mathcal{A})$. $\!\!$Then $\textup{Hom}_{H(\mathcal{Q})}(\mathcal{G}_3,\mathcal{G}_1)\!\rightarrow\!\textup{Hom}_{H(\mathcal{Q})}(\mathcal{G}_3,\mathcal{G}_2),$ $\langle S\rangle\mapsto\langle T\rangle\circ\langle S\rangle$ is an isomorphism too. 
	\end{itemize} 
\end{Lem}

\begin{Lem}\label{1.7}										
	Let $\mathcal{C}$ be any non-unital $A_\infty$-category. Suppose $\mathcal{G}\in nu\text{-}\!fun(\mathcal{A},\mathcal{B})$ is:
	\begin{itemize}[leftmargin=0.5cm]
		\item cohomologically full and faithful, then so is $\mathcal{L}_\mathcal{G}\!:\! nu\text{-}\!fun(\mathcal{C},\mathcal{A})\!\rightarrow\! nu\text{-}\!fun(\mathcal{C},\mathcal{B})$;
		\item a quasi-isomorphism, then so is $\mathcal{R}_\mathcal{G}: nu\text{-}\!fun(\mathcal{B},\mathcal{C})\rightarrow nu\text{-}\!fun(\mathcal{A},\mathcal{C})$.   
	\end{itemize} 
\end{Lem}

\subsection{Homological perturbation theory}\label{ch2.5}

As usual, let $\mathcal{Q}=nu\text{-}\!fun(\mathcal{A},\mathcal{B})$. Given $\mathcal{F}_0, \mathcal{F}_1\in\textup{obj}(\mathcal{Q})$, observe that the pre-natural transformation $D\coloneqq \mathcal{F}_0 - \mathcal{F}_1\in\textup{hom}_\mathcal{Q}^1(\mathcal{F}_0,\mathcal{F}_1)$ defined by $D^0\coloneqq 0$ and $D^d\coloneqq\mathcal{F}_{0}^d-\mathcal{F}_{1}^d$ for $d\geq 1$ is actually a natural transformation (indeed, $\mu_{\mathcal{Q}}^1(D)=0$ can be checked with help of \eqref{Ainffunceqs} and \eqref{shit1}).

\begin{Def}											
	Two non-unital $A_\infty$-functors $\mathcal{F}_0, \mathcal{F}_1\in\textup{obj}(\mathcal{Q})$ with $\mathcal{F}_0(X) = \mathcal{F}_1(X)$ for all $X\in\textup{obj}(\mathcal{A})$ are \textbf{homotopic}\index{aaa@$A_\infty$-functor!homotopic}, denoted $\mathcal{F}_0 \simeq\mathcal{F}_1$, if there exists some $T\in\textup{hom}_\mathcal{Q}^0(\mathcal{F}_0,\mathcal{F}_1)$ with $T^0 = 0$ such that $\mathcal{F}_0-\mathcal{F}_1 = \mu_{\mathcal{Q}}^1(T)$.
	
	Being homotopic is a well-defined equivalence relation. 
\end{Def}

As one would suspect, homotopic $A_\infty$-functors induce the same map in cohomology: $(H(\mathcal{F}_0)-H(\mathcal{F}_1))\langle c\rangle = H(\mu_{\mathcal{Q}}^1(T))\langle c\rangle = \langle\mu_{\mathcal{Q}}^1(T)^1(c)\rangle = 0$, while identity on objects is readily checked.
Moreover, if $\mathcal{F}_0 \simeq\mathcal{F}_1$ via $T$, so $\mathcal{G}\circ\mathcal{F}_0 \simeq\mathcal{G}\circ\mathcal{F}_1$ via $\mathcal{L}_\mathcal{G}^1(T)$ respectively $\mathcal{F}_0\circ\mathcal{G} \simeq\mathcal{F}_1\circ\mathcal{G}$ via $\mathcal{R}_\mathcal{G}^1(T)$.

Here is our first important result.

\begin{Pro}[\textbf{Homological Perturbation Lemma}]\index{Homological Perturbation Lemma}\label{perturbation}		
	Let $\mathcal{B}$ be a non-unital $A_\infty$-category. Suppose for each pair $X_0,X_1\in\textup{obj}(\mathcal{B})$ there is a diagram
	\begin{equation}\label{perturbdiagram}
	\begin{tikzcd}[cells={nodes={}}]
		{\textup{hom}_\mathcal{A}(X_0,X_1)}\arrow[r, yshift=1ex, "{\mathcal{F}^1[0]}"]
		& {\textup{hom}_\mathcal{B}(X_0,X_1)}\arrow[l, yshift=-1ex, "{\mathcal{G}^1[0]}"]\arrow[loop right, distance=3em, start anchor={[yshift=1ex]east}, end anchor={[yshift=-1ex]east}]{}{T^1[-1]}
	\end{tikzcd}
	\end{equation}
	where $(\textup{hom}_\mathcal{A}(X_0,X_1),\mu_\mathcal{A}^1)$ is a cochain complex \textup(as in Definition \ref{complex}\textup), $\mathcal{F}^1$ and $\mathcal{G}^1$ are degree $0$ chain maps, and $T^1$ is a degree $- 1$ graded map fulfilling $\mu_{\mathcal{B}}^1T^1+T^1\mu_{\mathcal{B}}^1 = \mathcal{F}^1\mathcal{G}^1-\textup{id}$ as endomorphisms of $\textup{hom}_\mathcal{B}(X_0,X_1)$. Then there exist:
	\begin{itemize}[leftmargin=0.5cm]
		\renewcommand{\labelitemi}{\textendash}
		\item a non-unital $A_\infty$-category $\mathcal{A}$ with $\textup{obj}(\mathcal{A})=\textup{obj}(\mathcal{B})$, $\textup{hom}_{\mathcal{A}}(X_0,X_1)$ as morphism spaces and associated first order composition maps $\mu_{\mathcal{A}}^1$,
		\item non-unital $A_\infty$-functors $\mathcal{F}:\mathcal{A}\rightarrow\mathcal{B}$ and $\mathcal{G}:\mathcal{B}\rightarrow\mathcal{A}$ which are identities on objects and have first order terms $\mathcal{F}^1$ respectively $\mathcal{G}^1$,
		\item a homotopy $T\in\textup{hom}_{nu\text{-}\!fun(\mathcal{B},\mathcal{B})}^0(\mathcal{F}\circ\mathcal{G},\mathcal{I}d_\mathcal{B})$ with first order term $T^1$.
	\end{itemize}
\end{Pro}

One can produce the desired formulae recursively. Indeed, 
\[
\mathcal{F}^d(a_d,...,a_1) \coloneqq\mkern-24mu\sum_{\substack{2\leq r\leq d\\ s_1+...+s_r=d}}\mkern-24mu T^1\Big(\mu_{\mathcal{B}}^r(\mathcal{F}^{s_r}(a_d,...,a_{d-s_r+1}),...,
\mathcal{F}^{s_1}(a_{s_1},...,a_1))\Big)
\]
\[
\mu_{\mathcal{A}}^d(a_d,...,a_1)\coloneqq\mkern-24mu\sum_{\substack{2\leq r\leq d\\ s_1+...+s_r=d}}\mkern-24mu \mathcal{G}^1\Big(\mu_{\mathcal{B}}^r(\mathcal{F}^{s_r}(a_d,...,a_{d-s_r+1}),...,
\mathcal{F}^{s_1}(a_{s_1},...,a_1))\Big)
\]
are seen to fulfill the defining equations \eqref{Ainffunceqs} respectively \eqref{Ainfcatasseqs} (ultimately because $\mu_{\mathcal{B}}^\bullet$ is an $A_\infty$-structure on $\mathcal{B}$). However, the proof leading to them is too technical, and we would rather move on to the following application.

\begin{Lem}\label{homotopyinverse}		
	Any quasi-isomorphism $\mathcal{K}:\mathcal{B}\rightarrow\tilde{\mathcal{B}}$ of non-unital $A_\infty$-categories has an inverse up to homotopy. That is, there exists some $\mathcal{J}:\tilde{\mathcal{B}}\rightarrow\mathcal{B}$ such that $\mathcal{K}\circ\mathcal{J}\simeq\mathcal{I}d_{\tilde{\mathcal{B}}}$ and $\mathcal{J}\circ\mathcal{K}\simeq\mathcal{I}d_{\mathcal{B}}$.
\end{Lem}

\begin{proof}
	We start by proving the following claim: for a suitable choice of complexes, we can construct an $A_\infty$-category $\mathcal{A}$ as in Proposition \ref{perturbation} such that $\mathcal{F}$ is a quasi-isomorphism and $\mathcal{G}\circ\mathcal{F}\simeq\mathcal{I}d_\mathcal{A}$ too. 
	
	Indeed, given two $X_0, X_1\in\textup{obj}(\mathcal{B})$, we can decompose the associated complex $(\textup{hom}_\mathcal{B}(X_0,X_1), \mu_{\mathcal{B}}^1)$ as $(\textup{hom}_\mathcal{A}(X_0,X_1), 0)\oplus(\textup{hom}_\mathcal{B'}(X_0,X_1), \mu_{\mathcal{B'}}^1)$ so that $\mathcal{F}^1$ and $\mathcal{G}^1$ are the inclusion respectively projection of the first summand, while the second summand is acyclic. But acyclicity is achieved if and only if there exists a contracting homotopy $(T^1)_n: \textup{hom}_\mathcal{B'}^n(X_0,X_1)\rightarrow\textup{hom}_\mathcal{B'}^{n-1}(X_0,X_1)$. This provides the desired $T^1$. 
	
	The Perturbation Lemma applies to construct an $A_\infty$-category $\mathcal{A}$ with vanishing boundary operator, as well as compatible $A_\infty$-functors $\mathcal{F}$ and $\mathcal{B}$. Since $\mathcal{F}^1$ is the inclusion and $\mathcal{F}\circ\mathcal{G}\simeq\mathcal{I}d_{\mathcal{B}}$, we see that $H(\mathcal{G})$ is bijective on objects and morphisms, making $\mathcal{G}$ (thus also $\mathcal{F}$) a quasi-isomorphism. By Lemma \ref{1.7}, $\mathcal{R}_\mathcal{G}$ is a quasi-isomorphism too, so that any
	\[
	H(\mathcal{R}_\mathcal{G})_{\mathcal{F}_1,\mathcal{F}_2}:\textup{Hom}_{H(nu\text{-}\!fun(\mathcal{A},\mathcal{A}))}(\mathcal{F}_1,\mathcal{F}_2) \longmapsto\textup{Hom}_{H(nu\text{-}\!fun(\mathcal{B},\mathcal{A}))}(\mathcal{F}_1\circ\mathcal{G},\mathcal{F}_2\circ\mathcal{G})
	\]
	is a bijection, also when restricting to (classes of) pre-natural transformations with vanishing zeroth order term. Choose $\mathcal{F}_1 \coloneqq  \mathcal{I}d_\mathcal{A}$ and $\mathcal{F}_2\coloneqq \mathcal{G}\circ\mathcal{F}$. As $\mathcal{I}d_{\mathcal{B}}\simeq\mathcal{F}\circ\mathcal{G}$, also $\mathcal{G}\simeq\mathcal{G}\circ\mathcal{F}\circ\mathcal{G}$, say via some pre-natural transformation $T\in\textup{hom}_{{nu\text{-}\!fun}(\mathcal{B},\mathcal{A})}^0(\mathcal{G},\mathcal{G}\circ\mathcal{F}\circ\mathcal{G})$. Then $\mu_{\mathcal{Q}}^1(T)=\mathcal{G}-\mathcal{G}\circ\mathcal{F}\circ\mathcal{G} = \mathcal{R}_\mathcal{G}^1(\mathcal{I}d_\mathcal{A}-\mathcal{G}\circ\mathcal{F})$, hence $\langle \mu_{\mathcal{Q}}^1(T)\rangle = H(\mathcal{R}_\mathcal{G})_{\mathcal{I}d_\mathcal{A},\mathcal{G}\circ\mathcal{F}}\langle\mathcal{I}d_\mathcal{A}-\mathcal{G}\circ\mathcal{F}\rangle$, and bijectivity of $H(\mathcal{R}_\mathcal{G})_{\mathcal{I}d_\mathcal{A},\mathcal{G}\circ\mathcal{F}}$ provides a class mapping to $\langle\mathcal{I}d_\mathcal{A}-\mathcal{G}\circ\mathcal{F}\rangle$. Any representative thereof gives by construction a homotopy $\mathcal{I}d_\mathcal{A}\simeq\mathcal{G}\circ\mathcal{F}$. This proves the claim.  
	
	We use this result to construct the following diagram:
	\[
	\begin{tikzcd}
	\mathcal{A}\arrow[r, yshift=0.8ex, "\mathcal{F}"]\arrow[d, dashed, "\mathcal{E}"'] & \mathcal{B}\arrow[l, yshift=-0.3ex, "\mathcal{G}"]\arrow[d, "\mathcal{K}"] \\
	\tilde{\mathcal{A}}\arrow[r, yshift=0.8ex, "\tilde{\mathcal{F}}"] & \tilde{\mathcal{B}}\arrow[l, yshift=-0.3ex, "\tilde{\mathcal{G}}"]
	\end{tikzcd}\quad,
	\]
	where all arrows, including $\mathcal{E}\coloneqq \tilde{\mathcal{G}}\circ\mathcal{K}\circ\mathcal{F}$, are quasi-isomorphisms. But then $\mu_{A}^1 =\mu_{\tilde{\mathcal{A}}}^1 = 0$ and the equations \eqref{Ainffunceqs} force $\mathcal{E}:\mathcal{A}\rightarrow\tilde{\mathcal{A}}\equiv\mathcal{E}_*\mathcal{A}$ into a formal diffeomorphism, which is thus invertible. Therefore, the non-unital $A_\infty$-functor $\mathcal{J}\coloneqq \mathcal{F}\circ\mathcal{E}^{-1}\circ\tilde{\mathcal{G}}:\tilde{\mathcal{B}}\rightarrow\mathcal{B}$ is well defined and fulfills: $\mathcal{K}\simeq\mathcal{I}d_{\tilde{\mathcal{B}}}\circ\mathcal{K}\circ\mathcal{I}d_\mathcal{B}\simeq\tilde{\mathcal{F}}\circ\tilde{\mathcal{G}}\circ\mathcal{K}\circ\mathcal{F}\circ\mathcal{G}=\tilde{\mathcal{F}}\circ\mathcal{E}\circ\mathcal{G}$, so $\mathcal{K}\circ\mathcal{J}\simeq\tilde{\mathcal{F}}\circ\mathcal{E}\circ\mathcal{I}d_{\mathcal{A}}\circ\mathcal{E}^{-1}\circ\tilde{\mathcal{G}}\simeq\mathcal{I}d_{\tilde{\mathcal{B}}}$ and $\mathcal{J}\circ\mathcal{K}\simeq\mathcal{F}\circ\mathcal{E}^{-1}\circ\mathcal{I}d_{\tilde{\mathcal{A}}}\circ\mathcal{E}\circ\mathcal{G}\simeq\mathcal{I}d_{\mathcal{B}}$. We conclude that $\mathcal{J}$ is a homotopy inverse of $\mathcal{K}$. 
\end{proof}

\subsection{$A_{\infty}$-modules over $A_{\infty}$-categories}\label{ch2.6}

Mirroring our definition for $A_\infty$-algebras, we extend the construction of $A_\infty$-modules to $A_\infty$-categories. 

Given any non-unital $A_\infty$-category $\mathcal{A}$, its opposite $\mathcal{A}^{opp}$ is simply the non-unital $A_\infty$-category of same objects, with $\textup{hom}_{\mathcal{A}^{opp}}(X_0,X_1) \coloneqq \textup{hom}_\mathcal{A}(X_1,X_0)$ and $\mu_{\mathcal{A}^{opp}}^d(a_d,...,a_1)\coloneqq (-1)^{\dagger_d}\mu_{\mathcal{A}}^d(a_1,...,a_d)$. Also, $H(\mathcal{A}^{opp})=H(\mathcal{A})^{opp}$.

\begin{Def}\label{Ainfmod}						
	Consider the $A_\infty$-category\footnote{cf. Remark \ref{Ainfcatrem}.} of complexes $\mathsf{Ch}$ and any non-unital $A_\infty$-category $\mathcal{A}$. Then $\mathsf{Q}= nu\text{-}mod(\mathcal{A})\coloneqq  nu\text{-}\!fun(\mathcal{A}^{opp},\mathsf{Ch})$ is the non-unital \textbf{$A_\infty$-category of non-unital (right) $A_\infty$-modules over $\mathcal{A}$}\index{aaa@$A_\infty$-category!of $A_\infty$-modules}, specified by the following data:
	\begin{itemize}[leftmargin=0.5cm]
	\renewcommand{\labelitemi}{\textendash}
	\item Objects $\mathcal{M}\in\textup{obj}(\mathsf{Q})$ are non-unital $A_\infty$-functors $\mathcal{M}:\mathcal{A}^{opp}\rightarrow\mathsf{Ch}$, thus collections of graded spaces  $\{\mathcal{M}(X)=\bigoplus_{p\in\mathbb{Z}}\mathcal{M}(X)^p\in\textup{obj}(\mathsf{Ch})\mid X\in\textup{obj}(\mathcal{A})\}$ with (multi)linear graded maps of degree $d-2$
	\begin{equation}
	\mu_{\mathcal{M}}^d:\mathcal{M}(X_{d-1})\!\otimes\!\textup{hom}_\mathcal{A}(X_{d-2},X_{d-1})\!\otimes...\otimes\!\textup{hom}_\mathcal{A}(X_0,X_1)\rightarrow\mathcal{M}(X_0)
	\end{equation}
	for all $d\geq 1$ and $X_0,...,X_{d-1}\in\textup{obj}(\mathcal{A})$. In particular, $\mu_{\mathcal{M}}^1$ serves as differential for the cochain complex $\mathcal{M}(X_0)$. The $\mu_{\mathcal{M}}^d$'s are subject to associativity equations analogous to \eqref{Ainffunceqs}:
	\begin{align}\label{Ainfmodeqs}
	&\sum_{0 \leq n\leq d-1}\mkern-18mu(-1)^{\dagger_n}\mu_{\mathcal{M}}^{n+1}(\mu_{\mathcal{M}}^{d-n}(b, a_{d-1},...,a_{n+1}),a_n,...,a_1)\, + \\ 
	&\!\sum_{\substack{1\leq m \leq d \\ 0\leq n< d-m}}\mkern-18mu(-1)^{\dagger_n}\mu_{\mathcal{M}}^{d-m+1}(b,a_{d-1},...,\mu_{\mathcal{A}}^{m}(a_{n+m},...,a_{n+1}),...,a_1) = 0\;. \nonumber
	\end{align}
	\item For any pair $\mathcal{M}_0, \mathcal{M}_1\in\textup{obj}(\mathsf{Q})$, morphisms $t\in\textup{hom}_{\mathsf{Q}}(\mathcal{M}_0,\mathcal{M}_1)$ consisting of sequences $(t^1,t^2,...)$ of (multi)linear graded maps
	\begin{align}\label{tshit}
	t^d:\;  \mathcal{M}_0(X_{d-1})\otimes\textup{hom}_\mathcal{A}(X_{d-2},X_{d-1}&)\otimes...\otimes\textup{hom}_\mathcal{A}(X_0,X_1)	\\
	&\mkern-15mu\longrightarrow\mathcal{M}_1(X_0)[|t|-d+1]\;, \nonumber
	\end{align}
	for $d\geq 1$ and $X_0,...,X_{d-1}\in\textup{obj}(\mathcal{A})$. Such $t$ is called a \textbf{pre-module homomorphism}\index{pre-module homomorphism}. 
	\item Composition maps $\mu_{\mathsf{Q}}^d$, so that for $t_k\in\textup{hom}_\mathsf{Q}(\mathcal{M}_{k-1},\mathcal{M}_k)$ the output $s\coloneqq \mu_{\mathsf{Q}}^d(t_d,...,t_1)\in\textup{hom}_\mathsf{Q}(\mathcal{M}_0,\mathcal{M}_d)[2-d]$ is again a sequence $(s^1,s^2,...)$ as above. The first two cases $d=1,2$ are:
	\begin{align}\label{tshit1}
	\mu_{\mathsf{Q}}^1&(t_1)^h(b,a_{h-1},...,a_1) \coloneqq\mkern-18mu\sum_{0\leq n\leq h-1}\mkern-18mu(-1)^{\natural_n}\mu_{\mathcal{M}_{1}}^{n+1}(t^{h-n}(b,a_{h-1},...,a_{n+1}),a_n,...,a_1)+	\nonumber	\\ 
	&\mkern-24mu\sum_{\quad 0\leq n\leq h-1}\mkern-18mu(-1)^{\natural_n}t^{n+1}(\mu_{\mathcal{M}_{0}}^{h-n}(b,a_{h-1},...,a_{n+1}),a_n,...,a_1)\,+ \nonumber	\\
	&\mkern-27mu\sum_{\substack{1\leq m \leq h \\ \quad 0\leq n< h-m}}\mkern-21mu(-1)^{\natural_n}t^{h-m+1}(b,a_{h-1},...,\mu_{\mathcal{A}}^{m}(a_{n+m},...,a_{n+1}),...,a_1)\;,
	\end{align}
	in $\mathcal{M}_1(X_0)$, and
	\begin{align}\label{tshit2}
	\mkern-15mu\mu_{\mathsf{Q}}^2(t_2,t_1)^h(b,a_{h-1},...,a_1) \coloneqq \mkern-24mu\sum_{0\leq n\leq h-1}\mkern-18mu(-1)^{\natural_n}t_{2}^{n+1}(t_{1}^{h-n}(b,a_{h-1},...),a_n,...,a_1),  
	\end{align}
	in $\mathcal{M}_2(X_0)$, where $h\geq 1$, $a_k\in\textup{hom}_\mathcal{A}(X_{k-1},X_k)$ and $\natural_n \coloneqq |b| + |a_{h-1}|+...+|a_{n+1}| - (h-1-n)$. Meanwhile, for $d>2$ one has $\mu_{\mathsf{Q}}^d \coloneqq 0$. 
	\end{itemize} 
\end{Def}

In the case of trivial $A_\infty$-categories, this data coincide with that given in Definitions \ref{Ainfalgmod} and  \ref{Ainfmodmorph} (compare the defining equations \eqref{Ainfmodeqs} and \eqref{tshit1} with \eqref{Ainfalgmodasseq} respectively \eqref{Ainfmodmorpheqs}, observing that pre-module homomorphisms simply reduce to morphisms of $A_\infty$-modules over $A_\infty$-algebras).

Any pre-module homomorphism $t\in\textup{hom}_\mathsf{Q}(\mathcal{M}_0,\mathcal{M}_1)$ identifies a pre-natural transformation $T\!\in\textup{hom}_{nu\text{-}\!fun(\mathcal{A}^{opp},\mathsf{Ch})}(\mathcal{M}_0,\mathcal{M}_1)$ by setting $T^{d-1}(a_1,...,a_{d-1})(b)$ $\coloneqq (-1)^\S t^d(b,a_{d-1},...,a_1)$, where $\S \coloneqq (|T|-1)|b|+|T|(|T|-1)/2$.

The associativity equation for $d=1$ reads $\mu_\mathcal{M}^1(\mu_\mathcal{M}^1(b)) = 0$ --- as expected --- while \eqref{tshit1} and \eqref{tshit2} become 
\begin{align*}
&\mu_\mathsf{Q}^1(t)^1(b) = (-1)^{|b|}\mu_{\mathcal{M}_1}^1(t^1(b))+(-1)^{|b|}t^1(\mu_{\mathcal{M}_0}^1(b))\;,	\\
&\mu_\mathsf{Q}^2(t_2,t_1)^1(b) = (-1)^{|b|}t_{2}^1(t_{1}^1(b))\;.
\end{align*}

\begin{Def}										
	A pre-module homomorphism which is closed with respect to the boundary operator of $\textup{hom}_\mathsf{Q}(\mathcal{M}_0,\mathcal{M}_1)$, that is, any $t\in\textup{hom}_\mathsf{Q}(\mathcal{M}_0,\mathcal{M}_1)$ with $\mu_{\mathsf{Q}}^1(t)^d = 0$ for all $d\geq 0$, is called an ($A_\infty$-)\textbf{module homomorphism}\index{module homomorphism}. 
\end{Def}  

Being an $A_\infty$-functor, $\mathcal{M}\in\textup{obj}(\mathsf{Q})$ induces a non-unital (contravariant) functor $H(\mathcal{M}):H(\mathcal{A})^{opp}\rightarrow H(\mathsf{Ch})\cong\mathsf{GradVect}$ to the category of graded vector spaces (as described in Section \ref{ch1.1}). 

\begin{Def}\label{cohommod}						
	Let $\mathcal{M}\in\textup{obj}(\mathsf{Q})$ be an $A_\infty$-module. The \textbf{cohomological module}\index{cohomological module} $H(\mathcal{M}):H(\mathcal{A})^{opp}\rightarrow \mathsf{GradVect}$ of $\mathcal{M}$ is an $H(\mathcal{A})$-module consisting of a family $\{(H(\mathcal{M})(X),\partial_X)\mid X\!\in\textup{obj}(\mathcal{A})\}$, where $H(\mathcal{M})(X)\coloneqq H(\mathcal{M}(X))  \in\textup{obj}(\mathsf{GradVect})$ and $\partial_X(b) \coloneqq  (-1)^{|b|}\mu_{\mathcal{M}}^1(b)\in\mathcal{M}(X)[1]$ for $b\in\mathcal{M}(X)$, endowed with module multiplication 
	\[
	H(\mathcal{M})(X_1)\otimes \textup{Hom}_{H(\mathcal{A})}(X_0,X_1)\rightarrow H(\mathcal{M})(X_0),\;\;\langle b\rangle\cdot\langle a\rangle \coloneqq (-1)^{|a|}\langle\mu_{\mathcal{M}}^2(b,a)\rangle\;,
	\]
	(cf. equation \eqref{cohomcatcompo}). Cohomological modules are the objects of the non-unital category $\mathcal{H}(\mathsf{Q})= Nu\text{-}\!Fun(H(\mathcal{A})^{opp},\mathsf{GradVect})$ ($=:\mathsf{Mod}^0(H(\mathcal{A}))$, if unital). 
	
	Moreover, any module homomorphism $t\in\textup{hom}_\mathsf{Q}(\mathcal{M}_0,\mathcal{M}_1)$ induces an ordinary module homorphism $H(t)\in\textup{Hom}_{\mathcal{H}(\mathsf{Q})}(H(\mathcal{M}_0),H(\mathcal{M}_1))$, namely a family of graded maps
	\[
	H(t)_X: H(\mathcal{M}_0)(X)\rightarrow H(\mathcal{M}_1)(X),\,\langle b\rangle\mapsto\langle (-1)^{|b|}t^1(b)\rangle\;,
	\]
	for all $X\in\textup{obj}(\mathcal{A})$.
\end{Def}

\subsection{Strict unitality and cohomological unitality}\label{ch2.7}

As we have already observed at the beginning of Section \ref{ch2.2}, $A_\infty$-categories generally lack identity morphisms, but --- as one would guess --- it is desirable to work with them, as they simplify the overall theory significantly. So, here we introduce two different notions of unitality (generalizing those encountered with $A_\infty$-algebras) which turn out to be equivalent.

\begin{Def}												
	Let $\mathcal{A}$ be a non-unital $A_\infty$-category. We call it:
	\begin{itemize}[leftmargin=0.5cm]
		\item \textbf{strictly unital}\index{aaa@$A_\infty$-category!strictly unital} if for each object $X$ there exists some $e_X\in\textup{hom}_{\mathcal{A}}^0(X,X)$ (then unique), called a \textbf{strict unit}\index{unit!strict}, fulfilling the conditions:
		\begin{align}\label{unitalrequirm}
			&\mu_\mathcal{A}^1(e_X) = 0\;,\nonumber\\ &(-1)^{|a_1|}\mu_\mathcal{A}^2(e_{X_1},a_1) = a_1 = \mu_\mathcal{A}^2(a_1,e_{X_0})\;,\\ &\mu_{\mathcal{A}}^d(a_{d-1},...,a_{n+1},e_{X_n},a_n,...,a_1) = 0\;,\nonumber
		\end{align}
		for $d>2$, $0\leq n<d$ and $a_i\in\textup{hom}_\mathcal{A}(X_{i-1},X_i)$ (in particular, $e_X$ must be a cocycle);
		\item \textbf{c-unital}\index{aaa@$A_\infty$-category!c-unital} (for cohomologically unital) if the cohomological category $H(\mathcal{A})$ is unital (thus an ``actual'' linear graded category). We call any $e_X\in\textup{hom}_{\mathcal{A}}^0(X,X)$ inducing the identity morphism on $X$ in $H(\mathcal{A})$ a \textbf{c-unit}\index{unit!cohomological} ($e_X$ must necessarily be a degree 0 cocycle). 
	\end{itemize}
\end{Def}

Strict unitality implies c-unitality, which can be proved by making use of the middle notion of homotopy unitality: roughly, a \textit{homotopy unital} $A_\infty$-category is one endowed with additional composition maps $\mu_\mathcal{A}^{d,(i_d,...,i_0)}$ of degree $2-d-2\sum_{k=0}^di_k$, for all $d+i_0+...+i_d>0$, coinciding with those of Definition \ref{Ainfcat} when $i_0=...=i_d=0$, and fulfilling some modified compatible $A_\infty$-associativity equations. 

Actually, the converse implication is true as well (up to quasi-isomorphism given by a formal diffeomorphism):

\begin{Lem}\label{cunitstrictunit}		
	Let $\mathcal{A}$ be a c-unital $A_\infty$-category such that for all $X\in\textup{obj}(\mathcal{A})$ either $\textup{hom}_\mathcal{A}(X,X) = \{0\}$ or $H(\textup{hom}_\mathcal{A}(X,X))$ is non-trivial. Then there exists a formal diffeomorphism $\Phi:\mathcal{A}\rightarrow\tilde{\mathcal{A}}=\Phi_*\mathcal{A}$ with $\Phi_{X_0,X_1}^1 = \textup{id}_{\textup{hom}_\mathcal{A}(X_0,X_1)}$ and such that $\Phi_*\mathcal{A}$ is strictly unital.
\end{Lem}

\begin{proof}
	(\textit{Sketch}) As in most cases, we lay out a recursive construction. First, for each $X\in\textup{obj}(\mathcal{A})$ there exists by assumption a c-unit $e_X\in\textup{hom}_\mathcal{A}^0(X,X)$, which is in particular a cocycle. Set $\mu_{\tilde{\mathcal{A}}}^1 \coloneqq \mu_\mathcal{A}^1$, then clearly $\mu_{\tilde{\mathcal{A}}}^1(e_X)=0$. Construct\footnote{Recall from Definition \ref{formaldiffeo} that $\Phi_*\mathcal{A}$ has same morphism spaces as $\mathcal{A}$!} $\mu_{\tilde{\mathcal{A}}}^2:\textup{hom}_\mathcal{A}(X_1,X_2)\otimes\textup{hom}_\mathcal{A}(X_0,X_1)\rightarrow\textup{hom}_\mathcal{A}(X_0,X_2)$ so as to represent the composition law \eqref{cohomcatcompo} for $H(\mathcal{A})$ and satisfy the second equation in \eqref{unitalrequirm}. Let $\Phi^1\coloneqq\textup{id}$, and define $\Phi^2$ to be a chain homotopy between $\mu_\mathcal{A}^2$ and $\mu_{\tilde{\mathcal{A}}}^2$, on each morphism space.
	
	The construction now becomes more involved, ultimately producing the desired formal diffeomorphism $\Phi:\mathcal{A}\rightarrow\tilde{\mathcal{A}}$ from composition maps:
	\begin{align*}
	\mu_{\tilde{\mathcal{A}}}^d&(a_d,...,a_1) \coloneqq \mu_\mathcal{A}^d(a_d,...,a_1)-(-1)^\gamma\mu_\mathcal{A}^2(\mu_\mathcal{A}^d(a_d,...,e_{X_{n+1}},a_{n+1},...,a_2),a_1)\\
	&-(-1)^{\dagger_{n+1}}\sum_{j<n}(-1)^{\alpha_j}\mu_\mathcal{A}^d(a_d,...,e_{X_{n+1}},a_{n+1},...,\mu_{\mathcal{A}}^2(a_{j+2},a_{j+1}),...,a_1)\\
	&-(-1)^{\dagger_n}\sum_{i,j}(-1)^{\beta_j}\mu_{\mathcal{A}}^{d+2-i}(b_{d+1},...,\mu_\mathcal{A}^i(b_{j+i},...,b_{j+1}),b_j,...,b_1)\;, 
	\end{align*}
	where $0\leq n<d$, $(b_1,...,b_{d+1}) \coloneqq (a_1,...,a_n,e_{X_n},a_{n+1},...,a_d)$,  $\gamma \coloneqq |a_2|+...+|a_{n+1}|-n$, $\alpha_j \coloneqq |a_1|+...+|a_j|-j$, $\beta_j \coloneqq |b_1|+...+|b_j|-j$ and the ranges for $i,j$ in last summand are: $i=1$, $j\neq n$; $i=2$, $j>n$; $i=d$, $j=0$; $i=d+1$. 
	
	Now, one must check that such higher order $\mu_{\tilde{\mathcal{A}}}^d$'s vanish when having an identity morphism among their entries. For the strategical approach to this we refer to \cite[Lemma 2.1]{[Sei08]}.
\end{proof}

\begin{Def}\label{unitalfunc}						
	Let $\mathcal{F}:\mathcal{A}\rightarrow\mathcal{B}$ be a functor of $A_\infty$-categories. If $\mathcal{A}$ and $\mathcal{B}$ are:
	\begin{itemize}[leftmargin=0.5cm]
		\item strictly unital, then we call $\mathcal{F}$ a \textbf{strictly unital $A_\infty$-functor}\index{aaa@$A_\infty$-functor!strictly unital} if $\mathcal{F}^1(e_X) = e_{\mathcal{F}(X)}$ for all $X\in\textup{obj}(\mathcal{A})$, and $\mathcal{F}^d(a_{d-1},...,a_{n+1},e_{X_n},a_n,...,a_1) = 0$ if $d>1$;
		\item c-unital, then we call $\mathcal{F}$ a \textbf{c-unital $A_\infty$-functor}\index{aaa@$A_\infty$-functor!c-unital} if $H(\mathcal{F})$ is unital (thus an ``actual'' linear graded functor).
	\end{itemize}
\end{Def} 

We put emphasis on c-unitality, as this is a characteristic of Fukaya categories, which we will introduce in the second part of this thesis. For now, let us investigate unitality in functor categories.

\begin{Pro}\label{funccatunitality}			
	Let $\mathcal{A}$ be an $A_\infty$-category, $\mathcal{C}$ any non-unital $A_\infty$-category. Then:
	\renewcommand{\theenumi}{\roman{enumi}}
	\begin{enumerate}[leftmargin=0.5cm]
		\item if $\mathcal{A}$ is strictly unital, then so is $\mathcal{Q}= nu\text{-}\!fun(\mathcal{C},\mathcal{A})$;
		\item if $\mathcal{A}$ is c-unital, then so is $\mathcal{Q}$. 
	\end{enumerate}
	In the latter case, if $E_\mathcal{F}\in\textup{hom}_\mathcal{Q}^0(\mathcal{F},\mathcal{F})$ is a c-unit, thus representing $\textup{id}_\mathcal{F}\in\textup{Hom}_{H(\mathcal{Q})}^0(\mathcal{F},\mathcal{F})$, then each $E_\mathcal{F}^0(X)\equiv E_{\mathcal{F},X}^0\in\textup{hom}_\mathcal{A}^0(\mathcal{F}(X),\mathcal{F}(X))$ for $X\in\textup{obj}(\mathcal{C})$ is a c-unit, thus representing $\textup{id}_{\mathcal{F}(X)}\in\textup{Hom}_{H(\mathcal{A})}^0(\mathcal{F}(X),\mathcal{F}(X))$.
\end{Pro} 

\begin{proof}
	\renewcommand{\theenumi}{\roman{enumi}}
	\ \vspace*{0.0cm}
	\begin{enumerate}[leftmargin=0.5cm]
		\item We check that strict units for $\mathcal{Q}$ are all the pre-natural transformations $E_\mathcal{F}\in\textup{hom}_\mathcal{Q}^0(\mathcal{F},\mathcal{F})$ given by $E_\mathcal{F}^0(X) \coloneqq  e_{\mathcal{F}(X)}\in\textup{hom}_\mathcal{A}^0(\mathcal{F}(X),\mathcal{F}(X))$ (the strict unit of $\mathcal{F}(X)$) and $E_\mathcal{F}^d \coloneqq  0$ if $d>0$. For example, substitution in \eqref{cutie0} gives $\mu_\mathcal{Q}^1(E_\mathcal{F})_{X}^0 = \mu_\mathcal{A}^1(e_{\mathcal{F}(X)}) = 0$, while \eqref{cutie1} yields $\mu_\mathcal{Q}^1(E_\mathcal{F})^1(a) = 0+0+\mu_\mathcal{A}^2(\mathcal{F}^1(a),e_{\mathcal{F}(X_0)})+(-1)^{1-|a|}\mu_\mathcal{A}^2(e_{\mathcal{F}(X_1)},\mathcal{F}^1(a)) = \mathcal{F}^1(a) + (-1)^{1-|a|+|a|}$ $\cdot\mathcal{F}^1(a) = 0$. Indeed, the equations \eqref{shit1} show that $\mu_{\mathcal{Q}}^1(E_\mathcal{F})^d = 0$ for all higher $d$, hence $\mu_\mathcal{Q}^1(E_\mathcal{F}) = 0$ (in particular, $E_\mathcal{F}$ is a natural transformation). By \eqref{shit2}, the second requirement for strict unitality is also met, while the third asks for higher degree $\mu_\mathcal{Q}^d$'s (see again \cite{[Fuk99]}).    
		
		\item We reduce this case to the strictly unital one from above. Thereto, consider a formal diffeomorphism $\Phi:\mathcal{A}\rightarrow\tilde{\mathcal{A}}$ with $\Phi^1=\textup{id}$ on morphism spaces as given by Lemma \ref{cunitstrictunit}, so that $\tilde{\mathcal{A}}$ is strictly unital (then so is $\tilde{\mathcal{Q}}\coloneqq nu\text{-}\!fun(\mathcal{C},\tilde{\mathcal{A}}) $, by part i.). By definition, $\Phi$ is a quasi-isomorphism, and hence so is its left composition $\mathcal{L}_{\Phi}:\mathcal{Q}\rightarrow\tilde{\mathcal{Q}}$ (by Lemma \ref{1.7}). For each $\mathcal{F}\in\textup{obj}(\mathcal{Q})$, this enables us to find some $E_\mathcal{F}\in\textup{hom}_\mathcal{Q}^0(\mathcal{F},\mathcal{F})$ such that $H(\mathcal{L}_{\Phi})\langle E_\mathcal{F}\rangle = \langle \mathcal{L}_{\Phi}^1(E_\mathcal{F})\rangle$ is the identity morphism  $\langle E_{\tilde{\mathcal{F}}}\rangle\in\textup{Hom}_{H(\tilde{\mathcal{Q}})}^0(\tilde{\mathcal{F}},\tilde{\mathcal{F}})$ induced by the unique strict unit for $\tilde{\mathcal{F}}\coloneqq \mathcal{L}_\Phi(\mathcal{F})\in\textup{obj}(\tilde{\mathcal{Q}})$. Functoriality of the isomorphism $H(\mathcal{L}_\Phi)$ guarantees that $\langle E_\mathcal{F}\rangle$ itself is an identity morphism in $H(\mathcal{Q})$, the unique such for $\mathcal{F}$, which proves that $H(\mathcal{Q})$ is unital, thus $\mathcal{Q}$ c-unital. 
	\end{enumerate}
	\vspace*{0.2cm}
	Now, the functor $\mathcal{H}$ of \eqref{Hfunctor} maps $\langle E_\mathcal{F}\rangle = \textup{id}_\mathcal{F}$ to the identity natural transformation $H(E_\mathcal{F})\in\textup{Hom}_{\mathcal{H}(\mathcal{Q})}(H(\mathcal{F}),H(\mathcal{F}))$ in $\mathcal{H}(\mathcal{Q}) = Nu\text{-}\!Fun(H(\mathcal{C}),H(\mathcal{A}))$. The definition of $\Phi$ (which makes $H(\mathcal{F}) = H(\tilde{\mathcal{F}})$ and $H(\mathcal{A})\cong H(\tilde{\mathcal{A}})$, thus $H(\mathcal{Q})\cong H(\tilde{\mathcal{Q}})$) allows us to construct the commutative diagram
	\[
	\begin{tikzcd}
	{\textup{Hom}_{H(\mathcal{Q})}(\mathcal{F},\mathcal{F})} \arrow[rd, "\mathcal{H}"']\arrow[rr, "H(\mathcal{L}_\Phi)"] & & {\textup{Hom}_{H(\tilde{\mathcal{Q}})}(\tilde{\mathcal{F}},\tilde{\mathcal{F}})}\arrow[ld, "\mathcal{H}"]	\\
	&{\textup{Hom}_{\mathcal{H}(\mathcal{Q})}(H(\mathcal{F}),H(\mathcal{F}))} &
	\end{tikzcd}\quad,
	\]
	which in turn gives $H(E_\mathcal{F}) = \mathcal{H}(\langle E_\mathcal{F}\rangle)=\mathcal{H}\circ H(\mathcal{L}_\Phi)\langle E_\mathcal{F}\rangle = \mathcal{H}(\langle\tilde{E}_\mathcal{F}\rangle) = H(E_{\tilde{\mathcal{F}}})$ as identity natural transformations, so that each $\langle E_\mathcal{F}^0(X)\rangle = \langle E_{\tilde{\mathcal{F}}}^0(X)\rangle$ for $X\in\textup{obj}(\mathcal{C})$ (cf. Definition \ref{Nu-Fun}) must be the identity $\textup{id}_{\mathcal{F}(X)}$ in $H^0(\mathcal{A})$. 
\end{proof}

\begin{Lem}\label{cunitalcompo}					
	Let $\mathcal{G}:\mathcal{A}\rightarrow\mathcal{B}$ be a c-unital $A_\infty$-functor. Then for any non-unital $A_\infty$-category $\mathcal{C}$, $\mathcal{L}_\mathcal{G}:nu\text{-}\!fun(\mathcal{C},\mathcal{A})\rightarrow nu\text{-}\!fun(\mathcal{C},\mathcal{B})$ is c-unital. A similar result holds for right composition.
\end{Lem}

\begin{proof}
	For $E_\mathcal{F}\in\textup{hom}_{nu\text{-}\!fun(\mathcal{C},\mathcal{A})}^0(\mathcal{F},\mathcal{F})$ a c-unit, we know by Proposition \ref{funccatunitality} that $\langle E_\mathcal{F}^0(X)\rangle = \textup{id}_{\mathcal{F}(X)}$ in $H(\mathcal{A})$ for each $X\in\textup{obj}(\mathcal{C})$. Then $\langle\mathcal{L}_\mathcal{G}^1(E_\mathcal{F})_{X}^0\rangle = \langle\mathcal{G}^1(E_\mathcal{F}^0(X))\rangle = H(\mathcal{G})\langle E_\mathcal{F}^0(X)\rangle$ (using \eqref{lshit1} for $d=0$) shows that also $\mathcal{L}_\mathcal{G}^1(E_\mathcal{F})_{X}^0$ $\in\textup{hom}_\mathcal{B}^0\big(\mathcal{G}(\mathcal{F}(X)),\mathcal{G}(\mathcal{F}(X))\big)$ is a c-unit (as $\mathcal{G}$ is c-unital, $H(\mathcal{G})$ will preserve identity morphisms).
	
	Now we apply Lemma \ref{1.6} to the following data: $\mathcal{Q}\coloneqq nu\text{-}\!fun(\mathcal{C},\mathcal{B})$, $\mathcal{G}_1=\mathcal{G}_2=\mathcal{G}_3\coloneqq \mathcal{G}\circ\mathcal{F}$,\, $T\coloneqq \mathcal{L}_\mathcal{G}^1(E_\mathcal{F})\in\textup{hom}_\mathcal{Q}^0(\mathcal{G}\circ\mathcal{F},\mathcal{G}\circ\mathcal{F})$, where the argument from last paragraph ensures that right composition with each $\langle T_{X}^0\rangle = \textup{id}_{\mathcal{G}(\mathcal{F}(X))}$ is indeed an automorphism of $\textup{Hom}_{H(\mathcal{B})}\big(\mathcal{G}(\mathcal{F}(X)),\mathcal{G}(\mathcal{F}(X))\big)$. Consequently, right composition with $\langle\mathcal{L}_\mathcal{G}^1(E_\mathcal{F})\rangle$ is an automorphism of $\textup{Hom}_{H(\mathcal{Q})}(\mathcal{G}\circ\mathcal{F},\mathcal{G}\circ\mathcal{F})$. 
	
	But $\mathcal{Q}$ is c-unital by Proposition \ref{funccatunitality}, and thus $H(\mathcal{Q})$ possesses an identity $\textup{id}_{\mathcal{G}\circ\mathcal{F}}$ and whichever $\langle S\rangle\in\textup{Hom}_{H(\mathcal{Q})}(\mathcal{G}\circ\mathcal{F},\mathcal{G}\circ\mathcal{F})$ maps to it under right composition with $\langle\mathcal{L}_\mathcal{G}^1(E_\mathcal{F})\rangle$ is an inverse for $\langle\mathcal{L}_\mathcal{G}^1(E_\mathcal{F})\rangle$. Moreover, by functoriality of $H(\mathcal{L}_\mathcal{G})$ and equation \eqref{cohomcatcompo}, $\langle\mathcal{L}_\mathcal{G}^1(E_\mathcal{F})\rangle\circ\langle\mathcal{L}_\mathcal{G}^1(E_\mathcal{F})\rangle = \langle\mathcal{L}_\mathcal{G}^1(\mu_\mathcal{Q}^2(E_\mathcal{F},E_\mathcal{F}))\rangle = \langle\mathcal{L}_\mathcal{G}^1(E_\mathcal{F})\rangle$ is idempotent. The only idempotent invertible map is the identity, whence we conclude that $H(\mathcal{L}_\mathcal{G}): \langle E_\mathcal{F}\rangle\mapsto\langle\mathcal{L}_\mathcal{G}^1(E_\mathcal{F})\rangle$ preserves identities, that is, $\mathcal{L}_\mathcal{G}$ is c-unital.
\end{proof}

\begin{Lem}\label{homotopicareiso}				
	Let $\mathcal{A}$ be a c-unital $A_\infty$-category, $\mathcal{C}$ any non-unital one. If $\mathcal{F}_0,\mathcal{F}_1: \mathcal{C}\rightarrow\mathcal{A}$ are homotopic non-unital $A_\infty$-functors, then $\mathcal{F}_0\cong\mathcal{F}_1$ as objects of $H^0(nu\text{-}\!fun(\mathcal{C},\mathcal{A}))$.
\end{Lem}

\begin{proof}
	As in the proof of Proposition \ref{funccatunitality}, we pass to the strictly unital image $\tilde{\mathcal{A}}=\Phi_*\mathcal{A}$ via a suitable formal diffeomorphism $\Phi$, recalling that $\mathcal{F}_0\simeq\mathcal{F}_1$ implies $\tilde{\mathcal{F}}_0\simeq\tilde{\mathcal{F}}_1$, where $\tilde{\mathcal{F}}_i \coloneqq  \Phi\circ\mathcal{F}_i$. 
	
	Set again $\mathcal{Q}=nu\text{-}\!fun(\mathcal{C},\mathcal{A})$ and $\tilde{\mathcal{Q}}=nu\text{-}\!fun(\mathcal{C},\tilde{\mathcal{A}})$. Let $T\in\textup{hom}_{\tilde{\mathcal{Q}}}^0(\tilde{\mathcal{F}}_0,\tilde{\mathcal{F}}_1)$ have $T^0=0$, and define $S\in\textup{hom}_{\tilde{\mathcal{Q}}}^0(\tilde{\mathcal{F}}_0,\tilde{\mathcal{F}}_1)$ such that $S_{X}^0 = e_{\tilde{\mathcal{F}}_0(X)} = e_{\tilde{\mathcal{F}}_1(X)}$\footnote{Recall that two $A_\infty$-functors can be homotopic only when coinciding on objects!} for all $X\in\textup{obj}(\mathcal{C})$ (then $\mu_{\tilde{\mathcal{Q}}}^1(S)_{X}^0 = \mu_{\tilde{\mathcal{A}}}^1(e_{\tilde{\mathcal{F}}_0(X)}) = 0$) and $S^d = T^d$ if $d>0$. The equations \eqref{shit1} show that $T$ is a homotopy from $\tilde{\mathcal{F}}_0$ to $\tilde{\mathcal{F}}_1$ if and only if $S$ is a natural transformation, as \eqref{cutie1} readily suggests:
	$\mu_{\tilde{\mathcal{Q}}}^1(S)^1(a) = \mu_{\tilde{\mathcal{A}}}^1(S^1(a))-S^1(\mu_{\mathcal{C}}^1(a))+\mu_{\tilde{\mathcal{A}}}^2(\tilde{\mathcal{F}}_{1}^1(a),S_{X_0}^0) +  (-1)^{1-|a|}\mu_{\tilde{\mathcal{A}}}^2(S_{X_1}^0,\tilde{\mathcal{F}}_{0}^1(a)) = \mu_{\tilde{\mathcal{A}}}^1(S^1(a))-S^1(\mu_{\mathcal{C}}^1(a))$ $+\tilde{\mathcal{F}}_{1}^1(a)-\tilde{\mathcal{F}}_{0}^1(a) = 0$ if and only if $\mu_{\tilde{\mathcal{Q}}}^1(T) = \tilde{\mathcal{F}}_0 - \tilde{\mathcal{F}}_1$.
	
	Now, apply Lemma \ref{1.6} to the data: $\mathcal{G}_1\coloneqq \tilde{\mathcal{F}}_0$, $\mathcal{G}_2=\mathcal{G}_3\coloneqq \tilde{\mathcal{F}}_1$ and natural transformation $S\in\textup{hom}_{\tilde{\mathcal{Q}}}(\tilde{\mathcal{F}}_0,\tilde{\mathcal{F}}_1)$, so that $\langle c\rangle\mapsto\langle c\rangle\circ\langle S_{X}^0\rangle = \langle\mu_{\tilde{\mathcal{A}}}(c,e_{\tilde{\mathcal{F}}_0(X)})\rangle = \langle c\rangle$ is indeed an isomorphism. Then $\textup{Hom}_{H(\tilde{\mathcal{Q}})}^0(\tilde{\mathcal{F}}_1,\tilde{\mathcal{F}}_1)\rightarrow\textup{Hom}_{H(\tilde{\mathcal{Q}})}^0(\tilde{\mathcal{F}}_0,\tilde{\mathcal{F}}_1)$ given by right composition with $\langle S\rangle$ is an isomorphism. $\tilde{\mathcal{Q}}$ being strictly unital (since $\tilde{\mathcal{A}}$ is), it is also c-unital, so that we can find an identity morphism $\textup{id}_{\tilde{\mathcal{F}}_1}\in\textup{Hom}_{H(\tilde{\mathcal{Q}})}^0(\tilde{\mathcal{F}}_1,\tilde{\mathcal{F}}_1)$. Then $\tilde{\mathcal{F}}_0 \cong\tilde{\mathcal{F}}_1\in\textup{obj}(H^0(\tilde{\mathcal{Q}}))$ via the isomorphism $\textup{id}_{\tilde{\mathcal{F}}_1}\circ\langle S\rangle:\tilde{\mathcal{F}}_0\rightarrow\tilde{\mathcal{F}}_1$.
	
	Finally, $\Phi$ being a quasi-isomorphism, it is also cohomologically full and faithful, hence so is $\mathcal{L}_\Phi:\mathcal{Q}\rightarrow\tilde{\mathcal{Q}}$ by Lemma \ref{1.7}. This implies that $H(\mathcal{L}_\Phi)_{\mathcal{F}_0,\mathcal{F}_1}$ is bijective, and hence preserves isomorphisms. Therefore, $\mathcal{F}_0 \cong\mathcal{F}_1\in\textup{obj}(H^0(\mathcal{Q}))$, as desired. 	
\end{proof}

\subsection{Quasi-equivalences}\label{ch2.8}

\begin{Def}														
	For $\mathcal{A}$ and $\mathcal{B}$ c-unital $A_\infty$-categories, let $fun(\mathcal{A},\mathcal{B})$ denote the full $A_\infty$-subcategory of $nu\text{-}\!fun(\mathcal{A},\mathcal{B})$ consisting of c-unital $A_\infty$-functors. 
	
	Any $\mathcal{F}\in\textup{obj}(fun(\mathcal{A},\mathcal{B}))$ is called \textbf{quasi-equivalence}\index{quasi-equivalence} if $H(\mathcal{F})$ is an equivalence. In particular, by Lemma \ref{aboutfunctors}, we have the following chain of implications for $\mathcal{F}$:
	\vspace*{0.2cm}
	
	\noindent\minibox[frame]{quasi-isomorphism $\implies$ quasi-equivalence $\implies$ cohomologically full, faithful \\ $\mkern+338mu\implies$ c-unital}	
\end{Def}

In the following, given any c-unital $A_\infty$-category $\mathcal{A}$, write $\tilde{\mathcal{A}}\subset\mathcal{A}$ for the full c-unital $A_\infty$-subcategory making the (c-unital) inclusion $\mathcal{I}:\tilde{\mathcal{A}}\rightarrow\mathcal{A}$ a quasi-equivalence.

\begin{Lem}\label{restrictionquasi}				
	Given another c-unital $A_\infty$-category $\mathcal{B}$, the restriction map $\mathcal{Q}=fun(\mathcal{A},\mathcal{B})\rightarrow \tilde{\mathcal{Q}} = fun(\tilde{\mathcal{A}},\mathcal{B})$ specified by $\mathcal{F}\mapsto\tilde{\mathcal{F}}\equiv\mathcal{F}|_{\tilde{\mathcal{A}}}$ and $T\in\textup{hom}_\mathcal{Q}(\mathcal{F}_0,\mathcal{F}_1)\mapsto \tilde{T}\in\textup{hom}_{\tilde{\mathcal{Q}}}(\tilde{\mathcal{F}}_0,\tilde{\mathcal{F}}_1)$ is a quasi-equivalence.
\end{Lem}

\begin{proof}
	(\textit{Sketch}) Arguing with Hochschild cohomology, one can prove that $T\in\textup{hom}_\mathcal{Q}(\mathcal{F}_0,\mathcal{F}_1)\mapsto \tilde{T}\in\textup{hom}_{\tilde{\mathcal{Q}}}(\tilde{\mathcal{F}}_0,\tilde{\mathcal{F}}_1)$ is a quasi-isomorphism, thus cohomologically full and faithful, and that any c-unital $\tilde{\mathcal{F}}:\tilde{\mathcal{A}}\rightarrow\mathcal{B}$ can be extended uniquely to an $\mathcal{F}\in\textup{obj}(\mathcal{Q})$, making the restriction cohomologically essentially surjective. By Lemma \ref{aboutfunctors}, these two conditions imply quasi-equivalence. 
\end{proof}

The extension procedure alluded to in last proof, when applied to the identity $\mathcal{I}d_{\tilde{\mathcal{A}}}\in fun(\tilde{\mathcal{A}},\tilde{\mathcal{A}})$, produces a c-unital $A_\infty$-functor $\mathcal{P}:\mathcal{A}\rightarrow\tilde{\mathcal{A}}$ such that $\mathcal{P}|_{\tilde{\mathcal{A}}} = \mathcal{P}\circ\mathcal{I} = \mathcal{I}d_{\tilde{\mathcal{A}}}$ and which is a quasi-equivalence (as $\mathcal{I}$ is). Actually, $\mathcal{I}\circ\mathcal{P}\cong{\mathcal{I}d_\mathcal{A}}\in\textup{obj}\big(H^0(fun(\mathcal{A},\mathcal{A}))\big)$, because $\mathcal{L}_\mathcal{P}$ is cohomologically full and faithful (by Lemma \ref{1.7}), and thus
\vspace*{-0.5cm}
\[
H(\mathcal{L}_\mathcal{P}):\textup{Hom}_{H(fun(\mathcal{A},\mathcal{A}))}^0(\mathcal{I}\circ\mathcal{P},\mathcal{I}d_\mathcal{A})\rightarrow\textup{Hom}_{H(fun(\mathcal{A},\tilde{\mathcal{A}}))}^0(\overbrace{\mathcal{L}_\mathcal{P}(\mathcal{I}\circ\mathcal{P})}^{=\mathcal{L}_\mathcal{P}(\mathcal{I}d_\mathcal{A})},\mathcal{L}_\mathcal{P}(\mathcal{I}d_\mathcal{A}))
\]
a bijection mapping the identity of $\mathcal{L}_\mathcal{P}(\mathcal{I}d_\mathcal{A})$ back to the desired isomorphism.

These maps take part in proving the following very important result, which explains why we care about quasi-equivalences. We will use it repeatedly throughout next section.  

\begin{Thm}\label{aboutquasiequiv}			
	Let $\mathcal{A}, \mathcal{B}$ be c-unital $A_\infty$-categories, $\mathcal{F}:\mathcal{A}\rightarrow\mathcal{B}$ a quasi-equivalence. Then there exists a quasi-equivalence $\mathcal{G}:\mathcal{B}\rightarrow\mathcal{A}$ such that 
	\[
	\mathcal{G}\circ\mathcal{F}\cong\mathcal{I}d_\mathcal{A}\in\textup{obj}\big(H^0(fun(\mathcal{A},\mathcal{A}))\big)\,,\quad\mathcal{F}\circ\mathcal{G}\cong\mathcal{I}d_\mathcal{B}\in\textup{obj}\big(H^0(fun(\mathcal{B},\mathcal{B}))\big)\,.
	\]
\end{Thm} 

\begin{proof}
	Denote by $\tilde{\mathcal{A}}\subset\mathcal{A}$ and $\tilde{\mathcal{B}}\subset\mathcal{B}$ the full c-unital $A_\infty$-subcategories making the inclusions $\mathcal{I}_\mathcal{A}:\tilde{\mathcal{A}}\rightarrow\mathcal{A}$ respectively $\mathcal{I}_\mathcal{B}:\tilde{\mathcal{B}}\rightarrow\mathcal{B}$ quasi-equivalences, invertible up to isomorphism via $\mathcal{P}_\mathcal{A}$ and $\mathcal{P}_\mathcal{B}$, by the argument given in last paragraph. Writing $\tilde{\mathcal{F}}\coloneqq {\mathcal{P}_\mathcal{B}}\circ\mathcal{F}\circ{\mathcal{I}_\mathcal{A}}:\tilde{\mathcal{A}}\rightarrow\tilde{\mathcal{B}}$, we see that $\tilde{\mathcal{F}}$ is actually a quasi-isomorphism, and hence by Lemma \ref{homotopyinverse} has a homotopy inverse $\tilde{\mathcal{G}}:\tilde{\mathcal{B}}\rightarrow\tilde{\mathcal{A}}$, so that $\tilde{\mathcal{F}}\circ\tilde{\mathcal{G}}\simeq\mathcal{I}d_{\tilde{\mathcal{B}}}$ and $\tilde{\mathcal{G}}\circ\tilde{\mathcal{F}}\simeq\mathcal{I}d_{\tilde{\mathcal{A}}}$. Now, Lemma \ref{homotopicareiso} yields $\tilde{\mathcal{F}}\circ\tilde{\mathcal{G}}\cong\mathcal{I}d_{\tilde{\mathcal{B}}}$ in $H^0(fun(\tilde{\mathcal{B}},\tilde{\mathcal{B}}))$ and $\tilde{\mathcal{G}}\circ\tilde{\mathcal{F}}\cong\mathcal{I}d_{\tilde{\mathcal{A}}}$ in $H^0(fun(\tilde{\mathcal{A}},\tilde{\mathcal{A}}))$.
\end{proof}

\begin{Cor}							
	Let $\mathcal{C}$ be any c-unital $A_\infty$-category. If $\mathcal{G}:\mathcal{A}\rightarrow\mathcal{B}$ is a quasi-equivalence, then so is its right composition $\mathcal{R}_\mathcal{G}:fun(\mathcal{B},\mathcal{C})\rightarrow fun(\mathcal{A},\mathcal{C})$ \textup(and its left $\mathcal{L}_\mathcal{G}:fun(\mathcal{C},\mathcal{A})\rightarrow fun(\mathcal{C},\mathcal{B})$\textup).
\end{Cor}

\begin{proof}
	By above, we have two cases which imply the general result: either $\mathcal{G}$ is a quasi-isomorphism or it is the inclusion of a full $A_\infty$-subcategory $\tilde{\mathcal{B}}\coloneqq \mathcal{A}$ (so that the restriction of Lemma \ref{restrictionquasi} is actually $\mathcal{R}_\mathcal{G}$). In both cases, $\mathcal{R}_\mathcal{G}$ is cohomologically full and faithful (by Lemmas \ref{1.7} respectively \ref{restrictionquasi}). 
	
	Now, by Theorem \ref{aboutquasiequiv} we can find a quasi-equivalence $\mathcal{F}:\mathcal{B}\rightarrow\mathcal{A}$ such that $\mathcal{F}\circ\mathcal{G}\cong\mathcal{I}d_\mathcal{A}$ as objects in cohomology, and Lemma \ref{cunitalcompo} ensures that $\mathcal{L}_\mathcal{E}:fun(\mathcal{A},\mathcal{A})\rightarrow fun(\mathcal{A},\mathcal{C})$ is well defined for any $\mathcal{E}\in\textup{obj}(fun(\mathcal{A},\mathcal{C}))$. We have $\mathcal{R}_\mathcal{G}(\mathcal{E}\circ\mathcal{F})=\mathcal{L}_\mathcal{E}(\mathcal{F}\circ\mathcal{G})\cong\mathcal{L}_\mathcal{E}(\mathcal{I}d_\mathcal{A}) = \mathcal{E}$ in $H^0(fun(\mathcal{A},\mathcal{C}))$. Hence $\mathcal{R}_\mathcal{G}$ is cohomologically essentially surjective. Together with fullness and faithfulness, this makes it a quasi-equivalence (the argument for $\mathcal{L}_\mathcal{G}$ is similar). 
\end{proof}

Finally, the notion of unitality for $A_\infty$-modules. 

\begin{Rem}											
	Notice that the dg category $\mathsf{Ch}$ is strictly unital: for any complex $(V,d_V)\in\textup{obj}(\mathsf{Ch})$, the unique strict unit is given by $e_V\coloneqq \textup{id}_V\in\textup{hom}_\mathsf{Ch}^0(V,V)$. Indeed, by Definition \ref{complex}: $\mu_\mathsf{Ch}^1(\textup{id}_V)= d(\textup{id}_V)= d_V\circ \textup{id}_V - (-1)^0\textup{id}_V\circ d_V = 0$, $(-1)^{|f|}\mu_\mathsf{Ch}^2(\textup{id}_W,f) = (-1)^{2|f|}\textup{id}_W\circ f = f = f\circ \textup{id}_V = \mu_\mathsf{Ch}^2(f,\textup{id}_V)$ for all $f\in$ $\textup{Hom}_\mathsf{Ch}(V,W)$, and readily $\mu_\mathsf{Ch}^d = 0$ if $d>2$. 
\end{Rem}

By last Remark and Proposition \ref{funccatunitality}, for any non-unital $A_\infty$-category $\mathcal{A}$, $\mathsf{Q}= nu\text{-}mod(\mathcal{A}) = nu\text{-}\!fun(\mathcal{A}^{opp},\mathsf{Ch})$ is strictly unital. The strict unit associated to $\mathcal{M}\in\textup{obj}(\mathsf{Q})$ is the module homomorphism $t\coloneqq  e_\mathcal{M}\in\textup{hom}_\mathsf{Q}^0(\mathcal{M},\mathcal{M})$ given by 
\begin{equation}\label{modunit}
e_\mathcal{M}^1(b)\coloneqq (-1)^{|b|}b\in\mathcal{M}(X_0)\qquad\text{and}\qquad e_\mathcal{M}^d \coloneqq  0 \quad\text{ if }d>1\,. 
\end{equation}

We verify this in dimension 1 with the help of equations \eqref{tshit1} and \eqref{tshit2}: $\mu_\mathsf{Q}^1(t)^1(b)$ $= (-1)^{|b|}\mu_\mathcal{M}^1(t^1(b)) + (-1)^{|b|}t^1(\mu_\mathcal{M}^1(b)) = \mu_\mathcal{M}^1(b) + (-1)^{2|b|+1}\mu_\mathcal{M}^1(b) = 0$ and $(-1)^{|t_1|}\mu_\mathsf{Q}^2(t,t_1)^1(b) = (-1)^{|t_1|+|b|}t^1(t_{1}^1(b)) = t_{1}^1(b) = \mu_\mathsf{Q}^2(t_1,t)^1(b)$.

\begin{Def}\label{mod}											
	Let $\mathcal{A}$ be a c-unital $A_\infty$-category. An $A_\infty$-module $\mathcal{M}\in\textup{obj}(\mathsf{Q})$ is \textbf{c-unital}\index{aaa@$A_\infty$-module!c-unital} if the $H(\mathcal{A})$-module $H(\mathcal{M})$ is unital, that is, if c-units $e_X\in\textup{hom}_\mathcal{A}^0(X,X)$ yield maps $\mu_\mathcal{M}^2(\;\cdot\;,e_X) :\mathcal{M}(X)\rightarrow\mathcal{M}(X)$ inducing the identity in cohomology. 
	
	Denote by $mod(\mathcal{A})\!\subset \!nu\text{-}mod(\mathcal{A})$ the full subcategory of c-unital $A_\infty$-modules.
\end{Def}

\subsection{The Yoneda embedding}\label{ch2.9}

Before moving to triangulation of $A_\infty$-categories, we discuss an important $A_\infty$-functor which will play a major role in the coming constructions.

\begin{Def}\label{Yonedaemb}						
	The \textbf{Yoneda embedding}\index{Yoneda embedding} for the non-unital $A_\infty$-category $\mathcal{A}$ is the non-unital $A_\infty$-functor $\Upsilon\equiv\Upsilon_\mathcal{A}:\mathcal{A}\rightarrow\mathsf{Q}= nu\text{-}mod(\mathcal{A})$ which:
	\begin{itemize}[leftmargin=0.5cm]
	\renewcommand{\labelitemi}{\textendash}
	\item maps $Y\in\textup{obj}(\mathcal{A})$ to $\mathcal{Y}\in\textup{obj}(\mathsf{Q})$ (called \textbf{Yoneda module}\index{Yoneda module} of $Y$) defined by $\mathcal{Y}(X)\coloneqq \textup{hom}_\mathcal{A}(X,Y)$ and $\mu_\mathcal{Y}^d\coloneqq \mu_\mathcal{A}^d:\textup{hom}_\mathcal{A}(X_{d-1},Y)\otimes...\otimes\textup{hom}_\mathcal{A}(X_0,X_1)\rightarrow\textup{hom}_\mathcal{A}(X_0,Y)$ for $d\geq 1$ and $X,X_0,...,X_{d-1}\in\textup{obj}(\mathcal{A})$,
	\item consists of a sequence of graded maps $\Upsilon^d$ as in \eqref{Ainffunc}, whose first order term $\Upsilon^1$ (of degree 0) maps $c_1\!\in\!\textup{hom}_\mathcal{A}(Y_0,Y_1)$ to the pre-module homorphism $\Upsilon^1(c_1)\!\in\!\textup{hom}_\mathsf{Q}(\mathcal{Y}_0,\mathcal{Y}_1)$ given by
	\begin{align}\label{Yonedamap}
	\Upsilon^1(c_1)^d\!:\;& \mathcal{Y}_0(X_{d-1})\otimes\textup{hom}_\mathcal{A}(X_{d-2},X_{d-1})\otimes...\otimes\textup{hom}_\mathcal{A}(X_0,X_1)\rightarrow\mathcal{Y}_1(X_0) \nonumber \\
	&(b,a_{d-1},...,a_1)\mapsto\mu_\mathcal{A}^{d+1}(c_1,b,a_{d-1},...,a_1),
	\end{align}
	The higher order $\Upsilon^d$ are defined similarly: $\Upsilon^2(c_2,c_1)^d(b,a_{d-1},...,a_1)\!\coloneqq \mu_\mathcal{A}^{d+2}(c_2,c_1,b,a_{d-1},...,a_1)\in\mathcal{Y}_2$ for $c_2\in\textup{hom}_\mathcal{A}(Y_1,Y_2)$, and so on.
	\end{itemize}
	If moreover $\mathcal{A}$ is c-unital, then by Definition \ref{mod} the Yoneda embedding is actually an $A_\infty$-functor $\Upsilon:\mathcal{A}\rightarrow mod(\mathcal{A})$.
\end{Def}

One can show that for a c-unital $A_\infty$-category $\mathcal{A}$, $\Upsilon_\mathcal{A}$ is cohomologically full and faithful (the proof involves Hochschild cohomology; see \cite[section (2g)]{[Sei08]}). Furthermore, for any $A_\infty$-module $\mathcal{M}$ and $Y\in\textup{obj}(\mathcal{A})$, we have a natural chain map 
\vspace*{-0.4cm}
\begin{align}\label{lambda}
&\lambda=\lambda_\mathcal{M}:\mathcal{M}(Y)\rightarrow\textup{hom}_\mathsf{Q}(\mathcal{Y},\mathcal{M}), \nonumber \\
&\lambda(c)^d(b,a_{d-1},...,a_1) \coloneqq \mu_\mathcal{M}^{d+1}(c,b,a_{d-1},...,a_1)  
\end{align}
extending $\Upsilon^1$. Actually, $\lambda(c)$ is a module homomorphism as soon as $c\in\mathcal{M}(Y)$ is a cocycle, by equation \eqref{tshit1}. Moreover, equation \eqref{tshit2} shows that:
\begin{equation}\label{lambdaproperties}
\begin{aligned}
\langle\lambda_\mathcal{M}(\mu_\mathcal{M}^2(c,b))\rangle &= \langle\mu_\mathsf{Q}^2(\lambda_\mathcal{M}(c),\Upsilon^1(b))\rangle\in\textup{Hom}_{H(\mathsf{Q})}(\mathcal{Y}_0,\mathcal{M})\;,\\
\langle\mu_\mathsf{Q}^2(t,\lambda_{\mathcal{M}_0}(\tilde{c}))\rangle &= \langle\lambda_{\mathcal{M}_1}(t^1(\tilde{c}))\rangle\in\textup{Hom}_{H(\mathsf{Q})}(\mathcal{Y},\mathcal{M}_1)\;,
\end{aligned}
\end{equation}
for any $\langle c\rangle\in H(\mathcal{M}(Y_1))$, $\langle b\rangle\in\textup{Hom}_{H(\mathcal{A})}(Y_0,Y_1)$, $\langle t\rangle\in\textup{Hom}_{H(\mathsf{Q})}(\mathcal{M}_0,\mathcal{M}_1)$ and $\langle \tilde{c}\rangle\in H(\mathcal{M}_0(Y))$.

In particular, the extension of the classical Yoneda Lemma holds in the c-unital case (see \cite[Lemma 2.12]{[Sei08]}).

\begin{Lem}[Yoneda Lemma]\label{Yonedalemma}
	Let $\mathcal{A}$ be a c-unital $A_\infty$-category, $\mathsf{Q}=mod(\mathcal{A})$. Then the module homomorphism $\lambda_\mathcal{M}$ from \eqref{lambda} is actually a quasi-isomorphism, for any $\mathcal{M}\in\textup{obj}(\mathsf{Q})$ and $Y\in\textup{obj}(\mathcal{A})$. 
\end{Lem}

There is a way to relate Yoneda embeddings of distinct $A_\infty$-categories $\mathcal{A}$, $\mathcal{B}$. It makes use of the following construction.

\begin{Def}\label{pullback}									
	 A non-unital $A_\infty$-functor $\mathcal{G}:\mathcal{A}\rightarrow\mathcal{B}$ induces a non-unital (dg) \textbf{pullback functor}\index{pullback functor} $\mathcal{G}^*\coloneqq \mathcal{R}_{\mathcal{G}^{opp}}: nu\text{-}mod(\mathcal{B})\rightarrow nu\text{-}mod(\mathcal{A})$, defined as follows:
	\begin{itemize}[leftmargin=0.5cm]
		\renewcommand{\labelitemi}{\textendash}
		\item $\mathcal{M}\in\textup{obj}(nu\text{-}mod(\mathcal{B}))$ is mapped to the module $\mathcal{G}^*\mathcal{M}\in\textup{obj}(nu\text{-}mod(\mathcal{A}))$ given by $(\mathcal{G}^*\mathcal{M})(X) \coloneqq \mathcal{M}(\mathcal{G}(X))\in\textup{obj}(\mathsf{Ch})$ for $X\in\textup{obj}(\mathcal{A})$, with associated
		\[
		\mu_{\mathcal{G}^*\!\mathcal{M}}^d(b,a_{d-1},...,a_1) \coloneqq\mkern-24mu \sum_{\substack{1 \leq r\leq d \\ s_1+...+s_r=d}}\mkern-24mu\mu_{\mathcal{M}}^r(b, \mathcal{G}^{s_r}(a_{d-1},...,a_{d-s_r}),...,\mathcal{G}^{s_1}(a_{s_1},...,a_1))\;,
		\]
		\item $t\in\textup{hom}_{nu\text{-}mod(\mathcal{B})}(\mathcal{M}_0,\mathcal{M}_1)$ is mapped to the pre-module homomorphism $\mathcal{G}^*t\in\textup{hom}_{nu\text{-}mod(\mathcal{A})}(\mathcal{G}^*\mathcal{M}_0,\mathcal{G}^*\mathcal{M}_1)$ given by
		\[
		(\mathcal{G}^*t)^d(b,a_{d-1},...,a_1) \coloneqq\mkern-24mu \sum_{\substack{1 \leq r\leq d \\ s_1+...+s_r=d}}\mkern-24mu t^r(b, \mathcal{G}^{s_r}(a_{d-1},...,a_{d-s_r}),...,\mathcal{G}^{s_1}(a_{s_1},...,a_1))\;.
		\]   
	\end{itemize}
	If $\mathcal{G}$ is c-unital, then so is $\mathcal{G}^*$ (by Lemma \ref{cunitalcompo}), resulting in a strictly unital (dg) functor $mod(\mathcal{B})\rightarrow mod(\mathcal{A})$. If moreover $\mathcal{G}$ is a quasi-equivalence, then so is $\mathcal{G}^*$.
\end{Def}

Consider now a non-unital $A_\infty$-functor $\mathcal{F}:\mathcal{A}\rightarrow\mathcal{B}$. Then the respective Yoneda embeddings $\Upsilon_\mathcal{A}$ and $\Upsilon_\mathcal{B}$ are related via the natural transformation $T:\Upsilon_\mathcal{A}\rightarrow\mathcal{F}^*\circ\Upsilon_\mathcal{B}\circ\mathcal{F}$ in $nu\text{-}\!fun(\mathcal{A},\mathsf{Q})$, whose $T^1:\textup{hom}_\mathcal{A}(Y_0,Y_1)\rightarrow\textup{hom}_\mathsf{Q}(\Upsilon_\mathcal{A}(Y_0),\mathcal{F}^*\Upsilon_\mathcal{B}\mathcal{F}(Y_1))$ assigns to each morphism in $\mathcal{A}$ a module homorphism in $\mathsf{Q}$, itself a collection of graded maps
\vspace*{-0.5cm}
\begin{align}
	&(\Upsilon_\mathcal{A}(Y))(X_{d-1})\!\otimes\!\textup{hom}_\mathcal{A}(X_{d-2},X_{d-1})\!\otimes\!...\!\otimes\!\textup{hom}_\mathcal{A}(X_0,X_1)\!\rightarrow\!\overbrace{(\mathcal{F}^*\Upsilon_\mathcal{B}\mathcal{F}(Y))(X_0)}^{= \textup{hom}_\mathcal{B}(\mathcal{F}(X_0),\mathcal{F}(Y))},\nonumber \\
	&(b,a_{d-1},...,a_1)\mapsto(-1)^{\dagger_{d-1}+|b|}\mathcal{F}^d(b,a_{d-1},...,a_1)\;.
\end{align}

\begin{Rem}\label{commutingYoneda}					
	If we assume that $\mathcal{A}$, $\mathcal{B}$ are c-unital and $\mathcal{F}$ is cohomologically full and faithful, it turns out that $T^1$ is an isomorphism of c-unital $A_\infty$-modules. Therefore, Lemma \ref{1.6} implies that $T$ is an isomorphism in the respective cohomological category, producing commutativity of the diagram
	\[
	\begin{tikzcd}
	\mathcal{A}\arrow[r, "\mathcal{F}"]\arrow[d, "\Upsilon_\mathcal{A}"'] & \mathcal{B}\arrow[d, "\Upsilon_\mathcal{B}"] \\
	{mod(\mathcal{A})} & {mod(\mathcal{B})}\arrow[l, "\mathcal{F}^*"]
	\end{tikzcd},
	\]
	up to canonical isomorphism in $H^0\big(fun(\mathcal{A},mod(\mathcal{A}))\big)$.
\end{Rem}
\vspace*{0.2cm}

\noindent\minibox[frame]{Henceforth, all $A_\infty$-categories, $A_\infty$-functors and $A_\infty$-modules are assumed to \\ be c-unital, unless otherwise stated.}

\newpage

\section{Triangulation}
\thispagestyle{plain}

\subsection{Category theoretic preliminaries}\label{ch3.1}

In this subsection we recall the definition of a triangulated category, adopting the concise version of \cite[section 1]{[May01]}, and later introduce the qualitative construction of derived categories. First some fundamentals.

\begin{Def}											
	\ \vspace*{0.0cm}
	\begin{itemize}[leftmargin=0.5cm]
		\item A (linear) category $\mathsf{C}$ is called \textit{additive} if it has a zero object, denoted 0 $\in\textup{obj}(\mathsf{C})$, and contains all finite products and coproducts, so that for any two $X, Y\in\textup{obj}(\mathsf{C})$ the canonical map $X\sqcup Y\rightarrow X \sqcap Y$ is an isomorphism and $\textup{Hom}_\mathsf{C}(X,Y)$ an abelian group.
		\item A functor $\mathsf{F}:\mathsf{C}\rightarrow\mathsf{D}$ between additive categories is \textit{additive} if it maps the zero object of $\mathsf{C}$ to that of $\mathsf{D}$, and if there are isomorphisms $\mathsf{F}(X\oplus Y)\cong \mathsf{F}(X)\oplus\mathsf{F}(Y)$ respecting inclusions and projections of direct summands, for any $X, Y\in\textup{obj}(\mathsf{C})$.
\end{itemize}
\end{Def}

Recall that an object $0$ of $\mathsf{C}$ is the zero object (then unique up to isomorphism) if both initial and terminal, that is, both $\textup{Hom}_\mathsf{C}(0,X)=\{\alpha_X\}$ and $\textup{Hom}_\mathsf{C}(X,0)=\{\omega_X\}$ possess just one morphism each, for all $X\in\textup{obj}(\mathsf{C})$.

\begin{Def}\label{Verdier}					
	Let $\mathsf{C}$ be an additive category. A \textbf{triangulation}\index{triangulation} on $\mathsf{C}$ is an additive equivalence $T:\mathsf{C}\rightarrow\mathsf{C}$ (with inverse $T^{-1}$), referred to as a \textbf{translation functor}\index{translation functor}, together with a collection of diagrams, called \textbf{distinguished triangles}\index{distinguished triangles}, of the form
	\[
	X\xrightarrow{f}Y\xrightarrow{g}Z\xrightarrow{h}T(X),
	\]
	for some $X, Y, Z\in\textup{obj}(\mathsf{C})$ and morphisms $f, g, h$ (where the target object of $h$ is the ``translated'' $T(X)\equiv X[1]\in\textup{obj}(\mathsf{C})$). When the objects are tacitly understood, we write $(f,g,h)$ for short. These give rise to periodic triangular diagrams --- whence their name --- of the form
	\[
	\begin{tikzcd}
	X\arrow[rr, "f"] & & Y\arrow[dl, "g"]\\
	& Z\arrow[ul, "{h[1]}"]
	\end{tikzcd}\quad,
	\]
	where $h[1]$ means that $h$ is actually a morphism $Z\rightarrow T(X)$. Distinguished triangles must satisfy the following axioms:
	\begin{itemize}[leftmargin=0.5cm]
		\renewcommand{\labelitemi}{\textendash}
		\item (\textbf{T1}) For any $X, Y\in\textup{obj}(\mathsf{C})$ and morphism $f:X\rightarrow Y$:
		\begin{enumerate}[leftmargin=0.5cm]
			\item $X\xrightarrow{\textup{id}_X}X\xrightarrow{\omega_X} 0\xrightarrow{\alpha_X} T(X)$ is distinguished (note that $\alpha_X$ and $\omega_X$ are unique);
			\item $f$ is part of some distinguished triangle $(f,g,h)$ as above (then we call $Y$ the \textit{extension of $Z$ by $X$}, and $Z$ the \textit{cone of $f$});
			\item any triangle isomorphic to a distinguished triangle $(f,g,h)$ is distinguished, that is, if $X\cong X', Y\cong Y', Z\cong Z'$ via $u,v$ respectively $w$, then $X'\xrightarrow{v\circ f\circ u^{-1}}Y'\xrightarrow{w\circ g\circ v^{-1}}Z'\xrightarrow{T(v)\circ h\circ w^{-1}}T(X')$ is distinguished.
		\end{enumerate}
	\item (\textbf{T2}) If $(f,g,h)$ is distinguished, so are $(g,h,-T(f))$ and $(-T^{-1}(h),f,g)$. These correspond to the ``counterclockwise and clockwise rotated'' triangles:
	\[
	\begin{tikzcd}
	Y\arrow[rr, "g"] & & Z\arrow[dl, "h"]\\
	& T(X)\arrow[ul, "{-T(f)[1]}"]
	\end{tikzcd}\quad\text{and}\quad
	\begin{tikzcd}
	T^{-1}(Z)\arrow[rr, "-T^{-1}(h)"] & & X\arrow[dl, "f"]\\
	& Y\arrow[ul, "{g[1]}"]
	\end{tikzcd}\quad.
	\]
	(This also explains the periodicity we mentioned above: by repeated application of $T$ and $T^{-1}$, we can circle around a triangle an indefinite number of times, obtaining the infinite chain $...\rightarrow T^{-1}(X)\rightarrow T^{-1}(Y)\rightarrow T^{-1}(Z)\rightarrow X\rightarrow Y\rightarrow Z\rightarrow T(X)\rightarrow T(Y)\rightarrow T(Z)\rightarrow...$ .)
	\item (\textbf{T3}) The diagram
	\[
	\begin{tikzcd}
	X \arrow[rr, bend left = 45, "h"']\arrow[dr, "f"] & & Z\arrow[rr, bend left = 45, "g'"']\arrow[dr, "h'"] & & W \arrow[rr, bend left = 45, "j''"']\arrow[dr, "g''"] & & \mkern-6mu T(U) \\
	& Y\arrow[ur, "g"]\arrow[dr, "f'"] & & V \arrow[ur, dashed, "j'"]\arrow[dr, "h''"] & & \mkern-6mu T(Y)\arrow[ur, "T(f')"] \\
	& & U\arrow[ur, dashed, "j"]\arrow[rr, bend right = 45, "f''"] & & \mkern-6mu T(X)\arrow[ur, "T(f)"]
	\end{tikzcd}
	\]
	made of distinguished triangles $(f,f',f'')$, $(g,g',g'')$, $(h,h',h'')$, where $h = g\circ f$ and $j''= T(f')\circ g''$, can be completed to a commutative diagram via dashed arrows $j$ and $j'$ forming a distinguished triangle $(j,j',j'')$.
	\end{itemize}
	Then $\mathsf{C}=(\mathsf{C},T_\mathsf{C})$ is said to be a \textbf{triangulated category}\index{triangulated category}.	
\end{Def}

Jean-Louis Verdier originally identified four axioms for triangulated categories, which he labeled (TR1)--(TR4) (see \cite[section II.1]{[Ver96]}). His (TR1) coincide with (T1) above, (TR2) with (T2) and (TR4) with (T3). The latter is known as \textbf{Verdier's axiom}\index{Verdier's axiom}, or \textit{octahedral axiom} (after one possible reshaping of the commutative braid we pictured). (TR3) can actually be deduced from the other axioms; we will need it later on in its $A_\infty$-categorical formulation, thus we show its derivation.

\begin{Lem}[TR3]\label{tr3}								
	Assume the rows of the diagram
	\[
	\begin{tikzcd}
	X\arrow[r,"f"]\arrow[d,"i"] & Y \arrow[r,"g"]\arrow[d,"j"] & Z \arrow[r,"h"]\arrow[d, dashed, "k"] & T(X)\arrow[d, "T(i)"] \\
	X'\arrow[r, "f'"] & Y'\arrow[r, "g'"] & Z'\arrow[r, "h'"] & T(X')
	\end{tikzcd}
	\]
	in $\mathsf{C}$ are two distinguished triangles and the left square commutes. Then there is a dashed arrow $k$ making the whole diagram commute; moreover, if $i$ and $j$ are isomorphisms, so is $k$.
\end{Lem}

\begin{proof}
	We apply point 2. of axiom (T1) to the morphisms $i$, $j$, $j\circ f$, finding distinguished triangles $X\xrightarrow{i}X'\xrightarrow{i'}X''\xrightarrow{i''}T(X)$, $Y\xrightarrow{j}Y'\xrightarrow{j'}Y''\xrightarrow{j''}T(Y)$ and $X\xrightarrow{j\circ f}Y'\xrightarrow{p}V\xrightarrow{q}T(X)$. Then we use (T3) twice with input data $(f,g,h)$, $(j,j',j'')$, $(j\circ f,p,q)$ respectively $(i,i',i'')$, $(f',g',h')$, $(f'\circ i,p,q)$ to obtain distinguished triangles 
	\[
	Z\xrightarrow{s}V\xrightarrow{t}Y''\xrightarrow{T(g)\circ j''}T(Z) \quad\text{resp.}\quad X''\xrightarrow{s'}V\xrightarrow{t'}Z'\xrightarrow{T(i')\circ h'}T(X'')\,.
	\]
	Now, it is just a matter of commutativity of the two resulting braids to show that $k\coloneqq t'\circ s: Z\rightarrow Z'$ does the job: $k\circ g = t'\circ(s\circ g) = t'\circ(p\circ j) = g'\circ j$ and $h'\circ k = (h'\circ t')\circ s = (T(i)\circ q)\circ s = T(i)\circ h$.
	
	Moreover, if both $i$ and $j$ are isomorphisms, then the equivalence $T$ preserves them. By axiom (T2) we can legitimately extend the rows to $T(Y)$ and $T(Y')$, via $-T(f)$ respectively $-T(f')$. Since, as established, the vertical arrows $i$, $j$, $T(i)$ and $T(j):T(Y)\rightarrow T(Y')$ are isomorphisms, the Five Lemma gives the desired result. 
\end{proof}

Actually, the ``clockwise rotated'' triangle in (T2) can itself be deduced from the remaining axioms, which are then the sufficient requirements for a category to be triangulated (more on this in \cite{[May01]}). Finally, we include some terminology that will ring a bell later on.

\begin{Def}\label{triangfunc}					
	Let $(\mathsf{C},T_\mathsf{C})$ and $(\mathsf{D},T_\mathsf{D})$ be triangulated categories. 
	\begin{itemize}[leftmargin=0.5cm]
	\item A functor $\mathsf{F}:\mathsf{C}\rightarrow\mathsf{D}$ is called \textbf{triangulated}\index{functor!triangulated/exact} (or \textit{exact}) if it commutes with translations and maps distinguished triangles $X\xrightarrow{f}Y\xrightarrow{g}Z\xrightarrow{h}T_\mathsf{C}(X)$ in $\mathsf{C}$ to distinguished triangles $\mathsf{F}(X)\xrightarrow{\mathsf{F}(f)}\mathsf{F}(Y)\xrightarrow{\mathsf{F}(g)}\mathsf{F}(Z)\xrightarrow{\mathsf{F}(h)}\mathsf{F}(T_\mathsf{C}(X))=T_\mathsf{D}(\mathsf{F}(X))$ in~$\mathsf{D}$.
	\item A triangle $(f,g,h)$ in $\mathsf{C}$ is called \textbf{exact}\index{triangle!exact} if it induces long exact sequences (of abelian groups) upon application of the functors $\textup{Hom}_\mathsf{C}(\cdot,C)$ and $\textup{Hom}_\mathsf{C}(C,\cdot)$, for every $C\in\textup{obj}(\mathsf{C})$.
	\end{itemize}
\end{Def}

The subsequent lemma is yet another exercise in applying (T1)--(T3); a proof is given by Verdier (cf. \cite[section II.1.2.1]{[Ver96]}).

\begin{Lem}										
	Every distinguished triangle $(f,g,h)$ is exact, thus yielding long exact sequences which, written as periodic triangular diagrams, have form:
	\[
	\begin{tikzcd}
	\textup{Hom}_\mathsf{C}(C,X)\arrow[rr, "{\textup{Hom}_\mathsf{C}(C,f)}"] & & {\textup{Hom}_\mathsf{C}(C,Y)}\arrow[dl, "{\textup{Hom}_\mathsf{C}(C,g)}"]\\
	& {\textup{Hom}_\mathsf{C}(C,Z)}\arrow[ul, "{\textup{Hom}_\mathsf{C}(C,h)[1]}"]
	\end{tikzcd}\quad,
	\]
	\[
	\begin{tikzcd}
	\textup{Hom}_\mathsf{C}(X,C)\arrow[dr, "{\textup{Hom}_\mathsf{C}(h,C)[1]}"'] & & {\textup{Hom}_\mathsf{C}(Y,C)}\arrow[ll, "{\textup{Hom}_\mathsf{C}(f,C)}"']\\
	& {\textup{Hom}_\mathsf{C}(Z,C)}\arrow[ur, "{\textup{Hom}_\mathsf{C}(g,C)}"']
	\end{tikzcd}\quad.
	\]
\end{Lem}

\newpage

Before returning to $A_\infty$-categories, we illustrate --- on a very superficial level --- the construction of derived categories, which are particular instances of triangulated categories. Our reference is again \cite[sections 4.1, 4.2]{[Kel01]}. However, the following procedure will \textit{not} be generalized to $A_\infty$-categories; it merely serves as a build-up towards Theorem \ref{DAtriangulated} below.

On the most fundamental level, given an associative unital $\mathbb{K}$-algebra $A$, we use its category $\mathsf{Kom}(A)$ of cochain complexes to construct the associated homotopy category $\mathsf{H}(A)$ as follows: $\textup{obj}(\mathsf{H}(A)) \coloneqq \textup{obj}(\mathsf{Kom}(A))$, and morphisms sets of $\mathsf{H}(A)$ are those of $\mathsf{Kom}(A)$ after quotienting by all null-homotopic morphisms, which are the morphisms chain homotopic to the trivial one. From this, we define the derived category $\mathsf{D}(A)$ to be the ``localization'' of $\mathsf{H}(A)$ with respect to all its quasi-isomorphisms, achievable by ``formally inverting'' them.

Now, defining suspensions of complexes as we did back in Section \ref{ch1.1} (namely, given any $K\in\textup{obj}(\mathsf{D}(A))=\textup{obj}(\mathsf{Kom}(A))$, set $(SK)^p\coloneqq K^{p+1}$ and $d_{SK}\coloneqq -d_K$), one finds that any short exact sequence $\{0\}\rightarrow K\rightarrow L\rightarrow M\rightarrow \{0\}$ of complexes induces a triangle $K\rightarrow L\rightarrow M\rightarrow SK$ satisfying the axioms (T1)--(T3), and all triangles in $\mathsf{D}(A)$ are isomorphic to one such. This makes $\mathsf{D}(A)$ a triangulated category.

Let us see a concrete example involving the category $\mathsf{C}_\infty(A)$ of $A_\infty$-modules from Section \ref{ch1.2}.\footnote{Fact: if $A$ is just an associative algebra, then one can construct a faithful functor $\mathsf{Kom}(A)\rightarrow\mathsf{C}_\infty (A)$.} 

\begin{Ex}							
	Let $A$ be an $A_\infty$-algebra. Given two $A_\infty$-modules $M, N$ over $A$, recall that a morphism of $A_\infty$-modules $f: M\rightarrow N$ is a sequence of graded maps $f_n: M\otimes A^{\otimes n-1}\rightarrow N[1-n]$ for $n\geq 1$; it is a quasi-isomorphism if and only if $f_1:M\rightarrow N$ is. By Definition \ref{Ainfmodmorph}, $f$ is nullhomotopic if and only if there is a sequence of graded maps $h_n: M\otimes A^{\otimes n-1}\rightarrow N[-n]$ such that
	\[
	f_n = \mkern-6mu\sum_{s+t=n}\mkern-6mu(-1)^{st}m_{1+t}^N(h_s\otimes \textup{id}_{A}^{\otimes t}) + \mkern-12mu\sum_{r+s+t=n}\mkern-12mu(-1)^{r+st}h_{r+1+t}(\textup{id}_{A}^{\otimes r}\otimes m_s^M\otimes \textup{id}_{A}^{\otimes t})\;.
	\]  
	One checks that this defines an equivalence relation $\sim$. As explained above, the homotopy category $\mathsf{H}_\infty (A)$ of $A$ is given by setting $\textup{obj}(\mathsf{H}_\infty (A))\coloneqq \textup{obj}(\mathsf{C}_\infty (A))$ and $\textup{Hom}_{\mathsf{H}_\infty (A)}(M,N)\coloneqq \textup{Hom}_{\mathsf{C}_\infty (A)}(M,N)/_\sim$ for each $M, N\in\textup{obj}(\mathsf{C}_\infty (A))$. It turns out (see for example \cite[section 2.4.2]{[Lef03]}) that all quasi-isomorphims of $\mathsf{H}_\infty (A)$ are already invertible, so that $\mathsf{D}_\infty (A)\coloneqq \mathsf{H}_\infty (A)$ is readily the derived category of $A$.
	
	Moreover, defining the suspension functor $S:\mathsf{D}_\infty(A)\rightarrow\mathsf{D}_\infty(A)$ by $(SM)^d\coloneqq M^{d+1}$ and $(m_{SM})_n\coloneqq (-1)^n (m_{M})_n$, one can show that each short exact sequence \vspace*{-0.2cm}
	\[
	\{0\}\rightarrow K\xrightarrow{i} L\xrightarrow{p} M\rightarrow \{0\}
	\]
	(for morphisms $i, p$ with $i_n=p_n=0$ if $n>1$) yields a canonical triangle in $\mathsf{D}_\infty (A)$, and that all triangles in $\mathsf{D}_\infty (A)$ are isomorphic to one such (consult again \cite{[Lef03]}). This shows that $\mathsf{D}_\infty (A)$ is a triangulated category. \hfill $\blacklozenge$
\end{Ex}

Attempting to formalize the procedure to generic ($A_\infty$-)categories would take us too far afield. We limit ourselves to give the main steps for \textit{abelian categories}, which are roughly additive categories containing kernel and cokernel of any morphism (see \cite[section II.5]{[GM03]}; the prototype is the category of abelian groups $\mathsf{Ab}$), so that we have well-defined notions of exact sequences and cohomology. A more rigorous treatment can be found in \cite[chapter III]{[GM03]}. 

\begin{Def}\label{derivedcat}					
	Let $\mathsf{A}$ be an abelian category. Its derived category is constructed as follows:
	\begin{enumerate}[leftmargin=0.5cm]
		\item Take the \textit{category} $\mathsf{Kom(A)}$ \textit{of cochain complexes} in $\mathsf{A}$, whose objects are cochain complexes $X^\bullet$ of the form
		\[
		...\rightarrow X^{-1}\xrightarrow{d_{-1}} X^0\xrightarrow{d_0} X^1\xrightarrow{d_1}X^2\rightarrow...\;,
		\]
		for $\{X^k\}_{k\in\mathbb{Z}}\subset\textup{obj}(\mathsf{A})$ and morphisms $\{d_k\}_{k\in\mathbb{Z}}$ in $\mathsf{A}$ such that $d_{k+1}\circ d_k=0$, and whose morphisms are standard chain maps between cochain complexes;
		\item Define the associated \textit{homotopy category $\mathsf{H(A)}$ of cochain complexes} by\break $\textup{obj}(\mathsf{H(A)})\coloneqq \textup{obj}(\mathsf{Kom(A)})$ and $\textup{Hom}_{\mathsf{H(A)}}(X^\bullet,Y^\bullet)\coloneqq \textup{Hom}_{\mathsf{Kom(A)}}(X^\bullet,Y^\bullet)/_\sim$ for any two $X^\bullet,Y^\bullet\in\textup{obj}(\mathsf{H(A)})$, where two chain maps $f,g$ are related, written $f\sim g$, if and only if they are chain homotopic (in the standard sense of homological algebra);
		\item Characterize the \textbf{derived category}\index{derived category} $\mathsf{D(A)}$ of $\mathsf{A}$ as solution of a universal problem:
		\begin{enumerate}[leftmargin=0.6cm]
			\item Pick the class $\mathsf{S}$ in $\mathsf{H(A)}$ of all quasi-isomorphisms, which are those (classes) of chain maps $f:X^\bullet\rightarrow Y^\bullet$ inducing isomorphisms $H^n(X^\bullet)\cong H^n(Y^\bullet)$ in cohomology for all $n\in\mathbb{Z}$;
			\item Observe that $\mathsf{S}$ is a localizing class of morphisms of $\mathsf{H(A)}$:\\ $\textup{id}_{X^\bullet}\!\in\mathsf{S}$ for all $X^\bullet\in\textup{obj}(\mathsf{H(A)})$, $\big(s,t\in\mathsf{S}\Rightarrow s\,\circ\, t\in\mathsf{S}\big)$ (whenever$~$defined),\\ $\big(\forall\, X^\bullet\!\xrightarrow{f}\!~Y^\bullet \text{ in }\mathsf{H(A)},\,Z^\bullet\!\xrightarrow{s}\! Y^\bullet \text{ in }\mathsf{S}\;\, \exists\,\;W^\bullet\!\xrightarrow{g}\! Z^\bullet \text{ in }\mathsf{H(A)},\,W^\bullet\!\xrightarrow{t}\! X^\bullet \text{ in }\mathsf{S}$ s.t. $f\circ t=s\circ g\big)$ and $\big(\forall f,g:X^\bullet\rightarrow Y^\bullet \text{ in } \mathsf{H(A)}\!\!: \exists s\in\mathsf{S} \text{ s.t. } s\circ f=s\circ g \iff \exists t\in\mathsf{S}\, \text{ s.t. } f\circ t=g\circ t\big)$;
			\item Define a category $\mathsf{H(A)[S^{-1}]}$ with $\textup{obj}(\mathsf{H(A)[S^{-1}]})\coloneqq\textup{obj}(\mathsf{H(A)})$ and morphisms modelled on those of $\mathsf{S}$, and a functor $\mathsf{L}:\mathsf{H(A)}\rightarrow\mathsf{H(A)[S^{-1}]}$ (the ``localization by $\mathsf{S}$'') such that $\mathsf{L}|_{\mathsf{H(A)}}=\mathsf{Id_{H(A)}}$ (see \cite[subsection III.2.2]{[GM03]});
			\item Define $\mathsf{D(A)}\coloneqq \mathsf{H(A)[S^{-1}]}$, then satisfying $\big(f: X^\bullet\rightarrow Y^\bullet$ quasi-isomor-\break phism in $\mathsf{H(A)} \implies \mathsf{L}(f):X^\bullet\rightarrow Y^\bullet$ isomorphism in  $\mathsf{D(A)}\big)$ and having the universal property: for any other $\big(\mathsf{B},\,\mathsf{L}':\mathsf{H(A)}\rightarrow\mathsf{B}\big)$ with the same property, there exists a unique functor $\mathsf{F}:\mathsf{D(A)}\rightarrow\mathsf{B}$ such that $\mathsf{F}\circ\mathsf{L}=\mathsf{L}'$.  
		\end{enumerate}
		In practice, $\textup{obj}(\mathsf{D(A)})=\textup{obj}(\mathsf{Kom(A)})$ and morphisms $X^\bullet\rightarrow Y^\bullet$ in $\mathsf{D(A)}$ are (classes of) ``roofs'' $(s,f)$, that is, diagrams $\tilde{X}^\bullet\xleftarrow{s}X^\bullet\xrightarrow{f}Y^\bullet$ with ${f\in\textup{Hom}_\mathsf{H(A)}(X^\bullet,Y^\bullet)}$ and $s\in\mathsf{S}$.
	\end{enumerate}
	Considering only the bounded cochain complexes in $\mathsf{Kom(A)}$, one obtains the \textbf{bounded derived category}\index{bounded derived category} $\mathsf{D^b(A)}$ of $\mathsf{A}$. 
\end{Def}

Next, one proceeds to adapt the concept of distinguished triangles in $\mathsf{D(A)}$. The property of being exact is more complicated to reshape, because $\mathsf{D(A)}$ turns out to be just an additive category, not abelian. In contrast, corresponding definitions of cones and translations are immediate to produce, allowing one to show that $\mathsf{H(A)}$ is always triangulated (proved in \cite[Theorem IV.1.9]{[GM03]}). What we are most interested in is the following result (see \cite[Corollary IV.2.7]{[GM03]}).

\begin{Thm}\label{DAtriangulated}				
	Let $\mathsf{A}$ be an additive category, with associated triangulated homotopy category $(\mathsf{H(A)},T)$. The localizing class $\mathsf{S}$ defining its derived category $\mathsf{D(A)}=\mathsf{H(A)[S^{-1}]}$ is compatible with triangulation, that is, it satisfies $(s\in\mathsf{S}\iff T(s)\in\mathsf{S})$ and $(i,j\in\mathsf{S}$ as in Lemma \ref{tr3} $\implies k\in\mathsf{S})$.
	
	If we define distinguished triangles in $\mathsf{D(A)}$ to be those isomorphic to images of distinguished triangles in $\mathsf{H(A)}$ under localization by $\mathsf{L}:\mathsf{H(A)}\rightarrow\mathsf{H(A)[S^{-1}]}$, it follows that $(\mathsf{D(A)},T)$ is a triangulated category \textup(the same holds for $\mathsf{D^b(A)}$\textup).
\end{Thm}

\subsection{Operations on $A_\infty$-modules}\label{ch3.2}

In the framework of $A_\infty$-categories, the constructions of mapping cones and ``shifted'' objects we will produce rely on our ability to manipulate $A_\infty$-modules through basic operations such as direct sums and tensor products. First, an important definition. 

\begin{Def}\label{quasirepresentative}					
	Let $\mathcal{A}$ be an $A_\infty$-category, $\mathsf{Q} = mod(\mathcal{A})=fun(\mathcal{A}^{opp},\mathsf{Ch})$. Suppose that for a given $A_\infty$-module $\mathcal{M}\in\textup{obj}(\mathsf{Q})$ there is some $Y\in\textup{obj}(\mathcal{A})$ such that $\mathcal{Y}=\Upsilon(Y)\in\textup{obj}(\mathsf{Q})$ is isomorphic to $\mathcal{M}$ in $H^0(\mathsf{Q})$, say via an isomorphism $\langle t\rangle\in\textup{Hom}_{H(\mathsf{Q})}^0(\mathcal{Y},\mathcal{M})$. Then we say that $(Y,\langle t\rangle)$ \textbf{quasi-represents}\index{quasi-represent} $\mathcal{M}$ (or is a \textbf{quasi-representative}\index{quasi-representative} of $\mathcal{M}$; sometimes $\langle t\rangle$ is omitted from the notation).
\end{Def}

\begin{Rem}									
	Suppose $\mathcal{M}$ is quasi-represented by both $(Y,\langle t\rangle)$ and $(Y',\langle t'\rangle)$, which yields an isomorphism $\langle t'\rangle^{-1}\circ\langle t\rangle\in\textup{Hom}_{H(\mathsf{Q})}^0(\mathcal{Y},\mathcal{Y'})$. Since, as pointed out in Section \ref{ch2.9}, the Yoneda embedding is cohomologically full and faithful, meaning that $H(\Upsilon)$ is bijective on morphisms, there is a unique isomorphism $\langle a\rangle\in\textup{Hom}_{H(\mathcal{A})}^0(Y,Y')$ such that $H(\Upsilon)\langle a\rangle = \langle \Upsilon^1(a)\rangle = \langle t'\rangle^{-1}\circ\langle t\rangle$, thus satisfying $\langle t'\rangle\circ\langle \Upsilon^1(a)\rangle = \langle t\rangle$.
\end{Rem}	

An explicit criterion for quasi-representability is the following (refer to \cite[Lemma 2.1]{[Sei08]} for a proof). 

\begin{Lem}\label{quasi-repr}			
	A pair $(Y,\langle t\rangle)$ quasi-represents $\mathcal{M}$ if and only if there is a degree $0$ cocycle $c\in\mathcal{M}(Y)^0$ such that $\langle t\rangle = \langle\lambda(c)\rangle$ and 
	\[
	t_{X}^1: \mathcal{Y}(X)\rightarrow\mathcal{M}(X),\,b\mapsto (-1)^{|b|}\mu_\mathcal{M}^2(c,b)
	\]
	is a quasi-isomorphism in $\mathsf{Ch}$ for all $X\in\textup{obj}(\mathcal{A})$. 
	Here, $\lambda = \lambda_\mathcal{M}: \mathcal{M}(Y)\rightarrow\textup{hom}_\mathsf{Q}(\mathcal{Y},\mathcal{M})$ is the map \eqref{lambda}. By choice of $c$, $\lambda(c)$ is a module homomorphism.
\end{Lem}

Let us now illustrate the promised operations on $A_\infty$-modules.

\begin{Def}												
	Given two $\mathcal{M}_0, \mathcal{M}_1\in\textup{obj}(\mathsf{Q})$, the \textbf{direct sum $A_\infty$-module}\index{aaa@$A_\infty$-module!direct sum} $\mathcal{M}_0\oplus\mathcal{M}_1\in\textup{obj}(\mathsf{Q})$ is obtained by trivially setting 
	\[
	(\mathcal{M}_0\oplus\mathcal{M}_1)(X)\coloneqq \mathcal{M}_0(X)\oplus\mathcal{M}_1(X)\quad \text{and}\quad \mu_{\mathcal{M}_0\oplus\mathcal{M}_1}^d \coloneqq  \mu_{\mathcal{M}_0}^d\oplus \mu_{\mathcal{M}_1}^d\;,
	\]
	for $X\in\textup{obj}(\mathcal{A})$ and $d\geq 1$.
	We write $Y_0\oplus Y_1\in\textup{obj}(\mathcal{A})$ for any object quasi-representing $\mathcal{Y}_0\oplus\mathcal{Y}_1\in\textup{obj}(\mathsf{Q})$.\footnote{So this must \textit{not} be misinterpreted as some sort of operation between objects! Same warning holds for the tensor product.} 
	
	By Lemma \ref{quasi-repr} and definition of $\Upsilon$, for any $X\in\textup{obj}(\mathcal{A})$ we obtain isomorphisms $\textup{Hom}_{H(\mathcal{A})}(X,Y_0\oplus Y_1)=H(\mathcal{Y}(Y_0\oplus Y_1)(X))\cong H((\mathcal{Y}_0\oplus\mathcal{Y}_1)(X))\cong \textup{Hom}_{H(\mathcal{A})}(X,Y_0)\oplus\textup{Hom}_{H(\mathcal{A})}(X,Y_1)$.
\end{Def} 

\begin{Def}										
	Let $\mathcal{M}\in\textup{obj}(\mathsf{Q})$, $Z = (Z,\partial_Z)\in\textup{obj}(\mathsf{Ch})$. Then the \textbf{tensor product $A_\infty$-module}\index{aaa@$A_\infty$-module!tensor product} $Z\otimes\mathcal{M}\in\textup{obj}(\mathsf{Q})$ is obtained by setting: 
	\begin{align*}
	&(Z\otimes\mathcal{M})(X_0) \coloneqq  Z\otimes\mathcal{M}(X_0)\qquad \text{(indeed a graded vector space),}\\
	&\mu_{Z\otimes\mathcal{M}}^1(z\otimes b)\coloneqq (-1)^{|b|-1}\partial_Z(z)\otimes b + z\otimes\mu_\mathcal{M}^1(b)\in(Z\otimes\mathcal{M})(X_0)[1]\;, \\
	&\mu_{Z\otimes\mathcal{M}}^d(z\otimes b, a_{d-1},...,a_1)\coloneqq  z\otimes\mu_\mathcal{M}^d(b,a_{d-1},...,a_1)\in(Z\otimes\mathcal{M})(X_0)[2-d]\;, 
	\end{align*}
	for $X_0\in\textup{obj}(\mathcal{A})$ and $d>1$.
	We write $Z\otimes Y\in\textup{obj}(\mathcal{A})$ for any object quasi-representing $Z\otimes\mathcal{Y}\in\textup{obj}(\mathsf{Q})$. 
	
	By Lemma \ref{quasi-repr} and definition of $\Upsilon$, for any $X\in\textup{obj}(\mathcal{A})$ we get the isomorphisms $\textup{Hom}_{H(\mathcal{A})}(X,Z\otimes Y)=H(\mathcal{Y}(Z\otimes Y)(X))\cong H((Z\otimes\mathcal{Y})(X))\cong$\break $H(Z)\otimes\textup{Hom}_{H(\mathcal{A})}(X,Y)$.
\end{Def}

\begin{Def}\label{shifts}									
	Take $Z = \mathbb{K}[\sigma]$ in last definition, for some $\sigma\in\mathbb{Z}$ (this corresponds to the one-dimensonal vector space $\mathbb{K}$ placed in degree $-\sigma$). Given $Y\in\textup{obj}(\mathcal{A})$ quasi-representing $\mathcal{M}\in\textup{obj}(\mathsf{Q})$, we write $S^\sigma Y\coloneqq \mathbb{K}[\sigma]\otimes Y$ for the quasi representative of $SS^\sigma \mathcal{M}\coloneqq  \mathbb{K}[\sigma]\otimes\mathcal{M}$, both called \textbf{$\sigma$-fold shifts}\index{sigma@$\sigma$-fold shift}, shortened $SY$ respectively  $SS\mathcal{M}$ if $\sigma = 1$.
\end{Def}

Notice that $Z\otimes\square: \mathsf{Q}\rightarrow\mathsf{Q}$ defines an $A_\infty$-functor assigning pre-module homomorphisms $t\in\textup{hom}_\mathsf{Q}(\mathcal{M}_0,\mathcal{M}_1)$ to $\textup{id}_Z\otimes t\in\textup{hom}_\mathsf{Q}(Z\otimes\mathcal{M}_0,Z\otimes\mathcal{M}_1)$ given by $(\textup{id}_Z\otimes t)^d(z\otimes b,a_{d-1},...)\coloneqq (-1)^{|z|(|t|-1)}z\otimes t^d(b,a_{d-1},...)$, and fulfilling $(Z\otimes\square)^d =  0$ if $d>1$. In particular, $SS^\sigma \coloneqq  \mathbb{K}[\sigma]\otimes\square: \mathsf{Q}\rightarrow\mathsf{Q}$ itself is an $A_\infty$-functor.

When they exist, the shifts of Definition \ref{shifts} come with canonical isomorphisms:
\begin{align}\label{Homiso}
	& \textup{Hom}_{H(\mathcal{A})}(Y_0,S^\sigma Y_1)\cong \mathbb{K}[\sigma]\otimes\textup{Hom}_{H(\mathcal{A})}(Y_0,Y_1)\cong \textup{Hom}_{H(\mathcal{A})}(Y_0,Y_1)[\sigma]\;, \nonumber \\
	& \textup{Hom}_{H(\mathcal{A})}(S^\sigma Y_0,S^\sigma Y_1)\cong\textup{Hom}_{H(\mathcal{A})}(Y_0,Y_1)\;,\\
	& \textup{Hom}_{H(\mathcal{A})}(S^\sigma Y_0,Y_1)[\sigma]\cong \textup{Hom}_{H(\mathcal{A})}(S^\sigma Y_0,S^\sigma Y_1)\cong\textup{Hom}_{H(\mathcal{A})}(Y_0,Y_1)\;. \nonumber
\end{align}
For the second one, we applied functoriality of $H(SS^\sigma)$.

\begin{Def}\label{shiftfunc}						
	Let $\mathcal{A}$ be an $A_\infty$-category such that $SY\in\textup{obj}(\mathcal{A})$ for all objects $Y$ of $\mathcal{A}$. Then we define the \textbf{$A_\infty$-shift functor}\index{aaa@$A_\infty$-functor!shift} $S:\mathcal{A}\rightarrow\mathcal{A}$ to be the unique $\big($up to isomorphism in $H^0(fun(\mathcal{A},\mathcal{A}))\big)$ $A_\infty$-functor making the diagram
	\[
	\begin{tikzcd}
	\mathcal{A}\arrow[r, "\Upsilon"]\arrow[d, "S"'] & \mathsf{Q}\arrow[d, "SS"] \\
	\mathcal{A}\arrow[r, yshift=0.5ex, "\Upsilon"] & \mathsf{Q}\arrow[l, dashed, yshift=-0.5ex, "\mathcal{K}"]
	\end{tikzcd}
	\]
	commute, up to isomorphism in $H^0(fun(\mathcal{A},\mathsf{Q}))$.
\end{Def}

In fact, we can be more explicit: let $\tilde{\mathcal{A}}\subset\mathcal{A}$ be the full $A_\infty$-subcategory of all ``shifted'' objects $SY$, $\tilde{\mathsf{Q}}\subset\mathsf{Q}$ that of $A_\infty$-modules isomorphic to the ``shifted'' $SS\mathcal{Y}$'s. Then $\Upsilon|_{\tilde{\mathcal{A}}}:\tilde{\mathcal{A}}\rightarrow\tilde{\mathsf{Q}}$ is cohomologically essentially surjective, full and faithful, thus a quasi-equivalence which can be inverted to some $\mathcal{K}$ as in the diagram above, by virtue of Theorem \ref{aboutquasiequiv}. We then set $S\coloneqq \mathcal{K}\circ SS\circ \Upsilon$.

\subsection{Mapping cones and exact triangles}\label{ch3.3}

We now mimic the construction of standard mapping cones from homological algebra. Here, $\mathcal{A}$ and $\mathsf{Q}$ are defined as usual.

\begin{Def}\label{abstractmappingcone}			
	Let $Y_0, Y_1\in\textup{obj}(\mathcal{A})$, $c\in\textup{hom}_\mathcal{A}^0(Y_0,Y_1)$ a degree 0 cocyle. The \textbf{abstract mapping cone}\index{abstract mapping cone} of $c$ is the $A_\infty$-module $\mathcal{C} = \mathcal{C}one(c)\in\textup{obj}(\mathsf{Q})$ defined by $\mathcal{C}(X) \coloneqq \textup{hom}_\mathcal{A}(X,Y_0)[1]\oplus\textup{hom}_\mathcal{A}(X,Y_1)\in\textup{obj}(\mathsf{Ch})$ for $X\in\textup{obj}(\mathcal{A})$, with graded maps
	\begin{align}
	\mu_\mathcal{C}^d\!:\;&\mathcal{C}(X_{d-1})\!\otimes\!\textup{hom}_\mathcal{A}(X_{d-2},X_{d-1})\!\otimes\!...\!\otimes\!\textup{hom}_\mathcal{A}(X_0,X_1)\rightarrow\mathcal{C}(X_0)[2-d] \\
	&\!((b_0,b_1),a_{d-1},...) \!\mapsto\! \big(\mu_\mathcal{A}^d(b_0,a_{d-1},...),\mu_\mathcal{A}^d(b_1,a_{d-1},...)\!+\!\mu_\mathcal{A}^{d+1}(c,b_0,a_{d-1},...)\big),\nonumber
	\end{align}
	for $d\geq1$ and $X_0,...,X_{d-1}\in\textup{obj}(\mathcal{A})$. In particular, $\mathcal{C}(X)$ is a cochain complex whose differential is: 
	\[
	\mu_\mathcal{C}^1 =\begin{pmatrix}\mu_\mathcal{A}^1 & 0 \\ \mu_\mathcal{A}^2(c,\cdot) & \mu_\mathcal{A}^1\end{pmatrix}\,.
	\]
	We denote by $Cone(c)\in\textup{obj}(\mathcal{A})$ any quasi-representative of $\mathcal{C}$, and call it the \textbf{mapping cone}\index{mapping cone} of $c$.
\end{Def} 

We check that $\mathcal{C}$ is indeed an object of $\mathsf{Q}$. First, the associativity equations \eqref{Ainfmodeqs} are satisfied (again, a non-trivial fact). Moreover, $\mathcal{C}$ is c-unital, by Definition \ref{mod}: given a c-unit $e_X\in\textup{hom}_\mathcal{A}^0(X,X)$, the map $(b_0,b_1)\mapsto\mu_\mathcal{C}^2((b_0,b_1),e_X) = (\mu_\mathcal{A}^2(b_0,e_X),\mu_\mathcal{A}^2(b_1,e_X)+\mu_\mathcal{A}^3(c,b_0,e_X)) = (b_0,b_1)$ clearly induces the identity in cohomology. Therefore, $\mathcal{C}\in\textup{obj}(\mathsf{Q})$.

One can show that $\mathcal{C}one(c)$, thus also $Cone(c)$, is determined just by the class $\langle c\rangle$, up to non-canonical isomorphism. Most remarkably, $Cone(0)$ quasi-represents $SS\mathcal{Y}_0\oplus\mathcal{Y}_1$, hence $\mathcal{C}one(0)\cong SS\mathcal{Y}_0\oplus\mathcal{Y}_1$, in analogy to the standard definition of mapping cones. 

\begin{Def}\label{iotaandpi}							
	Any abstract mapping cone $\mathcal{C}$ of some degree 0 cocycle $c\in\textup{hom}_\mathcal{A}^0(Y_0,Y_1)$ comes with two canonical module homomorphisms:
	\begin{itemize}[leftmargin=0.5cm]
		\renewcommand{\labelitemi}{\textendash}
		\item an ``inclusion'' into the second summand $\iota\in\textup{hom}_\mathsf{Q}^0(\mathcal{Y}_1,\mathcal{C})$, which is given by $\iota^1:\mathcal{Y}_1(X_0)\rightarrow\mathcal{C}(X_0),\; b_1\mapsto(0,(-1)^{|b_1|}b_1)$ for $X_0\in\textup{obj}(\mathcal{A})$, and $\iota^d \coloneqq 0$ if $d>1$;
		\item a ``projection'' onto the first summand $\pi\in\textup{hom}_\mathsf{Q}^1(\mathcal{C},\mathcal{Y}_0)$, which is given by  $\pi^1:\mathcal{C}(X_0)\rightarrow\mathcal{Y}_0(X_0)[1],\; (b_0,b_1)\mapsto(-1)^{|b_0|-1}b_0$ for $X_0\in\textup{obj}(\mathcal{A})$, and $\pi^d\coloneqq 0$ if $d>1$.
	\end{itemize}
	These fit into the periodic triangular diagram
	\begin{equation}\label{Qtriangle}
	\begin{tikzcd}
	\mathcal{Y}_0\arrow[rr, "\langle \Upsilon^1(c)\rangle"]& &\mathcal{Y}_1\arrow[dl, "\langle\iota\rangle"] \\
	& \mathcal{C}\arrow[ul, "{\langle\pi\rangle[1]}"]
	\end{tikzcd}
	\end{equation}
	in $H(\mathsf{Q})$. Here we adopted the same notation as in Section \ref{ch3.1} (in this case, [1] really just points out that the degree of $\pi$ is 1).
\end{Def}

\begin{Lem}\label{iandpi}							
	Let $e_{Y_0}\in\textup{hom}_\mathcal{A}^0(Y_0,Y_0)$ and $e_{Y_1}\in\textup{hom}_\mathcal{A}^0(Y_1,Y_1)$ be c-units. Then the module homomorphisms
	\begin{itemize}[leftmargin=0.5cm]
		\renewcommand{\labelitemi}{\textendash}
		\item $\tilde{\iota}\in\textup{hom}_\mathsf{Q}^0(\mathcal{Y}_1,\mathcal{C})$, $\tilde{\iota}^d(b,a_{d-1},...,a_1) \coloneqq  (0,\mu_\mathcal{A}^{d+1}(e_{Y_1},b,a_{d-1},...,a_1))\in\mathcal{C}(X_0)$ and
		\item $\tilde{\pi}\in\textup{hom}_\mathsf{Q}^1(\mathcal{C},\mathcal{Y}_0)$, $\tilde{\pi}^d((b_0,b_1),a_{d-1},...,a_1)\coloneqq -\mu_\mathcal{A}^{d+1}(e_{Y_0},b_0,a_{d-1},...,a_1)\in\mathcal{Y}_0(X_0)$, 
	\end{itemize} 
	for $d\geq 1$ and $X_0\in\textup{obj}(\mathcal{A})$, satisfy $\langle\tilde{\iota}\rangle = \langle\iota\rangle$ and $\langle\tilde{\pi}\rangle = \langle\pi\rangle$ in $H(\mathsf{Q})$.
\end{Lem}

\begin{proof}
	Consider first $\iota$. Let $e_{\mathcal{Y}_1}\in\textup{hom}_\mathsf{Q}^0(\mathcal{Y}_1,\mathcal{Y}_1)$ be the c-unit for $\mathcal{Y}_1\in\textup{obj}(\mathsf{Q})$. Since $\Upsilon$ is c-unital, we must have $\langle e_{\mathcal{Y}_1}\rangle = H(\Upsilon)\langle e_{Y_1}\rangle = \langle\Upsilon^1(e_{Y_1})\rangle$, so that there exists some $h\in\textup{hom}_\mathsf{Q}^{-1}(\mathcal{Y}_1,\mathcal{Y}_1)$ such that $e_{\mathcal{Y}_1}-\Upsilon^1(e_{Y_1}) = \mu_\mathsf{Q}^1(h)$. Set $k\coloneqq (0,h)\in\textup{hom}_\mathsf{Q}^{-1}(\mathcal{Y}_1,\mathcal{C})$, then we have (by equation \eqref{Yonedamap}):
	\begin{align*}
	\mu_\mathsf{Q}^1(k)^d(b,a_{d-1},...,a_1) &= (0,\mu_\mathsf{Q}^1(h)^d(b,a_{d-1},...,a_1)) \\
	&= (0,e_{\mathcal{Y}_1}^d(b,a_{d-1},...,a_1)-\Upsilon^1(e_{Y_1})^d(b,a_{d-1},...,a_1)) \\
	&= (0,0-\mu_\mathcal{A}^{d+1}(e_{Y_1},b,a_{d-1},...,a_1)) \\
	&= (\iota^d-\tilde{\iota}^d)(b,a_{d-1},...,a_1)\;,
	\end{align*}
	for $d>1$, while $\mu_\mathsf{Q}^1(k)^1(b) = (0,e_{\mathcal{Y}_1}^1(b)-\Upsilon^1(e_{Y_1})^1(b)) = (0,(-1)^{|b|}b-\mu_\mathcal{A}^2(e_{Y_1},b))$ $=\iota^1(b)-\tilde{\iota}^1(b)$ (by definition \eqref{modunit} of strict unit for $A_\infty$-modules). Therefore, $\mu_\mathsf{Q}^1(k) = \iota-\tilde{\iota}$, meaning that $\tilde{\iota}$ and $\iota$ are cohomologous.
	
	The argument for $\pi$ is similar: there exists some $h\in\textup{hom}_\mathsf{Q}^{-1}(\mathcal{Y}_0,\mathcal{Y}_0)$ such that $\mu_\mathsf{Q}^1(h)=-e_{\mathcal{Y}_0}+\Upsilon^1(e_{Y_0})$, whence
	\begin{align*}
	\mu_\mathsf{Q}^1(h)^d(b_0,a_{d-1},...,a_1) &= -e_{\mathcal{Y}_0}^d(b_0,a_{d-1},...,a_1)+\mu_\mathcal{A}^{d+1}(e_{Y_0},b_0,a_{d-1},...,a_1) \\
	&= 0 -\tilde{\pi}^d((b_0,b_1),a_{d-1},...,a_1) \\
	&= (\pi^d-\tilde{\pi}^d)((b_0,b_1),a_{d-1},...,a_1)\;,
	\end{align*}
	for $d>1$, and $\mu_\mathsf{Q}^1(h)^1(b_0) = -e_{\mathcal{Y}_0}^1(b_0)+\Upsilon^1(e_{Y_0})^1(b_0) = (-1)^{|b_0|-1}b_0 + \mu_\mathcal{A}^2(e_{Y_0},b_0)$ $=\pi^1(b_0,b_1)-\tilde{\pi}^1(b_0,b_1)$, meaning that $\tilde{\pi}$ and $\pi$ are cohomologous.
\end{proof}

\begin{Def}										
	An \textbf{exact triangle}\index{triangle!exact in the cohomological category} in $H(\mathcal{A})$ is any diagram
	\begin{equation}\label{Atriangle}
	\begin{tikzcd}
	Y_0\arrow[rr,"\langle c_1\rangle"] & & Y_1\arrow[dl, "\langle c_2\rangle"]\\
	& Y_2\arrow[ul, "{\langle c_3\rangle[1]}"]
	\end{tikzcd}
	\end{equation}
	isomorphic to the triangular diagram \eqref{Qtriangle} once applied the Yoneda embedding $\Upsilon:\mathcal{A}\rightarrow\mathsf{Q}$. Equivalently, \eqref{Atriangle} is exact if and only if there exists an isomorphism $\langle t\rangle\in\textup{Hom}_{H(\mathsf{Q})}^0(\mathcal{Y}_2,\mathcal{C}one(c_1))$ (that is, a quasi-representative $(Y_2,\langle t\rangle)$ of $\mathcal{C}one(c_1)$) making the following diagram commute in $H(\mathsf{Q})$:
	\begin{equation}\label{Qexacttriangle}
	\begin{tikzcd}
	\mathcal{Y}_0\arrow[rr, "\langle \Upsilon^1(c_1)\rangle"]& &\mathcal{Y}_1\arrow[dl, near end, "\langle\Upsilon^1(c_2)\rangle"]\arrow[dddl, near end, bend left=45, "\langle\iota\rangle"'] \\
	&\mathcal{Y}_2\arrow[ul, near start, "{\langle\Upsilon^1(c_3)\rangle[1]}"]\arrow[dd, dashed, "\wr"', "\langle t\rangle"]\\
	& \\
	&\mathcal{C}one(c_1)\arrow[uuul, near start, bend left=45, "{\langle\pi\rangle[1]}"']
	\end{tikzcd}
	\end{equation}
	(Note: the existence of $\langle t\rangle$ does not depend on the choice of representative of~$\langle c_1\rangle$.) 
\end{Def}

\begin{Lem}\label{whenexact}					
	A triangle \eqref{Atriangle} is exact if and only if there are $h_1\!\in\!\textup{hom}_\mathcal{A}^0(Y_1,Y_0)$, $h_2\in\textup{hom}_\mathcal{A}^0(Y_2,Y_1)$ and $k\in\textup{hom}_\mathcal{A}^{-1}(Y_1,Y_1)$ such that
	\begin{align}\label{hhk}
	&\mu_\mathcal{A}^1(h_1)=\mu_\mathcal{A}^2(c_3,c_2)\in\textup{hom}_\mathcal{A}^1(Y_1,Y_0)\nonumber \\
	&\mu_\mathcal{A}^1(h_2) = -\mu_\mathcal{A}^2(c_1,c_3)\in\textup{hom}_\mathcal{A}^1(Y_2,Y_1) \\
	&\mu_\mathcal{A}^1(k)=-\mu_\mathcal{A}^2(c_1,h_1)+\mu_\mathcal{A}^2(h_2,c_2)+\mu_\mathcal{A}^3(c_1,c_3,c_2)-e_{Y_1}\in\textup{hom}_\mathcal{A}^0(Y_1,Y_1)\nonumber
	\end{align}
	\textup(for $e_{Y_1}$ a c-unit\textup), and if the chain complex $\mathcal{Y}_2(X)[1]\oplus\mathcal{Y}_0(X)[1]\oplus\mathcal{Y}_1(X)$ with boundary operator
	\[
	\partial \coloneqq 
	\begin{pmatrix}
	\mu_\mathcal{A}^1 & 0 & 0 \\
	\mu_\mathcal{A}^2(c_3,\cdot) & \mu_\mathcal{A}^1 & 0 \\
	\mu_\mathcal{A}^2(h_2,\cdot) + \mu_\mathcal{A}^3(c_1,c_3,\cdot) & \mu_\mathcal{A}^2(c_1,\cdot) & \mu_\mathcal{A}^1
	\end{pmatrix}
	\]
	is acyclic for all $X\in\textup{obj}(\mathcal{A})$.
\end{Lem}

\begin{proof}
	``$\Leftarrow$'' Define $\mathcal{C} = \mathcal{C}one(c_1)$. Then $b\coloneqq  (c_3,h_2)$ is a degree 0 cocycle of $\mathcal{C}(Y_2)$ (indeed $\mu_\mathcal{C}^1(b)=(\mu_\mathcal{A}^1(c_3),\mu_\mathcal{A}^1(h_2)+\mu_\mathcal{A}^2(c_1,c_3))=0$, by the second equation in \eqref{hhk}), thus giving a module homomorphism $t\coloneqq \lambda_\mathcal{C}(b)\in\textup{hom}_\mathsf{Q}^0(\mathcal{Y}_2,\mathcal{C})$ (where $\lambda_\mathcal{C}$ is as in \eqref{lambda}). For $X\in\textup{obj}(\mathcal{A})$, $t_{X}^1$ is precisely the graded (chain) map $\mathcal{Y}_2(X)\rightarrow\mathcal{C}(X),\;a\mapsto(-1)^{|a|}\mu_\mathcal{C}^2(b,a) = (-1)^{|a|}(\mu_\mathcal{A}^2(c_3,a),\mu_\mathcal{A}^2(h_2,a)+\mu_\mathcal{A}^3(c_1,c_3,a))$ from Lemma \ref{quasi-repr}, whose mapping cone is, by definition, 
	\[
	\partial^{t_{X}^1} = \begin{pmatrix} \partial_{\mathcal{Y}_2(X)} & 0 \\ t_{X}^1 & \partial_{\mathcal{C}(X)}\end{pmatrix}.
	\] 
	The differential for $\mathcal{Y}_2(X)=\textup{hom}_\mathcal{A}(X,Y_2)$ is $\mu_\mathcal{A}^1$, that for $\mathcal{C}(X)$ was given in Definition \ref{abstractmappingcone}. Therefore, $\partial^{t_{X}^1} = \partial$, which is acyclic by assumption. Basic homological algebra tells us that this makes $t_{X}^1$ a chain equivalence, in turn implying that $\langle t\rangle$ is an isomorphism in $H^0(\mathsf{Q})$, and hence a valid vertical arrow for diagram \eqref{Qexacttriangle}.
	
	We need to show that \eqref{Qexacttriangle} commutes. We have:
	\begin{align*}
	\langle\iota\rangle &= \langle\tilde{\iota}\rangle &\text{(by Lemma \ref{iandpi})}\\
	&=\langle\lambda_\mathcal{C}((0,e_{Y_1}))\rangle &\text{(by definition of $\lambda$ and $\tilde{\iota}$)}\\
	&=\langle\lambda_\mathcal{C}(\mu_\mathcal{C}^2(b,c_2))\rangle &\text{(as $(0,e_{Y_1})$ is cohomologous to $\mu_\mathcal{C}^2(b,c_2)$, by \eqref{hhk})}\\
	&=\langle\mu_\mathsf{Q}^2(t,\Upsilon^1(c_2))\rangle &\text{(by the first equation in \eqref{lambdaproperties})}\\
	&=\langle t\rangle\circ\langle \Upsilon^1(c_2)\rangle &\text{(by definition, and since $|\Upsilon^1(c_2)| =|c_2|=0$)}	
	\end{align*}
	and
	\begin{align*}
	\langle\Upsilon^1(c_3)\rangle &= \langle\lambda_{\mathcal{Y}_0}(c_3)\rangle &\text{(by definition of $\lambda$)}\\
	&= \langle\lambda_{\mathcal{Y}_0}(\pi^1(b))\rangle &\text{(by definition of $\pi$, and since $|c_3|=1$)}\\
	&= \langle\mu_\mathsf{Q}^2(\pi,t)\rangle &\text{(by the second equation in \eqref{lambdaproperties})}\\
	&= \langle\pi\rangle\circ\langle t\rangle &\text{(by definition, and since $|t|=0$).}
	\end{align*}
	All requisites for \eqref{Atriangle} to be exact are satisfied.
	
	\noindent``$\Rightarrow$'' We argue backwards: any isomorphism $\langle t\rangle$ making \eqref{Qexacttriangle} commute is of the form $\langle\lambda_\mathcal{C}(b)\rangle$ for some cocycle $b=(b_0,h_2)\in\mathcal{C}(Y_2)^0$, by Lemma \ref{quasi-repr}. By bijectivity of $H(\Upsilon)$, $H(\Upsilon)\langle b_0\rangle = \langle\Upsilon^1(\pi^1(b))\rangle = H(\Upsilon)\langle c_3\rangle$ implies $\langle b_0\rangle =\langle c_3\rangle$ (here we used the last chain of equalities, as well as commutativity of \eqref{Qexacttriangle}). Hence, we can take $b=(c_3,h_2)$, for yet an unknown $h_2$. 
	
	On the other hand, $\langle\lambda_\mathcal{C}((0,e_Y))\rangle = \langle\iota\rangle = \langle\mu_\mathsf{Q}^2(t,\Upsilon^1(c_2))\rangle=\langle\lambda_\mathcal{C}(\mu_\mathcal{C}^2(b,c_2))\rangle$ (again, by commutativity and first equation in \eqref{lambdaproperties}), so that the inputs are cohomologous and thus provide coboundaries of some $h_1$ and $k$. By construction, the triple $(h_1,h_2,k)$ fulfills equations \eqref{hhk}. Acyclicity of $\mathcal{Y}_2(X)[1]\oplus\mathcal{Y}_0(X)[1]\oplus\mathcal{Y}_1$ can be achieved by exhibiting a contracting homotopy; we leave this out.  
\end{proof}

Although explicit, the criterion of Lemma \ref{whenexact} is very laborious. It is desirable to deduce exactness of a triangle by passing through a structurally simpler $A_\infty$-category. This is possible (see Proposition \ref{exactcriterium}), but we will be able to show it only after having developed the proper tools.

\subsection{Triangulated $A_\infty$-categories and derived categories} \label{ch3.4}

We extend the concept of triangulation studied in Section \ref{ch3.1} to $A_\infty$-categories.

\begin{Def}									
	An $A_\infty$-category $\mathcal{A}$ is \textbf{triangulated}\index{aaa@$A_\infty$-category!triangulated} if:
	\begin{itemize}[leftmargin=0.5cm]
		\item $\textup{obj}(\mathcal{A})\neq\emptyset$;
		\item for any $Y_0, Y_1\in\textup{obj}(\mathcal{A})$, every degree 0 morphism $\langle c_1\rangle\in\textup{Hom}_{H(\mathcal{A})}^0(Y_0,Y_1)$ can be completed to an exact triangle in $H(\mathcal{A})$ of the form \eqref{Atriangle};
		\item for any $Y\in\textup{obj}(\mathcal{A})$, there is some $\tilde{Y}\in\textup{obj}(\mathcal{A})$ such that $S\tilde{Y}\cong Y$ in $H^0(\mathcal{A})$, that is, every object is isomorphic to a shifted one.
	\end{itemize}
\end{Def}  

The last two conditions imply that we can construct an $A_\infty$-shift functor $S:\mathcal{A}\rightarrow\mathcal{A}$ as in Definition \ref{shiftfunc} which is a quasi-equivalence (because it is required to be cohomologically full, faithful and essentially surjective). 

Later on we will show that $\mathsf{Q}=mod(\mathcal{A})$ is always triangulated. More importantly, if $\mathcal{A}$ is a triangulated $A_\infty$-category, then its zeroth cohomological category $H^0(\mathcal{A})$ is a standard triangulated category, with shift functor $T\coloneqq H^0(S)$ and distinguished triangles corresponding to the exact ones in $H^0(\mathcal{A})\subset H(\mathcal{A})$ (see Proposition \ref{H0Atriangulated}), possible after identifying $\langle c_3\rangle$ as a degree 0 morphism in $\textup{Hom}_{H(\mathcal{A})}(Y_2,SY_0)$ ($\cong\textup{Hom}_{H(\mathcal{A})}(Y_2,Y_0)[1]$, by the first equation of \eqref{Homiso}).

\begin{Def}\label{triangenvelope}				
	Let $\mathcal{B}$ be a triangulated $A_\infty$-category, $\mathcal{A}\subset\mathcal{B}$ a full non-empty $A_\infty$-subcategory. The smallest triangulated, full $A_\infty$-subcategory $\tilde{\mathcal{B}}\subset\mathcal{B}$ which is closed under isomorphism (that is, $X_0\cong X_1$ in $H(\mathcal{B})$ and $X_0\in\textup{obj}(\tilde{\mathcal{B}})$ imply $X_1\in\textup{obj}(\tilde{\mathcal{B}})$) and contains $\mathcal{A}$ is called the \textbf{triangulated subcategory of $\mathcal{B}$ generated by}\index{triangulated subcategory generated by an $A_\infty$-category} $\mathcal{A}$. 
	
	Specifically, $\textup{obj}(\tilde{\mathcal{B}})$ can be constructed from $\textup{obj}(\mathcal{A})$ by forming all possible mapping cones and shifts, and iterating. If $\tilde{\mathcal{B}}=\mathcal{B}$, then $\mathcal{A}$ \textbf{generates}\index{generate (an $A_\infty$-category)} $\mathcal{B}$.
	
	A \textbf{triangulated envelope}\index{triangulated envelope} for an $A_\infty$-category $\mathcal{A}$ is a pair $(\mathcal{B},\mathcal{F})$ where $\mathcal{B}$ is a triangulated $A_\infty$-category and $\mathcal{F}:\mathcal{A}\rightarrow\mathcal{B}$ a cohomologically full and faithful functor whose image objects generate $\mathcal{B}$. By definition, if $\mathcal{A}$ generates $\mathcal{B}$, then $\mathcal{B}$ is a triangulated envelope for $\mathcal{A}$ (choosing $\mathcal{F}$ to be the inclusion).
\end{Def}

We will show that triangulated  envelopes $(\mathcal{B},\mathcal{F})$ always exist and are determined up to quasi-equivalence (cf. Proposition \ref{envelopeexists}), thus uniquely identifying $H^0(\mathcal{B})$ up to equivalence. 

\begin{Def}\label{Ainfderivedcat}					
	Let $\mathcal{B}$ be the triangulated envelope of an $A_\infty$-category $\mathcal{A}$. Then we call $\mathsf{D}(\mathcal{A})\coloneqq  H^0(\mathcal{B})$ the \textbf{derived category}\index{derived category!of an $A_\infty$-category} of $\mathcal{A}$.\footnote{This terminology is somewhat improper: our construction of triangulated envelopes does not involve the localization procedure highlighted in Definition \ref{derivedcat}, but a definite link can be found upon looking at twisted complexes (cf. Proposition \ref{H0TwAtriangulated}). The same remark holds for the \textit{bounded} derived category of Definition \ref{boundeddercat}.}
\end{Def}

Now, in order to simplify our discussion and prove all promised results, we study a particular case of triangulated envelope: the $A_\infty$-category of twisted complexes.

\subsection{The $A_\infty$-category of twisted complexes}\label{ch3.5}

For the sake of generality, we temporarily suspend the assumption that every $A_\infty$-category and functor is c-unital.

\begin{Def}\label{addenl}									
	Let $\mathcal{A}$ be a non-unital $A_\infty$-category. We define the \textbf{additive enlargement}\index{additive enlargement} $\varSigma\mathcal{A}$ to be the non-unital $A_\infty$-category with:
	\begin{itemize}[leftmargin=0.5cm]
		\renewcommand{\labelitemi}{\textendash}
		\item objects all formal direct sums $X \equiv (I,\{X^i\}_{i\in I},\{V^i\}_{i\in I})\coloneqq \bigoplus_{i\in I}V^i\otimes X^i$, where $I$ is a finite index set, all $X^i\in\textup{obj}(\mathcal{A})$ and all $V^i$ are finite-dimensional graded vector spaces (called \textit{multiplicity spaces}),
		\item for any pair $X_0, X_1\in\textup{obj}(\varSigma\mathcal{A})$, a graded vector space
		\begin{align}\label{addenlmorphspaces}
		\textup{hom}_{\varSigma\mathcal{A}}(X_0,X_1) &= \textup{hom}_{\varSigma\mathcal{A}}\Big(\bigoplus_{i_0\in I_0}V_{0}^{i_0}\otimes X_{0}^{i_0},\bigoplus_{i_1\in I_1}V_{1}^{i_1}\otimes X_{1}^{i_1}\Big) \\
		&\coloneqq\mkern-24mu\bigoplus_{\quad i_0\in I_0, i_1\in I_1}\mkern-24mu \textup{Hom}_\mathbb{K}(V_{0}^{i_0},V_{1}^{i_1})\otimes\textup{hom}_\mathcal{A}(X_{0}^{i_0},X_{1}^{i_1})\nonumber
		\end{align}
		(graded because each summand is), with 
		\[
		\textup{hom}_{\varSigma\mathcal{A}}^n(X_0,X_1)=\!\bigoplus_{\!\!l+m=n}\mkern-12mu\bigoplus_{\quad i_0\in I_0,i_1\in I_1}\mkern-24mu\textup{Hom}_\mathbb{K}^l(V_{0}^{i_0},V_{1}^{i_1})\otimes\textup{hom}_\mathcal{A}^m(X_{0}^{i_0},X_{1}^{i_1})\;;
		\]
		we write a morphism $a_1\in	\textup{hom}_{\varSigma\mathcal{A}}(X_0,X_1)$ in the matrix form $(a_{1}^{i_1,i_0})$ for $i_0\in I_0, i_1\in I_1$, where each entry is a finite linear combination of the form $a_{1}^{i_1,i_0} = \sum_k\phi_{1}^{k,i_1,i_0}\otimes x_{1}^{k,i_1,i_0}\in\textup{Hom}_\mathbb{K}(V_{0}^{i_0},V_{1}^{i_1})\otimes\textup{hom}_\mathcal{A}(X_{0}^{i_0},X_{1}^{i_1})$ (though, for practicality, we will implicitly assume linearity and simply write $a_{1}^{i_1,i_0} = \phi_{1}^{i_1,i_0}\otimes x_{1}^{i_1,i_0}$),
		\item composition maps $\mu_{\varSigma\mathcal{A}}^d$ as in \eqref{compomap}, where, for $a_k\in\textup{hom}_{\varSigma\mathcal{A}}(X_{k-1},X_k)$, the output morphism $\mu_{\varSigma\mathcal{A}}^d(a_d,...,a_1)\in\textup{hom}_{\varSigma\mathcal{A}}(X_d,X_0)$ is componentwise defined as
		\begin{equation}\label{compomapaddenl}
		\mu_{\varSigma\mathcal{A}}^d(a_d,...,a_1)^{i_d,i_0}\coloneqq \mkern-18mu\sum_{\substack{i_1\in I_1,...,\\i_{d-1}\in I_{d-1}}}\mkern-18mu (-1)^\star(\phi_{d}^{i_d,i_{d-1}}\circ...\circ\phi_{1}^{i_1,i_0})\otimes\mu_\mathcal{A}^d(x_{d}^{i_d,i_{d-1}},...,x_{1}^{i_1,i_0})\;,
		\end{equation}
		for $i_0\in I_0, i_d\in I_d$ and $\star \coloneqq \sum_{1\leq p<q\leq d}|\phi_{p}^{i_p,i_{p-1}}|\!\cdot\!(|x_{q}^{i_q,i_{q-1}}|-1)$.
	\end{itemize}
\end{Def}

Choosing $I\coloneqq\{\ast\}, X^*\coloneqq X$ and $V^*\coloneqq\mathbb{K}[0]=\mathbb{K}$ for all $X\in\textup{obj}(\mathcal{A})$ makes $\mathcal{A}\subset\varSigma\mathcal{A}$ a full $A_\infty$-subcategory. 

\begin{Rem}\label{alternateaddenl}			
	Before proceeding, it is worth mentioning an alternative construction for the additive enlargement, compatible only up to quasi-equivalence with the one given, but notationally friendlier (consult \cite[section 2b]{[Sei13]} for the precise formulation, \cite[section 7.6]{[Kel01]} for an even more concise version). 
	
	Namely, we just allow shifted copies of $\mathbb{K}$ as multiplicity spaces. Then objects have form $\bigoplus_iS^{-\sigma_i}X^i$ for $\sigma_i\in\mathbb{Z}$, that is, formal sums of shifted objects of $\mathcal{A}$ as in Section \ref{ch3.2}. Morphism spaces become $\bigoplus_{i_0,i_1}\textup{hom}_\mathcal{A}(X_{0}^{i_0},X_{1}^{i_1})[\sigma_{1,i_1}-\sigma_{0,i_0}]$, and compositions $\mu_{\varSigma\mathcal{A}}^d(a_d,...,a_1)^{i_d,i_0}\!\coloneqq\! (-1)^{\sigma_{0,i_0}}\!\sum_{i_1,...,i_{d-1}}\!\mu_\mathcal{A}^d(a_{d}^{i_d,i_{d-1}}\!\!,...,a_{1}^{i_1,i_0})$. 
	
	This simplifies the discussion, especially in matter of signs, but we will remain faithful to the more general exposition of our main source \cite{[Sei08]}.
\end{Rem}

\begin{Def}											
	Let $\mathcal{A}$ and $\mathcal{B}$ be non-unital $A_\infty$-categories. A non-unital $A_\infty$-functor $\mathcal{F}:\mathcal{A}\rightarrow\mathcal{B}$ induces an $A_\infty$-functor $\varSigma\mathcal{F}:\varSigma\mathcal{A}\rightarrow\varSigma\mathcal{B}$ defined by
	\begin{align}\label{addenlfunc}
	&(\varSigma\mathcal{F})(X) = (\varSigma\mathcal{F})\Big(\bigoplus_{i\in I}V^i\otimes X^i\Big) \coloneqq \bigoplus_{i\in I}V^i\otimes\mathcal{F}(X^i)\;,\\
	&(\varSigma\mathcal{F})^d(a_d,...,a_1)^{i_d,i_0} \coloneqq \mkern-18mu\sum_{\substack{i_1\in I_1,...,\\i_{d-1}\in I_{d-1}}}\mkern-18mu (-1)^\star(\phi_{d}^{i_d,i_{d-1}}\circ...\circ\phi_{1}^{i_1,i_0})\otimes\mathcal{F}^d(x_{d}^{i_d,i_{d-1}},...,x_{1}^{i_1,i_0}), \nonumber
	\end{align}
	for $X\in\textup{obj}(\varSigma\mathcal{A})$ and $a_k\in\textup{hom}_{\varSigma\mathcal{A}}(X_{k-1},X_k)$. 
	
	There is a non-unital $A_\infty$-functor $\varSigma:nu\text{-}\!fun(\mathcal{A},\mathcal{B})\rightarrow nu\text{-}\!fun(\varSigma\mathcal{A},\varSigma\mathcal{B})$ sending $T\in\textup{hom}_{nu\text{-}\!fun(\mathcal{A},\mathcal{B})}(\mathcal{F}_0,\mathcal{F}_1)$ to $\varSigma^1(T)\in\textup{hom}_{nu\text{-}\!fun(\varSigma\mathcal{A},\varSigma\mathcal{B})}(\varSigma\mathcal{F}_0,\varSigma\mathcal{F}_1)$, given by
	\begin{equation}
	\varSigma^1(T)^d(a_d,...,a_1)^{i_d,i_0} \coloneqq \mkern-18mu\sum_{\substack{i_1\in I_1,...,\\i_{d-1}\in I_{d-1}}}\mkern-18mu (-1)^\star(\phi_{d}^{i_d,i_{d-1}}\circ...\circ\phi_{1}^{i_1,i_0})\otimes T^d(x_{d}^{i_d,i_{d-1}},...,x_{1}^{i_1,i_0}),
	\end{equation}
	and with $\varSigma^d\coloneqq 0$ for $d>1$. 
\end{Def}

\begin{Def}											
	Let $\mathcal{A}$ be a non-unital $A_\infty$-category. A \textbf{pre-twisted complex}\index{pre-twisted complex} in $\mathcal{A}$ is a pair $(X,\delta_X)$ where $X=(I,\{X^i\},\{V^i\})\in\textup{obj}(\varSigma\mathcal{A})$ and $\delta_X=(\delta_{X}^{j,i})\in\textup{hom}_{\varSigma\mathcal{A}}^1(X,X)$ (so that each $\delta_{X}^{j,i}$ is linear combination of tensors whose degree is 1).
	
	A \textbf{twisted complex}\index{twisted complex} in $\mathcal{A}$ is a pre-twisted complex $(X,\delta_X)$ such that $\delta_X$ is ``strictly lower-triangular'' (that is, $\delta_{X}^{j,i} = 0$ if $j\leq i$, after suitably\footnote{One resorts to filtrations by subcomplexes; see \cite[section (3l)]{[Sei08]}.} ordering $I$) and satisfies the \textit{generalized Maurer-Cartan equation}
	\begin{equation}
	\sum_{r=1}^\infty \mu_{\varSigma\mathcal{A}}^r(\delta_X,...,\delta_X) = 0\;.
	\end{equation}
	(The last sum is actually finite due to lower triangularity, which eventually makes the compositions of $\phi$'s in \eqref{compomapaddenl} vanish.)   
\end{Def}

\begin{Def}									
	Let $\mathcal{A}$ be a non-unital $A_\infty$-category. We define the \textbf{category $Tw\mathcal{A}$ of twisted complexes}\index{aaa@$A_\infty$-category!of twisted complexes} in $\mathcal{A}$ to be the non-unital $A_\infty$-category with:
	\begin{itemize}[leftmargin=0.5cm]
		\renewcommand{\labelitemi}{\textendash}
		\item objects all twisted complexes $X=(X,\delta_X)$ in $\mathcal{A}$,
		\item for any pair $X_0, X_1\in\textup{obj}(Tw\mathcal{A})$, the morphism space 
		\begin{equation}
			\textup{hom}_{Tw\mathcal{A}}(X_0,X_1) \coloneqq \textup{hom}_{\varSigma\mathcal{A}}(X_0,X_1)\,,
		\end{equation}
		\item composition maps $\mu_{Tw\mathcal{A}}^d$ given by
		\begin{equation}\label{twistedcompo}
		\mkern-6mu\mu_{Tw\mathcal{A}}^d(a_d,...,a_1) \coloneqq \mkern-5mu\sum_{\substack{s_0,...,\\s_d\geq 0}}\mkern-5mu \mu_{\varSigma\mathcal{A}}^{d+s_0+...+s_d}(\underbrace{\delta_{X_d},...,\delta_{X_d}}_{s_d},a_d,...,a_1,\underbrace{\delta_{X_0},...,\delta_{X_0}}_{s_0}),
		\end{equation}
		for $d\geq 1$ and $a_k\in\textup{hom}_{Tw\mathcal{A}}(X_{k-1},X_k)$ (the sum is again finite). 
	\end{itemize}
\end{Def}

That $Tw\mathcal{A}$ is indeed an $A_\infty$-category is proved, for example, in \cite[section 8.6]{[Kel01]} (there the author works with the simplified definition stemming from Remark \ref{alternateaddenl}). Setting $\delta_X \coloneqq 0$ for all $X\in\textup{obj}(\varSigma\mathcal{A})$ makes $\varSigma\mathcal{A}\subset Tw\mathcal{A}$ a full $A_\infty$-subcategory. Hence, with the identifications above, $\mathcal{A}\subset Tw\mathcal{A}$ is a full $A_\infty$-subcategory as well. 

\begin{Def}										
	Let $\mathcal{A}$ and $\mathcal{B}$ be non-unital $A_\infty$-categories. A non-unital $A_\infty$-functor $\mathcal{F}:\mathcal{A}\rightarrow\mathcal{B}$ induces an $A_\infty$-functor $Tw\mathcal{F}:Tw\mathcal{A}\rightarrow Tw\mathcal{B}$ defined by
	\begin{align}
	&(Tw\mathcal{F})(X,\delta_X) \coloneqq \Big(\varSigma\mathcal{F}(X), \sum_{d\geq 1}(\varSigma\mathcal{F})^d(\delta_X,...,\delta_X)\Big)\;,\\
	&(Tw\mathcal{F})^d(a_d,...,a_1) \coloneqq \mkern-6mu\sum_{\substack{s_0,...,\\s_d\geq 0}}\mkern-6mu (\varSigma\mathcal{F})^{d+s_0+...+s_d}(\underbrace{\delta_{X_d},...,\delta_{X_d}}_{s_d},a_d,...,a_1,\underbrace{\delta_{X_0},...,\delta_{X_0}}_{s_0})\;, \nonumber
	\end{align}
	for $X=(X,\delta_X)\in\textup{obj}(Tw\mathcal{A})$ and $a_k\in\textup{hom}_{\varSigma\mathcal{A}}(X_{k-1},X_k)$. 
	
	There is a non-unital $A_\infty$-functor $Tw:nu\text{-}\!fun(\mathcal{A},\mathcal{B})\rightarrow nu\text{-}\!fun(Tw\mathcal{A},Tw\mathcal{B})$, $T\in\textup{hom}_{nu\text{-}\!fun(\mathcal{A},\mathcal{B})}(\mathcal{F}_0,\mathcal{F}_1)\mapsto Tw^1(T)\in\textup{hom}_{nu\text{-}\!fun(Tw\mathcal{A},Tw\mathcal{B})}(Tw\mathcal{F}_0,$ $Tw\mathcal{F}_1)$ given by
	\begin{equation}
	Tw^1(T)^d(a_d,...,a_1) \coloneqq \mkern-6mu\sum_{\substack{s_0,...,\\s_d\geq 0}}\mkern-6mu \varSigma^1(T)^{d+s_0+...+s_d}(\underbrace{\delta_{X_d},...,\delta_{X_d}}_{s_d},a_d,...,a_1,\underbrace{\delta_{X_0},...,\delta_{X_0}}_{s_0}), \nonumber
	\end{equation}
	and with $Tw^d\coloneqq 0$ for $d>1$.
\end{Def}

\begin{Lem}\label{Twfullfaith}					
	If $\mathcal{F}:\mathcal{A}\rightarrow\mathcal{B}$ is a cohomologically full and faithful non-unital $A_\infty$-functor, then so is $Tw\mathcal{F}:Tw\mathcal{A}\rightarrow Tw\mathcal{B}$.
\end{Lem}

\begin{proof}
	A comparison argument for the spectral sequences \eqref{spectral} associated to filtrations of the spaces $\textup{hom}_{Tw\mathcal{A}}(X_0,X_1)$ and $\textup{hom}_{Tw\mathcal{B}}(Tw\mathcal{F}(X_0),Tw\mathcal{F}(X_1))$ reduces the problem to $\varSigma\mathcal{F}$. But for any two fixed objects $X,Y\in\textup{obj}(\varSigma\mathcal{A})$, $H(\varSigma\mathcal{F})_{X,Y}:\textup{Hom}_{H(\varSigma\mathcal{A})}(X,Y)\rightarrow\textup{Hom}_{H(\varSigma\mathcal{B})}((\varSigma\mathcal{F})(X),(\varSigma\mathcal{F})(Y))$ maps $\langle a\rangle$ to $\langle(\varSigma\mathcal{F})^1(a)\rangle$, which, according to equation \eqref{addenlfunc}, has $(i_1,i_0)$-th entry equal to $\phi_{1}^{i_1,i_0}\otimes H(\mathcal{F})\langle x_{1}^{i_1,i_0}\rangle$. Therefore, bijectivity of $H(\varSigma\mathcal{F})$ on morphisms can be deduced from that of $H(\mathcal{F})$. 
\end{proof}

We return to discuss unitality. If $\mathcal{A}$ is strictly unital, so that for all $Y\in\textup{obj}(\mathcal{A})$ we can find a strict unit $e_Y\in\textup{hom}_{\mathcal{A}}^0(Y,Y)$, then we define the strict unit $E_X\in\textup{hom}_{Tw\mathcal{A}}(X,X)$ of a twisted complex $X=(I,\{X^i\},\{V^i\})$ to have diagonal matrix form $E_{X}^{j,i}=\phi^{j,i}\otimes x^{j,i} \coloneqq \delta_{ji}\cdot \textup{id}_{V^i}\otimes e_{X^i}$
(thinking of \eqref{twistedcompo} as matrix multiplication, it is clear from strictly lower triangularity of the $\delta_{X_k}$'s that the requirements \eqref{unitalrequirm} for strict unitality are fulfilled). 

One can also show that a strictly unital $A_\infty$-functor $\mathcal{F}:\mathcal{A}\rightarrow\mathcal{B}$ induces a strictly unital $Tw\mathcal{F}$, and that $Tw$ itself is strictly unital on the strictly unital $A_\infty$-category $nu\text{-}\!fun(\mathcal{C},\mathcal{A})$, for any $\mathcal{C}$. What about c-unitality?

\begin{Lem}\label{TwAcunital}								
	If $\mathcal{A}$ is c-unital, then so is $Tw\mathcal{A}$. Also, a c-unital $\mathcal{F}:\mathcal{A}\rightarrow\mathcal{B}$ between c-unital $A_\infty$-categories yields a c-unital $Tw\mathcal{F}$. Finally, the $A_\infty$-functor $Tw:nu\text{-}\!fun(\mathcal{C},\mathcal{A})$ $\rightarrow nu\text{-}\!fun(Tw\mathcal{C},Tw\mathcal{A})$ is c-unital for any $\mathcal{C}$ and c-unital $\mathcal{A}$. 
\end{Lem}

\begin{proof}
	We adopt a strategy similar to the proof of Proposition \ref{funccatunitality}.ii, reducing to the strictly unital case $\tilde{\mathcal{A}}$ by means of a formal diffeomorphism $\Phi:\mathcal{A}\rightarrow\tilde{\mathcal{A}}$ with $\Phi^1=\textup{id}$, cohomologically full and faithful (this is provided by Lemma \ref{cunitstrictunit}). Then Lemma \ref{Twfullfaith} applies, making $Tw\Phi:Tw\mathcal{A}\rightarrow Tw\tilde{\mathcal{A}}$ cohomologically full and faithful, with target the strictly unital $Tw\tilde{\mathcal{A}}$ (by last paragraph). Consequently, $Tw\mathcal{A}$ must be c-unital as well.
	
	The argument for c-unitality of $Tw\mathcal{F}$ is more involved; we take it on faith and refer to \cite[Lemma 3.24]{[Sei08]} for a proof.
	
	Concerning $Tw$, start from a non-unital $A_\infty$-functor $\mathcal{G}:\mathcal{C}\rightarrow\mathcal{A}$ and define $\tilde{\mathcal{G}}\coloneqq\Phi\circ\mathcal{G}:\mathcal{C}\rightarrow\tilde{\mathcal{A}}$. A straight computation shows that $Tw$ is well-behaved with respect to left composition $\mathcal{L}_{\Phi}:nu\text{-}\!fun(\mathcal{C},\mathcal{A})\rightarrow nu\text{-}\!fun(\mathcal{C},\tilde{\mathcal{A}})$, which is a c-unital $A_\infty$-functor by Lemma \ref{cunitalcompo} (and so is $\mathcal{L}_{Tw\Phi^{-1}}$, since $Tw\Phi^{-1}$ is c-unital by the previous point we skipped). This means that, on morphisms, $Tw^1=\mathcal{L}_{Tw\Phi^{-1}}^1\circ Tw^1\circ\mathcal{L}_{\Phi}^1:\textup{hom}_{nu\text{-}\!fun(\mathcal{C},\mathcal{A})}(\mathcal{G},\mathcal{G})\rightarrow\textup{hom}_{nu\text{-}\!fun(Tw\mathcal{C},Tw\mathcal{A})}(Tw\mathcal{G},Tw\mathcal{G})$, where $Tw^1:\textup{hom}_{nu\text{-}\!fun(\mathcal{C},\tilde{\mathcal{A}})}(\tilde{\mathcal{G}},\tilde{\mathcal{G}})\rightarrow\textup{hom}_{nu\text{-}\!fun(Tw\mathcal{C},Tw\tilde{\mathcal{A}})}(Tw\tilde{\mathcal{G}},Tw\tilde{\mathcal{G}})$ on the right-hand side is strictly unital by last paragraph. All together, $Tw$ must be c-unital, as claimed.
\end{proof}

Combining c-unitality of $Tw\mathcal{F}$ with Theorem \ref{aboutquasiequiv} and Lemma \ref{Twfullfaith}, we also obtain: 

\begin{Cor}\label{TwFquasiequiv}				
	Let $\mathcal{A}$ and $\mathcal{B}$ be c-unital $A_\infty$-categories, $\mathcal{F}:\mathcal{A}\rightarrow\mathcal{B}$ a quasi-equivalence. Then $Tw\mathcal{F}:Tw\mathcal{A}\rightarrow Tw\mathcal{B}$ is a quasi-equivalence too.
\end{Cor}

Now, we study the constructions of Sections \ref{ch3.2} and \ref{ch3.3} for $Tw\mathcal{A}$, where \textit{$\mathcal{A}$ resumes to be c-unital along with all associated $A_\infty$-functors.}

\begin{Def}\label{TwAops}								
	Let $\mathcal{A}$ be a (c-unital) $A_\infty$-category.
	\begin{itemize}[leftmargin=0.5cm]
		\item For $Y_0=(I_0,\{Y_{0}^{i_0}\},\{V_{0}^{i_0}\}),Y_1=(I_1,\{Y_{1}^{i_1}\},\{V_{1}^{i_1}\})\in\textup{obj}(\varSigma\mathcal{A})$, we define $Y_0\oplus Y_1\in\textup{obj}(\varSigma\mathcal{A})$ as the obvious direct sum, with indices now running over the set $I_0\times I_1$. \\		
		Interpreting $Y_0\oplus Y_1$ as an object of $Tw\mathcal{A}$ (its definition is identical), so $\Upsilon(Y_0\oplus Y_1) = \mathcal{Y}_0\oplus\mathcal{Y}_1\in\textup{obj}(mod(Tw\mathcal{A}))$ (indeed for any $X\in\textup{obj}(Tw\mathcal{A})$ holds $\Upsilon(Y_0\oplus Y_1)(X)=\textup{hom}_{Tw\mathcal{A}}(X,Y_0\oplus Y_1) = \textup{hom}_{Tw\mathcal{A}}(X,Y_0)\oplus\textup{hom}_{Tw\mathcal{A}}(X,Y_1) = \mathcal{Y}_0(X)\oplus\mathcal{Y}_1(X)$).
		\item For $Z$ a finite-dimensional graded vector space, we define the non-unital $A_\infty$-functor $Z\otimes\square:\varSigma\mathcal{A}\rightarrow\varSigma\mathcal{A}$ sending objects $Y=(I,\{Y^i\},\{V^i\})$ to $Z\otimes Y\coloneqq (I,\{Y^i\},\{Z\otimes V^i\})$ and morphisms $a=(\phi^{j,i}\otimes x^{j,i})$ to $\textup{id}_Z\otimes a\coloneqq (\textup{id}_Z\otimes\phi^{j,i}\otimes x^{j,i})$, with $(Z\otimes\square)^d \coloneqq 0$ if $d>1$. \\
		Similarly, there is a non-unital $A_\infty$-functor $Z\otimes\square:Tw\mathcal{A}\rightarrow Tw\mathcal{A}$ assigning $(Y,\delta_Y)\mapsto (Z\otimes Y,\textup{id}_Z\otimes\delta_Y)$, on $mod(Tw\mathcal{A})$ precisely corresponding to the tensoring functor of Section \ref{ch3.2} (once applied $\Upsilon$).
		\item Choosing $Z\coloneqq\mathbb{K}[\sigma]$ (with $\partial_Z\coloneqq0$), we set again $S^\sigma Y \coloneqq  \mathbb{K}[\sigma]\otimes Y \cong\bigoplus_{i\in I}V^i[\sigma]\otimes Y^i\in\textup{obj}(\varSigma\mathcal{A})$ to indicate $\sigma$-shifted objects having morphism $\delta_{S^\sigma Y}\coloneqq S^\sigma(\delta_Y)=S^\sigma((\phi^{j,i}\otimes x^{j,i})) = ((-1)^{\sigma\cdot|\phi^{j,i}|}\phi^{j,i}\otimes x^{j,i})$. Then we obtain a non-unital $A_\infty$-functor $S^\sigma:Tw\mathcal{A}\rightarrow Tw\mathcal{A}$, coinciding with that given in Definition \ref{shiftfunc} and representing $SS^\sigma$ on objects. \\
		The isomorphisms \eqref{Homiso} are still valid, on the non-cohomological level too. Most importantly, $S:Tw\mathcal{A}\rightarrow Tw\mathcal{A}$ is a quasi-equivalence, by construction (cf. Remark \ref{alternatecone} below).
	\end{itemize}
\end{Def}

\begin{Def}\label{mappingcone}						
	Let $c\in\textup{hom}_{Tw\mathcal{A}}^0(Y_0,Y_1)$ be a cocycle. The \textbf{twisted mapping cone}\index{twisted mapping cone} $Cone(c)=(C,\delta_C)\in\textup{obj}(Tw\mathcal{A})$ is given by
	\begin{equation}
	Cone(c)\coloneqq\Big(SY_0\oplus Y_1, \begin{pmatrix} S(\delta_{Y_0}) & 0 \\ -S(c) & \delta_{Y_1} \end{pmatrix}\Big)\;,
	\end{equation}
	where $\delta_C\in\textup{hom}_{Tw\mathcal{A}}^1(C,C)$ is well defined since $S(\delta_{Y_0})\in\textup{hom}_{Tw\mathcal{A}}^1(SY_0,SY_0)$ and $S(c)\in\textup{hom}_{Tw\mathcal{A}}^0(SY_0,SY_1)$ $\cong\textup{hom}_{Tw\mathcal{A}}^0(Y_0,Y_1)\cong\textup{hom}_{Tw\mathcal{A}}^1(SY_0,Y_1)$ (due to the non-cohomological counterparts of \eqref{Homiso}). Then $C$ fulfills:
	\begin{align}\label{coneidentification}
	\Upsilon(C)(X)=\textup{hom}_{Tw\mathcal{A}}(X,C) &= \textup{hom}_{Tw\mathcal{A}}(X,SY_0)\oplus\textup{hom}_{Tw\mathcal{A}}(X,Y_1)\nonumber \\
	&\cong \textup{hom}_{Tw\mathcal{A}}(X,Y_0)[1]\oplus\textup{hom}_{Tw\mathcal{A}}(X,Y_1)\nonumber \\
	&=\mathcal{C}(X)\;,
	\end{align}
	that is, we can identify $\Upsilon(C)$ with the abstract mapping cone $\mathcal{C}=\mathcal{C}one(c)$ from Definition \ref{abstractmappingcone}.
\end{Def}

\begin{Rem}\label{alternatecone}				
	Recall the alternative construction for $\varSigma\mathcal{A}$ from Remark \ref{alternateaddenl}, which easily extends to produce a similar version of $Tw\mathcal{A}$. The respective twisted mapping cone $C=Cone(c)$ is shown to have
	\[
	\delta_C = \begin{pmatrix} \delta_{Y_0} & 0 \\ c & \delta_{Y_1}\end{pmatrix}\;.
	\]
	As already seen, signs are conveniently absent. This is because here $S$ acts as $a\in\textup{hom}_{Tw\mathcal{A}}^k(X_0,X_1)\mapsto (-1)^{k-1}a\in\textup{hom}_{Tw\mathcal{A}}^k(X_0,X_1)$ (whence we clearly deduce its cohomological fullness and faithfulness). 
\end{Rem}

\begin{Def}										
	Let $\mathcal{A}$ be a c-unital $A_\infty$-category. Given $c\in\textup{hom}_{Tw\mathcal{A}}^0(Y_0,Y_1)$, the twisted mapping cone $C=Cone(c)=SY_0\oplus Y_1\in\textup{obj}(Tw\mathcal{A})$ comes with two canonical morphisms ($e_{Y_k}$ are c-units):
	\begin{itemize}[leftmargin=0.5cm]
		\renewcommand{\labelitemi}{\textendash}
		\item the inclusion into the second summand $i\coloneqq\begin{pmatrix}0\\e_{Y_1}\end{pmatrix}\in\textup{hom}_{Tw\mathcal{A}}^0(Y_1,C)\cong\mathcal{C}(Y_1)$; since
		\vspace*{-0.2cm}
		\begin{align*}
		\Upsilon^1(i)(b,a_{d-1},...) &= \mu_{Tw\mathcal{A}}^{d+1}(i,b,a_{d-1},...) = \begin{pmatrix} 0\\\mu_{Tw\mathcal{A}}^{d+1}(e_{Y_1},b,a_{d-1},...)\end{pmatrix} \\
		&= \tilde{\iota}^d(b,a_{d-1},...)\;,
		\end{align*}
		we have $H(\Upsilon)\langle i\rangle = \langle\tilde{\iota}\rangle=\langle\iota\rangle$ (see Definition \ref{iotaandpi} and Lemma \ref{iandpi}).
		\item the projection onto the first summand $p\coloneqq(S(e_{Y_0}),0)\in\textup{hom}_{Tw\mathcal{A}}^1(C,Y_0)\cong\mathcal{Y}_0(C)$; since
		\begin{align*}
		\Upsilon^1(p)^d((b_0,b_1),a_{d-1},...)&=\mu_{Tw\mathcal{A}}^{d+1}(p,(b_0,b_1),a_{d-1},...)\\
		&=\mu_{Tw\mathcal{A}}^{d+1}(S(e_{Y_0}),b_0,a_{d-1},...)+0\\
		&=\mu_{Tw\mathcal{A}}^{d+1}(-e_{Y_0},b_0,a_{d-1},...) = \tilde{\pi}^d((b_0,b_1),a_{d-1},...)\;,
		\end{align*}
		we have $H(\Upsilon)\langle p\rangle =\langle\tilde{\pi}\rangle=\langle\pi\rangle$ (in the second-to-last equation we used the action of $S$ explained in Remark \ref{alternatecone}).
	\end{itemize}
\end{Def}

\begin{Lem}\label{exactinTwA}						
	A triangle $(\langle c_1\rangle,\langle c_2\rangle,\langle c_3\rangle)$ in $H^0(Tw\mathcal{A})$ like \eqref{Atriangle} is exact if and only if there is an isomorphism $\langle b\rangle\in\textup{Hom}_{H(Tw\mathcal{A})}^0(Y_2,Cone(c_1))$ making the following diagram commute in $H^0(Tw\mathcal{A})$:
	\begin{equation}\label{Aexacttriangle}
	\begin{tikzcd}
	Y_0\arrow[rr, "\langle c_1\rangle"]& & Y_1\arrow[dl, near end, "\langle c_2\rangle"]\arrow[ddl, bend left=20, "\langle i\rangle"] \\
	&Y_2\arrow[ul, near start, "{\langle c_3\rangle[1]}"]\arrow[d, dashed, "\wr"', "\langle b\rangle"]\\
	&Cone(c_1)\arrow[uul, bend left=20, "{\langle p\rangle[1]}"]
	\end{tikzcd}
	\end{equation}
\end{Lem}

\begin{proof}
	As usual, we write $\mathsf{Q}=mod(Tw\mathcal{A})$ and $C=Cone(c_1)$, and identify $\mathcal{C}=\Upsilon(C)$ (via the isomorphism \eqref{coneidentification}). Commutativity of diagram \eqref{Aexacttriangle} means that $\langle p\rangle\circ\langle b\rangle = \langle\mu_{Tw\mathcal{A}}^2(p,b)\rangle = \langle c_3\rangle$ and $\langle b\rangle\circ\langle c_2\rangle =\langle \mu_{Tw\mathcal{A}}(b,c_2)\rangle = \langle i\rangle$ (note both $b$ and $c_2$ have degree 0). 
	
	That the triangle \eqref{Atriangle} is exact in $H^0(Tw\mathcal{A})$ amounts to say there is an isomorphism $\langle t\rangle\in\textup{Hom}_{H(\mathsf{Q})}^0(\mathcal{Y}_2,\mathcal{C})$ making the diagram \eqref{Qexacttriangle} commute in $H(\mathsf{Q})$. That is, by Lemma \ref{quasi-repr} we have a quasi-representative $\langle t\rangle =\langle \lambda_\mathcal{C}(b)\rangle = H(\Upsilon)\langle b\rangle$ for a uniquely defined isomorphism $\langle b\rangle\in H(\mathcal{C}(Y_2))^0=\textup{Hom}_{H(Tw\mathcal{A})}^0(Y_2,C)$ (using again \eqref{coneidentification}), such that
	\begin{align*}
	&H(\Upsilon)\langle p\rangle\circ H(\Upsilon)\langle b\rangle = \langle\pi\rangle\circ\langle t\rangle = H(\Upsilon)\langle c_3\rangle \qquad\qquad\qquad\text{and} \\
	&H(\Upsilon)\langle b\rangle\circ H(\Upsilon)\langle c_2\rangle = \langle t \rangle\circ H(\Upsilon)\langle c_2\rangle = \langle\iota\rangle = H(\Upsilon)\langle i\rangle\;.
	\end{align*}
	Bijectivity on morphisms of $H(\Upsilon)$ concludes the proof.  
\end{proof}

\begin{Cor}\label{TwAtriang}					
	Let $\mathcal{A}$ be a c-unital $A_\infty$-category. Then the $A_\infty$-category $Tw\mathcal{A}$ is triangulated.
\end{Cor}

\begin{proof}
	We always (implicitly) assumed $\mathcal{A}$ to be a non-trivial $A_\infty$-category, hence, by construction, so is $Tw\mathcal{A}$: $\textup{obj}(Tw\mathcal{A})\neq\emptyset$. Moreover, any degree 0 morphism $\langle c_1\rangle\in\textup{Hom}_{H(Tw\mathcal{A})}^0(Y_0,Y_1)$ is part of the exact triangle \eqref{Atriangle} in $H^0(Tw\mathcal{A})$ where $Y_2\coloneqq Cone(c_1)$, $c_2\coloneqq i$ and $c_3\coloneqq  p$ (in Lemma \ref{exactinTwA}, take $b$ equal the c-unit $e_{Y_2}$ in $Tw\mathcal{A}$). Finally, we observed in Definition \ref{TwAops} that $S:Tw\mathcal{A}\rightarrow Tw\mathcal{A}$ is a quasi-equivalence, so that every twisted complex $Y$ of $Tw\mathcal{A}$ is isomorphic to some shift $S\tilde{Y}\in\textup{obj}(Tw\mathcal{A})$ in $H^0(Tw\mathcal{A})$.
\end{proof}

\begin{Pro}\label{H0TwAtriangulated}				
	Let $\mathcal{A}$ be a c-unital $A_\infty$-category. The zeroth cohomological category $H^0(Tw\mathcal{A})$, if equipped with $T\coloneqq H^0(S):H^0(Tw\mathcal{A})\rightarrow H^0(Tw\mathcal{A})$ as translation functor and exact triangles as the distinguished ones, is a classical triangulated category. 
	
	Moreover, the functor $H^0(Tw\mathcal{F}):H^0(Tw\mathcal{A})\rightarrow H^0(Tw\mathcal{B})$ induced by any c-unital $A_\infty$-functor $\mathcal{F}:\mathcal{A}\rightarrow\mathcal{B}$ is triangulated \textup(cf. Definition \ref{triangfunc}\textup).
\end{Pro}

\begin{proof}
	(\textit{Sketch}) By Corollary \ref{TwAtriang}, $Tw\mathcal{A}$ is a triangulated $A_\infty$-category. Its twisted mapping cones behave like the classical mapping cones of cochain complexes, which allows us to prove the proposition by mimicking the proof of Theorem \ref{DAtriangulated}, where $\mathcal{A}$ ``plays the role'' of the category $\mathsf{A}$ and $Tw\mathcal{A}$ that of the associated homotopy category $\mathsf{H(A)}$. Consequently, $\mathsf{D}(\mathcal{A})\coloneqq H^0(Tw\mathcal{A})$ is a classical triangulated category.\footnote{This notation for $A_\infty$-categories was introduced in Definition \ref{Ainfderivedcat}, and will become legitimate in this setting when we prove that indeed $Tw\mathcal{A}$ is a triangulated envelope for $\mathcal{A}$ (see Proposition \ref{envelopeexists}).}
	
	Now, one can show that the c-unital $A_\infty$-functor $Tw\mathcal{F}:Tw\mathcal{A}\rightarrow Tw\mathcal{B}$ commutes with $S$ and allows the identification $Tw\mathcal{F}(Cone(c))\cong Cone((Tw\mathcal{F})^1(c))$ between cones which is compatible with their canonical morphisms $i$ and $p$. Therefore, $H^0(Tw\mathcal{F}):H^0(Tw\mathcal{A})\rightarrow H^0(Tw\mathcal{B})$ maps exact triangles to exact triangles, meaning that it is triangulated. 
\end{proof}

\subsection{The bounded derived category}\label{ch3.6}

In this section we apply the results obtained for $Tw\mathcal{A}$ to prove that $H^0(\mathcal{A})$ is triangulated as soon as $\mathcal{A}$ is. Furthermore, we show that $Tw\mathcal{A}$ is a valid choice of triangulated envelope, yielding the definition of bounded derived category.

\begin{Lem}\label{exactinclpreserved}			
	A triangle \eqref{Atriangle} in $H(\mathcal{A})$ is exact if and only if its image under the embedding $H(\mathcal{A})\hookrightarrow H(Tw\mathcal{A})$ is exact in $H(Tw\mathcal{A})$.
\end{Lem}

\begin{proof}
	We specified in Section \ref{ch3.5} how to construct a cohomologically full and faithful embedding $\imath:\mathcal{A}\hookrightarrow Tw\mathcal{A}$: $\imath(X)\coloneqq(X,0)\equiv\big((\{\ast\},\{X\},\{\mathbb{K}\}),0\big)$ on objects, and it preserves morphisms. By Lemma \ref{whenexact}, if \eqref{Atriangle} is exact in $H(\mathcal{A})$, we can find morphisms $h_1$, $h_2$ and $k$ in $H(\mathcal{A})$ fulfilling the equations \eqref{hhk}. Due to bijectivity of $H(\imath)$, there exist counterparts in $H(Tw\mathcal{A})$ with the same property, and vice versa, so that it really makes no difference whether we consider \eqref{Atriangle} as sitting in $H(\mathcal{A})$ or in $H(Tw\mathcal{A})$. 
	
	We still need to prove acyclicity of $K(X)\coloneqq \mathcal{Y}_2(X)[1]\oplus\mathcal{Y}_0(X)[1]\oplus\mathcal{Y}_1(X)$ for all objects $X$. One direction is clear. For the converse, we can filtrate any $X\in\textup{obj}(Tw\mathcal{A})$ so that in the corresponding ascending filtration of $K(X)$ all subcomplexes are acyclic, being finite direct sums of shifted, acyclic copies $K(\tilde{X})$ with $\tilde{X}\in\textup{obj}(\mathcal{A})$.
\end{proof}  

\begin{Cor}\label{HFexact}							
	Let $\mathcal{F}:\mathcal{A}\rightarrow\mathcal{B}$ be any $A_\infty$-functor. Then $H(\mathcal{F})$ maps exact triangles in $H(\mathcal{A})$ to exact triangles in $H(\mathcal{B})$.
\end{Cor}

\begin{proof}
	Consider our usual exact triangle \eqref{Atriangle} in $H(\mathcal{A})$, then exact also in $H(Tw\mathcal{A})\supset H(\mathcal{A})$, by Lemma \ref{exactinclpreserved}. Its image under $H(Tw\mathcal{F})$ (coinciding with that of $H^0(Tw\mathcal{F})$ after identifying $\langle c_3\rangle$ as a degree 0 morphism) is still exact in $H(Tw\mathcal{B})$, by Proposition \ref{H0TwAtriangulated}. Clearly, $H(Tw\mathcal{F})|_{H(\mathcal{A})}: H(\mathcal{A})\rightarrow H(\mathcal{B})$ coincides with $H(\mathcal{F})$, and thus maps \eqref{Atriangle} to an exact triangle in $H(\mathcal{B})\subset H(Tw\mathcal{B})$. Applying Lemma \ref{exactinclpreserved} again, the latter is equivalent to exactness in the category $H(\mathcal{B})$ itself.   
\end{proof}

We are finally ready to state and prove a more efficient criterion for exactness of triangles in $H(\mathcal{A})$. First, we give an apparently arbitrary definition.

\begin{Def}												
	We define $\mathcal{D}$ to be the strictly-unital $A_\infty$-category  with $\textup{obj}(\mathcal{D})$ $\coloneqq\{Z_0,Z_1,Z_2\}$ and 
	\vspace*{-0.1cm}
	\begin{itemize}[leftmargin=0.5cm]
		\renewcommand{\labelitemi}{\textendash}
		\item only non-trivial morphism spaces:
		\begin{align*}
		&\textup{hom}_\mathcal{D}(Z_k,Z_k) \coloneqq \mathbb{K}\cdot e_{Z_k}\equiv\{k\cdot e_{Z_k}\mid k\in\mathbb{K}\},\qquad &\textup{hom}_\mathcal{D}(Z_0,Z_1)\coloneqq\mathbb{K}\cdot x_1, \\
		&\textup{hom}_\mathcal{D}(Z_1,Z_2)\coloneqq\mathbb{K}\cdot x_2,
		&\textup{hom}_\mathcal{D}(Z_2,Z_0)\coloneqq\mathbb{K}\cdot x_3,
		\end{align*}
		for some fixed (forcedly non-invertible) morphisms $x_i$ with $|x_1|=|x_2|=0$ and $|x_3|=1$;
		\item only non-trivial compositions $\mu_\mathcal{D}^3(x_3,x_2,x_1)\coloneqq e_{Z_0}$, $\mu_\mathcal{D}^3(x_1,x_3,x_2)\coloneqq e_{Z_1}$ and $\mu_\mathcal{D}^3(x_2,x_1,x_3)\coloneqq e_{Z_2}$. 
	\end{itemize}
\end{Def}

\begin{Pro}\label{exactcriterium}				
	Let $\mathcal{A}$ be a c-unital $A_\infty$-category, $\mathcal{D}$ as above. A triangle \eqref{Atriangle} in $H(\mathcal{A})$ is exact if and only if there exists an $A_\infty$-functor $\mathcal{F}:\mathcal{D}\rightarrow\mathcal{A}$ such that $\mathcal{F}(Z_k) = Y_k$ \textup(for $k=0,1,2$\textup) and $H(\mathcal{F})\langle x_k\rangle=\langle c_k\rangle$ \textup(for $k=1,2,3$\textup).
\end{Pro}

\begin{proof}
	``$\Leftarrow$'' The objects $Z_0,Z_1,Z_2$ form a triangle $(\langle x_1\rangle,\langle x_2\rangle,\langle x_3\rangle)$ in $H(\mathcal{D})$ which is exact, since the choice $h_1=h_2=k=0$  satisfies all the requisites of Lemma \ref{whenexact} (indeed: 2-compositions vanish, $\mu_\mathcal{D}^3(x_1,x_3,x_2) = e_{Z_1}$ and the complex $\textup{hom}_\mathcal{D}(Z,Z_2)[1]\oplus\textup{hom}_\mathcal{D}(Z,Z_0)[1]\oplus\textup{hom}_\mathcal{D}(Z,Z_1)=\{0\}$ is acyclic for every $Z\in\textup{obj}(\mathcal{D})$). By Corollary \ref{HFexact}, the objects $Y_k\coloneqq\mathcal{F}(Z_k)\in\textup{obj}(\mathcal{A})$ and morphisms $\langle c_k\rangle\coloneqq H(\mathcal{F})\langle x_k\rangle$ form an exact triangle $(\langle c_1\rangle,\langle c_2\rangle,\langle c_3\rangle)$ in $H(\mathcal{A})$.
	
	``$\Rightarrow$'' Consider first the case where $\mathcal{A}$ is strictly unital (hence $Tw\mathcal{A}$ is, with strict units as specified in Section \ref{ch3.5}). Fixed some cocycle $c\in\textup{hom}_\mathcal{A}^0(Y_0,Y_1)$, whose cone has canonical morphisms $i$ and $p$, define the $A_\infty$-subcategory $\tilde{\mathcal{D}}\subset Tw\mathcal{A}$ to have objects $\tilde{Z}_0\coloneqq Y_0$, $\tilde{Z}_1\coloneqq Y_1$, $\tilde{Z}_2\coloneqq Cone(c)=SY_0\oplus Y_1$ and morphism spaces:
	\begin{alignat*}{2}
	&\textup{hom}_{\tilde{\mathcal{D}}}(\tilde{Z}_0,\tilde{Z}_0)\coloneqq\mathbb{K}\cdot e_{Y_0}\,, &\;\; &\textup{hom}_{\tilde{\mathcal{D}}}(\tilde{Z}_1,\tilde{Z}_1)\coloneqq\mathbb{K}\cdot e_{Y_1}\,, \\ 
	&\textup{hom}_{\tilde{\mathcal{D}}}(\tilde{Z}_2,\tilde{Z}_2)\coloneqq
	\bigg\{\!\!\!\begin{pmatrix}\lambda_{00}\!\cdot\!S(e_{Y_0}) & 0 \\ \lambda_{10}\!\cdot\!S(c) & \lambda_{11}\!\cdot\! e_{Y_1}\end{pmatrix}\!\!\!\bigg\}, & 
	& \hfill	\\
	&\textup{hom}_{\tilde{\mathcal{D}}}(\tilde{Z}_1,\tilde{Z}_0)\coloneqq0\,, &
	&\textup{hom}_{\tilde{\mathcal{D}}}(\tilde{Z}_0,\tilde{Z}_1)\coloneqq\mathbb{K}\cdot c\,, \\
	&\textup{hom}_{\tilde{\mathcal{D}}}(\tilde{Z}_2,\tilde{Z}_1)\coloneqq\{(\nu_{10}\cdot S(c), \nu_{11}\cdot e_{Y_1})\}\,, &
	&\textup{hom}_{\tilde{\mathcal{D}}}(\tilde{Z}_1,\tilde{Z}_2)\coloneqq\mathbb{K}\cdot i=\mathbb{K}\!\cdot\!\begin{pmatrix} 0 \\ e_{Y_1}\end{pmatrix}\!\!, \\
	&\textup{hom}_{\tilde{\mathcal{D}}}(\tilde{Z}_2,\tilde{Z}_0)\coloneqq\mathbb{K}\cdot p=\mathbb{K}\cdot(S(e_{Y_0}),0)\,, &
	&\textup{hom}_{\tilde{\mathcal{D}}}(\tilde{Z}_0,\tilde{Z}_2)\coloneqq\bigg\{\!\!\!\begin{pmatrix} \sigma_{00}\cdot e_{Y_0} \\ \sigma_{10}\cdot c\end{pmatrix}\!\!\!\bigg\}\,,
	\end{alignat*}
	where all the $\lambda_{ij}$'s, $\nu_{ij}$'s and $\sigma_{ij}$'s range over $\mathbb{K}$. 
	
	Our goal is to construct $\mathcal{D}$ using the Homological Perturbation Lemma like we did in the proof of Lemma \ref{homotopyinverse}. Define $T^1\in\text{End}(\textup{hom}_{\tilde{\mathcal{D}}}(\tilde{Z}_i,\tilde{Z}_j))$ such that $T^1(\lambda_{00},\lambda_{10},\lambda_{11})=(0,0,-\lambda_{10})$ on $\textup{hom}_{\tilde{\mathcal{D}}}(\tilde{Z}_2,\tilde{Z}_2)$, $T^1(\nu_{10},\nu_{11})=(0,-\nu_{10})$ on $\textup{hom}_{\tilde{\mathcal{D}}}(\tilde{Z}_2,\tilde{Z}_1)$ and $T^1(\sigma_{00},\sigma_{10})=(-\sigma_{10},0)$ on $\textup{hom}_{\tilde{\mathcal{D}}}(\tilde{Z}_0,\tilde{Z}_2)$, while $T^1=0$ otherwise. By construction, $\mathcal{E}^1\coloneqq \textup{id}+\mu_{\tilde{\mathcal{D}}}^1\circ T^1+T^1\circ\mu_{\tilde{\mathcal{D}}}^1$ is a strictly unital projection of $\textup{hom}_{\tilde{\mathcal{D}}}(\tilde{Z}_i,\tilde{Z}_j)$ onto a subcomplex $\textup{hom}_{\mathcal{B}}(\tilde{Z}_i,\tilde{Z}_j)$ with vanishing differential. Taking $\mathcal{G}^1$ to be the associated inclusion, diagram \eqref{perturbdiagram} reads:
	\vspace*{-0.3cm}
	\[
	\begin{tikzcd}[cells={nodes={}}]
	{\textup{hom}_\mathcal{B}(\tilde{Z}_i,\tilde{Z}_j)}\arrow[r, yshift=1ex, "{\mathcal{G}^1}"]
	& {\textup{hom}_{\tilde{\mathcal{D}}}(\tilde{Z}_i,\tilde{Z}_j)}\arrow[l, yshift=-1ex, "{\mathcal{E}^1}"]\arrow[loop right, distance=3em, start anchor={[yshift=1ex]east}, end anchor={[yshift=-1ex]east}]{}{T^1}
	\end{tikzcd}
	\vspace*{-0.3cm}
	\]
	Proposition \ref{perturbation} constructs an $A_\infty$-category $\mathcal{B}$ quasi-isomorphic to $\tilde{\mathcal{D}}$. But a closer inspection reveals that $\mathcal{B}=\mathcal{D}$. Indeed, defining $Z_k\coloneqq \mathcal{E}(\tilde{Z}_k)$, $x_1\coloneqq c$, $x_2\coloneqq i$ and $x_3\coloneqq p$, we obtain:
	\begin{align*}
	&\textup{hom}_\mathcal{B}(Z_k,Z_k)=\mathcal{E}^1(\textup{hom}_{\tilde{\mathcal{D}}}(\tilde{Z}_k,\tilde{Z}_k)) = \mathbb{K}\cdot e_{\mathcal{E}(Y_k)} = \mathbb{K}\cdot e_{Z_k}\;(\text{for }k=0,1)\,,\\
	&\textup{hom}_\mathcal{B}(Z_0,Z_1)=\mathcal{E}^1(\textup{hom}_{\tilde{\mathcal{D}}}(\tilde{Z}_0,\tilde{Z}_1)) =\mathbb{K}\cdot(c+\mu_{\tilde{\mathcal{D}}}^1(0)+T^1(0)) = \mathbb{K}\cdot x_1\,, \\
	&\textup{hom}_\mathcal{B}(Z_1,Z_2)=\mathcal{E}^1(\textup{hom}_{\tilde{\mathcal{D}}}(\tilde{Z}_1,\tilde{Z}_2)) = \mathbb{K}\cdot(i+\mu_{\tilde{\mathcal{D}}}^1(0)+\cancel{T^1(\mu_{\tilde{\mathcal{D}}}^1(i)})) = \mathbb{K}\cdot x_2\,,\\
	&\textup{hom}_\mathcal{B}(Z_2,Z_0)=\mathcal{E}^1(\textup{hom}_{\tilde{\mathcal{D}}}(\tilde{Z}_2,\tilde{Z}_0)) = \mathbb{K}\cdot(p+\mu_{\tilde{\mathcal{D}}}^1(0)+\cancel{T^1(\mu_{\tilde{\mathcal{D}}}^1(p)})) = \mathbb{K}\cdot x_3\,,
	\end{align*}
	while the other morphism spaces can be shown to map to $\{0\}$, except for $\textup{hom}_\mathcal{B}(Z_2,Z_2)=\mathcal{E}^1(\textup{hom}_{\tilde{\mathcal{D}}}(\tilde{Z}_k,\tilde{Z}_k)) = \mathbb{K}\cdot\ e_{Z_2}$ (the computations are just messier because this time $T^1$ acts non-trivially). Therefore, we have a quasi-isomorphism $\mathcal{G}:\mathcal{D}\rightarrow\tilde{\mathcal{D}}\subset Tw\mathcal{A}$ fulfilling $\mathcal{G}(Z_k)=\tilde{Z}_k$, $H(\mathcal{G})\langle x_1\rangle = \langle c\rangle$, $H(\mathcal{G})\langle x_2\rangle = \langle i\rangle$ and $H(\mathcal{G})\langle x_3\rangle = \langle p\rangle$.
	
	Now, take $c=c_1$ as in the exact triangle \eqref{Atriangle} in $H(\mathcal{A})$ and define the full $A_\infty$-subcategories $\mathcal{C}\subset\mathcal{A}$ with $\textup{obj}(\mathcal{C})\coloneqq\{Y_0,Y_1,Y_2\}$, respectively $\tilde{\mathcal{C}}\subset Tw\mathcal{A}$ with $\textup{obj}(\tilde{\mathcal{C}})\coloneqq\{\imath(Y_0),\imath(Y_1),\imath(Y_2),Cone(c_1)\}$, where $\imath:\mathcal{A}\hookrightarrow Tw\mathcal{A}$ is the embedding from Lemma \ref{exactinclpreserved}. The latter says that \eqref{Atriangle} has exact image in $H(Tw\mathcal{A})$, hence $Y_2\cong Cone(c_1)$ (by Lemma \ref{exactinTwA}), making the trivial embedding $\imath|_\mathcal{C}:\mathcal{C}\rightarrow\tilde{\mathcal{C}}$ a quasi-equivalence. Theorem \ref{aboutquasiequiv} yields a quasi-inverse $\mathcal{K}:\tilde{\mathcal{C}}\rightarrow\mathcal{C}$ such that $\mathcal{K}(Y_k)=Y_k$ for $k=0,1$, $\mathcal{K}(Cone(c_1))\cong\mathcal{K}(Y_2)=Y_2$, $H(\mathcal{K})\langle c_1\rangle = \langle c_1\rangle$, $H(\mathcal{K})\langle i\rangle = \langle c_2\rangle$ and $H(\mathcal{K})\langle p\rangle = \langle c_3\rangle$. Together with the construction of previous paragraph, it is clear that the $A_\infty$-functor $\mathcal{F}\coloneqq \mathcal{K}\circ\mathcal{G}:\mathcal{D}\rightarrow\mathcal{C}\subset\mathcal{A}$ has the desired properties.
	
	In order to generalize to merely c-unital $A_\infty$-categories $\mathcal{A}$, we take our usual formal diffeomorphism $\Phi:\mathcal{A}\rightarrow\tilde{\mathcal{A}}$ from Lemma \ref{cunitstrictunit}, which preserves exactness of \eqref{Atriangle} by Corollary \ref{HFexact}. Repeating the construction above for the strictly unital $\tilde{\mathcal{A}}$ in lieu of $\mathcal{A}$, we get an $\tilde{\mathcal{F}}:\mathcal{D}\rightarrow\tilde{\mathcal{A}}$ as in the statement. Finally, the $A_\infty$-functor $\mathcal{F}\coloneqq \Phi^{-1}\circ\tilde{\mathcal{F}}:\mathcal{D}\rightarrow\mathcal{A}$ fulfills all requisites.   
\end{proof}

\begin{Cor}\label{preimagetriangle}				
	Let $\mathcal{F}:\mathcal{A}\rightarrow\mathcal{B}$ be a cohomologically full and faithful $A_\infty$-functor, and suppose there is an exact triangle in $H(\mathcal{B})$ formed by objects in the image of $\mathcal{F}$. Then the preimage of this triangle is an exact triangle in $H(\mathcal{A})$.
\end{Cor}

Although cumbersome to prove, Proposition \ref{exactcriterium} makes good use of the machinery we developed in both this and last chapter. More importantly, it highly simplifies our study of exact triangles, moving the focus to a ``sleeker'' $A_\infty$-category such as $\mathcal{D}$. Here is another fundamental result.

\begin{Lem}\label{TwAgeneratedbyA}						
	$Tw\mathcal{A}$ is generated by its full subcategory $\mathcal{A}$. 
\end{Lem} 

The proof essentially builds exact triangles from subcomplexes obtained by filtration of twisted complexes in $Tw\mathcal{A}$, showing that the latters just originate from objects of $\mathcal{A}$ after applying shifts and cones (see \cite[Lemma 3.32]{[Sei08]}). By definition, this means that $\mathcal{A}$ generates $Tw\mathcal{A}$ (however, in contrast to Definition \ref{triangenvelope}, here we are making no assumptions on triangulation yet --- $\mathcal{A}$ need not be triangulated!).

\begin{Lem}\label{ifembquasiequiv}			
	A c-unital $A_\infty$-category $\mathcal{A}$ is triangulated if and only if the embedding $\imath:\mathcal{A}\hookrightarrow Tw\mathcal{A}$ is a quasi-equivalence.
\end{Lem}

\begin{proof}
	Suppose $\imath$ is a quasi-equivalence. For any cocycle $c\in\textup{hom}_{Tw\mathcal{A}}^0(Y_0,Y_1)$, the abstract mapping cone $\mathcal{C}one(c)\in\textup{obj}(mod(Tw\mathcal{A}))$ is quasi-represented by $Cone(c)\in\textup{obj}(Tw\mathcal{A})$, via \eqref{coneidentification}. Since $Tw\mathcal{A}$ is generated by $\mathcal{A}$,  $\mathcal{C}one(c)$ is also quasi-represented by an object of $\mathcal{A}$ (here bijectivity of $H(\imath)$ is crucial). The same holds for upwards and downwards shifts. Therefore, the properties for a triangulated $A_\infty$-category are inferred from those of $Tw\mathcal{A}$, which we know to be triangulated (by Corollary \ref{TwAtriang}).
	
	Conversely, suppose $\mathcal{A}$ is triangulated, thus closed under taking mapping cones and shifts. As $\mathcal{A}$ generates $Tw\mathcal{A}$, any object of $H^0(Tw\mathcal{A})$ can be constructed from one in $H^0(\mathcal{A})$ by repeated application of cones and shifts still belonging to $H^0(\mathcal{A})$, by assumption. Therefore, $H^0(\imath):H^0(\mathcal{A})\rightarrow H^0(Tw\mathcal{A})$ must be an equivalence of categories, and hence $\imath$ a quasi-equivalence. 
\end{proof}

We are ready to prove the main result of this chapter:

\begin{Pro}\label{H0Atriangulated}					
	If $\mathcal{A}$ is a triangulated c-unital $A_\infty$-category, then $H^0(\mathcal{A})$ is a classical triangulated category. Moreover, if $\mathcal{A}$, $\mathcal{B}$ are triangulated c-unital $A_\infty$-categories and $\mathcal{F}:\mathcal{A}\rightarrow\mathcal{B}$ any c-unital $A_\infty$-functor, then the induced functor $H^0(\mathcal{F}):H^0(\mathcal{A})\rightarrow H^0(\mathcal{B})$ is exact. 
\end{Pro}

\begin{proof}
	By Proposition \ref{H0TwAtriangulated}, we already know that $H^0(Tw\mathcal{A})$ is a triangulated category, and by Lemma \ref{ifembquasiequiv}, $\imath:\mathcal{A}\hookrightarrow Tw\mathcal{A}$ is a quasi-equivalence. Therefore, the axioms (T1)--(T3) for a classical triangulated category can be inferred from those for $H^0(Tw\mathcal{A})$, exploiting fullness, faithfulness and essential surjectivity of $H(\imath)$. In particular, the full subcategory $H^0(\mathcal{A})\subset H^0(Tw\mathcal{A})$ has same translation functor $T\coloneqq H^0(S):H^0(\mathcal{A})\rightarrow H^0(\mathcal{A})$ as $H^0(Tw\mathcal{A})$, and its distinguished triangles are those which are distinguished in $H^0(Tw\mathcal{A})$ (the exact ones, by choice).
	
	Moreover, given any two such triangulated categories $H^0(\mathcal{A})$, $H^0(\mathcal{B})$, the exactness of  $H^0(\mathcal{F}):H^0(\mathcal{A})\rightarrow H^0(\mathcal{B})$ can be deduced from that of $H^0(Tw\mathcal{F})$ (cf. again Proposition \ref{H0TwAtriangulated}) and upon using the respective cohomolgical embedding functors, which we just saw to preserve distinguished triangles. 
\end{proof}

\begin{Rem}\label{t2rem}								
	For example, let us spell out axiom (T2). Let $\mathcal{A}$ be a triangulated $A_\infty$-category. Then, if \eqref{Atriangle} is a distinguished triangle in $H^0(\mathcal{A})$, by our convention denoted $(\langle c_1\rangle,\langle c_2\rangle,\langle c_3\rangle)$, then so are the ``rotated'' triangles $(\langle c_2\rangle,\langle c_3\rangle,$ $-\langle S^1(c_1)\rangle)$ and $(-\langle (S^{-1})^1(c_3)\rangle,\langle c_1\rangle,\langle c_2\rangle)$ in $H^0(\mathcal{A})$, corresponding to diagrams:
	\begin{equation}\label{rotatedtriang}
	\begin{tikzcd}
	Y_1\arrow[rr, "\langle c_2\rangle"] & & Y_2\arrow[dl, "\langle c_3\rangle"]\\
	& SY_0\arrow[ul, "{-\langle S^1(c_1)\rangle[1]}"]
	\end{tikzcd}\quad\text{and}\quad
	\begin{tikzcd}
	S^{-1}Y_2\arrow[rr, "-\langle (S^{-1})^1(c_3)\rangle"] & & Y_0\arrow[dl, "\langle c_1\rangle"]\\
	& Y_1\arrow[ul, "{\langle c_2\rangle[1]}"]
	\end{tikzcd},
	\end{equation}
	where $S^{-1}$ induces the quasi-inverse translation $T^{-1}=H^0(S^{-1})$ in $H^0(\mathcal{A})$. With help of the isomorphisms \eqref{Homiso}, we identified $\langle c_3\rangle\!\in\!\textup{Hom}_{H^0(\mathcal{A})}(Y_2,Y_0)[1]$ $\cong\textup{Hom}_{H^0(\mathcal{A})}(Y_2,SY_0)$ and $\langle S^1(c_1)\rangle\in\textup{Hom}_{H^0(\mathcal{A})}(SY_0,SY_1)$ $\cong\textup{Hom}_{H^0(\mathcal{A})}(Y_0,Y_1)\cong\textup{Hom}_{H^0(\mathcal{A})}(SY_0,Y_1)[1]$, respectively $\langle (S^{-1})^1(c_3)\rangle\in$ $\textup{Hom}_{H^0(\mathcal{A})}(S^{-1}Y_2,S^{-1}Y_0)[1]$ $\cong\textup{Hom}_{H^0(\mathcal{A})}(Y_2,Y_0)[1]\cong \textup{Hom}_{H^0(\mathcal{A})}(S^{-1}Y_2,Y_0)$ and $\langle c_2\rangle\in\textup{Hom}_{H^0(\mathcal{A})}(Y_1,Y_2)\cong\textup{Hom}_{H^0(\mathcal{A})}(Y_1,S^{-1}Y_2)[1]$. 
	
	As we mentioned right after Definition \ref{abstractmappingcone}, the vanishing of $c=c_1$ in \eqref{Qtriangle} implies that $\mathcal{Y}_2\cong\mathcal{C}one(c_1)\cong SS\mathcal{Y}_0\oplus\mathcal{Y}_1$ (in $H(\mathsf{Q})$). Using that $S$ quasi-represents $SS$, we get the following interesting implications (the last two can be deduced analogously from the diagrams \eqref{rotatedtriang} and \eqref{Atriangle}):
	\[
	\langle c_1\rangle = 0 \implies Y_2\cong SY_0\oplus Y_1\;,\qquad\qquad
	\langle c_2\rangle = 0 \implies Y_0\cong Y_1\oplus S^{-1}Y_2\;,
	\]
	\[
	\langle c_3\rangle = 0 \implies Y_1\cong Y_2\oplus Y_0\;.
	\] 
\end{Rem}

Finally, we return to the concepts of triangulated envelope and derived category of an $A_\infty$-category, proving their existence and uniqueness up to equivalence.

\begin{Lem}\label{uniqueenvelope}				
	Let $\mathcal{F}:\mathcal{A}\rightarrow\mathcal{B}$ be a cohomologically full and faithful $A_\infty$-functor, $\mathcal{B}$ a triangulated c-unital $A_\infty$-category generated by the image objects of $\mathcal{F}$. Then there exists a quasi-equivalence $\tilde{\mathcal{F}}:Tw\mathcal{A}\rightarrow\mathcal{B}$ such that $\tilde{\mathcal{F}}|_\mathcal{A}\cong \mathcal{F}$ in $H^0(fun(\mathcal{A},\mathcal{B}))$.  
\end{Lem}

\begin{proof}
	By Lemma \ref{ifembquasiequiv}, $\mathcal{B}\hookrightarrow Tw\mathcal{B}$ is a quasi-equivalence, for which we can then find a quasi-inverse $\mathcal{K}:Tw\mathcal{B}\rightarrow\mathcal{B}$ (as in Theorem \ref{aboutquasiequiv}). Define the $A_\infty$-functor $\tilde{\mathcal{F}}\coloneqq \mathcal{K}\circ Tw\mathcal{F}:Tw\mathcal{A}\rightarrow\mathcal{B}$, clearly satisfying $\tilde{\mathcal{F}}|_\mathcal{A}\cong\mathcal{K}\circ\mathcal{F}=\mathcal{F}$. By Lemma \ref{Twfullfaith}, $Tw\mathcal{F}$ is cohomologically full and faithful, and hence so is $\tilde{\mathcal{F}}$. 
	
	As to cohomological essential surjectivity: by assumption, any $Y\in\textup{obj}(\mathcal{B})$ is isomorphic in $\textup{obj}(H^0(\mathcal{B}))$ to an image object $\mathcal{F}(X)=\tilde{\mathcal{F}}(X)$ for some $X\in\textup{obj}(\mathcal{A})$, possibly after applying mapping cones and shifts, which remain in the image of $\tilde{\mathcal{F}}$ because $H(Tw\mathcal{F})$ preserves exact triangles (see Corollary \ref{HFexact}). This amounts to say that $\tilde{\mathcal{F}}$ is cohomologically essentially surjective. Therefore, it is a quasi-equivalence, restricting to $\mathcal{F}$ as an object of $H^0(fun(\mathcal{A},\mathcal{B}))$.
\end{proof}

\begin{Pro}\label{envelopeexists}	
	Every \textup(non-empty\textup) c-unital $A_\infty$-category $\mathcal{A}$ has a triangulated envelope $(\tilde{\mathcal{B}},\mathcal{G})$. Given another such, say $(\mathcal{B},\mathcal{F})$, there is a quasi-equivalence $\tilde{\mathcal{F}}:\tilde{\mathcal{B}}\rightarrow\mathcal{B}$ such that $\tilde{\mathcal{F}}\circ\mathcal{G}\cong\mathcal{F}\in\textup{obj}\big(H^0(fun(\mathcal{A},\mathcal{B}))\big)$.
\end{Pro}

\begin{proof}
	The pair $(\tilde{\mathcal{B}},\mathcal{G})\coloneqq (Tw\mathcal{A},\imath)$ is a valid choice of triangulated envelope for $\mathcal{A}$: $Tw\mathcal{A}$ is triangulated, $\mathcal{A}\equiv\imath(\mathcal{A})$ generates it and the embedding $\imath:\mathcal{A}\hookrightarrow Tw\mathcal{A}$ is cohomologically full and faithful (by construction). 
	
	Moreover, given another envelope $(\mathcal{B},\mathcal{F})$, by Lemma \ref{uniqueenvelope} there exists a quasi-equivalence $\tilde{\mathcal{F}}:Tw\mathcal{A}\rightarrow\mathcal{B}$ such that indeed $\tilde{\mathcal{F}}|_\mathcal{A}=\tilde{\mathcal{F}}\circ\imath:\mathcal{A}\rightarrow\mathcal{B}$ is isomorphic to $\mathcal{F}$ in $H^0(fun(\mathcal{A},\mathcal{B}))$.
\end{proof}

\begin{Def}\label{boundeddercat}			
	Let $\mathcal{A}$ be a c-unital $A_\infty$-category. We call the derived category $\mathsf{D^b}(\mathcal{A})\coloneqq H^0(Tw\mathcal{A})$ the \textbf{bounded derived category}\index{bounded derived category!of an $A_\infty$-category} of $\mathcal{A}$.\footnote{This notation and terminology is purely strategical: as mentioned in Definition \ref{derivedcat}, ``bounded'' just refers to the fact that we started with bounded cochain complexes in the localization construction, here sidestepped!}
\end{Def}

Putting together Corollary \ref{TwFquasiequiv} and Proposition \ref{H0TwAtriangulated}, one easily deduces (see also \cite[Lemma~2.4]{[Sei13]}):

\begin{Cor}\label{equivboundercat}	
	Let $\mathcal{A}$ and $\mathcal{B}$ be c-unital $A_\infty$-categories which are quasi-equivalent. Then $\mathsf{D^b}(\mathcal{A})$ and $\mathsf{D^b}(\mathcal{B})$ are equivalent triangulated categories.
\end{Cor}

\subsection{Example of triangulated $A_\infty$-category: $mod(\mathcal{A})$} \label{ch3.7}

One prominent example of triangulated $A_\infty$-category is that of c-unital $A_\infty$-modules over a c-unital $A_\infty$-category: $\mathsf{Q}=mod(\mathcal{A})$. Let us briefly justify this, first by generalizing the definition of abstract mapping cone.

\begin{Def}\label{abstractconemod}			
	Let $\mathcal{M}_0,\mathcal{M}_1\in\textup{obj}(\mathsf{Q})$, and let $t\in\textup{hom}_\mathsf{Q}^0(\mathcal{M}_0,\mathcal{M}_1)$ be a degree 0 cocycle (that is, a module homomorphism). Its abstract mapping cone $\mathcal{C}=\mathcal{C}one(t)\in\textup{obj}(\mathsf{Q})$ is given by $\mathcal{C}(X)\coloneqq\mathcal{M}_0(X)[1]\oplus\mathcal{M}_1(X)\in\textup{obj}(\mathsf{Ch})$ for $X\in\textup{obj}(\mathcal{A})$, with graded maps
	\begin{align*}
	\mu_\mathcal{C}^d:\;&\mathcal{C}(X_{d-1})\!\otimes\!\textup{hom}_\mathcal{A}(X_{d-2},X_{d-1})\!\otimes\!...\!\otimes\!\textup{hom}_\mathcal{A}(X_0,X_1)\rightarrow\mathcal{C}(X_0)[2-d] \\
	&((b_0,b_1),a_{d-1},...) \!\mapsto\! (\mu_{\mathcal{M}_{0}}^d(b_0,a_{d-1},...),\mu_{\mathcal{M}_{1}}^d(b_1,a_{d-1},...)\!+\!t^d(b_0,a_{d-1},...)),\nonumber
	\end{align*}
	for $d\geq1$ and $X_0,...,X_{d-1}\in\textup{obj}(\mathcal{A})$.
\end{Def}

Upon setting $\mathcal{M}_k\coloneqq\Upsilon(Y_k)=\mathcal{Y}_k$ and $t\coloneqq\Upsilon^1(c)\in\textup{hom}_\mathsf{Q}^0(\mathcal{Y}_0,\mathcal{Y}_1)$, we recover Definition \ref{abstractmappingcone}. Similarly, an adaptation of what we saw for $\mathcal{A}$ and $Tw\mathcal{A}$ shows (cf. \cite[Lemma 3.35]{[Sei08]}):

\begin{Lem}											
	Given a module homomorphism $t\in\textup{hom}_\mathsf{Q}^0(\mathcal{M}_0,\mathcal{M}_1)$, the corresponding abstract mapping cone $\mathcal{C}=\mathcal{C}one(t)$ comes with the canonical inclusion $\iota\in\textup{hom}_\mathsf{Q}^0(\mathcal{M}_1,\mathcal{C})$ and projection $\pi\in\textup{hom}_\mathsf{Q}^1(\mathcal{C},\mathcal{M}_0)$ making the triangle
	\[
	\begin{tikzcd}
	\mathcal{M}_0\arrow[rr, "\langle t\rangle"] & & \mathcal{M}_1\arrow[dl, "\langle \iota\rangle"] \\
	& \mathcal{C}\arrow[ul, "{\langle\pi\rangle[1]}"]
	\end{tikzcd}
	\]
	exact in $H^0(\mathsf{Q})$. In particular, any morphism $\langle t\rangle$ in $H^0(\mathsf{Q})$ can be completed to an exact triangle as above.
\end{Lem} 

Existence of a shift functor $SS:\mathsf{Q}\rightarrow\mathsf{Q}$ which is in particular a quasi-equivalence was already proved in Section \ref{ch3.2}. Together with last lemma, this shows that $\mathsf{Q}=mod(\mathcal{A})$ is a triangulated $A_\infty$-category.

\begin{Rem}\label{alternenv}					
	Considering $\mathcal{F}=\Upsilon_\mathcal{A}:\mathcal{A}\rightarrow\tilde{\mathsf{Q}}$ (cohomologically full and faithful), where $\tilde{\mathsf{Q}}\subset\mathsf{Q}$ is the triangulated full $A_\infty$-subcategory generated by the image objects of $\Upsilon_\mathcal{A}$, we obtain another canonical triangulated envelope for $\mathcal{A}$. Indeed, by Proposition \ref{envelopeexists}, there exists a quasi-equivalence $\tilde{\Upsilon}: Tw\mathcal{A}\rightarrow\tilde{\mathsf{Q}}$ such that $\tilde{\Upsilon}\circ\imath\cong\Upsilon_\mathcal{A}$ in $H^0(fun(\mathcal{A},\tilde{\mathsf{Q}}))$. One can prove it to be exactly $\tilde{\Upsilon}=\imath^*\circ\Upsilon_{Tw\mathcal{A}}$ (with pullback $\imath^*:mod(Tw\mathcal{A})\rightarrow mod(\mathcal{A})$ as in Definition \ref{pullback}). We obtain the commutative (up to isomorphism) diagram
	\[
	\begin{tikzcd}
	\mathcal{A}\arrow[r, hook, "\imath"]\arrow[d, "\Upsilon_\mathcal{A}"'] & Tw\mathcal{A}\arrow[d, "\Upsilon_{Tw\mathcal{A}}"]\arrow[dl, "\tilde{\Upsilon}"']\\
	\tilde{\mathsf{Q}} & {mod(Tw\mathcal{A})}\arrow[l, "\imath^*"]
	\end{tikzcd}\quad,
	\]
	which is precisely the one given in Remark \ref{commutingYoneda}. (See \cite[section 7.5]{[Kel01]}, where the author suggests this can be used to characterize $Tw\mathcal{A}$ by a universal property.)  
\end{Rem}

\newpage

\section{Split-generation and twisting}
\thispagestyle{plain}

\subsection{Category theoretic preliminaries}\label{ch4.1}

\begin{Def}\label{imageofidem}							
	Let $\mathsf{C}$ be a linear category over $\mathbb{K}$. Consider an \textbf{idempotent endomorphism}\index{idempotent endomorphism} $p$ of some $Y\in\textup{obj}(\mathsf{C})$, that is, a morphism $p\in\textup{Hom}_\mathsf{C}(Y,Y)$ such that $p\circ p=p$. A triple $(Z,r,k)$, where $Z\in\textup{obj}(\mathsf{C})$, $r\in\textup{Hom}_\mathsf{C}(Y,Z)$ (retraction) and $k\in\textup{Hom}_\mathsf{C}(Z,Y)$ (inclusion), is called the \textbf{image}\index{image} of $p$ if fulfilling 
	\[
	r\circ k=\textup{id}_Z\qquad\text{and}\qquad k\circ r=p\,.
	\]
	Accordingly, we say that $p$ \textbf{splits}\index{split (idempotent endomorphism)} if such a triple exists. We encode the data in the following diagram:
	\vspace*{-0.2cm}
	\begin{equation}\label{psplit}
	\begin{tikzcd}
	Z\arrow[r, yshift=1ex, "k"]\arrow[loop left, distance=1.5em, start anchor={[yshift=1ex]west}, out=180, end anchor={[yshift=-0.5ex]west}, in=180, "\textup{id}_Z"']
	& Y\arrow[l, yshift=-0.5ex, "r"]\!\arrow[loop right, distance=1.5em, start anchor={[yshift=-0.5ex]east}, out=0, end anchor={[yshift=1ex]east}, in=0, "p"']
	\end{tikzcd}
	\end{equation}
	The category $\mathsf{C}$ is \textbf{split-closed}\index{category!split-closed} if all its idempotent endomorphisms split.
\end{Def}

Note that \textit{the} image of an idempotent $p$, if it exists, is unique: given another suitable triple $(Z',r',k')$, $r'\circ k: Z\rightarrow Z'$ is an isomorphism of objects with inverse $r\circ k'$. 

Observe also that if $\mathsf{C}$ is additive (for example, triangulated) and split-closed, given an idempotent $p\in\textup{Hom}_\mathsf{C}(Y,Y)$, we may consider the pair $Z\coloneqq \text{im}(p)$ and $Z^\perp\coloneqq \text{im}(\textup{id}_Y-p)$, yielding the decomposition $Y=Z\oplus Z^\perp$.

\begin{Def}										
	A \textbf{split-closure}\index{split-closure} of a category $\mathsf{C}$ is a pair $(\mathsf{D},\mathsf{F})$ where $\mathsf{D}$ is a split-closed category and $\mathsf{F}:\mathsf{C}\rightarrow\mathsf{D}$ a fully faithful embedding such that any $Z\in\textup{obj}(\mathsf{D})$ is the image of an idempotent endomorphism $p\in\textup{Hom}_\mathsf{D}(\mathsf{F}(X),\mathsf{F}(X))$ for some $X\in\textup{obj}(\mathsf{C})$.
\end{Def}

\begin{Rem}									
	Any category $\mathsf{C}$ has a tautological split-closure $\bar{\mathsf{C}}$, called \textit{Karoubi completion} (or \textit{idempotent completion}). It is defined as follows: objects are pairs $(Y,p)$ for $Y\in\textup{obj}(\mathsf{C})$ and $p\in\textup{Hom}_\mathsf{C}(Y,Y)$ an idempotent endomorphism (every object of $\mathsf{C}$ has at least one such, namely $\textup{id}_Y$), and morphisms in $\textup{Hom}_{\bar{\mathsf{C}}}((Y_0,p_0),(Y_1,p_1))$ $\coloneqq  p_1\circ\textup{Hom}_\mathsf{C}(Y_0,Y_1)\circ p_0$ are triples $(p_0,f,p_1)\coloneqq p_1\circ f\circ p_0$ for $f\in\textup{Hom}_\mathsf{C}(Y_0,Y_1)$ fulfilling $f=p_1\circ f\circ p_0$. Compositions is the obvious one induced by $\mathsf{C}$, and identity morphisms are all $\textup{id}_{(Y,p)}\coloneqq (p,p,p)=p$. The full and faithful embedding is $\mathsf{E}:\mathsf{C}\rightarrow\bar{\mathsf{C}}$, $Y\mapsto(Y,\textup{id}_Y)$ and $p\mapsto(\textup{id}_Y,p,\textup{id}_Y)$.
	
	Indeed, $\bar{\mathsf{C}}$ is split-closed: for any idempotent $\bar{p}=(p,s,p)\!\in\!\textup{Hom}_{\bar{\mathsf{C}}}((Y,p),(Y,p))$ (where $s\in\textup{Hom}_\mathsf{C}(Y,Y)$ necessarily fulfills $s=p\circ s\circ p$ and $s=s\circ p\circ s$), $Z\coloneqq (Y,\textup{id}_Y)$, $r\coloneqq (p,s,\textup{id}_Y):(Y,p)\rightarrow Z$ and $k\coloneqq (\textup{id}_Y,s,p):Z\rightarrow(Y,p)$ fulfill $r\circ k=\textup{id}_Y\circ s\circ \textup{id}_Y = \mathsf{E}(s)$ and $k\circ r=p\circ s\circ p = \bar{p}$. 
\end{Rem}

In fact, one can prove that any split-closure $(\mathsf{D},\mathsf{F})$ of $\mathsf{C}$ is equivalent to the Karoubi completion $(\bar{\mathsf{C}},\mathsf{E})$. Moreover, if $\mathsf{C}$ is triangulated, so is $\bar{\mathsf{C}}$, along with the embedding $\mathsf{E}$. Here is a more specific result.

\begin{Lem}\label{Atildetriang}						
	Let $\mathsf{B}=(\mathsf{B},T)$ be a split-closed triangulated category, $\mathsf{A}\subset\mathsf{B}$ a triangulated full subcategory \textup(thus closed under translations and exact triangles\textup). Let $\tilde{\mathsf{A}}\subset\mathsf{B}$ be the full subcategory of all objects isomorphic to a direct summand of an object in $\mathsf{A}$ \textup(particularly, $\tilde{\mathsf{A}}\cong\bar{\mathsf{A}}$ because the latter is split-closed\textup). Then $\tilde{\mathsf{A}}$ is triangulated.
\end{Lem}

\begin{proof}
	Firstly, note that we are working with additive categories, so that it makes sense to talk about direct sum decompositions (cf. above). By assumption, any $\tilde{Y}_i\in\textup{obj}(\tilde{\mathsf{A}})$ appears in some $Y\cong\tilde{Y}_i\oplus\tilde{Y}_{i}^\perp\in\textup{obj}(\mathsf{A})$. Hence, we can define $r_i, k_i$ to be the corresponding retraction and inclusion, so that $r_i\circ k_j=\textup{id}_{\tilde{Y}_i}$. 
	
	By axiom (T1) for triangulated categories (see Definition \ref{Verdier}), we can complete the morphisms $\tilde{c}:\tilde{Y}_0\rightarrow \tilde{Y}_1$ and $c\coloneqq k_1\circ\tilde{c}\circ r_0:Y_0\rightarrow Y_1$ to exact triangles, involving some third objects $\tilde{Y}_2$ respectively $Y_2$. These form two diagrams as in Lemma \ref{tr3} (one with inclusions $k_0, k_1$ as vertical arrows, the other with retractions $r_0, r_1$), whose leftmost squares commute, by definition of $c$. Therefore, we can find complementing morphisms $k_2:\tilde{Y}_2\rightarrow Y_2$ and $r_2:Y_2\rightarrow\tilde{Y}_2$ making the whole diagrams commute. Composing them, we obtain:
	\[
	\begin{tikzcd}
	\tilde{Y}_0\arrow[d, "r_0\circ k_0"']\arrow[r, "\tilde{c}"] & \tilde{Y}_1\arrow[d, "r_1\circ k_1"']\arrow[r] & \tilde{Y}_2\arrow[d, dashed, "r_2\circ k_2"]\arrow[r] & T(\tilde{Y}_0)\arrow[d, "T(r_0\circ k_0)"] \\
	\tilde{Y}_0\arrow[r, "\tilde{c}"'] & \tilde{Y}_1\arrow[r] & \tilde{Y}_2\arrow[r] & T(\tilde{Y}_0) 
	\end{tikzcd}
	\]
	Since the first two vertical arrows are isomorphisms (being identity morphisms), a second application of Lemma \ref{tr3} tells us that $r_2\circ k_2$ is an isomorphim as well, thus invertible, and moreover commutativity of the diagram forces it to be the identity $\textup{id}_{\tilde{Y}_2}$. Hence, taking as endomorphism $p\coloneqq k_2\circ r_2\in\textup{Hom}_\mathsf{A}(Y_2,Y_2)$, which is clearly idempotent, we see that it splits via $(\tilde{Y}_2,r_2,k_2)$, making $\tilde{Y}_2$ isomorphic to a direct summand of $Y_2$. By definition, this means that $\tilde{Y}_2\in\textup{obj}(\tilde{\mathsf{A}})$, so that $\tilde{\mathsf{A}}$ wholly contains the exact triangle pictured above. Therefore, $\tilde{\mathsf{A}}$ is a triangulated category.
\end{proof}

\subsection{Abstract images}\label{ch4.2}

Let us translate the theory we have just seen in the language of $A_\infty$-modules. Thereto, identify our ground field $\mathbb{K}$ as the trivial $A_\infty$-category with structure $\textup{obj}(\mathbb{K}) \coloneqq \{\ast\}$ and $\textup{hom}_\mathbb{K}(\ast,\ast)$ $\coloneqq \{\textup{id}_\ast\}$, thus coinciding with its cohomological category, $H(\mathbb{K})\cong\mathbb{K}$.

\begin{Def}									
	Let $\mathcal{A}$ be an $A_\infty$-category, $\mathsf{Q}=mod(\mathcal{A})$. An \textbf{idempotent up to homotopy}\index{idempotent up to homotopy} for $Y\in\textup{obj}(\mathcal{A})$ is a non-unital $A_\infty$-functor $\wp:\mathbb{K}\rightarrow\mathcal{A}$ such that $\wp(\ast)=Y$ and $\wp^d\equiv\wp^d(\textup{id}_\ast,...,\textup{id}_\ast)\in\textup{hom}_\mathcal{A}^{1-d}(Y,Y)$ for $d\geq 1$, fulfilling the corresponding equations \eqref{Ainffunceqs}:
	\begin{equation*}
	\sum_{\substack{1\leq r\leq d \\ s_1+...+s_r = d}}\mkern-18mu\mu_{\mathcal{A}}^r(\wp^{s_r},...,\wp^{s_1}) = 
	\begin{cases}
	\wp^{d-1} & \text{if } d \text{ is even} \\ 
	0 & \text{if } d \text{ is odd}
	\end{cases},
	\end{equation*}
	in $\textup{hom}_\mathcal{A}^{2-d}(Y,Y)$.
\end{Def}

One can show the following relation between idempotent endomorphisms and idempotents up to homotopy (see \cite[Lemma 4.2]{[Sei08]}).

\begin{Lem}\label{forwpp}								
	Let $p\in\textup{Hom}_{H(\mathcal{A})}^0(Y,Y)$ be idempotent, $Y\in\textup{obj}(\mathcal{A})$. Then there is an idempotent up to homotopy $\wp:\mathbb{K}\rightarrow\mathcal{A}$ for $Y$ such that $\langle\wp^1\rangle = p$.
\end{Lem}

We work towards the description of $\wp$ as image of some $A_\infty$-module, similarly to what we did in previous section for idempotents $p$.

\begin{Def}											
	Let $\wp:\mathbb{K}\rightarrow\mathcal{A}$ be an idempotent up to homotopy for some $Y\in\textup{obj}(\mathcal{A})$. Then we define the \textbf{abstract image}\index{abstract image} of $\wp$ to be the $A_\infty$-module $\mathcal{Z}\in\textup{obj}(\mathsf{Q})$ which to any $X\in\textup{obj}(\mathcal{A})$ associates $\mathcal{Z}(X)\coloneqq\textup{hom}_\mathcal{A}(X,Y)[q]$, a graded vector space of formal polynomials in one variable $q$ of degree $-1$. Writing $\delta_q$ for the operator of formal differentiation, so that $\delta_q(q^0)=0$ and $\delta_q(q^k)=q^{k-1}$, the composition maps are
	\[
	\mu_\mathcal{Z}^1(b(q)) \coloneqq \mkern-18mu\sum_{\substack{r\geq 2 \\ s_2,...,s_r \geq 1}}\mkern-18mu\delta_{q}^{s_2+...+s_r}\mu_\mathcal{A}^r(\wp^{s_r},...,\wp^{s_2},b(q)) + (-1)^{|b|}\frac{b(q)\!-\! b(-q)}{2q}\in\mathcal{Z}(X_0)[1], 
	\]\vspace*{-0.3cm}
	\begin{align*}
	\mu_\mathcal{Z}^d:\; &\mathcal{Z}(X_{d-1})\otimes\textup{hom}_\mathcal{A}(X_{d-2},X_{d-1})\otimes...\otimes\textup{hom}_\mathcal{A}(X_0,X_1)\rightarrow\mathcal{Z}(X_0)[2-d],\\
	& (b(q),a_{d-1},...,a_1)\mapsto\mkern-24mu\sum_{\substack{r\geq d+1 \\ s_{d+1},...,s_r \geq 1}} \mkern-21mu\delta_{q}^{s_{d+1}+...+s_r}\mu_\mathcal{A}^r(\wp^{s_r},...,\wp^{s_{d+1}},b(q),a_{d-1},...,a_1)\,,
	\end{align*}
	for $d>1$.
\end{Def}

Recurring to filtrations of $\mathcal{Z}(X)$ and studying the associated spectral sequences, one eventually obtains the following characterizations on the level of cohomological modules (recall Definition \ref{cohommod}).

\begin{Lem}\label{aboutcohomZ}		
	Consider the abstract image $\mathcal{Z}$ of an idempotent up to homotopy $\wp$ for some $Y\in\textup{obj}(\mathcal{A})$, and choose $p\in\textup{Hom}_{H(\mathcal{A})}^0(Y,Y)$ such that $\langle\wp^1\rangle = p$ is idempotent \textup(possible by Lemma \ref{forwpp}\textup).
	\renewcommand{\theenumi}{\roman{enumi}}
	\begin{enumerate}[leftmargin=0.5cm]
		\item The cohomological module $H(\mathcal{Z})\in\textup{obj}\big(\mathsf{Mod}^0(H(\mathcal{A}))\big)$\footnote{Since we are implicitly assuming c-unitality of $\mathcal{A}$, indeed $H(\mathcal{Z}):H(\mathcal{A})^{opp}\rightarrow\mathsf{GradVect}$ is a contravariant \textit{unital} functor, that is, an object of $\mathsf{Mod}^0(H(\mathcal{A}))$ (cf. Definition \ref{cohommod}).} is described as follows: $H(\mathcal{Z})(X)=H(\mathcal{Z}(X))\cong p\circ\textup{Hom}_{H(\mathcal{A})}(X,Y)\in\textup{obj}(\mathsf{GradVect})$ for any $X\in\textup{obj}(\mathcal{A})$ \textup{\big(}in particular, $H(\mathcal{Z})$ is a direct summand of $H(\Upsilon(Y))$\textup{\big)}, with module multiplication
		\[
		H(\mathcal{Z})(X_1)\otimes\textup{Hom}_{H(\mathcal{A})}(X_0,X_1)\rightarrow H(\mathcal{Z})(X_0),\, (p\circ \langle a_1\rangle, \langle a_0\rangle)\mapsto p\circ(\langle a_1\rangle\circ \langle a_0\rangle).
		\]
		\item There is an isomorphism $\textup{Hom}_{H(\mathsf{Q})}(\mathcal{Z},\mathcal{M})\cong H(\mathcal{M})(Y)\circ p$ as graded groups for any $\mathcal{M}\in\textup{obj}(\mathsf{Q})$ \textup(which is natural with respect to left composition with module homomorphisms\textup). Together with the previous point, it follows that $\textup{Hom}_{H(\mathsf{Q})}(\mathcal{Z},\mathcal{Z})\cong p\circ\textup{Hom}_{H(\mathcal{A})}(Y,Y)\circ p$. 
	\end{enumerate}
\end{Lem}

\begin{Cor}\label{Zisimage}								
	Consider the abstract image $\mathcal{Z}$ of an idempotent up to homotopy $\wp$ for some $Y\in\textup{obj}(\mathcal{A})$, and choose $p\in\textup{Hom}_{H(\mathcal{A})}^0(Y,Y)$ such that $\langle\wp^1\rangle = p$\break is idempotent. Then $\mathcal{Z}$ is the image of $\langle\Upsilon^1(\wp^1)\rangle$ in $H^0(\mathsf{Q})$, in the sense of Definition \ref{imageofidem}. In particular, $\mathcal{Z}$ only depends on the class $p$ \textup(and not its representative $\wp^1$\textup), up to isomorphism in $H^0(\mathsf{Q})$. 
\end{Cor}

\begin{proof}
	Writing $\mathcal{Y}=\Upsilon(Y)\in\textup{obj}(\mathsf{Q})$ for the Yoneda module, observe that $p\in H(\mathcal{Y}(Y))=H(\mathcal{Y})(Y)$. We check that the class $\langle\Upsilon^1(\wp^1)\rangle = H(\Upsilon)\langle\wp^1\rangle=H(\Upsilon)(p)\in\textup{Hom}_{H(\mathsf{Q})}^0(\mathcal{Y},\mathcal{Y})$ is an idempotent endomorphism of $\mathcal{Y}$ in $H^0(\mathsf{Q})$. This is actually immediate from functoriality of $H(\Upsilon)$ and idempotency of $p$: $H(\Upsilon)(p)\circ H(\Upsilon)(p)= H(\Upsilon)(p\circ p) = H(\Upsilon)(p)$. Applying $H(\Upsilon)$ to the diagram \eqref{psplit}, we obtain:
	\begin{equation}\label{wpsplit}
	\begin{tikzcd}
	\Upsilon(Z)\arrow[r, yshift=1ex, "\kappa"]\arrow[loop left, distance=1.5em, start anchor={[yshift=1ex]west}, out=180, end anchor={[yshift=-0.5ex]west}, in=180, "\textup{id}_{\Upsilon(Z)}"']
	& \mathcal{Y}\arrow[l, yshift=-0.5ex, "\rho"]\!\arrow[loop right, distance=1.5em, start anchor={[yshift=-0.5ex]east}, out=0, end anchor={[yshift=1ex]east}, in=0, "\langle\Upsilon^1(\wp^1)\rangle"']
	\end{tikzcd}
	\end{equation}
	in $H^0(\mathsf{Q})$, where we set $\rho\coloneqq H(\Upsilon)(r)$ and $\kappa\coloneqq H(\Upsilon)(k)$. Recall that $\textup{id}_{\Upsilon(Z)}$ is induced by the strict unit $e_{\Upsilon(Z)}\in\textup{hom}_\mathsf{Q}(\Upsilon(Z),\Upsilon(Z))$ as described in Section \ref{ch2.8}. The two characterizing relations for images are trivially fulfilled, again by functoriality of $H(\Upsilon)$. 
	
	It remains to show that $\mathcal{Z}\cong\Upsilon(Z)$ in $H^0(\mathsf{Q})$. By Lemma \ref{aboutcohomZ}, we have $\textup{Hom}_{H(\mathsf{Q})}(\mathcal{Z},\Upsilon(Z))\cong H(\Upsilon(Z)(Y))\circ p = \textup{Hom}_{H(\mathcal{A})}^0(Y,Z)\circ p$. Thus, an isomorphism in the right-hand side suffices: one such is given by $r\circ p$, with inverse $k\in\text{Hom}_{H(\mathcal{A})}^0(Z,Y)$. 
\end{proof}

As a consequence, $H^0(\mathsf{Q})$ is a split-closed ordinary category: any idempotent endomorphism in $H^0(\mathsf{Q})$ identifies an idempotent in $H^0(\mathcal{A})$ (again, by bijectivity and functoriality of $H(\Upsilon)$), whose image in $H^0(\mathsf{Q})$ Corollary \ref{Zisimage} just proved to split.

\subsection{The split-closed derived category}\label{ch4.3}

\begin{Def}								
	An $A_\infty$-category $\mathcal{A}$ is \textbf{split-closed}\index{aaa@$A_\infty$-category!split-closed} if the abstract image $\mathcal{Z}$ of any idempotent up to homotopy is quasi-represented by some $Y\in\textup{obj}(\mathcal{A})$, that is, if we can find an isomorphism $\langle t\rangle:\Upsilon(Y)=\mathcal{Y}\xrightarrow{_\sim}\mathcal{Z}$ in $H^0(\mathsf{Q})$. 
	
	A \textbf{split-closure}\index{split-closure!of an $A_\infty$-category} of $\mathcal{A}$ is a pair $(\mathcal{B},\mathcal{F})$ where $\mathcal{B}$ is a split-closed $A_\infty$-category and $\mathcal{F}:\mathcal{A}\rightarrow\mathcal{B}$ a cohomologically full and faithful $A_\infty$-functor such that for all $Z\in\textup{obj}(\mathcal{B})$ there exists a $Y\in\textup{obj}(\mathcal{A})$ making $Z$ the image of some idempotent $p\in\textup{Hom}_{H(\mathcal{B})}(\mathcal{F}(Y),\mathcal{F}(Y))$ in $H^0(\mathcal{B})$ (in the classical sense). 
\end{Def}

\begin{Lem}\label{splitclosediff}					
	An $A_\infty$-category $\mathcal{A}$ is split-closed if and only if $H^0(\mathcal{A})$ is \textup(in the classical sense\textup).
\end{Lem}

\begin{proof}
	If $H^0(\mathcal{A})$ is split-closed and we are given an idempotent up to homotopy $\wp:\mathbb{K}\rightarrow\mathcal{A}$ with $\wp(\ast)=Y\in\textup{obj}(\mathcal{A})$, we first apply Lemma \ref{forwpp} to find an idempotent $p=\langle\wp^1\rangle\in\textup{Hom}_{H(\mathcal{A})}^0(Y,Y)$, which splits, by assumption. Applying $H(\Upsilon)$ to the corresponding diagram \eqref{psplit}, we know that $Z$ is a valid quasi-representative for the abstract image $\mathcal{Z}$ of $\wp$ (the necessary isomorphism $\langle t\rangle\in\textup{Hom}_{H(\mathsf{Q})}(\Upsilon(Z),\mathcal{Z})$ is the one constructed in Corollary \ref{Zisimage}). Therefore, $\mathcal{A}$ is split-closed.
	
	Conversely, assume that $\mathcal{A}$ is split-closed and consider an idempotent $p=\langle\wp^1\rangle\in\textup{Hom}_{H(\mathcal{A})}(Y,Y)$ (using again Lemma \ref{forwpp}). By assumption, there exists some $X\in\textup{obj}(\mathcal{A})$ whose Yoneda module $\mathcal{X}\in\textup{obj}(\mathsf{Q})$ quasi-represents $\mathcal{Z}$ of $\wp$. By Lemma \ref{aboutcohomZ}, $\textup{Hom}_{H(\mathsf{Q})}^0(\mathcal{Z},\mathcal{X})\cong H^0(\mathcal{X})(Y)\circ p = \textup{Hom}_{H(\mathcal{A})}^0(Y,X)\circ p$, hence we can find an isomorphism $r=\tilde{r}\circ p:Y\rightarrow X$ in $H^0(\mathcal{A})$. Take $Z\coloneqq \text{im}(r)$ and $k\coloneqq p\circ r^{-1}|_Z$, then we have $k\circ r= p$ and $r\circ k=\tilde{r}\circ p\circ r^{-1}|_Z = \textup{id}_Z$. Therefore, any $p$ splits via such a triple $(Z,r,k)$. This makes $H^0(\mathcal{A})$ split-closed. 
\end{proof}

\begin{Cor}\label{splitclosureiff}					
	Let $\mathcal{A}$ and $\mathcal{B}$ be $A_\infty$-categories, $\mathcal{F}:\mathcal{A}\rightarrow\mathcal{B}$ an $A_\infty$-functor. The pair $(\mathcal{B},\mathcal{F})$ is a split-closure of $\mathcal{A}$ if and only if $(H^0(\mathcal{B}),H^0(\mathcal{F}))$ is a split-closure of $H^0(\mathcal{A})$.
\end{Cor}

Since $H^0(\mathsf{Q})$ is split-closed, we immediately deduce that $\mathsf{Q}$ is a split-closed $A_\infty$-category. Therein lies the unique split-closure of $\mathcal{A}$:

\begin{Pro}\label{uniqueclosure}			 	
	Every $A_\infty$-category $\mathcal{A}$ has a split-closure $(\tilde{\mathcal{B}},\tilde{\mathcal{F}})$. Given another such, say $(\mathcal{B},\mathcal{F})$, there is a quasi-equivalence $\mathcal{G}:\mathcal{B}\rightarrow\tilde{\mathcal{B}}$ such that $\mathcal{G}\circ\mathcal{F}\cong\tilde{\mathcal{F}}\in\textup{obj}\big(H^0(fun(\mathcal{A},\tilde{\mathcal{B}}))\big)$.
\end{Pro}

\begin{proof}
	Let $\Pi\mathcal{A}\subset\mathsf{Q}$ be the full $A_\infty$-subcategory whose objects are all those $A_\infty$-modules isomorphic to abstract images of some idempotent up to homotopy: 
	\[
	\textup{obj}(\Pi\mathcal{A})\coloneqq\{\mathcal{M}\in\textup{obj}(\mathsf{Q})\mid \exists\wp\text{ with }\mathcal{Z}\cong\mathcal{M}\text{ in }H^0(\mathsf{Q})\}\,.
	\]
	Notice that every Yoneda module $\mathcal{Y}=\Upsilon(Y)\in\textup{obj}(\mathsf{Q})$ lies in $\Pi\mathcal{A}$: by Lemma \ref{forwpp}, there exists some $\wp$ for the idempotent endomorphism $\textup{id}_Y$ such that $\textup{id}_Y=\langle\wp^1\rangle\in\textup{Hom}_{H(\mathcal{A})}^0(Y,Y)$, and whose abstract image $\mathcal{Z}$ fulfills $\textup{Hom}_{H(\mathcal{A})}^0(Y,Y)= H^0(\mathcal{Y})(Y)\circ \textup{id}_Y\cong\textup{Hom}_{H(\mathsf{Q})}^0(\mathcal{Z},\mathcal{Y})$ (by Lemma \ref{aboutcohomZ}); under this isomorphism, $\textup{id}_Y$ is necessarily mapped to an isomorphism $\mathcal{Z}\xrightarrow{\sim}\mathcal{Y}$ in $H^0(\mathsf{Q})$, so that $\mathcal{Y}\in\textup{obj}(\Pi\mathcal{A})$, as claimed. 
	
	By Corollary \ref{splitclosureiff}, $(\tilde{\mathcal{B}},\tilde{\mathcal{F}})\coloneqq (\Pi\mathcal{A},\Upsilon)$ is a valid split-closure of $\mathcal{A}$ because $(H^0(\Pi\mathcal{A}),H^0(\Upsilon))$ is for $H^0(\mathcal{A})$: any idempotent endomorphism in $H^0(\Pi\mathcal{A})$ splits (argue as in Lemma \ref{splitclosediff}) and $H^0(\Upsilon)$ is a full and faithful embedding such that any $\mathcal{M}\cong\mathcal{Z}\in\textup{obj}(H^0(\Pi\mathcal{A}))$ is the image of an idempotent endomorphism of the form $\langle\Upsilon^1(\wp^1)\rangle$ (as in diagram \eqref{wpsplit}).
	
	Consider now another split-closure $(\mathcal{B},\mathcal{F})$ of $\mathcal{A}$. By assumption, $\mathcal{F}$ is c-unital, thus admits a (strictly unital) pullback functor $\mathcal{F}^*:mod(\mathcal{B})\rightarrow\mathsf{Q}$ as in Definition \ref{pullback}. Let $\mathcal{G}\coloneqq \mathcal{F}^*\circ\Upsilon_\mathcal{B}:\mathcal{B}\rightarrow\Pi\mathcal{A}\subset\mathsf{Q}$. Looking at the diagram of Remark \ref{commutingYoneda}, which commutes up to isomorphism, we deduce that $\mathcal{G}\circ\mathcal{F}\cong\Upsilon_\mathcal{A}$ in $H^0(fun(\mathcal{A},\Pi\mathcal{A}))$. In turn, this implies that $\mathcal{G}$ is cohomologically full and faithful if restricted to the image objects of $\mathcal{F}$ (since $\Upsilon_\mathcal{A}$ is), and thus on the whole of $\mathcal{B}$ as well, by definition of split-closure. Just as we argued above, we can prove that any object of $\Pi\mathcal{A}$ is isomorphic to some image object of $\mathcal{G}$, providing essential surjectivity. This makes $\mathcal{G}$ a quasi-equivalence, invertible by Theorem \ref{homotopyinverse}, and hence identifying any two split-closures of $\mathcal{A}$ up to quasi-equivalence. 
\end{proof}

\begin{Lem}\label{splitclosuretriang}	
	The split-closure of a triangulated $A_\infty$-category $\mathcal{A}$ is triangulated.
\end{Lem}

\begin{proof}
	In Section \ref{ch3.7} we argued that $\mathsf{Q}$ is a triangulated $A_\infty$-category, hence so is $H^0(\mathsf{Q})$ as an ordinary category (by Proposition \ref{H0Atriangulated}), also split-closed by Corollary \ref{Zisimage}. In Remark \ref{alternenv} we identified $\mathcal{A}$ with $\tilde{\mathsf{Q}}\subset\mathsf{Q}$, the triangulated full $A_\infty$-subcategory generated by all Yoneda modules $\mathcal{Y}=\Upsilon(Y)$, then also objects of the triangulated full subcategory $H^0(\mathcal{A})\subset H^0(\mathsf{Q})$. Moreover, by definition, $H^0(\Pi\mathcal{A})\subset H^0(\mathsf{Q})$ is the full subcategory of all objects isomorphic to a direct summand of some $\mathcal{Y}\in\textup{obj}(H^0(\mathcal{A}))$ (by Lemma \ref{aboutcohomZ}, since the cohomological module $H(\mathcal{Z})$ of any abstract image appears in the decomposition of some $H(\mathcal{Y})$). Applying Lemma \ref{Atildetriang}, it follows that $H^0(\Pi\mathcal{A})$ is a triangulated category as well.
	
	But then we easily deduce that $\Pi\mathcal{A}$ itself is a triangulated $A_\infty$-category: it is of course non-empty, any object is isomorphic to a shifted one in $H^0(\Pi\mathcal{A})$ (because it is in $H^0(\mathsf{Q})$, which contains $H^0(\Pi\mathcal{A})$ as \textit{full} subcategory) and any morphism in $H^0(\Pi\mathcal{A})$ can be completed to an exact triangle (it can in $H^0(\mathsf{Q})$, and all its constituents must then belong to the triangulated $H^0(\Pi\mathcal{A})$; consequently, the preimage under $\Pi\mathcal{A}\hookrightarrow\mathsf{Q}$ of this exact triangle is still exact in $H^0(\Pi\mathcal{A})$, by Corollary \ref{preimagetriangle}).    
\end{proof}

\begin{Def}								
	Let $\mathcal{B}$ be a split-closed triangulated $A_\infty$-category, $\mathcal{A}\subset\mathcal{B}$ a full non-empty $A_\infty$-subcategory. The smallest split-closed triangulated full $A_\infty$-subcategory $\tilde{\mathcal{B}}\subset\mathcal{B}$ which is closed under isomorphism and contains $\mathcal{A}$ is called the (triangulated) \textbf{subcategory of $\mathcal{B}$ split-generated by}\index{subcategory split-generated by an $A_\infty$-subcategory} $\mathcal{A}$. 
	
	Specifically, $\textup{obj}(\tilde{\mathcal{B}})$ can be constructed from $\textup{obj}(\mathcal{A})$ by forming all possible mapping cones and shifts, iterating and finally taking images of all possible idempotent endomorphisms. 
	
	$\mathcal{A}$ \textbf{split-generates}\index{split-generate (an $A_\infty$-category)} $\mathcal{B}$ if $\tilde{\mathcal{B}}=\mathcal{B}$ (then $\mathcal{B}$ is a split-closure of $\mathcal{A}$; moreover, $\mathcal{A}$ also \textit{generates} $\mathcal{B}$ in the sense of Definition \ref{triangenvelope}).
\end{Def}

\begin{Def}									
	Let $\mathcal{A}$ be an $A_\infty$-category and consider the split-closure $(\Pi(Tw\mathcal{A}),\Upsilon_{Tw\mathcal{A}})$ of the standard triangulated envelope $(Tw\mathcal{A},\imath)$ of $\mathcal{A}$ (by construction, $\Pi(Tw\mathcal{A})$ is split-generated by $\mathcal{A}$). We call $\mathsf{D^\pi}(\mathcal{A})\coloneqq H^0(\Pi(Tw\mathcal{A}))$ the \textbf{split-closed derived category}\index{split-closed derived category} of $\mathcal{A}$.
\end{Def}

\begin{Cor}									
	Let $\mathcal{B}$ be a split-closed triangulated $A_\infty$-category which is split-generated by a full $A_\infty$-subcategory $\mathcal{A}\subset\mathcal{B}$. Then there is a quasi-equivalence $\Pi(Tw\mathcal{A})\rightarrow\mathcal{B}$, so that $H^0(\mathcal{B})$ is equivalent to $\mathsf{D^\pi}(\mathcal{A})$ as a split-closed triangulated category.
\end{Cor}

\begin{proof}
	Together with the inclusion $\mathcal{A}\hookrightarrow\mathcal{B}$, $\mathcal{B}$ is a valid triangulated envelope and split-closure of $\mathcal{A}$. The corresponding uniqueness statements (Propositions \ref{envelopeexists} and \ref{uniqueclosure}) yield the result.
\end{proof}

As a direct consequence of Corollary \ref{equivboundercat}, one can show:

\begin{Cor}						 
	Let $\mathcal{A}$ and $\mathcal{B}$ be c-unital $A_\infty$-categories which are quasi-equivalent. Then $\mathsf{D^\pi}(\mathcal{A})$ and $\mathsf{D^\pi}(\mathcal{B})$ are equivalent split-closed triangulated categories.
\end{Cor}

Combining last two corollaries, we obtain:

\begin{Cor}\label{splitequivgen}  
	Let $\mathcal{A}$ and $\mathcal{B}$ be c-unital $A_\infty$-categories split-generated by full $A_\infty$-subcategories $\tilde{\mathcal{A}}\subset\mathcal{A}$ respectively $\tilde{\mathcal{B}}\subset\mathcal{B}$ which are quasi-equivalent. Then $\mathsf{D^\pi}(\mathcal{A})$ and $\mathsf{D^\pi}(\mathcal{B})$ are equivalent split-closed triangulated categories.
\end{Cor}

\subsection{Twisting}\label{ch4.4}

To round-off our discussion of triangulation and give a little more grounding to the theory of split-generation --- which we will eventually exploit in the context of Fukaya categories --- we take a step back to abstract mapping cones for $A_\infty$-modules. However, we opt for a less detailed exposition. More about twisting can be found in \cite[chapter 5]{[Sei08]}.

\begin{Def}	\label{abstracttwistfunc}			
	Let $\mathcal{A}$ be an $A_\infty$-category and $Y\in\textup{obj}(\mathcal{A})$. We define the \textbf{abstract twist functor}\index{abstract twist functor} $\mathcal{T}_Y:\mathsf{Q}\rightarrow\mathsf{Q}$ along $Y$ to be the $A_\infty$-functor mapping:
	\begin{itemize}[leftmargin=0.5cm]
		\renewcommand{\labelitemi}{\textendash}
		\item $\mathcal{M}\in\textup{obj}(\mathsf{Q})$ to $\mathcal{C}=\mathcal{T}_Y(\mathcal{M})\in\textup{obj}(\mathsf{Q})$, the $A_\infty$-module sending $X\in\textup{obj}(\mathcal{A})$ to $\mathcal{C}(X)\coloneqq(\mathcal{M}(Y)\otimes\textup{hom}_\mathcal{A}(X,Y))[1]\oplus\mathcal{M}(X)\in\textup{obj}(\mathsf{GradVect})$, with graded maps
		\[
		\mu_{\mathcal{C}}^1(c_0\otimes b,c_1) \coloneqq \big((-1)^{|b|-1}\mu_{\mathcal{M}}^1(c_0)\otimes b + c_0\otimes\mu_{\mathcal{A}}^1(b),\mu_{\mathcal{M}}^1(c_1)+\mu_{\mathcal{M}}^2(c_0,b)\big)\;,
		\]
		\vspace*{-0.7cm}
		\begin{align*}
		\mkern-36mu\mu_\mathcal{C}^d((c_0\otimes b,c_1),a_{d-1},...)\coloneqq \big(c_0\otimes\mu_\mathcal{A}^d(b,a_{d-1},...),\,  &\mu_\mathcal{M}^d(c_1,a_{d-1},...)\,+ \\
		&\mu_\mathcal{M}^{d+1}(c_0,b,a_{d-1},...)\big)\;,
		\end{align*}
		for $d>1$.
		\item $t\in\textup{hom}_\mathsf{Q}(\mathcal{M}_0,\mathcal{M}_1)$ to $\tilde{t}=\mathcal{T}_Y^1(t)\in\textup{hom}_\mathsf{Q}(\mathcal{T}_Y(\mathcal{M}_0),\mathcal{T}_Y(\mathcal{M}_1))$ given by
		\begin{align*}
		& \tilde{t}^1(c_0\otimes b,c_1) \coloneqq \big((-1)^{|b|-1}t^1(c_0)\otimes b, t^1(c_1)+t^2(c_0,b)\big)\;,\\
		& \tilde{t}^d((c_0\otimes b,c_1),a_{d-1},...)\coloneqq\big(0,t^d(c_1,a_{d-1},...)+t^{d+1}(c_0,b,a_{d-1},...)\big)\;,
		\end{align*}
		for $d>1$, while $\mathcal{T}_Y^d \coloneqq 0$ for $d>1$.
	\end{itemize}
	Given two $Y_0, Y_1\in\textup{obj}(\mathcal{A})$, we write $T_{Y_0}(Y_1)\in\textup{obj}(\mathcal{A})$ for any quasi-represent- ative of $\mathcal{T}_{Y_0}(\mathcal{Y}_1)$ and call it the \textbf{twist}\index{twist (along an object)} of $Y_1$ along $Y_0$.
\end{Def}

\begin{Def}\label{evalmorph}				
	Let $\mathcal{M},\,\mathcal{Y}=\Upsilon(Y)\in\textup{obj}(\mathsf{Q})$, for $Y\in\textup{obj}(\mathsf{Q})$. We define the\break \textbf{abstract evaluation morphism}\index{abstract evaluation morphism} $\epsilon\in\textup{hom}_\mathsf{Q}^0(\mathcal{M}(Y)\otimes\mathcal{Y},\mathcal{M})$ as the module homomorphism given by $\epsilon^d(c_0\otimes b,a_{d-1},...,a_1)\coloneqq\mu_\mathcal{M}^{d+1}(c_0,b,a_{d-1},...,a_1)$. Then $\mathcal{T}_Y(\mathcal{M}) = \mathcal{C}one(\epsilon)$, as in Definition \ref{abstractconemod}, fitting into the exact triangle
	\[
	\begin{tikzcd}
	\mathcal{M}(Y)\otimes\mathcal{Y}\arrow[rr, "\langle\epsilon\rangle"] & & \mathcal{M}\arrow[dl, "\langle\iota\rangle"] \\
	& \mathcal{T}_Y\mathcal{M}\arrow[ul, "{\langle\pi\rangle[1]}"]
	\end{tikzcd}
	\]
	in $H^0(\mathsf{Q})$. Assuming $H(\mathcal{A})$ to be closed under finite direct sums and shifts, and given two $Y_0,Y_1\in\textup{obj}(\mathcal{A})$ such that $\textup{Hom}_{H(\mathcal{A})}(Y_0,Y_1)$ is finite-dimensional (so that we may tensor by it as in Section \ref{ch3.2}), there is an obvious evaluation morphism $\langle \textup{ev}\rangle:\textup{Hom}_{H(\mathcal{A})}(Y_0,Y_1)\otimes Y_0\rightarrow Y_1$ such that $H(\Upsilon)\langle \textup{ev}\rangle = \langle\epsilon\rangle$. If the twisting $T_{Y_0}(Y_1)$ exists, then the triangle
	\[
	\begin{tikzcd}
	\textup{Hom}_{H(\mathcal{A})}(Y_0,Y_1)\otimes Y_0\arrow[rr, "\langle \textup{ev}\rangle"] & & Y_1\arrow[dl, "\langle i\rangle"] \\
	& T_{Y_0}(Y_1)\arrow[ul, "{\langle p\rangle[1]}"]
	\end{tikzcd}
	\]
	maps to the preceding one under $H(\Upsilon)$ for the choice $\mathcal{M}=\mathcal{Y}_1$, and is therefore exact (since preimage of an exact triangle; see Corollary \ref{preimagetriangle}). 
\end{Def}

\begin{Rem}\label{abouttwists}			
	\ \vspace*{0.0cm}
	\begin{itemize}[leftmargin=0.5cm]
	\item If $\mathcal{A}$ is triangulated and $H(\mathcal{A})$ has finite-dimensional Hom-spaces, $\langle \textup{ev}\rangle$ has standard mapping cone $S(H(\mathcal{Y}_1)(Y_0)\otimes Y_0)\oplus Y_1$. Using Lemma \ref{tr3}, with vertical arrows equal to $H(\Upsilon)$, this cone maps to $\mathcal{T}_{Y_0}(\mathcal{Y}_1)$, always existing. Hence, $T_{Y_0}(Y_1)\coloneqq S(H(\mathcal{Y}_1(Y_0))\otimes Y_0)\oplus Y_1$ is well defined for all $Y_0,Y_1\in\textup{obj}(\mathcal{A})$, allowing us to determine the \textbf{twist functor}\index{twist functor} $T_Y:\mathcal{A}\rightarrow\mathcal{A}$ as the $A_\infty$-functor making the diagram
	\[
	\begin{tikzcd}
	\mathcal{A}\arrow[r, "T_Y"]\arrow[d, "\Upsilon"'] & \mathcal{A}\arrow[d, "\Upsilon"] \\
	\mathsf{Q}\arrow[r, "\mathcal{T}_Y"'] & \mathsf{Q}
	\end{tikzcd}
	\]
	commute \big(up to isomorphism in $H^0(fun(\mathcal{A},\mathsf{Q}))$\big).
	\item Recalling the map $\tilde{\Upsilon}=\imath^*\circ\Upsilon_{Tw\mathcal{A}}:Tw\mathcal{A}\rightarrow\mathsf{Q}$ from Remark \ref{alternenv}, the last diagram tells us --- after operating the substitutions $\mathcal{A}$\textleftarrow $Tw\mathcal{A}$ and $\Upsilon$\textleftarrow$\tilde{\Upsilon}$ --- that $\tilde{\Upsilon}(T_{Y_0}(Y_1))\cong\mathcal{T}_{Y_0}(\tilde{\Upsilon}(Y_1))\in\textup{obj}(H^0(\mathsf{Q}))$, for any $Y_0\in\textup{obj}(\mathcal{A})$ and $Y_1\in\textup{obj}(Tw\mathcal{A})$.
	\item For $\mathcal{F}:\mathcal{A}\rightarrow\mathcal{B}$ a cohomologically full and faithful $A_\infty$-functor, and $Y_0,Y_1\in\textup{obj}(\mathcal{A})$ such that both $T_{Y_0}(Y_1)$ and $T_{\mathcal{F}(Y_0)}(\mathcal{F}(Y_1))$ exist, there is a canonical isomorphism $\mathcal{F}(T_{Y_0}(Y_1))\cong T_{\mathcal{F}(Y_0)}(\mathcal{F}(Y_1))$ in $H^0(\mathcal{B})$, actually extendible to $\mathcal{F}\circ T_{Y_0}\cong T_{\mathcal{F}(Y_0)}\circ\mathcal{F}\in\textup{obj}\big(H^0(fun(\mathcal{A},\mathcal{B}))\big)$.
	\end{itemize}
\end{Rem}

Here is a criterion which detects objects of $\mathcal{A}$ quasi-representing a given abstract twisted module of $\mathsf{Q}$.

\begin{Lem}\label{detecttwist}
	Let $\mathcal{A}$ be an $A_\infty$-category, $Y_0, Y_1, Y_2\in\textup{obj}(\mathcal{A})$. Assume there are a cocycle $c\in\textup{hom}_\mathcal{A}(Y_1,Y_2)$ and a map $k:\textup{hom}_\mathcal{A}(Y_0,Y_1)\rightarrow\textup{hom}_\mathcal{A}(Y_0,Y_2)[-1]$ which satisfy $\mu_\mathcal{A}^1(k(c_0))-k(\mu_\mathcal{A}^1(c_0))+\mu_\mathcal{A}^2(c,c_0)=0$ for all $c_0\in\textup{hom}_\mathcal{A}(Y_0,Y_1)$. Then the module homomorphism $t:\mathcal{T}_{Y_0}(\mathcal{Y}_1)\rightarrow\mathcal{Y}_2$ given by
	\begin{align*}
	t^d((c_0\otimes b,c_1),a_{d-1},...,a_1)\coloneqq\; &\mu_\mathcal{A}^{d+1}(c,c_1,a_{d-1},...)+\mu_\mathcal{A}^{d+1}(k(c_0),b,a_{d-1},...) \\
	&\!+\mu_\mathcal{A}^{d+2}(c,c_0,b,a_{d-1},...)
	\end{align*}
	is well defined, fulfills $\langle\mu_\mathsf{Q}^2(t,\iota)\rangle=\langle\Upsilon^1(c)\rangle\in\textup{Hom}_{H(\mathsf{Q})}(\mathcal{Y}_1,\mathcal{Y}_2)$, and --- given suitable additional data \textup{(\textit{see} \cite[Lemma 5.3]{[Sei08]})} --- induces an isomorphism $\mathcal{T}_{Y_0}(\mathcal{Y}_1)$ $\cong\mathcal{Y}_2$ in $H^0(\mathsf{Q})$.
\end{Lem}

Of course, the twist operation can be applied iteratively. Consider a triangulated $A_\infty$-category $\mathcal{A}$ whose Hom-spaces are all finite-dimensional. Choose a set $\{Y_1,...,Y_m\}\subset\textup{obj}(\mathcal{A})$ and, starting from $X\coloneqq X_m\in\textup{obj}(\mathcal{A})$, define subsequently $X_{k-1}\coloneqq T_{Y_k}(X_k)\in\textup{obj}(\mathcal{A})$, which sit into exact triangles like that of Definition \ref{evalmorph} each. Using axiom (T3) for triangulated categories, these can be composed to produce another exact triangle of the form
\[
\begin{tikzcd}
Y\arrow[rr] & & X\arrow[dl, "\langle i_1\rangle\circ...\circ\langle i_m\rangle"] \\
& (T_{Y_1}\circ...\circ T_{Y_m})(X)\arrow[ul, "{[1]}"]
\end{tikzcd}\quad,
\]
where $Y\in\text{A}$ is suitably constructed from the spaces $\textup{Hom}_{H(\mathcal{A})}(Y_k,X_k)\otimes Y_k$, while $\langle i_k\rangle:X_k\rightarrow X_{k-1}$ are canonical inclusions. As we observed in Remark \ref{t2rem}, the vanishing of $\langle i_1\rangle\circ...\circ\langle i_m\rangle$ implies $X\oplus (T_{Y_1}\circ...\circ T_{Y_m})(X)[-1]\cong Y$. One builds upon this to prove the following result.

\begin{Cor}									
	Let $\mathcal{A}$ be a triangulated $A_\infty$-category whose \textup{Hom}-spaces are all finite-dimensional.
	\renewcommand{\theenumi}{\roman{enumi}}
	\begin{enumerate}[leftmargin=0.5cm]
		\item Suppose there are $Y_1,...,Y_m\in\textup{obj}(\mathcal{A})$ such that $(T_{Y_1}\circ...\circ T_{Y_m})(X)\cong 0\in\textup{obj}(H^0(\mathcal{A}))$ \textup(the zero object of $H^0(\mathcal{A})$\textup) for all $X\in\textup{obj}(\mathcal{A})$. Then $\{Y_1,...,Y_m\}$ generates $\mathcal{A}$, which is actually split-closed.
		\item Suppose $\mathcal{A}$ is readily split-closed and there are $Y_1,...,Y_m\in\textup{obj}(\mathcal{A})$, $\sigma\neq 0$ such that $(T_{Y_1}\circ...\circ T_{Y_m})(X)\cong S^\sigma X\in\textup{obj}(H^0(\mathcal{A}))$ for all $X\in\textup{obj}(\mathcal{A})$. Then $\{Y_1,...,Y_m\}$ split-generates $\mathcal{A}$.
	\end{enumerate}
\end{Cor} 

\begin{Rem}										
	We can actually improve the description of twist functors upon considering twisted complexes in $Tw\mathcal{A}$, where $\mathcal{A}$ is chosen strictly unital. 
	
	Pick $Y_0,Y_1\in\textup{obj}(Tw\mathcal{A})$ and consider $(\textup{hom}_{Tw\mathcal{A}}(Y_0,Y_1),\partial\!\coloneqq \!\mu_{Tw\mathcal{A}}^1)$, whose dual complex is $\textup{hom}_{Tw\mathcal{A}}(Y_0,Y_1)^\vee\coloneqq\textup{hom}_{Tw\mathcal{A}}(\textup{hom}_{Tw\mathcal{A}}(Y_0,Y_1),Y_1)$ with differential $\partial^\vee$ given by $(\partial^\vee(a^\vee))(b) \coloneqq (-1)^{|a^\vee|+1}a^\vee(\partial(b))$, for $a^\vee\in\textup{hom}_{Tw\mathcal{A}}(Y_0,Y_1)^\vee$ and $b\in\textup{hom}_{Tw\mathcal{A}}(Y_0,Y_1)$. We can take the tensor product as in Definition \ref{TwAops}: 
	\begin{align*}
	\textup{hom}_{Tw\mathcal{A}}(\textup{hom}_{Tw\mathcal{A}}(Y_0,Y_1)\otimes Y_0,Y_1) &= \textup{hom}_{Tw\mathcal{A}}(Y_0,Y_1)^\vee\otimes\textup{hom}_{Tw\mathcal{A}}(Y_0,Y_1)\\
	&\cong \text{End}_\mathbb{K}(\textup{hom}_{Tw\mathcal{A}}(Y_0,Y_1))\;,
	\end{align*}
	which has differential $D: a^\vee\otimes b\mapsto a^\vee(\cdot)b$, and $\textup{ev} \cong \textup{id}\in\text{End}_\mathbb{K}(\textup{hom}_{Tw\mathcal{A}}(Y_0,Y_1))$ is a cocycle with respect to it. Therefore, we can look at the twisted mapping cone $Cone(\textup{ev})$, turning under $H(\Upsilon)$ into $\mathcal{C}one(\epsilon)=\mathcal{T}_{Y_0}(\mathcal{Y}_1)$ (cf. Definition \ref{evalmorph}). This motivates us to set 
	\[
	T_{Y_0}(Y_1)\coloneqq Cone(\textup{ev})=\bigoplus_{i\in I}S^{-|b^i|+1}Y_1\oplus Y_0\;,
	\]
	where $\{b^i\}_{i\in I}$ is a basis of $\textup{hom}_\mathcal{A}(Y_0,Y_1)$.
\end{Rem}

Finally, we mention the adjoint notion of the twist operation.

\begin{Def}									
	Let $\mathcal{A}$ be an $A_\infty$-category, $Y_0,Y_1\in\textup{obj}(\mathcal{A})$. The $A_\infty$-module $\mathcal{C}^\vee\in\textup{obj}(\mathsf{Q})$ is given by 
	\begin{align*}
	&\mathcal{C}^\vee(X)\coloneqq\textup{hom}_\mathcal{A}(X,Y_1)\oplus\textup{Hom}_\mathbb{K}(\textup{hom}_\mathcal{A}(Y_1,Y_0),\textup{hom}_\mathcal{A}(X,Y_0))[-1]\;,\\
	&\mu_{\mathcal{C}^\vee}^1(c,\beta)\coloneqq(\mu_\mathcal{A}^1(c),\mu_\mathcal{A}^1\circ\beta+(-1)^{|\beta|-1}\beta\circ\mu_\mathcal{A}^1+\mu_\mathcal{A}^2(\cdot,c))\;,\\
	&\mu_{\mathcal{C}^\vee}^d((c,\beta),a_{d-1},...) \coloneqq (\mu_\mathcal{A}^d(c,a_{d-1},...),\mu_\mathcal{A}^d(\beta(\cdot),a_{d-1},...)+\mu_\mathcal{A}^{d+1}(\cdot,c,a_{d-1},...))\;.
	\end{align*}
	for $d>1, X\in\textup{obj}(\mathcal{A})$. 
	
	We write $T_{Y_0}^\vee (Y_1)$ for any quasi-representative of $\mathcal{C}^\vee$ and call it the \textbf{adjoint twist}\index{adjoint twist (along an object)} of $Y_1$ along $Y_0$. One can show that $T_{Y_0}(T_{Y_0}^\vee(Y))\cong Y$ in $H^0(\mathcal{A})$ for all $Y\in\textup{obj}(\mathcal{A})$ (indeed, $T_{Y_0}^\vee$ is actually an inverse of $T_{Y_0}$ in $H^0(fun(\mathcal{A},\mathcal{A}))$, up to isomorphism).
\end{Def}

As was the case for the normal twist, taking $\mathcal{A}$ to be triangulated with finite-dimensional Hom-spaces makes $T_{Y_0}^\vee (Y_1)$ always well-defined, thus fitting into the exact triangle which results after ``dualization'':
\[
\begin{tikzcd}
\textup{Hom}_{H(\mathcal{A})}(Y_1,Y_0)^\vee\otimes Y_0\arrow[dr, "{\langle p^\vee\rangle[1]}"'] & & Y_1\arrow[ll, "\langle \textup{ev}^\vee\rangle"'] \\
& T_{Y_0}^\vee(Y_1)\arrow[ur, "\langle i^\vee\rangle"']
\end{tikzcd}
\]
Taking again $\mathcal{A}$ to be strictly unital, $T_{Y_0}^\vee (Y_1)$ can be identified in $Tw\mathcal{A}$ as the mapping cone of the ``dual'' evaluation morphism $\langle \textup{ev}^\vee\rangle$ (in parallel with last remark).

\newpage

\part{Fukaya Category Theory}
\markboth{II\quad FUKAYA CATEGORY THEORY}{II\quad FUKAYA CATEGORY THEORY}
\thispagestyle{plain}

\vspace*{2cm}

Fukaya categories were first studied by Kenji Fukaya in his \textit{Morse homotopy, $A_\infty$-category and Floer homologies} \cite{[Fuk93]} from 1993, one year prior to Kontsevich's exploit (\cite{[Kon94]}), building on the original work \textit{Morse theory for Lagrangian intersections} \cite{[Flo88]} by Andreas Floer. Fukaya successfully packed a lot of geometric information within an algebraic structure, preparing the ground for a richer interpretation of mirror symmetry. 

Indeed, a Fukaya category $\mathscr{F}(M)$ can be associated to any (specific enough) symplectic manifold $M$, its objects being suitably ``decorated'' compact Lagrangian submanifolds, whose modes of intersection are dictated by the $A_\infty$-structure. Specifically, the definition makes use of a particular version of Floer theory, the Lagrangian (intersection) Floer homology, which studies solutions to Cauchy--Riemann problems with boundary conditions given by configurations of intersecting Lagrangian submanifolds.  

The underlying language is that of symplectic and complex geometry. Moreover, as is true of the other Floer-flavoured theories, a lot of functional analysis is involved. And this is our major compromise as well: we will not insist on the analytical subtleties which would unavoidably take us on a tangent.      

Before moving on, we should also mention that the classical analytical approach from above originally adopted by Fukaya has been substantially revised in the meantime, ultimately replaced by the groundbreaking \textit{Lagrangian intersection Floer theory: anomaly and obstruction} \cite{[FOOO09]} by Fukaya, Oh, Ohta and Ono (last updated in 2009), which explains how to define Fukaya categories in full generality. Though tempting, this is unfortunately a luxury we cannot afford, due to the even more advanced analysis and algebra required. 

These words of warning set aside, our complete agenda for this part reads as follows:
\begin{itemize}[leftmargin=0.5cm]
	\item In Chapter 5, we give an overview of Lagrangian Floer homology. After introducing the basic notions from symplectic and complex geometry, we anticipate the definition of Fukaya category of a symplectic manifold so to motivate each decoration required for its objects; we immediately tackle also a couple aspects of the general framework. 
	\newline Then we introduce Floer complexes and explain how their boundary operator is defined through the count of $J$-holomorphic strips solving some specific Cauchy--Riemann equations with Lagrangian boundary conditions. Next, we discuss Maslov indices and how to give complexes and Lagrangian submanifolds an integer grading. Then we address the first of two major obstructions to the well-definiteness of the boundary operator, namely that of transversality of the moduli spaces formed by the $J$-holomorphic strips, which is solved by applying a Hamiltonian perturbation to the Cauchy--Riemann problem. Moreover, we rectify the issue of compactness of the moduli spaces by imposing a topological constraint on the Lagrangian submanifolds. 
	\newline All together, these assumptions enable us to properly define the Floer differential and its resulting cohomology, and prove the latter to be independent of the perturbative data. Lastly, we observe the advantage of dealing with exact Lagrangian submanifolds, which are a priori unobstructed and thus get us rid of most technical burden.
	
	\item In Chapter 6, we construct the $A_\infty$-structure of Fukaya categories and highlight the geometrical peculiarities making them very interesting cases of $A_\infty$-categories. We start by adapting the Cauchy--Riemann problem to fit an additional Lagrangian submanifold, and see how this defines the Floer product, which has the desirable properties of a second order $A_\infty$-composition map. Then we generalize the procedure for any given number of Lagrangian submanifolds, fine-tuning the perturbation data, so to produce higher order composition maps fulfilling the $A_\infty$-associativity equations. At last, we give the formal definition of Fukaya category (and a few warning remarks), prove it to be c-unital and summarize what we already know from Part I.
	\newline Next, we describe two geometric constructions, Dehn twists and Lagrangian connected sums, which algebraically are just suitable mapping cones in the $A_\infty$-category of twisted complexes --- a highlight of Fukaya categories --- and later justify the necessity of taking split-closures by looking at the example of the 2-torus. We conclude with a glance at the wrapped version of Fukaya category one must resort to when dealing with exact not necessarily compact Lagrangian submanifolds.
	
	\item In Chapter 7, we inject what learned about Fukaya categories into the context of Homological Mirror Symmetry, paying regard to the historical path which led to it. Thereto, we start by giving a very superficial description of what string theory is and how some physical dualities observed there inspired mathematical research to solve longstanding questions, eventually giving birth to the new branch of Mirror Symmetry. The latter we discuss next, importing all the language from complex algebraic topology and differential geometry (Dolbeault cohomology, Hodge theory, Calabi--Yau manifolds) necessary to state the main mirror symmetry conjecture.
	\newline Finally, we collect all relevant notions from algebraic geometry (coherent sheaves) and homological algebra (Ext-groups), and suitably extend the definition of Fukaya category in order to state Kontsevich's homological conjecture, backed by some specific cases we list. A quick look at the geometrical SYZ conjecture, hinting how to ``explicitly'' fabricate mirror pairs of Calabi--Yau manifolds, completes our journey.  
\end{itemize}

The main reference for Chapters 5 and 6 is Denis Auroux's introductory course in Fukaya categories, summarized in \cite{[Aur13]}. Basic material entering the frame via symplectic geometry and Floer theory is distributed among many more sources, most prominently \cite{[Sil12]}, \cite{[McS94]}, \cite{[Mer14]} and \cite{[AS09]} (always with an eye on \cite{[Sei08]}). Chapter 7 instead borrows, among others, from \cite{[Aur09]}, \cite{[Fuk02a]} and \cite{[GHJ03]}. Once again, preliminary aspects of the theory are embedded within the flow of our discussion, but can be safely skipped by the expert reader.

\newpage
\pagestyle{fancy}

\section{Lagrangian Floer (co)homology}
\thispagestyle{plain}

\subsection{Preliminaries from symplectic and complex geometry}\label{ch5.1}

In this section, we introduce in a brief if somewhat dense manner all geometric notions underlying the theory of Fukaya categories. Let us start from the branch of symplectic geometry (the reader is referred to \cite{[Sil12]} for more details).

\begin{Def}									
	A \textbf{symplectic manifold}\index{symplectic manifold} $(M,\omega)$ of dimension $2n$ is a smooth manifold $M^{2n}$ equipped with a \textbf{symplectic form}\index{symplectic form} $\omega\in\Omega^2(M)=\Gamma(\Lambda^2(T^*M))$, that is, a closed non-degenerate 2-form (thus canonically identifying, for each $x\in M$, antisymmetric bilinear maps $\omega_x:T_xM\times T_xM\rightarrow \mathbb{R}$ such that $\omega_x(v,w)=0$ for all $w\in T_xM$ implies $v=0\in T_xM$). Closedness means that $d\omega=0$, with respect to the exterior differential operator $d:\Omega^\bullet(M)\rightarrow\Omega^{\bullet+1}(M)$.
	
	$M$ is \textbf{exact}\index{symplectic manifold!exact} if $\omega$ is an exact 2-form, that is, if there exists some $\nu\in\Omega^1(M)$ such that $\omega=d\nu$. In case $M$ is compact, one can prove that $\omega$ is never exact, meaning that $\langle\omega\rangle\in H_{dR}^2(M)$ is a non-trivial class in de Rham cohomology. 
\end{Def}

We observe that the very definition of $\omega$ forces $M$ to be even-dimensional and oriented.

\begin{Def}\label{Lagrangian}			
	Let $(M,\omega)$ be a $2n$-dimensional symplectic manifold. A \textbf{Lagrangian submanifold}\index{Lagrangian submanifold} $L\subset M$ is a maximal isotropic submanifold of $M$, that is, one of dimension $\frac{1}{2}$dim$(M)=n$ fulfilling $\omega|_L=0$ (or equivalently, $\iota^*\omega=0$ for the inclusion $\iota:L\hookrightarrow M$).
	
	Assuming $M$ is exact with $\omega=d\nu$ for some $\nu\in\Omega^1(M)$, $L$ is itself \textbf{exact}\index{Lagrangian submanifold!exact} if the restriction $\nu|_L$ (closed as $d(\nu|_L)=(d\nu)|_L=0$) is an exact 1-form, that is, if there exists some smooth function $f\in C^\infty(L)$ such that $\nu|_L=df$.  
\end{Def}

\begin{Rem}\label{LGr}					
	As a particular case, a vector space $V$ is called \textit{symplectic} when endowed with an antisymmetric non-degenerate map $\omega:V\times V\rightarrow\mathbb{R}$ (called \textit{symplectic structure}). Then a linear subspace $W\subset V$ is \textit{Lagrangian} if $W=W^\perp\coloneqq \{v\in V\mid \omega(v,w)=0\;\forall w\in W\}$. This is an element of the \textbf{Lagrangian Grassmannian}\index{Lagrangian Grassmannian} $LGr(V)\coloneqq \{W\subset V\mid W\text{ linear Lagrangian subspace}\}$, a smooth manifold (diffeomorphically) realizable as the homogeneous space $\text{U}(n)/\text{O}(n)$ (that of unitary matrices modulo the orthogonal ones).
	
	For example, in $V\coloneqq\mathbb{R}^{2n}$ endowed with $\omega_0\!\in\!\Omega^2(\mathbb{R}^{2n})$ represented by $\big(\begin{smallmatrix}
	0 & \mathds{1}_n\\
	-\mathds{1}_n & 0
	\end{smallmatrix}\big)$, a subspace $W\subset\mathbb{R}^{2n}$ is Lagrangian if and only if it is $n$-dimensional and $\omega_0((x_1,y_1),(x_2,y_2))\coloneqq  x_1\cdot y_2 - x_2\cdot y_1 = 0$ for all $(x_1,y_1),(x_2,y_2)\in W$. In this case, we simply write $LGr(n)$ instead of $LGr(\mathbb{R}^{2n})$. Recall also the symplectic group $\text{Sp}(2n,\mathbb{R})\!=\!\{B\!\in\!\text{GL}(2n,\mathbb{R})\mid \omega_0(B(v),B(w))=\omega_0(v,w)\;\;\forall v,w\!\in\!\mathbb{R}^{2n}\}\equiv$ $\{B\in\mathbb{R}^{2n\times 2n}\mid B^{\text{T}}\omega_0B=\omega_0^\text{T}\}$.
	
	Clearly, if $(M^{2n},\omega)$ is a symplectic manifold, then $(T_xM,\omega_x)$ is a symplectic vector space with $LGr(T_xM)\cong LGr(n)$ for every $x\in M$ (we say that the vector bundle $TM\rightarrow M$ is \textit{symplectic}). For any point $x\in L$ of a Lagrangian submanifold $L^n$ holds $T_xL\in LGr(T_xM)$, by definition.
\end{Rem}

We also recollect the fundamentals of Hamiltonian dynamics.

\begin{Def}	\label{Hamiltonian}					
	Let $(M,\omega)$ be a symplectic manifold. To each smooth function $H\in C^\infty(M)$ we can assign a vector field $X_H\in\mathfrak{X}(M)$ satisfying $i_{X\!_H}(\omega)=\omega(\,\cdot\,,X_H)=dH:\mathfrak{X}(M)\rightarrow C^\infty(M)$. $X_H$ is called the \textbf{Hamiltonian vector field}\index{Hamiltonian vector field} associated to the \textbf{Hamiltonian function}\index{Hamiltonian function} $H$.
	
	If the Hamiltonian is time-dependent, $H\in C^\infty(I\times M)$ with $H(t,\,\cdot\,)\equiv H_t(\cdot)\in C^\infty(M)$ and $I=[0,1]$ (or $\mathbb{R}$), then it yields a family of Hamiltonian vector fields $\{X_{H_t}\}_{t\in I}\subset\mathfrak{X}(M)$, and we can associate to $X_H$ (assumed complete) its \textbf{Hamiltonian ``flow''}\index{Hamiltonian ``flow''} $\phi_H: I\!\times\! M\rightarrow M$, which satisfies $\frac{d}{dt}|_t(\phi_H^t)=X_{H_t}\!\circ\!\phi_H^t$\break and $\phi_H^0=\textup{id}$ (we write $\phi_H(t,\,\cdot\,)=\phi_H^t(\cdot)$). In particular, we call $\phi_H^1$ the \textbf{Hamiltonian diffeomorphism}\index{Hamiltonian diffeomorphism} generated by $H$. This is an element of the group 
	\[
	\text{Ham}(M,\omega)\!\coloneqq \!\{\phi\!\in\! \text{Diff}(M)\mid \phi=\phi_H^1 \text{ Ham.\,diffeo.\;of\;some }H\!\in\! C^\infty(I\times M)\}\,.
	\]
	Denoting by $\text{Symp}(M,\omega)\coloneqq \{\phi\in\text{Diff}(M)\mid \phi^*\omega=\omega\}$ the group of symplectomorphisms, we have the following inclusions of subgroups:
	\[
	\text{Ham}(M,\omega)\leq\text{Symp}(M,\omega)\leq\text{Diff}(M).
	\] 
\end{Def}

\begin{Def}\label{isotopy}							
	Let $M$ be a smooth manifold, $L_0,L_1\subset M$ compact submanifolds. 
	\begin{itemize}[leftmargin=0.5cm]
		\item We say that $L_0$ and $L_1$ \textbf{intersect transversely}\index{intersect transversely} (or are \textbf{transverse}\index{submanifolds!transverse pair}) if $T_xL_0\pitchfork T_xL_1$, that is, $T_xL_0\oplus T_xL_1 = T_xM$ as vector spaces, for all $x\in L_0\cap L_1\subset M$. Then we write $L_0\pitchfork L_1$. 
		\item A \textbf{(smooth) isotopy}\index{isotopy} between $L_0$ and $L_1$ is a compact submanifold $\tilde{L}\subset M\times[0,1]$ such that $\tilde{L}\cap(M\times\{i\})=L_i\times\{i\}$ for $i=0,1$ and $\tilde{L}\pitchfork(M\times\{t\})$ for all $t\in[0,1]$.\newline 
		Then $L_0$ and $L_1$ are also \textbf{ambient isotopic}\index{submanifolds!ambient isotopic pair}, that is, there exists a path $\phi:t\in[0,1]\mapsto \phi_t\in\text{Diff}(M)$ such that $\phi_0=\textup{id}_M$ and $\phi_1(L_0)=L_1$.
		\item Assume now that $M=(M,\omega)$ is also symplectic and $L_0,L_1$ Lagrangian. The isotopy $\tilde{L}$ from above is a \textbf{Lagrangian isotopy}\index{isotopy!Lagrangian} if additionally each $\tilde{L}\;\cap\;(M\times\{t\})\subset M\times\{t\}$ is a Lagrangian submanifold, whence it follows that $(\pi_1^*\omega)|_{\tilde{L}}=\theta\wedge dt\in\Omega^2(\tilde{L})$ for $\pi_1:M\times[0,1]\rightarrow M$ and some $\theta\in\Omega^1(\tilde{L})$. \newline
		Then $L_0$ and $L_1$ are also \textbf{symplectic isotopic}\index{Lagrangian submanifold!symplectic isotopic pair}, that is, there exists a path $\phi:t\in[0,1]\mapsto \phi_t\in\text{Symp}(M,\omega)$ such that $\phi_0=\textup{id}_M$ and $\phi_1(L_0)=L_1$. 
		\item The Lagrangian isotopy $\tilde{L}$ is in particular a \textbf{Hamiltonian isotopy}\index{isotopy!Hamiltonian} if actually $\theta\wedge dt=d(Hdt)=dH\wedge dt$ for some Hamiltonian function $H\in C^\infty(\tilde{L})$. Choosing above\footnote{We can be slightly more precise: let $\text{Symp}_0(M,\omega)\coloneqq \{\varphi\in\text{Symp}(M,\omega)\mid \exists\phi:[0,1]\rightarrow\text{Symp}(M,\omega)\text{ s.t. } \phi_0=\textup{id}_M,\,\phi_1=\varphi\}$. Then $\phi_H:[0,1]\rightarrow\text{Symp}_0(M,\omega)$ (see \cite[Lemma 2.7]{[Mer14]} for a proof).} $\phi\coloneqq\phi_H:[0,1]\rightarrow \text{Symp}(M,\omega)$, we obtain that $L_0$ and $L_1$ are \textbf{Hamiltonian isotopic}\index{Lagrangian submanifold!Hamiltonian isotopic pair}. 
	\end{itemize}
\end{Def}
\vspace*{0.35cm}

\noindent The study of Lagrangian Floer homology requires some knowledge of complex geometry as well.  

\begin{Def}\label{almostcomplex}		
	Let $M$ be a smooth manifold. An \textbf{almost complex structure}\index{almost complex structure} on $M$ is a tensor field $J\in\mathcal{T}^{1,1}(M)\cong\Gamma(\text{End}(TM))$ such that $J_x\circ J_x = -\textup{id}_{T_xM}$ for all $x\in M$. Then $M=(M,J)$ is an \textbf{almost complex manifold}\index{almost complex manifold} and $\pi:(TM,J)\rightarrow M$ a \textit{complex vector bundle}.
	
	If moreover $M\!=\!(M,\omega)$ is a symplectic manifold, $J$ is said to be \textbf{$\omega$-compatible}\index{almost complex structure!$\omega$-compatible} if $g_J(\,\cdot\,,\,\cdot\,)\coloneqq \omega(J\cdot\,,\,\cdot\,)$ defines\footnote{Warning: literature mostly adopts the alternative definition, namely $g_J(\,\cdot\,,\,\cdot\,)\coloneqq \omega(\,\cdot\,,J\cdot\,)$! Observe that if $M$ is endowed with a Riemannian metric $g$ (always possible) such that $g(u,v)=\omega(u, Jv)$, then this relation can be used to \textit{define} a $\omega$-compatible $J$ on $M$. In particular, symplectic manifolds can always be regarded as almost complex manifolds.} a Riemannian metric on $M$ (that is, for each $x\in M$ the map $g_{J_x}(\,\cdot\,,\,\cdot\,)\coloneqq \omega_x(J_x\cdot\,,\,\cdot\,):T_xM\times T_xM\rightarrow \mathbb{R}$ is a positive definite inner product, whence $\omega_x(J_x(v),J_x(w))=\omega_x(v,w)$ for all $v,w\in T_xM$).
	
	We denote by $\mathfrak{J}(M,\omega)$ the (nonempty) set of all $\omega$-compatible almost complex structures on $M$ (it is path-connected, cf. \cite[Corollary 12.9]{[Sil12]}).
\end{Def}

In Chapter 7, when discussing homological mirror symmetry, we will deal with \textit{complex manifolds}, which are analytic $n$-dimensional (over $\mathbb{C}$) manifolds locally homeomorphic to $\mathbb{C}^n$ through charts whose transition functions are holomorphic. We can interpret them as real smooth $2n$-dimensional manifolds. They are naturally equipped with an almost complex structure, that is, they are automatically almost complex manifolds (see \cite[Proposition 15.2]{[Sil12]}), though the converse is not true a priori (cf. Lemma \ref{maybecomplex}).

The maps we are most interested in when discussing Floer theory are\break J-holomorphic (or pseudoholomorphic) curves. A thorough description of them is provided by \cite[sections 2.1--2.2]{[Wen14]}.

\begin{Def}\label{Jholocurve}			
	Let $M=(M,J)$ be a smooth manifold and $\Sigma=(\Sigma,j)$ a Riemann surface (that is, a complex analytic 1-dimensional manifold), where $J$ and $j$ are fixed almost complex structures. Then a \textbf{$J$-holomorphic curve}\index{jholomorphic@$J$-holomorphic curve} is a smooth map $u:\Sigma\rightarrow M$ fulfilling $(du)_x\circ j_x=J_x\circ (du)_x:T_x\Sigma\rightarrow T_{u(x)}M$ for all $x\in\Sigma$. 
	
	Using that $J$ squares to $-\textup{id}$, this condition is equivalent to the \textit{Cauchy--Riemann equation} $\bar{\partial}_{j,J}(u)\coloneqq \frac{1}{2}(du+J\circ du\circ j)=0$ (which reduces to the homonymous pair of equations from basic complex analysis after choosing $M=\Sigma\coloneqq \mathbb{C}$ and $J=j\coloneqq \big(\begin{smallmatrix} 0 & 1 \\ -1 & 0\end{smallmatrix}\big)$). In case $\bar{\partial}_{j,J}(u)=v$ for some perturbative bundle morphism $0\neq v:T\Sigma\rightarrow TM$, we talk about \textit{$(j,J,v)$-holomorphic curve} satisfying a \textit{perturbed Cauchy--Riemann equation}. In particular, if:
	\begin{itemize}[leftmargin=0.5cm] 			
		\item $\Sigma=\mathbb{R}\times[0,1]\subset\mathbb{R}^2\cong\mathbb{C}$, we call $u$ a \textbf{$J$-holomorphic strip}\index{jholomorphic@$J$-holomorphic curve!strip};
		\item $\Sigma=\mathbb{D}^2\coloneqq \{z\in\mathbb{C}\mid |z|\leq 1\}\subset\mathbb{C}$, $u$ is a \textbf{$J$-holomorphic disk}\index{jholomorphic@$J$-holomorphic curve!disk}; 
		\item $\Sigma=\mathbb{C}P^1\equiv\mathbb{C}\sqcup\{\infty\}(\simeq\mathbb{S}^2\subset\mathbb{R}^3)$, then $u$ is a \textbf{$J$-holomorphic sphere}\index{jholomorphic@$J$-holomorphic curve!sphere}.
	\end{itemize} 
	Since $\mathbb{R}\times[0,1]$ is biholomorphic to $\mathbb{D}^2\setminus\{\pm1\}$ by the Riemann Mapping Theorem, strips solving the (unperturbed!) Cauchy--Riemann equation can be extended to disks.
\end{Def} 

\begin{Def}								
	Let $(M,\omega,J)$ be a symplectic almost complex manifold, $(\Sigma,j)$ a Riemann surface. Then we can assign to any $J$-holomorphic curve $u:\Sigma\rightarrow M$ its \textbf{symplectic area}\index{symplectic area} (or energy):
	\begin{equation}\label{symparea}
	\text{Area}_{g_J}(u)\equiv E(u)\coloneqq \int_\Sigma u^*\omega =\int_\Sigma|du|^2\geq 0\;,
	\end{equation}
	where $|du|^2\equiv |du|_{g_J}^2=g_J(du,du)$. In particular, $E(u)>0$ whenever $u$ is non-constant.
\end{Def}

For a $J$-holomorphic strip $u$ parametrized by (orthogonal) coordinates $(s,t)\equiv s+it\in\mathbb{D}^2\setminus\{\pm 1\}$, we can clearly see the geometric interpretation of symplectic area: the Cauchy--Riemann equation becomes $J\partial_su - \partial_tu = 0$, and hence the area spanned by the orthogonal vectors $\partial_su$ and $\partial_tu$ is $|\partial_su|\cdot|\partial_tu| = |\partial_su|^2 = g_J(\partial_su,\partial_su) = \omega(J\partial_su,\partial_su) = \omega(\partial_tu,\partial_su)$; integration over all $z=s+it\in\mathbb{D}^2\setminus\{\pm 1\}$ yields exactly \eqref{symparea}. The same reasoning applies to any Riemann surface, by choosing holomorphic local coordinates $(s,t)$ as above.

\subsection{The Fukaya category: a sneak peek}\label{ch5.2}

There are several definitions of Fukaya category, which differ by the set of data and constraints one wishes to attach to its objects, starting with the underlying manifold itself. 

We work in one of the most simplified such settings, namely a symplectic manifold $M^{2n}=(M,\omega)$ with 2-torsion first Chern class, $2c_1(TM)=0$. Before going into the details, let us give a brief overview of all the necessary ingredients which will guide our construction; a thorough exposition is provided in the ensuing sections. 

Objects of the (compact) Fukaya category $\mathscr{F}(M)$ of $M$ are compact Lagrangian submanifolds $L^n\subset M$ decorated with the following data:
\begin{itemize}[leftmargin=0.5cm]
	\item a vanishing Maslov class $H^1(L;\mathbb{Z})$ --- which we can talk about only by assuming that $2c_1(TM)=0$ in the first place --- thus enabling us to equip each $L$ with a ``grading'' and consequently make Floer complexes graded;
	\item when char$(\mathbb{K})\neq 2$, a choice of orientation and spin structure for each $L$, eventually relative to some fixed $\langle b\rangle\in H^2(M;\mathbb{Z}_2)$;
	\item $\langle\omega\rangle\cdot\pi_2(M,L)=0$ for each $L$, which gets us rid of ``transversality issues'' preventing Floer cohomology from being well-defined;
	\item for every pair of Lagrangian submanifolds $(L,L')$, as well as for all tuples $(L_0,...,L_k)$, the choice of a time-dependent almost complex structure $J\in C^\infty([0,1],\mathfrak{J}(M,\omega))$ and Hamiltonian $H_{L,L'}\in C^\infty([0,1]\times M,\mathbb{R})$ which is consistent with all perturbation data, in order to once again take care of transversality and thus obtain well-defined composition maps.  
\end{itemize}

We also require that Hamiltonian isotopic Lagrangian submanifolds (cf. Definition \ref{isotopy}) yield isomorphic objects in the Fukaya category (or rather, in its associated cohomological category). 

Morphism spaces are Floer complexes, $\textup{hom}_{\mathscr{F}(M)}(L,L')\coloneqq CF^\bullet(L,L')$, and the choice of suitable perturbation data as mentioned above allows us to produce composition maps between these spaces.

At this point, the reader may have already guessed what we are aiming for: Fukaya categories have the structure of $\mathbb{Z}$-graded $A_\infty$-categories! However, unless we are working in the exact symplectic framework --- which we will eventually describe in Section \ref{ch5.8} --- the choice of ground field is no longer arbitrary: for the morphisms in $\mathscr{F}(M)$ to be well-defined, we must work over the Novikov field.

\begin{Def}											
	Let $\mathbb{K}$ be any field. The \textit{Novikov ring} over $\mathbb{K}$ is
	\[
	\Lambda_\mathbb{K}'\coloneqq  \Bigg\{ \sum_{i=0}^\infty a_iT^{\lambda_i}\,\bigg|\,a_i\in\mathbb{K},\,\lambda_i\in\mathbb{R}_{\geq 0}\text{ such that }\lim_{i\to\infty}\lambda_i = \infty\Bigg\}\,.
	\]
	Then the \textbf{Novikov field}\index{Novikov field} is simply its fraction field:
	\begin{equation}\label{Novikov}
	\Lambda_\mathbb{K}\coloneqq  \text{Frac}(\Lambda_\mathbb{K}')=\Bigg\{ \sum_{i=0}^\infty a_iT^{\lambda_i}\,\bigg|\,a_i\in\mathbb{K},\,\lambda_i\in\mathbb{R}\text{ such that }\lim_{i\to\infty}\lambda_i = \infty\Bigg\}\,.
	\end{equation}
\end{Def}    

Indeed, the higher-order composition maps $\mu_{\mathscr{F}(M)}^d$ ``count'' certain pseudo-holomorphic polygons, encoding them as coefficients of power series which generally do not converge. To remedy this issue, we are forced to work over a field of formal power series such as $\Lambda_\mathbb{K}$. As already hinted, assuming $M$ and all Lagrangian submanifolds to be exact, a suitable rescaling allows us to eliminate the formal variable $T$, so that we can directly consider coefficients in $\mathbb{K}$.

An aspect worth addressing immediately is our standing assumption that $2c_1(TM)=0$. We just provide a formal definition, encouraging the reader to look up in the literature for more details (see for example \cite{[MS74]}).

\begin{Def}												
	Let $M^{2n}=(M,\omega,J)$ be an almost complex symplectic manifold, and hence $\pi:TM\rightarrow M$ a complex vector bundle of rank $2n$. The \textit{Chern classes} $c_k(TM)\in H^{2k}(M;\mathbb{Z})$ of $M$ are implicitly defined to be the coefficients in the expression
	\[
	\text{det}\Bigg(\frac{it\Omega}{2\pi}+\mathds{1}_n\Bigg)=\sum_{k\geq 0}c_k(TM)t^k\;.
	\]
	Here $t$ is a formal variable and $\Omega=d\varpi+\frac{1}{2}[\varpi,\varpi]\in\Omega^2(P,\mathbb{K}^{2n\times 2n})$ is the curvature form of the principal $\text{GL}(2n)$-bundle $P$ associated\footnote{The principal bundle $P$ is the frame bundle $\text{Fr}(TM)$ of $TM$, whose fibers are given by $\text{Fr}(TM)_x\coloneqq \{f:\mathbb{K}^{2n}\rightarrow T_xM\mid f \text{ linear isomorphism}\}$.} to $TM$, identifying a $C^\infty(P)$-bilinear alternating map $\mathfrak{X}(P)\times\mathfrak{X}(P)\rightarrow C^\infty(P,\mathbb{K}^{2n\times 2n})$ of which we can indeed consider the standard determinant.
	
	Note that, given another almost complex structure $\tilde{J}\in\mathfrak{J}(M,\omega)$, one can show that $c_1(TM,J)=c_1(TM,\tilde{J})$ (similarly for the other $c_k$'s), and therefore unambiguously write $c_1(TM)\in H^2(M;\mathbb{Z})$ for the \textbf{first Chern class}\index{first Chern class} of $M$. 
\end{Def}  
\vspace*{0.2cm}

\noindent\minibox[frame]{Henceforth, we assume that the symplectic manifold $M$ fulfills $2c_1(TM)=0$, \\ and consider only compact Lagrangian submanifolds $L\subset M$.}
\vspace*{0.2cm}

\subsection{The Floer complex}\label{ch5.3}

As apparent from the previous section, to understand Fukaya categories we must first understand Lagrangian Floer theory. The rest of Chapter 5 is devoted to gain an essential grasp on this subject. However, the reader should be aware that the following exposition deliberately takes the many underlying analytical results and constructions for granted; whenever sensitive, we provide technical comments so to orientate the more expert eyes. Detail-filled accounts on (general) Floer theory can easily be found throughout literature (see for example \cite{[Sal97]} or \cite{[Mer14]}); here we limit ourselves to the concise version of \cite[section 1]{[Aur13]}.

Firstly, we should mention that Lagrangian Floer homology can be seen as an infinite-dimensional analogue of Morse homology\footnote{For an overview on Morse theory consult for example \cite[chapter 1]{[Mer14]}.} for the action functional on the universal cover of $\mathcal{P}(L_0,L_1)\coloneqq\{\gamma:[0,1]\rightarrow M\mid \gamma(0)\in L_0,\,\gamma(1)\in L_1\}$, given any two Lagrangian submanifolds $L_0, L_1\subset M$ (see \cite[section 1.1]{[Aur13]}). We will return to this in Section \ref{ch5.5}. 

However, the infinite-dimensional framework makes the analytical problems too hard to tackle, so that one actually resorts to an alternative point of view: moduli spaces of pseudo-holomorphic strips.

\begin{Def}\label{Floercomp}						
	Let $M=(M,\omega)$ be a symplectic manifold equipped with some $J\in\mathfrak{J}(M,\omega)$, and let $L_0, L_1\subset M$ be two compact Lagrangian submanifolds such that $L_0\pitchfork L_1$ (whence $|L_0\cap L_1|<\infty$, because $L_0\cap L_1$ is a 0-dimensional compact set, by definition). Then the \textbf{Floer complex}\index{Floer complex} of $(L_0,L_1)$ is the free $\Lambda_\mathbb{K}$-module generated by $\mathcal{X}(L_0,L_1)\coloneqq L_0\cap L_1$,
	\begin{equation}\label{Floercmplx}
	CF(L_0,L_1)\coloneqq \mkern-18mu\bigoplus_{p\in\mathcal{X}(L_0,L_1)}\mkern-18mu \Lambda_\mathbb{K}\cdot p\;.
	\end{equation} 
\end{Def}

Of course, Definition \ref{Floercomp} lacks information about the corresponding boundary operator $\partial:CF(L_0,L_1)\rightarrow CF(L_0,L_1)$. This is given by ``counting'' $J$-holomorphic strips in $M$ with boundary on $L_0$ and $L_1$. More precisely, the coefficient of $q$ in $\partial(p)$ depends on the space of solutions $u:\mathbb{R}\times[0,1]\simeq\mathbb{D}^2\setminus \{\pm 1\}\subset\mathbb{C}\rightarrow M, (s,t)\equiv(s+it)\mapsto u(s,t)$ of the \hypertarget{certainprob}{\textbf{Cauchy--Riemann equation}}\index{Cauchy--Riemann equation}
\begin{equation}\label{CRequation}
\bar{\partial}_J(u) \coloneqq  \partial_su+J(u)\partial_tu = 0      
\end{equation}
(where $J(u)\equiv J_{u(\cdot,\cdot)}$) with boundary conditions
\begin{equation}\label{CRboundarycond}
\left\{\begin{array}{llll}
	u(s,0)\in L_0 & \text{ and } & u(s,1)\in L_1 & \forall s\in\mathbb{R} \\ 
	\!\!\!\displaystyle\lim_{s\to-\infty}\!\!u(s,t)=q & \text{ and } & \!\!\!\displaystyle\lim_{s\to+\infty}\!\!u(s,t)=p & \forall t\in[0,1]
\end{array}\right.\,,
\end{equation}
and the finite \textbf{energy constraint}\index{energy constraint}\footnote{When treating the \textit{perturbed} Cauchy--Riemann equation \eqref{perturbedCR}, one can in fact prove that $E(u)<\infty$ is equivalent to the asymptotic conditions $\lim_{s\to+\infty}u(s,t)=p$ and $\lim_{s\to-\infty}u(s,t)=q$ from \eqref{CRboundarycond}; see \cite[Proposition 1.21]{[Sal97]}.}  
\begin{equation}\label{CRenergy}
E(u)=\int_{\mathbb{R}\times[0,1]}\mkern-18mu u^*\omega = \frac{1}{2}\int_{-\infty}^{+\infty}\int_{0}^{1}|\partial_su|^2 ds dt\overset{!}{<}\infty\;.
\end{equation}
(equivalently a bound on the symplectic area). Any solution $u$ is indeed a $J$-holomorphic map (compare to Definition \ref{Jholocurve}), as depicted in Figure \ref{holomorphicstrip2}.

\begin{figure}[htp]
	\centering
	\includegraphics[width=0.9\textwidth]{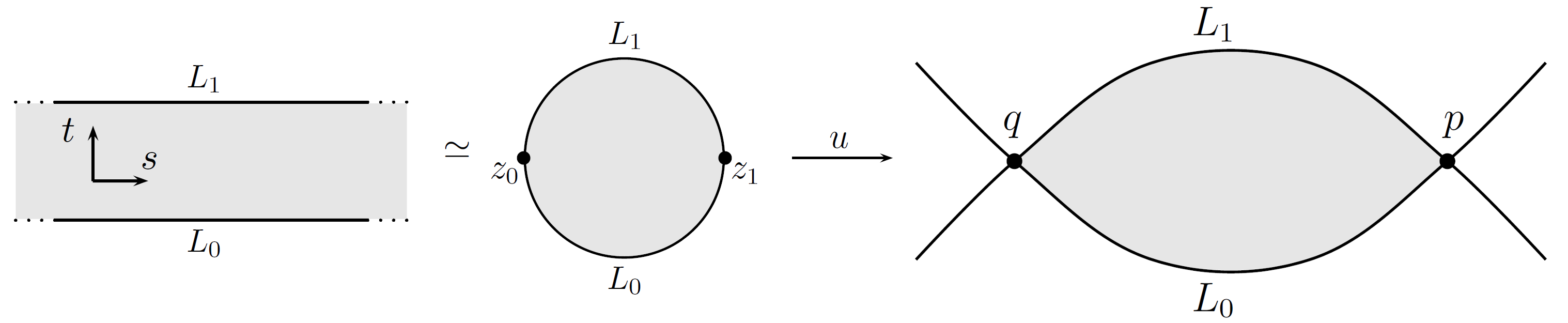}	
	\caption{A $J$-holomorphic strip $u$ belonging to $\widehat{\mathcal{M}}(p,q;[u],J)$. [Source:\cite{[Aur13]}]}
	\label{holomorphicstrip2}	
\end{figure}

\begin{Def}\label{modulispace}							
	Let $M=(M,\omega,J)$ be an almost complex symplectic manifold and $L_0,L_1\subset M$ compact Lagrangian submanifolds such that $L_0\pitchfork L_1$. Consider a solution $u:\mathbb{R}\times[0,1]\rightarrow M$ of the Cauchy--Riemann problem \eqref{CRequation} with constraints \eqref{CRboundarycond} and \eqref{CRenergy}, thus yielding a homotopy class $[u]\!\in\!\pi_2(M,L_0\cup L_1)$ \big($=\textup{Hom}_\mathsf{hTop}((B^2,\mathbb{S}^1),(M,L_0\cup L_1))$, for pointed pairs of topological spaces $(B^2,\mathbb{S}^1)\simeq(\mathbb{D}^2\setminus\{\pm1\},\mathbb{R}\times\{0\}\cup\mathbb{R}\times\{1\})$ and $(M,L_0\cup L_1)$\big). 
	
	For fixed $p,q\in\mathcal{X}(L_0,L_1)$, we denote by $\widehat{\mathcal{M}}(p,q;[u],J)$ the \textbf{moduli space}\index{moduli space} of solutions to \eqref{CRequation}--\eqref{CRenergy} representing $[u]$ and set $\mathcal{M}(p,q;[u],J)\!\coloneqq \!\widehat{\mathcal{M}}(p,q;[u],J)/\mathbb{R}$.
\end{Def}

We observe that $\mathbb{R}$ acts freely on $\widehat{\mathcal{M}}(p,q;[u],J)$ by translation in the $s$-coordinate: if $u(s,t)\in\widehat{\mathcal{M}}(p,q;[u],J)$, then $u_{\tilde{s}}(s,t)\coloneqq u(s-\tilde{s},t)\in\widehat{\mathcal{M}}(p,q;[u],J)$ for all $\tilde{s}\in\mathbb{R}$ as well. This is taken into account in the definition of $\mathcal{M}(p,q;[u],J)$.

There are now three issues hindering our attempt to define a boundary operator for $CF(L_0,L_1)$. Namely, for $\widehat{\mathcal{M}}(p,q;[u],J)$ to be a ``manageable'' space, we need it to have the following properties:
\begin{itemize}[leftmargin=0.5cm]
	\item \textit{Transversality}, which ensures that $\widehat{\mathcal{M}}(p,q;[u],J)$ can be given the structure of a finite-dimensional smooth manifold, whose dimension is given by the Maslov index of $[u]$.
	\item \textit{Compactness}, which is achieved only in absence of ``energy build-ups'' making $\widehat{\mathcal{M}}(p,q;[u],J)$ ill-defined.
	\item \textit{Orientability}, which should be consistent with the choice of orientation of the underlying Lagrangian submanifolds.
\end{itemize}
However, none of these requisites are a priori met, and hence we need additional assumptions. 

Let us start by briefly addressing orientability. The following definition is kept intentionally vague, since part of the aforementioned technicalities we would rather avoid. The pragmatic reader may skip it.

\begin{Def}										
	Let $L^n\subset M^{2n}$ be a Lagrangian submanifold. Assume $L$ is oriented, that is, endowed with (the class of) a nowhere-vanishing $n$-form $\theta_L\in\Omega^n(L)$, a volume form. Observe that, since smooth manifolds are in particular cell complexes, we can consider skeleton filtrations thereof. Then $L$ is:
	\begin{itemize}[leftmargin=0.5cm]
		\item a \textbf{spin manifold}\index{spin manifold} if we can trivialize the tangent bundle $TL$ over the 1-skeleton $L^{(1)}$ in a way extendable to the 2-skeleton $L^{(2)}$. The homotopy class of such trivialization is called a \textit{spin structure} for $L$, in fact existing if and only if the second \textit{Stiefel--Whitney class} $w_2(TL)\in H^2(L;\mathbb{Z}_2)$ --- an invariant of $TL$ --- vanishes.
		\item a \textbf{relatively spin manifold}\index{relatively spin manifold} relative to some fixed $\langle b\rangle\in H^2(M;\mathbb{Z}_2)$ if $w_2(TL)=\langle b\rangle|_L$.
	\end{itemize}
\end{Def}

Assuming $L_0$ and $L_1$ are oriented and spin, one can show that their spin structures yield a canonical orientation for $\widehat{\mathcal{M}}(p,q;[u],J)$ (this is proved in \cite[chapter 12]{[Sei08]}). However, these assumptions are unnecessary if $\text{char}(\mathbb{K})=2$.\\

\noindent\minibox[frame]{Henceforth, we will always assume that all Lagrangian submanifolds come \;\;\;\, \\ with an (implicit) choice of orientation and spin structure.}   
\vspace*{0.2cm}

\subsection{The Maslov index}\label{ch5.4}

The Maslov index of the homotopy class $[u]\in\pi_2(M,L_0\cup L_1)$, where $u$ solves \eqref{CRequation}--\eqref{CRenergy}, determines the dimension of $\widehat{\mathcal{M}}(p,q;[u],J)$. Actually, the situation is more subtle (the unfamiliar reader should skip the following remark):

\begin{Rem}\label{technically}								
	The Cauchy--Riemann problem \eqref{CRequation}--\eqref{CRenergy} is actually a \textit{Fredholm problem}, meaning that the linearization $D_{\bar{\partial}_J,u}$ of $\bar{\partial}_J$ at any zero $u$ (that is, solution) is a \textit{Fredholm operator}. We can assign to $[u]$ its \textit{Fredholm index} $\text{ind}([u]) \coloneqq  \text{dim}(\text{ker}(D_{\bar{\partial}_J,u})) - \text{dim}(\text{coker}(D_{\bar{\partial}_J,u}))$, which in fact coincides with the Maslov index we will compute below. We remark that \textit{Hamiltonian} Floer theory is actually discussed in terms of another quantity, the \textit{Conley--Zehnder index}.
\end{Rem}	

Recall from Remark \ref{LGr} the definition of Lagrangian Grassmannian $LGr(n)$, diffeomorphic to $\text{U}(n)/\text{O}(n)$. It can be shown that the square of the determinant map $\text{det}:\text{U}(n)/\text{O}(n)\rightarrow\mathbb{S}^1$ induces an isomorphism $\pi_1(LGr(n))\cong\pi_1(\mathbb{S}^1)=\mathbb{Z}$.

\begin{Def}\label{Maslovforpaths}				
	The \textbf{Maslov index of a loop}\index{Maslov index!of a loop} $\gamma:[0,1]\rightarrow LGr(n)$ based at some $W\in LGr(n)$ is given by the winding number $\text{ind}(\gamma)\coloneqq w(\gamma,W)\in\mathbb{Z}$ of $\text{im}(\text{det}^2\circ\gamma)\subset\mathbb{S}^1$ (identifying $LGr(n)\equiv\text{U}(n)/\text{O}(n)$).
	
	On the other hand, given two paths $l_0,l_1:[0,1]\rightarrow LGr(n)$ with $l_0(0)\pitchfork l_1(0)$ and $l_0(1)\pitchfork l_1(1)$, define $l\coloneqq (l_0,l_1):t\!\in\![0,1]\mapsto (l_0(t),l_1(t))\!\in\! LGr(n)\times LGr(n)$. The \textbf{Maslov index of the paths}\index{Maslov index!of a pair of paths} $l_0,l_1$ is 
	\begin{equation}
		\text{ind}(l_0,l_1)\coloneqq \big|\text{im}(l)\cap\overline{LGr(n)}\big|\,,
	\end{equation}
	where $\overline{LGr(n)}\coloneqq \{(W,W')\in LGr(n)\times LGr(n)\mid W\,\cancel{\pitchfork}\,W'\}$. Put differently, $\text{ind}(l_0,l_1)$ is the (signed) number of times $l_0(t)$ is not transverse to $l_1(t)$, for increasing $0\leq t\leq 1$.
\end{Def}

Let us parse the geometric meaning of Maslov index for paths with help of an easy example.

\begin{Ex}												
	For simplicity, we take $l_0(t)\coloneqq \mathbb{R}^n\subset\mathbb{R}^{2n}\cong\mathbb{C}^n$ to be constant, and define $l_1(t)\coloneqq (e^{i\varphi_1(t)}\mathbb{R})\times...\times (e^{i\varphi_n(t)}\mathbb{R})$ so that all angles $\varphi_k$ vary uniformly from $-\pi/2$ to $\pi/2$. The Maslov index of $l_0$ and $l_1$ is precisely $\text{ind}(l_0,l_1)=n$. This is best seen for the case $n=1$, where $l_0=x\text{-axis}$: the straight line $l_1(t)$ rotates counterclockwise from $l_1(0)=y\text{-axis}$ (where $\varphi_1(0)=-\pi/2$) to $l_1(1)=y\text{-axis}$ (where $\varphi_1(1)=\pi/2$), intersecting $l_0$ transversely at all times, except at $\varphi(1/2)=0$, where the two lines actually coincide.\hfill $\blacklozenge$ 
\end{Ex}

Consider now a $J$-holomorphic strip $u:\mathbb{R}\times[0,1]\rightarrow M$ solving the Cauchy--Riemann problem \eqref{CRequation}--\eqref{CRenergy}. Write $q=u(-\infty,t)$ and $p=u(\infty,t)$ for any $t\in[0,1]$. Then the smooth map $u$ yields the symplectic pullback bundle\break $u^*TM=\{((s,t),(x,v))\in(\mathbb{R}\times[0,1])\times TM\mid u(s,t)=x\}$, and we can define the Maslov index of $u$ as follows.

\begin{Def}\label{Maslovforstrips}						
	After choosing a trivialization for $u^*TM\rightarrow\mathbb{R}\times[0,1]$ (see for example \cite[chapter 5]{[Mer14]}), we define two paths $l_i(s)\coloneqq ((u|_{\mathbb{R}\times\{i\}})^*TL_i)|_{(-s,i)}$ $\cong T_{u(-s,i)}L_i$ for $s\in\overline{\mathbb{R}}=[-\infty,+\infty]$ and $i=0,1$, whose image actually lies in $LGr(n)$ (by Remark \ref{LGr}) and such that each $l_i$ goes from $T_pL_i$ to $T_qL_i$, satisfying $l_0(-\infty)\pitchfork l_1(-\infty)$ and $l_0(\infty)\pitchfork l_1(\infty)$ (because $L_0\pitchfork L_1$).
	
	The Maslov index of the $J$-holomorphic strip $u$ is just the Maslov index of the two paths $l_0$ and $l_1$, $\text{ind}([u])\coloneqq \text{ind}(l_0,l_1)$, as in Definition \ref{Maslovforpaths} (which is still valid for the enlarged domain $\overline{\mathbb{R}}$).
\end{Def}

We remark that $\text{ind}(l_0,l_1)$ is actually independent from the chosen trivialization for $u^*TM$, and invariant under conjugation by an element in $\text{Sp}(2n,\mathbb{R})$. This makes $\text{ind}([u])$ well-defined. Moreover, it is additive under concatenation (further properties, along with a slightly modified construction, can be found in \cite[section 4]{[Sei99]}). Let us see the equivalent formulation as index of a suitable loop.

\begin{Def}\label{canonicalshortpath}					
	Consider $\lambda_0,\lambda_1\in LGr(n)$ such that $\lambda_0\pitchfork\lambda_1$. Identify the unique transformation $A\in \text{Sp}(2n,\mathbb{R})$ such that $A(\lambda_0)=\mathbb{R}^n\subset\mathbb{C}^n$ and $A(\lambda_1)=(i\mathbb{R})^n\subset\mathbb{C}^n$. Then we call $\lambda:t\in[0,1]\mapsto A^{-1}((e^{-i\pi t/2}\mathbb{R})^n)\in LGr(n)$ the \textit{canonical short path} from $\lambda_0=\lambda(0)$ to $\lambda_1=\lambda(1)$. 
\end{Def}

\begin{Def}\label{Maslovindex}					 
	Let $u:\mathbb{R}\times[0,1]\rightarrow M$ be a $J$-holomorphic strip solving the Cauchy--Riemann problem \eqref{CRequation}--\eqref{CRenergy}. Define the paths $l_0,l_1:\overline{\mathbb{R}}\rightarrow LGr(n)$ as in the paragraph above, and consider canonical short paths $\lambda_p, \lambda_q:[0,1]\rightarrow LGr(n)$ from $T_pL_0$ to $T_pL_1$ respectively from $T_qL_0$ to $T_qL_1$, as in Definition \ref{canonicalshortpath}. 
	
	We thus have a loop $\gamma_u\coloneqq \overline{l_0}\ast\lambda_p\ast l_1\ast\overline{\lambda_q}$ in $LGr(n)$ based at $T_qL_0$ (bars denote inverse paths). The \textbf{Maslov index}\index{Maslov index} of $u$ coincides with the corresponding winding number,
	\begin{equation}
		\text{ind}([u])\coloneqq\text{ind}(\gamma_u)=w(\gamma_u, T_qL_0)\,.
	\end{equation}   
\end{Def} 

\begin{Ex}											
	Take $M=\mathbb{R}^2$. Then the Maslov index of the $J$-holomorphic strip $u$ depicted in Figure \ref{holomorphicstrip2} can be computed using the following trick: we choose a trivialization $u^*TM$ making the path $l_0(s)=T_{u(-s,0)}L_0$ equal the constant tangent line $\mathbb{R}\times\{0\}$, and without loss of generality parametrizing $L_1$ as the downwards open parabola segment $\{(x,-x^2)\in\mathbb{R}^2\mid x\in[q,p]\}$ (refer to Figure \ref{holomorphicstrip2} for clarity). 
	
	It is then obvious that the path $l_1(s)=T_{u(-s,1)}L_1$ is non-transverse to $l_0$ only once, namely at $s=0$ (where the tangent line $l_1(0)$ is parallel to $\mathbb{R}\times\{0\}$).\break Therefore, completing the loop at $T_qL_0$ with canonical short paths as in Definition \ref{Maslovindex}, it holds that ind($[u]$)=1. \hfill $\blacklozenge$ 
\end{Ex}
\vspace*{0.3cm}

\noindent Let us discuss how Floer complexes can be equipped with a $\mathbb{Z}$-grading. Here is where our assumption of a 2-torsion first Chern class comes into play.

Consider the fibre bundle $LGr(TM)\rightarrow M$ of fibres $LGr(TM)_x\coloneqq LGr(T_xM)$ $\cong LGr(n)$ (this is the bundle $P\times_{\text{Sp}(2n,\mathbb{R})}LGr(n)$, where $P$ is the associated symplectic frame bundle, a principal $\text{Sp}(2n,\mathbb{R})$-bundle over $M$), each having universal cover $\upsilon_x:\widetilde{LGr}(T_xM)\rightarrow LGr(T_xM)$. Then $2c_1(TM)=0$ allows\footnote{Unfortunately, this is yet another technical bit which falls out of scope. It is addressed in \cite[Lemma 2.6]{[Fuk02b]} or \cite[section (11j)]{[Sei08]}, in the context of \textit{Lagrangian branes}.} us to glue all these covers to form a covering space $\widetilde{LGr}(TM)\rightarrow LGr(TM)$ such that $\widetilde{LGr}(TM)_x = \widetilde{LGr}(T_xM)$ for all $x\in M$. Elements of $\widetilde{LGr}(TM)$ are called \textit{graded Lagrangian planes} in $TM$.

Now, given the section $s\in\Gamma(L,LGr(TM)),\;x\in L\mapsto T_xL\in LGr(T_xM)$ on some Lagrangian submanifold $L\subset M$, we may wonder what is the obstruction to lifting $s$ to some section $\tilde{s}\in\Gamma(L,\widetilde{LGr}(TM))$, so to produce the commutative diagram
\vspace*{-0.1cm}
\begin{equation}\label{liftsection}
\begin{tikzcd}
& \widetilde{T_xL}\in\arrow[d, mapsto, "\upsilon_x"]\widetilde{LGr}(T_xM) \\
x\in L\arrow[ur, mapsto, end anchor={[yshift=-0.5ex]west}, "\tilde{s}"]\arrow[r, mapsto, "s"'] & T_xL\in LGr(T_xM)
\end{tikzcd}\quad.
\end{equation}

\noindent It turns out that we need the \textit{Maslov class} of $L$ to vanish:

\begin{Def}\label{phase}										
	Let $(M,\omega)$ be a symplectic manifold with $2c_1(TM)=0$, so that to each $W\in LGr(T_xM)$ we can assign a ``phase'' $\phi(W)\in\mathbb{S}^1$ and choose a lift $\tilde{\phi}(W)\in\mathbb{R}$ fulfilling $\phi(W)=e^{2\pi i\tilde{\phi}(W)}$. Then, given a Lagrangian submanifold $L\subset M$, we can define its ``phase function'' $\phi_L:L\rightarrow\mathbb{S}^1$ by $\phi_L(x)\coloneqq \phi(T_xL)$.
	
	The \textbf{Maslov class}\index{Maslov class} of $L$ is\footnote{Observe that, $\mathbb{Z}$ being divisible, $\text{Ext}(H_0(L),\mathbb{Z})$ vanishes, and hence $H^1(L;\mathbb{Z})\cong\textup{Hom}(H_1(L),\mathbb{Z})$. Now, working on connected components of $L$, the Hurewicz Theorem tells us that $H_1(L)\cong\pi_1(L)^{ab}=\pi_1(L)/[\pi_1(L),\pi(L)]$, and abelianity of $\mathbb{Z}$ guarantees that $\textup{Hom}(\pi_1(L),\mathbb{Z})\cong\textup{Hom}(\pi_1(L)^{ab},\mathbb{Z})$.} precisely 
	\begin{equation}
		\mu_L\equiv[\phi_L]\coloneqq \pi_1(\phi_L)\in\textup{Hom}(\pi_1(L),\pi_1(\mathbb{S}^1))\cong H^1(L;\mathbb{Z})\,.
	\end{equation}
	Moreover, the choice of a map $\tilde{\phi}_L:L\rightarrow\mathbb{R}$ lifting $\phi_L$, that is, fitting into the commutative diagram 
	\vspace*{-0.5cm}
	\begin{equation}
	\begin{tikzcd}
	& \tilde{\phi}(T_xL)\in\mathbb{R}\arrow[d, mapsto, "\text{exp}(2\pi i\cdot)"] \\
	x\in L\arrow[ur, mapsto, end anchor={[yshift=-0.5ex]west}, "\tilde{\phi}_L"]\arrow[r, mapsto, "\phi_L"'] & \phi(T_xL)\in\mathbb{S}^1
	\end{tikzcd}\quad,
	\end{equation}
	yields a \textbf{graded Lagrangian submanifold}\index{Lagrangian submanifold!graded} $(L,\tilde{\phi}_L)$ with \textbf{graded lift}\index{graded lift} $\tilde{\phi}_L$. 
\end{Def}                                                                                    
Choose now transversely-intersecting graded Lagrangian submanifolds\break $(L_0,\tilde{\phi}_{L_0})$, $(L_1,\tilde{\phi}_{L_1})\subset M$ with $[\phi_{L_0}]=[\phi_{L_1}]=0$, so that there are sections as in diagram \eqref{liftsection}. Pick any $p\in L_0\cap L_1$. Then there exists a favoured (homotopy class of a) path $l_p:[0,1]\rightarrow LGr(T_pM)$ such that $l_p(0)=T_pL_0$ and $l_p(1)=T_pL_1$, namely the one selected by the covering map $\upsilon_p$ after looking at the universal cover $\widetilde{LGr}(T_pM)$.

\begin{Def}\label{degree}								
	Let $(L_0,\tilde{\phi}_{L_0}),(L_1,\tilde{\phi}_{L_1})\subset M$ be graded Lagrangian submanifolds with vanishing Maslov class, and $p\in L_0\cap L_1$. Let $l_p:[0,1]\rightarrow LGr(T_pM)$ be the path from previous paragraph, and $\lambda_p:[0,1]\rightarrow LGr(T_pM)$ the canonical short path from $T_pL_0$ to $T_pL_1$. Then the \textbf{degree}\index{degree (of a Floer generator)} of $p$ is given by 
	\begin{equation}
		\text{deg}(p)\coloneqq w(\gamma_p)\,,
	\end{equation}
	where $\gamma_p\coloneqq l_p\ast\overline{\lambda_p}$ is a loop based at $T_pL_0$.
\end{Def}

After suitably identifying the tangent spaces $T_x(L_0\cap L_1)$, an immediate consequence of Definitions \ref{Maslovindex} and \ref{degree} is:

\begin{Cor}\label{generatordegree}						
	Let $u:\mathbb{R}\times[0,1]\rightarrow M$ be a $J$-holomorphic strip solving the Cauchy--Riemann problem \eqref{CRequation}--\eqref{CRenergy}, where $L_0$ and $L_1$ are transversely-intersecting graded Lagrangian submanifolds with $[\phi_{L_0}]=[\phi_{L_1}]=0$. Then: 
	\[
	\textup{ind}([u])=\textup{deg}(q)-\textup{deg}(p).
	\]
	In particular, $\textup{ind}([u])$ only depends on the two generators $p$, $q$ connected by $u$, so that $CF(L_0,L_1)$ can be equipped with a $\mathbb{Z}$-grading:
	\[
	CF^n(L_0,L_1)\coloneqq \mkern-18mu\bigoplus_{\substack{p\in\mathcal{X}(L_0,L_1): \\ \textup{deg}(p)=n}}\mkern-18mu \Lambda_\mathbb{K}\cdot p\;,\quad\text{for }n\in\mathbb{Z}.
	\] 
\end{Cor}

Without the assumptions $2c_1(TM)=0$ and $[\phi_{L_0}]=[\phi_{L_1}]=0$, we merely obtain a $\mathbb{Z}_2$-grading (this alternative is briefly discussed in \cite[section 4]{[Sei99]}).\\

\noindent\minibox[frame]{Henceforth, we assume that all (oriented spin) Lagrangian submanifolds are \\ graded, with vanishing Maslov class.}
\vspace*{0.2cm}

\subsection{Transversality}\label{ch5.5}

Transversality of the moduli spaces $\widehat{\mathcal{M}}(p,q;[u],J)$ is somewhat more challenging to address, especially in terms of the underlying theory.

\begin{Rem}\label{surjectivity}					
	Building on the technical bit of Remark \ref{technically}, transversality is achieved if all zeros of the linearization $D_{\bar{\partial}_J,u}$ of $\bar{\partial}_J$ are regular, that is, if $D_{\bar{\partial}_J,u}$ is surjective for all $u\in\widehat{\mathcal{M}}(p,q;[u],J)$.
\end{Rem}

A solution for the case $L_0\pitchfork L_1$ is provided by allowing our $\omega$-compatible almost complex structure to be $t$-dependent, that is, by substituting $J$ in \eqref{CRequation} with $J_t$, for $t\in[0,1]$ (formally, we just take $J\in C^\infty([0,1],\mathfrak{J}(M,\omega))$ and write $J_t\coloneqq J(t,\,\cdot\,)$).

This strategy comes about also if we try to define $CF(L_0,L_1)$ in the most general case when $L_0\,\cancel{\pitchfork}\,L_1$ --- which we really need to study if we are to find identity morphisms in $CF(L_i,L_i)$! Indeed, in order to drop transversality of the Lagrangian submanifolds, we consider a time-dependent Hamiltonian function $H\in C^\infty([0,1]\times M)$ such that $L_0\,\pitchfork\,(\phi_H^1)^{-1}(L_1)$, where $\phi_H^1:M\rightarrow M$ is the Hamiltonian diffeomorphism of the associated $t$-dependent Hamiltonian vector field $X_H\in\mathfrak{X}([0,1]\times M)$. Equation \eqref{CRequation} then becomes a perturbed PDE, with perturbation given by $X_H$:
\begin{equation}\label{perturbedCR}
\bar{\partial}_{J,H}(u)\coloneqq \partial_su + J_t(u)(\partial_tu - X_{H_t}(u))=0\;.
\end{equation}
(This is sometimes called \textbf{Floer's equation}\index{Floer's equation}, and $\bar{\partial}_{J,H}$ Floer's differential.) The boundary conditions must be modified accordingly. The finite energy constraint \eqref{CRenergy} becomes:
\begin{equation}\label{perturbedCRenergy}
E(u)=\frac{1}{2}\int_{-\infty}^{+\infty}\int_{0}^{1}(|\partial_su|^2+|\partial_tu-X_{H_t}(u)|^2)dsdt \overset{!}{<}\infty\;,
\end{equation}
and we still want $u(s,i)\in L_i$ for $i=0,1$. However, the asymptotes for $s\rightarrow\pm\infty$ change due to the next observation.

\begin{Rem}\label{CRreform}								
	The perturbed Cauchy--Riemann equation \eqref{perturbedCR} can be made into a homogeneous Cauchy--Riemann equation of the form \eqref{CRequation} as follows. Let $\phi_H^t:M\rightarrow M$ be the flow of $X_H$, and consider the preimage of $u$ under it, $\tilde{u}:\mathbb{R}\times [0,1]\rightarrow M,\,\tilde{u}(s,t)\coloneqq (\phi_H^t)^{-1}(u(s,t))$, so that the chain rule yields $\partial_t\tilde{u}=(\phi_H^t)_*^{-1}(\partial_tu-X_H)$. Define also $\tilde{J}_t\coloneqq (\phi_H^t)_*^{-1}(J_t)$. Then \eqref{perturbedCR} acquires the familiar aspect
	\[
	\partial_s\tilde{u} + \tilde{J}_t(\tilde{u})\partial_t\tilde{u}=0\;,
	\]
	a PDE which we know to have as solutions $\tilde{J}$-holomorphic strips $\tilde{u}$ fulfilling $\tilde{u}(s,0)=u(s,0)\in L_0$ and $\tilde{u}(s,1)=(\phi_H^1)^{-1}(u(s,1))\in(\phi_H^1)^{-1}(L_1)$.
\end{Rem}

By last remark, the Definition \ref{Floercomp} of Floer complex now requires the modification $\mathcal{X}(L_0,L_1)\equiv\mathcal{X}(L_0,L_1;H)\coloneqq L_0\cap (\phi_H^1)^{-1}(L_1)$, which is again a discrete finite set due to $L_0\,\pitchfork\,(\phi_H^1)^{-1}(L_1)$. The generators are intersection points perturbed according to flow lines $\gamma:[0,1]\rightarrow M$ (fulfilling $\dot{\gamma}(t)=X_{H_t}(t,\gamma(t))$, by definition) such that $\gamma(0)\in L_0$ and $\gamma(1)\in L_1$. Accordingly, the Floer differential counts \textit{perturbed} $J$-holomorphic strips in moduli spaces of the form $\mathcal{M}(p,q;[u],J,H)$, now depending on $H$ as well.

\begin{Rem}(\textit{Warning})							
	We have overlooked a non-trivial fact when picking a suitable $H\in C^\infty([0,1]\times M)$ such that $L_0\,\pitchfork\,(\phi_H^1)^{-1}(L_1)$: $H$ should actually be chosen in $\mathcal{H}_\text{reg}\subset C^\infty([0,1]\times M)$, the subset of all Hamiltonian functions whose Hamiltonian diffeomorphisms fulfill such transversality property for some pair of Lagrangian submanifolds. However, this does not invalidate our analysis since $\mathcal{H}_\text{reg}$ turns out to be a \textit{dense} subset of $C^\infty([0,1]\times M)$ (more precisely, of second category in the sense of Baire; cf. \cite[Theorem 1.24]{[Sal97]}). 
	
	Similarly, the time-dependent almost complex structure $J$ assuring surjectivity of the linearization $D_{\bar{\partial}_J,u}$ (see Remark \ref{surjectivity}) should actually belong to a dense subset $\mathfrak{J}_\text{reg}\subset C^\infty([0,1],\mathfrak{J}(M,\omega))$.
	
	We bear these considerations in mind when selecting the data $(J,H)$.
\end{Rem} 
\vspace*{0.3cm}

\noindent Let us consider the perturbed problem \eqref{perturbedCR} from a slightly different perspective, namely that offered by \cite{[Oh91]}, \cite{[McS94]} and \cite[chapter 3]{[Mer14]}.

\begin{Def}											
	Let $(M,\omega)$ be a symplectic manifold, $L_0, L_1\subset M$ \textit{any} Lagrangian submanifolds and $H\in C^\infty([0,1]\times M)$ a time-dependent Hamiltonian. Define the \textbf{path space}\index{path space}
	\[
	\mathcal{P}=\mathcal{P}(L_0,L_1)\coloneqq \{\gamma\in C^\infty([0,1],M)\mid \gamma(0)\in L_0, \gamma(1)\in L_1\}\,. 
	\]
	This is an infinite-dimensional smooth manifold of countably many connected components, with tangent spaces 
	\[
	T_\gamma\mathcal{P}=\Gamma(\gamma^*TM)\cong\Gamma_\gamma(TM)\coloneqq \{\xi\in C^\infty([0,1],TM)\mid \xi(t)\in T_{\gamma(t)}M\}\,,
	\]
	for every $\gamma\in\mathcal{P}$. We define the closed 1-form $a_H\in\Omega^1(\mathcal{P})$ via the $C^\infty(\mathcal{P})$-linear map $\mathfrak{X}(\mathcal{P})\rightarrow C^\infty(\mathcal{P})$ it identifies:
	\begin{align}\label{Hamform}
	(a_H)_\gamma(\xi):\!&=\int_0^1\big(\omega_{\gamma(t)}(\xi(t),\dot{\gamma}(t))-(dH_t)_{\gamma(t)}(\xi(t))\big)dt \nonumber\\
	&= \int_0^1\omega_{\gamma(t)}\big(\xi(t),\dot{\gamma}(t)-X_{H_t}(\gamma(t))\big)dt\;,
	\end{align}
	for $\gamma\in\mathcal{P}$ (so that $\dot{\gamma}\coloneqq \frac{d}{dt}|_t\gamma\in\Gamma_\gamma(TM)$) and $\xi\in T_\gamma\mathcal{P}$, where last equality makes use of $dH_t = \omega(\,\cdot\,,X_{H_t}):\mathfrak{X}(M)\rightarrow C^\infty(M)$, by definition of Hamiltonian vector field (see Definition \ref{Hamiltonian}). Observe that closedness of $a_H$ is an immediate consequence of the closedness of $\omega$.
\end{Def}

It is then straightforward to see:

\begin{Cor}\label{flowlines}						
	Let $(M,\omega)$ be a symplectic manifold, $L_0, L_1\subset M$ any Lagrangian submanifolds and $H\in C^\infty([0,1]\times M)$ a time-dependent Hamiltonian function generating a flow $\phi_H^t:M\rightarrow M$. Define $a_H\in\Omega^1(\mathcal{P})$ as above. Then $\gamma\in\mathcal{P}$ is a zero of $a_H$, $\gamma\in\textup{zeros}(a_H)$, if and only if $\gamma$ is a flow line, that is, $\gamma(t)=\phi_H^t(x)$ for some $x\in L_0$. In particular, there is a bijection given by evaluation at time $0$\textup: 
	\begin{equation}\label{zerosflowlines}
	\{\gamma\in\textup{zeros}(a_H)\}\longleftrightarrow\{\gamma(0)=x\in L_0\cap(\phi_H^1)^{-1}(L_1)\}\;.
	\end{equation} 
\end{Cor}

Corollary \ref{flowlines} tells us that we may identify generators of the Floer complex $CF(L_0,L_1)$ with zeros of the closed 1-form $a_H$, which we call \textbf{time-1 Hamiltonian chords}\index{time-1 Hamiltonian chords} from $L_0$ to $L_1$. 

However, we can still improve our interpretation.

\begin{Def}											
	Let $(M,\omega)$ be a symplectic manifold, $L_0, L_1\subset M$ any Lagrangian submanifolds. Define the path space $\mathcal{P}=\mathcal{P}(L_0,L_1)$ as above, and choose fixed reference paths $\gamma_k\in\mathcal{P}_k$ for each connected component of $\mathcal{P}=\bigsqcup_{k\in I}\mathcal{P}_k$. We call\footnote{This terminology is non-standard.} a \textbf{bridge}\index{bridge} for $\gamma\in\mathcal{P}_k$ any $\bar{\gamma}\in C^\infty(\overline{\mathbb{R}}\times[0,1],M)$ such that $\bar{\gamma}(-\infty,\,\cdot\,)=\gamma_k$, $\bar{\gamma}(+\infty,\,\cdot\,)=\gamma$ and $\bar{\gamma}(s,\cdot)\in\mathcal{P}_k$ for all $s\in\overline{\mathbb{R}}$ (so that we may also write $\bar{\gamma}\in C^\infty(\overline{\mathbb{R}},\mathcal{P}_k)$).
\end{Def}

\begin{Def}										
	Let $(M,\omega)$ be a symplectic manifold, $L_0, L_1\subset M$ Lagrangian submanifolds and $H\in C^\infty([0,1]\times M)$. Let $\mathcal{P}=\mathcal{P}(L_0,L_1)$ be the corresponding path space, and fix some $\gamma_0\in\mathcal{P}$ such that $\gamma_0(0)=q\in L_0\cap(\phi_H^1)^{-1}(L_1)$ is a perturbed generator of the Floer complex $CF(L_0,L_1)$ (or equivalently by \eqref{zerosflowlines}, such that $\gamma_0\in\text{zeros}(a_H)$). We define the \textbf{Hamiltonian action functional}\index{Hamiltonian action functional} $A_H:\mathcal{P}\rightarrow\mathbb{R}$ by
	\begin{equation}\label{Hamaction}
	A_H(\gamma)\coloneqq \int_{\overline{\mathbb{R}}\times [0,1]}\bar{\gamma}^*\omega - \int_0^1 H_t(\gamma(t))dt\;,
	\end{equation}
	where $\bar{\gamma}\in C^\infty(\overline{\mathbb{R}}\times[0,1],M)$ is any bridge for $\gamma\in\mathcal{P}$ such that $\bar{\gamma}(-\infty)=\gamma_0$ (for this to make sense, observe that $\gamma$ must belong to the same connected component of $\gamma_0$!). 
\end{Def}

As it is given, the definition of $A_H$ is still incomplete, since we must prove it to be independent from the specific choice of bridge $\bar{\gamma}$ --- in fact, we did cheat a little: without suitable assumptions on $M$, $L_0$ and $L_1$, the definition of a Hamiltonian action functional on $\mathcal{P}=\mathcal{P}(L_0,L_1)$ is in general impossible, and one is forced to work instead on its Novikov covering $\overline{\mathcal{P}}$. This is the covering space of elements all pairs $[\gamma,\bar{\gamma}]$ such that $\bar{\gamma}$ is a bridge for $\gamma\in\mathcal{P}$, obtained after identifying all bridges which are homotopic relative to $\partial(\overline{\mathbb{R}}\times[0,1])$. With some more work, one can further adapt $\overline{\mathcal{P}}$ to produce a covering map $\pi:\overline{\mathcal{P}}\rightarrow\mathcal{P}$ fulfilling $dA_H = \pi^*(a_H)$ when $H=0$ (see for example \cite[Proposition 13.5.1]{[Oh91]}). This is still valid in the perturbed case $H\neq 0$:

\begin{Lem}\label{zeroscrit}						
	Let $A_H:\overline{\mathcal{P}}\rightarrow\mathbb{R}$ be the Hamiltonian action functional \eqref{Hamaction}, $a_H\in\Omega^1(\mathcal{P})$ the $1$-form \eqref{Hamform}. Then it holds that $dA_H=\pi^*(a_H)$. Consequently, the critical points of $A_H$ coincide with the zeros of $a_H$, $\textup{crit}(A_H)=\textup{zeros}(a_H)$.
\end{Lem}

This, together with Corollary \ref{flowlines}, helps us refine our interpretation of the generators of generic Floer complexes $CF(L_0,L_1;H)$:
\begin{equation}
	\{\gamma(0)=x\in \mathcal{X}(L_0,L_1;H)\}\longleftrightarrow\{\gamma\in\text{crit}(A_H)\}\;.
\end{equation}

Lemma \ref{zeroscrit} also motivates the following couple of definitions.

\begin{Def}											
	Let $(M,\omega,J)$ be an almost complex symplectic manifold with $J\in\mathfrak{J}(M,\omega)$, $L_0, L_1\subset M$ Lagrangian submanifolds and $H\in C^\infty([0,1]\times M)$. Define the $L^2$-inner product
	\[
	\langle\!\langle\xi,\zeta\rangle\!\rangle_J\coloneqq \int_0^1(g_{J})_{\gamma(t)}(\xi(t),\zeta(t))dt = \int_0^1 \omega_{\gamma(t)}(J_{\gamma(t)}(\xi(t)),\zeta(t))dt\;,
	\]
	for $\gamma\in\mathcal{P}$ and $\xi,\zeta\in T_\gamma\mathcal{P}$. Then we denote by $\nabla_JA_H\in\mathfrak{X}(\mathcal{P})$ the vector field
	\begin{equation}
		\nabla_JA_H(\gamma)\coloneqq J(\gamma)(\dot{\gamma}-X_H(\,\cdot\,,\gamma))\in T_\gamma\mathcal{P}
	\end{equation}
	(where $J(\gamma)\equiv J_{\gamma(\cdot)}$). Explicitly, we have $\nabla_JA_H(\gamma)(t)\coloneqq J_{\gamma(t)}\big(\dot{\gamma}(t)-X_{H_t}(\gamma(t))\big)$ $\in T_{\gamma(t)}M$ for any $t\in[0,1]$.
\end{Def}

\begin{Lem}									
	The vector field $\nabla_JA_H\in\mathfrak{X}(\mathcal{P})$ is the $L^2$-gradient of $A_H$ with respect to $\langle\!\langle\cdot,\cdot\rangle\!\rangle_J$, that is, $(dA_H)_\gamma(\xi)=\langle\!\langle\nabla_JA_H(\gamma),\xi\rangle\!\rangle_J$ for all $\gamma\in\mathcal{P}$ and $\xi\in T_\gamma\mathcal{P}$.
\end{Lem}

\begin{proof}
	This is a straightforward computation. Pick any $\gamma\in\mathcal{P}$ and $\xi\in T_\gamma\mathcal{P}$, then: 
	\begin{align*}
		(dA_H)_\gamma(\xi)&=\int_0^1\omega_{\gamma(t)}\big(\xi(t),\dot{\gamma}(t)-X_{H_t}(\gamma(t))\big)dt \\
		&=\int_0^1\omega_{\gamma(t)}\big(J_{\gamma(t)}^2\big(\dot{\gamma}(t)-X_{H_t}(\gamma(t))\big),\xi(t)\big)dt \\
		&=\int_0^1(g_J)_{\gamma(t)}\big(J_{\gamma(t)}\big(\dot{\gamma}(t)-X_{H_t}(\gamma(t))\big),\xi(t)\big)dt \\
		&=\int_0^1(g_J)_{\gamma(t)}\big(\nabla_JA_H(\gamma)(t),\xi(t)\big)dt \\
		&=\langle\!\langle\nabla_JA_H(\gamma),\xi\rangle\!\rangle_J\;.
	\end{align*}
\end{proof}

In the language just formulated, the perturbed Cauchy--Riemann equation \eqref{perturbedCR} reads
\begin{equation}\label{illODE}
\frac{d}{ds}u(s)+\nabla_JA_H(u(s))=0\;,
\end{equation}
an ODE on $\mathcal{P}$ where we reinterpreted $J$-holomorphic strips $u:\mathbb{R}\times[0,1]\rightarrow M$ as maps $u:\mathbb{R}\rightarrow\mathcal{P}$, the (bounded) \textbf{negative gradient flow lines}\index{negative gradient flow lines} of $A_H$. The energy of a solution $u\in\mathcal{M}(p,q;[u],J,H)$ can be expressed through the action functional:
\[
E(u)=A_H(\gamma_-)-A_H(\gamma_+)\,,
\]
where $\gamma_\pm\in\text{crit}(A_H)$ are such that $\gamma_-(0)=q$ and $\gamma_+(0)=p$. Clearly this makes sense, meaning $\mathcal{M}(p,q;[u],J,H)$ can be nonempty, if and only if $A_H(\gamma_-) > A_H(\gamma_+)$ or $\gamma_-=\gamma_+$.

However, this alternative standpoint earns us no further advantage, and we rather keep our focus\footnote{Indeed, this was Andreas Floer's brilliant intuition when developing Floer homology: to regard the ill-posed ODE \eqref{illODE} as the perturbation of the PDE \eqref{CRequation}, whose solutions are the well-known $J$-holomorphic strips.} on the perturbed PDE \eqref{perturbedCR} on $M$. The reader interested in the underlying Floer theoretic machinery is once again referred to \cite{[Oh91]}. \\

\noindent\minibox[frame]{Henceforth, we implicitly choose a time-dependent almost complex structure \\ $J\in C^\infty([0,1],\mathfrak{J}(M,\omega))$ and Hamiltonian $H\in C^\infty([0,1]\times M)$, turning our \\ analysis to the perturbed Cauchy--Riemann problem \eqref{perturbedCR} for not necessarily \\ transversely-intersecting Lagrangian submanifolds.} 
\vspace*{0.2cm}

\subsection{Compactness}\label{ch5.6}

Before proceeding to the definition of Floer's differential, we must address the issue of compactness of the smooth manifold $\widehat{\mathcal{M}}(p,q;[u],J,H)$. This is perhaps the most delicate of the three requirements listed in Section \ref{ch5.3}, due to the following result (see for example \cite[Proposition 3.3]{[Sal97]}).

\begin{Pro}[\textbf{Gromov's Compactness Theorem}]\index{Gromov's Compactness Theorem}\label{Gromov}		
	Consider a sequence of $J$-holomorphic strips $(u_n)_{n\in\mathbb{N}}\subset\widehat{\mathcal{M}}(p,q;[u],J,H)$ with uniformly bounded energy, $\sup_{n\in\mathbb{N}} E(u_n)<\infty$. Then there exist finitely many points $(s_j,t_j)\equiv s_j+it_j=:z_j\in\mathbb{D}^2\setminus\{\pm 1\}$, say for $j=1,...,l$, and a subsequence $(u_{n_k})_{k\in\mathbb{N}}\subset \widehat{\mathcal{M}}(p,q;[u],J,H)$ which converges uniformly to some solution $u\in\widehat{\mathcal{M}}(p,q;[u],J,H)$ of \eqref{perturbedCR}, with derivatives on compact subsets of $(\mathbb{D}^2\setminus\{\pm1\})\setminus\{z_1,...,z_l\}$.
\end{Pro}

\begin{Def}											
	Consider the subsequence of $J$-holomorphic curves $(u_{n_k})_{k\in\mathbb{N}}$ $\subset \widehat{\mathcal{M}}(p,q;[u],J,H)$ as in previous proposition. Then the derivatives $(du_{n_k})_{k}$ at any among those $z_j=(s_j,t_j)\in\mathbb{R}\times[0,1]$ are unbounded, which causes an ``energy concentration'' for $k\rightarrow\infty$. We say that the limit strip $u$ exhibits a \textbf{bubble}\index{bubble}. There are three possible such ``bubbling'' phenomena (see Figure \ref{bubbling}):
	\begin{itemize}[leftmargin=0.5cm]
		\item \textbf{Strip breaking}\index{strip breaking}, which occurs when $\lim_{k\to\infty}|(du_{n_k})(z)|=\infty$ at $z=(s,t)\in\mathbb{R}\times[0,1]$ for $s\rightarrow\pm\infty$, so that the limit strip is ``broken'' (that is, has an intermediate node). Regarding each component as a $J$-holomorphic disk, this is equivalent to having a bubble at one of the endpoints $p,q$. Formally, this means there is a sequence $a_n\rightarrow\pm\infty$ such that the translated strips $u_n(\cdot-a_n,\cdot)$ converge to a non-constant limit strip. 
		\item \textbf{Disk bubbling}\index{bubbling!disk}, which occurs when $\lim_{k\to\infty}|(du_{n_k})(z)|=\infty$ at some $z=(s,t)\in\mathbb{R}\times[0,1]$ for $t\in\{0,1\}$, so that the limit strip has a bubble consisting of a $J$-holomorphic disk whose boundary intersects --- and after suitable rescalings is entirely contained in --- $L_0$ or $L_1$.
		\item \textbf{Sphere bubbling}\index{bubbling!sphere}, which occurs when $\lim_{k\to\infty}|(du_{n_k})(z)|=\infty$ at some $z=(s,t)\in(\mathbb{R}\times[0,1])^\circ$, so that the limit strip exhibits a bubble which --- after suitable rescalings --- is a $J$-holomorphic sphere $\mathbb{C}P^1\rightarrow M$ intersecting $\text{im}(u)$ transversely at some interior point of $\text{im}(u)$. 
	\end{itemize}
\end{Def}

\begin{figure}[htp]
	\centering
	\includegraphics[width=0.9\textwidth]{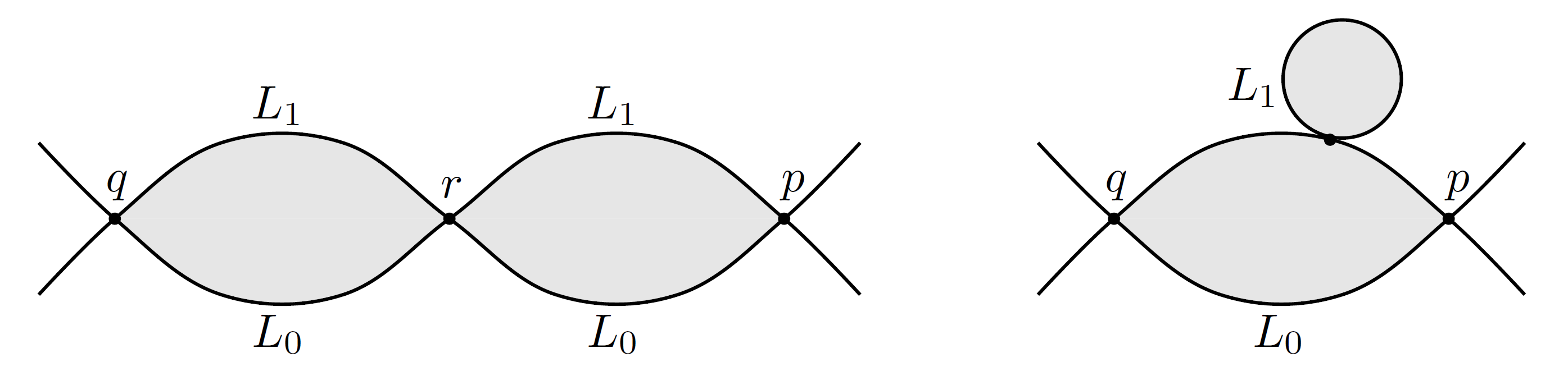}		  
	\caption{Two possible ``bubbling'' phenomena: strip breaking (left) and disk bubbling (right). [Source:\cite{[Aur13]}]}
	\label{bubbling}	
\end{figure}

Strip-breaking is the reason why Floer's differential squares to zero in absence of disk and sphere bubbles. The latters are undesirable since they also obstruct the transversality of $\widehat{\mathcal{M}}(p,q;[u],J,H)$; not even the perturbations $J\in C^\infty([0,1],\mathfrak{J}(M,\omega))$ and $H\in C^\infty([0,1]\times M)$ adopted in Section \ref{ch5.5} can solve the problem this time! Therefore, we wish to retain strip-breaking phenomena while discarding disk and sphere bubbling. This is achieved by simply requiring that the symplectic area of the associated $J$-holomorphic disks and spheres vanishes.

\begin{Lem}											
	Let $(M,\omega)$ be a symplectic manifold, $L_0, L_1\subset M$ Lagrangian submanifolds such that\footnote{Observe that $\langle\omega\rangle\in H_{dR}^2(M)\cong H^2(M;\mathbb{R}) \cong\textup{Hom}(H_2(M),\mathbb{R})\oplus\text{Ext}(H_1(M),\mathbb{R})=\textup{Hom}(H_2(M),\mathbb{R})$, hence, by identifying $\pi_2(M)$ with its image under the Hurewicz map $h:\pi_2(M)\rightarrow H_2(M)$, we can restrict $\langle\omega\rangle$ to $\pi_2(M)$. In other words, we impose the pairing $\langle\langle\omega\rangle, h(\cdot)\rangle\equiv \langle\omega\rangle\cdot\text{im}(h)$ between cohomology and homology to vanish, whence the notation adopted. The same holds when regarding $\langle\omega\rangle\in H_{dR}^2(M,L_i)$.} $\langle\omega\rangle\cdot\pi_2(M,L_i)=0$. Then $L_0$ and $L_1$ bound no $J$-holomorphic disk or sphere.
\end{Lem} 

This is just one of several possible conditions preventing the formation of disk and sphere bubbles. Later we will see that this comes for free when working with exact symplectic manifolds. \\

\noindent\minibox[frame]{Henceforth, we assume that all Lagrangian submanifolds $L\subset M$ satisfy the \; \\ ``no-bubbling'' condition $\langle\omega\rangle\cdot\pi_2(M,L)=0$.}
\vspace*{0.2cm}

\subsection{The Floer differential and Floer cohomology}\label{ch5.7}

We have finally collected all the necessary assumptions which make the ensuing theory well-posed. We proceed to define the second piece of data for Floer complexes.

\begin{Def}\label{Floerdifferential}			
	Let $(M,\omega)$ be a symplectic manifold, $L_0,L_1\subset M$ Lagrangian submanifolds with corresponding Floer complex $CF(L_0,L_1)\equiv CF(L_0,L_1;J,H)$ as in equation \eqref{Floercmplx} (for $\mathcal{X}(L_0,L_1)=L_0\cap(\phi_H^1)^{-1}(L_1)$). The \textbf{Floer differential}\index{Floer differential} $\partial\equiv \partial_{J,H}: CF(L_0,L_1)\rightarrow CF(L_0,L_1)$ is the $\Lambda_\mathbb{K}$-linear extension of
	\begin{equation}\label{Floersdiff}
		\partial(p) \coloneqq \mkern-18mu\sum_{\substack{q\in\mathcal{X}(L_0,L_1) \\ [u]:\text{ind}([u])=1}}\mkern-18mu (\#\mathcal{M}(p,q;[u],J,H))T^{\omega([u])}q\;,
	\end{equation}
	where $\#\mathcal{M}(p,q;[u],J,H)\in\mathbb{Z}$ (or $\mathbb{Z}_2$) is the (un)signed cardinality of the moduli space $\mathcal{M}(p,q;[u],J,H)$ coming from the perturbed Cauchy--Riemann problem \eqref{perturbedCR}, and $\omega([u])$ is the symplectic area (independent from the choice of representative, by the no-bubbling condition). Due to Corollary \ref{generatordegree}, $\partial$ is of degree 1. (Notice that the coefficient of $q$ in $\partial(p)$,
	\[
	\sum_{[u]:\text{ind}([u])=1}\mkern-18mu(\#\mathcal{M}(p,q;[u],J,H))T^{\omega([u])}\in\Lambda_\mathbb{K}\;,
	\]
	is an element of the Novikov field.)	
\end{Def}

To prove that $\partial$ is well defined, we need one last (non-trivial) result from manifold theory.

\begin{Lem}\label{boundarypoints}					
	Every compact $1$-dimensional smooth manifold $M$ has total signed number of boundary points equal zero.
\end{Lem}

\begin{proof}(\textit{Sketch})
	This is a consequence of the Classification Theorem for compact 1-dimensional smooth manifolds: any such must be diffeomorphic to a disjoint, finite union of closed real intervals and copies of $\mathbb{S}^1$. The formers have two boundary points each, the latters none. Since boundary points are preserved under diffeomorphism, it follows that $M$ has an even number of boundary points. 
	
	This readily proves the case for the unsigned count (when we work modulo 2). If we do care about orientations, thus taking into account signs, a careful analysis reveals that the signed cardinality of $\partial M$ still cancels out. 
\end{proof}

\begin{Thm}\label{HFwelldef}					
	Let $(M,\omega)$ be a symplectic manifold, $L_0,L_1\subset M$ Lagrangian submanifolds such that $\langle\omega\rangle \cdot\pi_2(M,L_i)=0$ for $i=0,1$. Then $\partial$ is well defined and satisfies $\partial^2=0$. Moreover, the induced cohomology group does not depend \textup(up to isomorphism\textup) on the choice of time-dependent $J$ and $H$.
\end{Thm}

\begin{proof}
	(\textit{Sketch}) The counting in \eqref{Floersdiff} happens for all homotopy classes $[u]\in\pi_2(M,L_0\cup L_1)$ with $\text{ind}([u])=\text{dim}(\widehat{\mathcal{M}}(p,q;[u],J,H))=1$, hence the corresponding moduli spaces $\mathcal{M}(p,q;[u],J,H)=\widehat{M}(p,q;[u],J,H)/\mathbb{R}$ (which are smooth manifolds by transversality) are compact, oriented and 0-dimensional, thus are finite sets which can be counted with signs. Moreover, the sum in \eqref{Floersdiff} is well defined because, as a consequence of Gromov's Compactness Theorem, the finite energy bound \eqref{perturbedCRenergy} is satisfied by only finitely many classes $[u]$ whose moduli space is nonempty. Recall also that $\mathcal{X}(L_0,L_1)$ is finite.
	
	Fix two generators $p,q\in\mathcal{X}(L_0,L_1)$ and some $[u]\in\pi_2(M,L_0\cup L_1)$ with Maslov index 2. Then $\text{dim}(\widehat{\mathcal{M}}(p,q;[u],J,H)) = 2$, hence the smooth manifold $\mathcal{M}(p,q;[u],J,H)=\widehat{\mathcal{M}}(p,q;[u],J,H)/\mathbb{R}$ is 1-dimensional. Again by Proposition \ref{Gromov}, it can be compactified to a space $\overline{\mathcal{M}}(p,q;[u],J,H)$ by just adding broken strips connecting $p$ to $q$ and representing $[u]$ (by assumption, there is no risk to incur into disk/sphere bubbling). Accordingly, $\partial\overline{\mathcal{M}}(p,q;[u],J,H)$ is contained in the disjoint union of products of moduli spaces having the form $\mathcal{M}(p,r;[u'],J,H)\times\mathcal{M}(r,q;[u''],J,H)$, where $r\in CF(L_0,L_1)$ is any generator and $[u']+[u'']=[u]\in\pi_2(M,L_0\cup L_1)$ (see Figure \ref{bubbling} left). Since the Maslov index is additive under strip-breaking (by construction) and any non-constant strip has Maslov index at least 1 (by transversality), $\text{ind}([u])=2$ implies that $\text{ind}([u'])=\text{ind}([u'']) = 1$. Conversely, a \textit{gluing theorem}\footnote{Refer to \cite[section 8]{[Mer14]} or \cite[section 3.3]{[Sal97]} for some aimed formulations, and think about their Morse theoretic analogues.} tells us that $\overline{\mathcal{M}}(p,q;[u],J,H)$ is actually a 1-dimensional manifold with boundary
	\begin{equation}
	\partial\overline{\mathcal{M}}(p,q;[u],J,H) = \mkern-36mu\bigsqcup_{\begin{footnotesize}\substack{r\in\mathcal{X}(L_0,L_1) \\ [u']+[u'']=[u] \\ \text{ind}([u'])=\text{ind}([u''])=1}\end{footnotesize}}\mkern-36mu \mathcal{M}(p,r;[u'],J,H)\times\mathcal{M}(r,q;[u''],J,H)\;,
	\end{equation}
	whose orientation is compatible with the natural ones of its components (coming from the implicit choice of orientations and spin structures for $L_0$ and $L_1$). By Lemma \ref{boundarypoints}, the boundary of the compact 1-dimensional smooth manifold $\overline{\mathcal{M}}(p,q;[u],J,H)$ has zero signed cardinality:
	\[
	\mkern-36mu\sum_{\begin{footnotesize}\substack{r\in\mathcal{X}(L_0,L_1) \\ [u']+[u'']=[u] \\ \text{ind}([u'])=\text{ind}([u''])=1}\end{footnotesize}}\mkern-36mu (\#\mathcal{M}(p,r;[u'],J,H))\cdot(\#\mathcal{M}(r,q;[u''],J,H)) = 0\;.
	\]
	Multiplying last expression by the formal variable $T$ of $\Lambda_\mathbb{K}$ weighted by the total strip energy, $T^{\omega([u])}=T^{\omega([u'])+\omega([u''])}$, and summing over all classes $[u]\in\pi_2(M,L_0\cup L_1)$, we obtain precisely $\partial^2(p)$. Therefore, $\partial^2(p)=0$.
	
	We must still prove that the cohomology groups induced by $\partial=\partial_{J,H}$ do not depend on the choice of perturbation data $J$ and $H$. We just give an overview.\footnote{More details can be found in \cite[chapter 9]{[Mer14]}, for the non-Lagrangian case.} Take another pair $(J',H')$. Then, since $C^\infty([0,1],\mathfrak{J}(M,\omega))$ and $C^\infty([0,1]\times M)$ are both path-connected, there is an \textit{asymptotically constant path} $\chi:\tau\in\mathbb{R}\mapsto(J(\tau),H(\tau))\in C^\infty([0,1],\mathfrak{J}(M,\omega))\times C^\infty([0,1]\times M)$ such that $\chi(-\tau)=(J,H)$ and $\chi(\tau)=(J',H')$ for all $\tau\geq N$, given $N\geq 0$ big enough. One proceeds to construct a \textit{continuation map} $\Phi_\chi:CF(L_0,L_1;J,H)\rightarrow CF(L_0,L_1;J',H')$ by counting solutions to
	\[
	\partial_su + J_t(\tau,u)(\partial_tu -X_{H_t}(\tau,u)) = 0
	\]
	as follows: the coefficient of $p'\in\mathcal{X}(L_0,L_1;H')$ in $F(p)$ is the signed count of index 0 solutions converging to $p$ respectively $p'$ for $s\rightarrow \pm\infty$. Symmetrically, one can define a continuation map $\Phi_\chi':CF(L_0,L_1;J',H')\rightarrow CF(L_0,L_1;J,H)$. Then it can be shown that $\partial_{J',H'}\circ\Phi_\chi = \Phi_\chi\circ\partial_{J,H}$ and $\partial_{J,H}\circ\Phi_\chi' = \Phi_\chi'\circ\partial_{J',H'}$, meaning $\Phi_\chi$ and $\Phi_\chi'$ are inverse chain maps, so that $\Phi_\chi'\circ\Phi_\chi$ is chain homotopic to $\textup{id}_{CF(L_0,L_1;J,H)}$ respectively $\Phi_\chi\circ\Phi_\chi'$ to $\textup{id}_{CF(L_0,L_1;J',H')}$. Passing to cohomology, it follows that $H(\Phi_\chi):H(CF(L_0,L_1;J,H))\rightarrow H(CF(L_0,L_1;J',H'))$ is an isomorphism, as desired.
\end{proof}

\begin{Def}										
	Let $(M,\omega)$ be a symplectic manifold, $L_0,L_1\subset M$ Lagrangian submanifolds such that $\langle\omega\rangle \cdot\pi_2(M,L_i)=0$ for $i=0,1$. Then we call
	\begin{equation}\label{Floercohomring}
		HF(L_0,L_1)\coloneqq H(CF(L_0,L_1),\partial) = \frac{\text{ker}(\partial)}{\text{im}(\partial)}
	\end{equation}
	\textbf{Lagrangian Floer cohomology}\index{Lagrangian Floer cohomology} (by Theorem \ref{HFwelldef}, we can unambiguously choose any $J\in C^\infty([0,1],\mathfrak{J}(M,\omega))$ and $H\in C^\infty([0,1]\times M)$ to define $CF(L_0,L_1)$ and $\partial$). Explicitly, the $n$-th group reads $HF^n(L_0,L_1)=\text{ker}(\partial_n)/\text{im}(\partial_{n-1})$. 
\end{Def}

From the constructions of Section \ref{ch5.5} and the proof of Theorem \ref{HFwelldef}, it readily follows that Hamiltonian isotopic Lagrangian submanifolds $L_1, L_1'\subset M$ induce an isomorphism on Floer cohomology: $HF(L_0,L_1)\cong HF(L_0,L_1')$. On the other hand, if $L_0\cap L_1=\emptyset$, then $HF(L_0,L_1)=\{0\}$.

\begin{Rem}											
	To obtain Floer \textit{homology}, we must look at all previous constructions ``backwards''. For example, the strip of Figure \ref{holomorphicstrip2} goes from $q$ to $p$, ultimately causing a reversal of all gradings.
	
	Returning to cohomology, with a suitable choice of data $(J,H)$ one can show that the Floer complexes $CF(L_0,L_1)$ and $CF(L_1,L_0)$ along with their differentials are dual to one another (the strip from $q$ to $p$ is then an element of the latter). 
\end{Rem}

\begin{Ex}										
	It is instructive to see what goes wrong when we drop the assumptions $\langle\omega\rangle\cdot\pi_2(M,L_i)=0$, that is, when we allow Lagrangian submanifolds to bound $J$-holomorphic disks and spheres. Consider the 2-dimensional manifold $M\coloneqq\mathbb{R}\times\mathbb{S}^1$ with standard area form $\omega\coloneqq ds\wedge dt$, and Lagrangian submanifolds $L_0, L_1$ as in Figure \ref{cylinder} left. Then $L_0$ and $L_1$ bound two $J$-holomorphic strips $u, v$ of Maslov index 1, which go from $p$ to $q$ respectively from $q$ to $p$. By definition, $\partial(p)=\pm T^{\omega([u])}q$ and $\partial(q)=\pm T^{\omega([v])}p$, whence $\partial^2(p)=\pm T^{\omega([u])+\omega([v])}p\neq 0$.
	
	\begin{figure}[htp]
		\centering
		\includegraphics[width=0.5\textwidth]{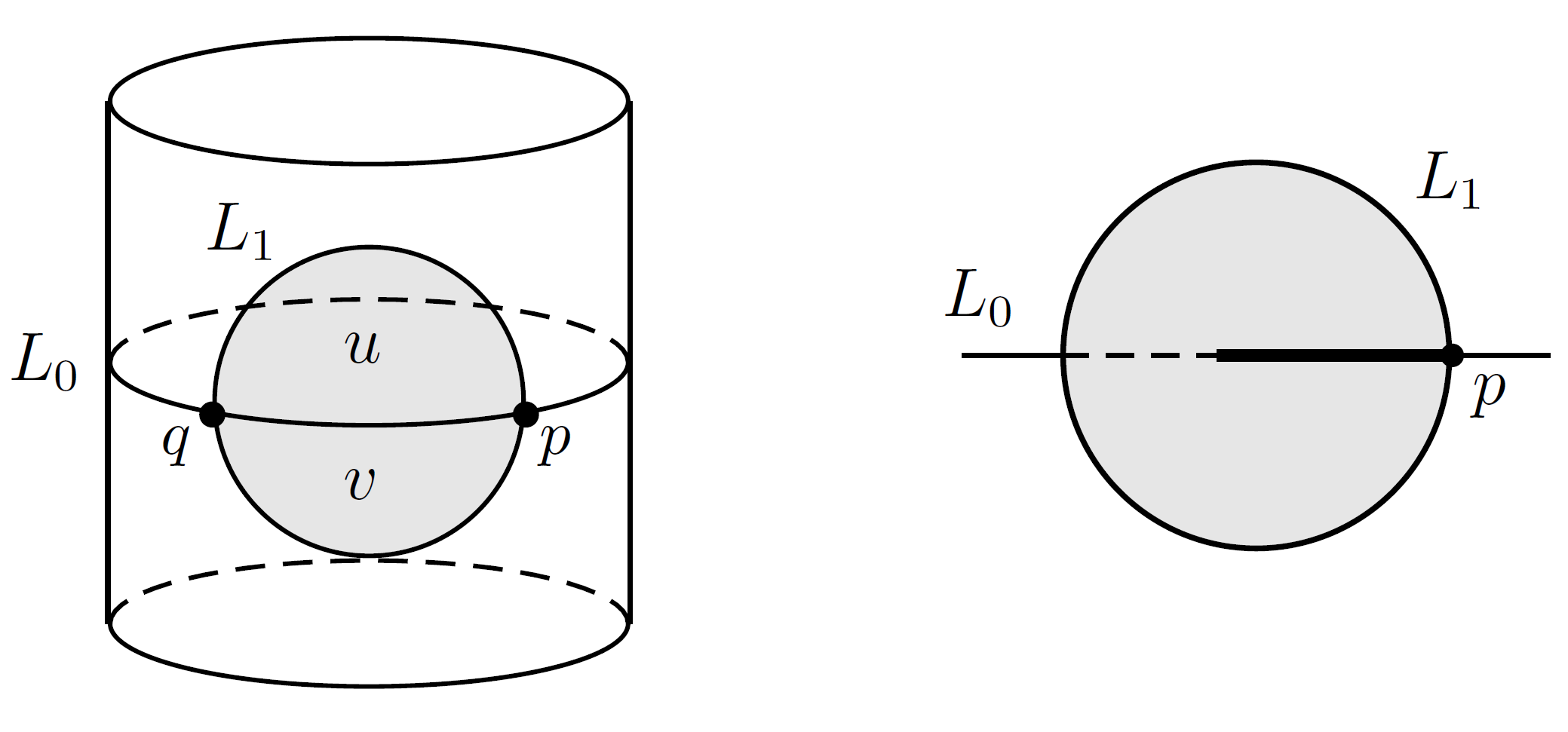}		  
		\caption{If Lagrangian submanifolds bound $J$-holomorphic disks, then Floer's differential does not square to 0. [Source:\cite{[Aur13]}]}
		\label{cylinder}	
	\end{figure}

	What happens algebraically? Look at the local picture of $M$ in Figure \ref{cylinder} right, after identifying $\mathbb{R}\times[0,1]\simeq\mathbb{D}^2\setminus\{\pm 1\}\subset\mathbb{C}$ so that $L_0\supset[-1,1]$ lies on $\mathbb{R}$ and $L_1$ coincides with $\mathbb{S}^1$. Focus on the index 2 moduli space of $J$-holomorphic strips from $p$ to itself whose image arises by removing a slit along $L_0$ (bold in figure). Any such strip is given by
	\[
	w_\alpha(z)\coloneqq\frac{z^2+\alpha}{1+\alpha z^2}\,,\quad\text{with }dw_\alpha(z)=2z\frac{1-\alpha^2}{(1+\alpha z^2)^2}\,,
	\]
	for $\alpha\in(-1,1)$ regulating the slit length. There are two limit cases:
	\begin{itemize}[leftmargin=0.5cm]
		\item If $\alpha\rightarrow -1$, the index two strips $(w_\alpha)_\alpha$ converge to a broken strip $w$ with index 1 components $u$ and $v$, just like in Figure \ref{bubbling} left (indeed, we can ``open up'' the disk of Figure \ref{cylinder} along the slit and diffeomorphically map it to the broken configuration). This means that energy concentrates at $z=\pm 1$, as confirmed by:
		\[
		\lim_{\alpha\to -1}|dw_\alpha(\pm 1)| = \lim_{\alpha\to -1}\Big|\frac{-4\alpha}{2(1+\alpha)}\Big| = \infty\;.  
		\]
		\item If $\alpha\rightarrow 1$, $(w_\alpha)_\alpha$ converges to a constant strip at $p$ with a disc bubble at $z=i$ mapping to the disk enclosed by $L_1$. Hence, energy concentrates at $i$:
		\[
		\lim_{\alpha\to 1}|dw_\alpha(i)| = \lim_{\alpha\to 1}\Big|\frac{4\alpha}{2(1-\alpha)}\Big| = \infty\;.
		\]
	\end{itemize} 
	The disk bubbling occurring for $\alpha\rightarrow 1$ causes $\partial^2(p)\neq 0$! \hfill $\blacklozenge$
\end{Ex}

Finally, we describe Floer cohomology in the case $L_0=L_1$. Again, the proof of the following statement is rather advanced for our level --- it involves Morse theory --- hence we just give a sketch and point the reader to \cite[Example 1.12]{[Aur13]}. 

\begin{Pro}\label{CFLL}								
	Let $(M^{2n},\omega)$ be a symplectic manifold, $L^n\subset M$ a Lagrangian submanifold such that $\langle\omega\rangle\cdot\pi_2(M,L)=0$. Then $HF(L,L)\cong H(L;\Lambda_\mathbb{K})$ \textup(singular cohomology with coefficients in the Novikov field\textup).
\end{Pro}

\begin{proof}
	(\textit{Sketch}) First, let $L$ be any compact real $n$-dimensional manifold. It is known\footnote{By \textit{Weinstein's Lagrangian Neighbourhood Theorem}; see \cite[Theorem 8.4]{[Sil12]}.} that a neighbourhood of $L\subset M$ is symplectomorphic to one of $L_0\coloneqq \mathfrak{o}(L)\subset T^*L$, where $\mathfrak{o}\in\Gamma(L,T^*L)$ denotes the zero section and $T^*L=$ $(T^*L,\theta=\sum_i dq_i\wedge dp_i)$ the cotangent bundle of $L$. Taking a Morse function $f\in C^\infty(L)$ and some $\varepsilon>0$, let $L_1\coloneqq \text{graph}(\varepsilon df)=\{(x,\varepsilon(df)_x)\in T_x^*L\mid x\in L\}$. Then $L_0,L_1\subset T^*L$ are Hamiltonian isotopic exact Lagrangian submanifolds, intersecting transversely at $\text{crit}(f)$, the set of critical points of $f$. 
	
	Choosing graded lifts, we obtain a graded Floer complex $CF(L_0,L_1)$ with generators any $p\in\text{crit}(f)$, of $\text{deg}(p)=n-i(p)$ (for $i(p)$ the Morse index of $p$). After suitable perturbations and rescalings, one can show that $HF(L_0,L_1)\cong HM(f;\Lambda_\mathbb{K})$, the Morse cohomology of $L$, in turn isomorphic to the ordinary singular cohomology $H(L;\Lambda_\mathbb{K})$. Together with invariance under Hamiltonian isotopy, we get $HF(L,L)\cong HF(L_0,L_0)\cong HF(L_0,L_1)\cong H(L;\Lambda_\mathbb{K})$.
	
	Now, let $L^n\subset M$ be a Lagrangian submanifold such that $\langle\omega\rangle\cdot\pi_2(M,L)=0$. For a small enough perturbing Hamiltonian, any $J$-holomorphic strip determining $HF(L,L)$ must be contained in a small tubular neighbourhood of $L$ (due to energy constraints). Therefore, the analysis carried above also applies to this general setting.
\end{proof}

Once again, the no-bubbling assumption $\langle\omega\rangle\cdot\pi_2(M,L)=0$ is essential: otherwise, $HF(L,L)$ might be --- whenever defined --- a smaller subgroup of $H(L;\Lambda_\mathbb{K})$.

\subsection{The exact manifold setting}\label{ch5.8}

Before moving on, let us briefly revisit the theory of last section in presence of exact Lagrangian submanifolds $L$ in an exact symplectic manifold $M$. As before, we assume that $2c_1(TM)=0$ and all $L$ are compact, oriented, spin, graded, with vanishing Maslov class.

\begin{Lem}\label{nobubbles}				
	Let $(M^{2n},\omega)$ be an exact symplectic manifold, $L_0, L_1\subset M$ exact Lagrangian submanifolds. Then the energy of any $J$-holomorphic strip defining $CF(L_0,L_1)$ is constant. Moreover, disk/sphere bubbling phenomena do not occur \textup(even though we do not require $\langle\omega\rangle\cdot\pi_2(M,L_i)=0$\textup!\textup).
\end{Lem}

\begin{proof}
	Let $\omega = d\nu\in\Omega^2(M)$ be the exact form attached to $M$, for some $\nu\in\Omega^1(M)$, and let $f_i\in C^\infty(L_i)$ be such that $\nu|_{L_i}=df_i\in\Omega^1(L_i)$ for $i=0,1$ (cf. Definition \ref{Lagrangian}). Consider a $J$-holomorphic strip $u:\mathbb{R}\times[0,1]\rightarrow M$ from $p$ to $q$ with boundary in $L_0\cup L_1$, as in Figure \ref{holomorphicstrip2}. Its energy is:
	\begingroup
	\allowdisplaybreaks
	\begin{align}\label{constenergy}
		E(u)&=\int_{\mathbb{R}\times[0,1]}\mkern-24mu u^*\omega=\int_{\mathbb{R}\times[0,1]}\mkern-24mu d(u^*\nu) & \text{(since $d$ commutes with pullbacks)} \nonumber \\
		&=\int_{\partial(\mathbb{R}\times[0,1])}\mkern-30mu u^*\nu & \text{(by Stokes' Theorem)} \nonumber \\
		&=\int_{\mathbb{R}\times\{1\}}\mkern-24mu u^*(\nu|_{L_1}) \;-\! \int_{\mathbb{R}\times\{0\}}\mkern-24mu u^*(\nu|_{L_0}) & \text{(because $\mathbb{R}\times\{i\}$ maps to $L_i$)} \nonumber \\
		&=\int_{\mathbb{R}\times\{1\}}\mkern-24mu d(u^*f_1) \;-\! \int_{\mathbb{R}\times\{0\}}\mkern-24mu d(u^*f_0) & \text{(since $d$ commutes with pullbacks)} \nonumber \\
		&=\int_{\partial(\mathbb{R}\times\{1\})}\mkern-30mu f_1\circ u \;-\! \int_{\partial(\mathbb{R}\times\{0\})}\mkern-30mu f_0\circ u & \text{(by Stokes' Thm and def. of $u^*$)} \nonumber \\
		&=f_1(u(\,\cdot\,,1))\Big|_{+\infty}^{-\infty} - f_0(u(\,\cdot\,,0))\Big|_{+\infty}^{-\infty} & \text{(because we integrate $p\rightarrow q$)} \nonumber \\
		&=(f_1(q)-f_1(p))-(f_0(q)-f_0(p))\,. 
	\end{align}
	\endgroup
	Therefore, the energy of $u$ is constant, as claimed.
	
	Consider now possible bubbling phenomena. If $v:\mathbb{D}^2\rightarrow M$ is a $J$-holomorphic disk whose image intersects $u$ at some $r\in L_1$ (see Figure \ref{bubbling} right), then its energy is:
	\[
	E(v)=\int_{\mathbb{D}^2}v^*\omega = \int_{\mathbb{S}^1}v^*\nu = \int_{\partial\mathbb{S}^1}v^*f_1 = 0
	\]
	(using commutativity of $d$ with $v^*$ and Stokes' Theorem twice, and $\partial\mathbb{S}^1=\emptyset$). Similarly, let $w:\mathbb{C}P^1\simeq\mathbb{S}^2\rightarrow M$ be a $J$-holomorphic sphere intersecting $u$ at some interior point. The associated energy is:
	\[
	E(w)=\int_{\mathbb{S}^2} w^*\omega = \int_{\partial\mathbb{S}^2} w^*\nu = 0
	\]
	(since $\partial\mathbb{S}^2=\emptyset$). It follows that, in both cases, there are no singular energy concentrations at any region of the strip. Accordingly, disk and sphere bubbling is excluded.
\end{proof}

Therefore, if we work with exact Lagrangian submanifolds, the compactness property required to make Floer's differential well-defined is automatically achieved. However, if $M$ is noncompact, then one needs an additional assumption on it (specifically, $M$ must be either convex at infinity or geometrically bounded; consult for example \cite{[CGK02]}). Also, in light of equation \eqref{constenergy} --- which tells us that all symplectic areas are equal --- we can rescale generators $p\in\mathcal{X}(L_0,L_1)$ by a factor $T^{f_1(p)-f_0(p)}$, so that the weights $T^{\omega([u])}$ in \eqref{Floersdiff} can be ignored, allowing us to directly work over the field $\mathbb{K}$. 

Most importantly, Lemma \ref{nobubbles} tells us that exact Lagrangian submanifolds cannot be obstructed, so that each one of them is a (non-zero) object of the Fukaya category. We summarize this scenario.

\begin{Def}											
	Let $(M,\omega)$ be an exact symplectic manifold equipped with some $J\in C^\infty([0,1],\mathfrak{J}(M,\omega))$ and $H\in C^\infty([0,1]\times M)$. Let $L_0, L_1\subset M$ be exact Lagrangian submanifolds (endowed with all data listed above). The Floer complex of $(L_0,L_1)$ is the free module
	\[
	CF(L_0,L_1)\equiv CF(L_0,L_1;H)\coloneqq \mkern-18mu\bigoplus_{\begin{footnotesize}p\in L_0\cap(\phi_H^1)^{\text{-1}}(L_1)\end{footnotesize}}\mkern-18mu\mathbb{K}\cdot p
	\]
	together with the boundary operator $\partial\equiv \partial_{J,H}:CF(L_0,L_1)\rightarrow CF(L_0,L_1)$ defined as the $\mathbb{K}$-linear extension of
	\[
	\partial(p)\coloneqq \mkern-18mu\sum_{\begin{footnotesize}\substack{q\in L_0\cap(\phi_H^1)^{\text{-1}}(L_1) \\ [u]:\text{ind}([u])=1}\end{footnotesize}}\mkern-18mu (\#\mathcal{M}(p,q;[u],J,H))q
	\]
	\big(in case one wishes to ignore orientations, it suffices to substitute $\mathbb{K}\leftarrow\mathbb{Z}_2$ and $\#\mathcal{M}(p,q;[u],J,H)\leftarrow |\mathcal{M}(p,q;[u],J,H)|\,(\text{mod}\,2)\,$\big). We obtain the Lagrangian Floer cohomology $HF(L_0,L_1)\coloneqq H(CF(L_0,L_1;H),\partial_{J,H})$, which by Theorem \ref{HFwelldef} we know to be $(J,H)$-independent.
\end{Def}

\newpage

\section{The Fukaya category} 
\thispagestyle{plain}

\subsection{The Floer product}\label{ch6.1}

We bring into play a third Lagrangian submanifold and adapt the geometrical constructions of the previous chapter so to define a product between two Floer complexes, the \textit{Floer product}. Our primary reference is \cite[section 2]{[Aur13]}.

Consider as usual an almost complex symplectic manifold $M^{2n}=(M,\omega,J)$ with $2c_1(TM)=0$ and $J\in\mathfrak{J}(M,\omega)$, and three Lagrangian submanifolds $L_0,L_1,L_2\subset M$, which we assume to be compact, oriented, spin, graded, with vanishing Maslov classes and enclosing no bubbles. Furthermore, we take them to be pairwise transversely intersecting, $L_i\pitchfork L_j$ for $i\neq j \in\{0,1,2\}$. Then transversality, compactness and orientability are readily achieved (the general case will be discussed afterwards). We seek a product operation between the associated Floer complexes of the form:
\[
CF(L_1,L_2)\otimes CF(L_0,L_1)\rightarrow CF(L_0,L_2)\;.
\]

With reference to Figure \ref{holomorphicstrip3}, we study the following problem: given $p_1\in\mathcal{X}(L_0,L_1)$, $p_2\in\mathcal{X}(L_1,L_2)$, $q\in\mathcal{X}(L_0,L_2)$ (where $\mathcal{X}(L_i,L_j)=L_i\cap L_j$ for now), and points $z_1=e^{-i\pi/3}$, $z_2=e^{i\pi/3}$, $z_0=-1\in\mathbb{S}^1$, we want $J$-holomorphic maps $u:\mathbb{D}^2\setminus\{z_0,z_1,z_2\}\rightarrow M$ with boundary conditions
\begin{equation}\label{CRboundarycond3}
	\left\{\begin{array}{llll}
		u(e^{i\alpha})\in L_0 & \text{if }\alpha\in\big(-\pi,-\frac{\pi}{3}\big) & \text{and} & \displaystyle\lim_{z\to z_0}u(z)=q\\
		u(e^{i\alpha})\in L_1 & \text{if }\alpha\in\big(-\frac{\pi}{3},\frac{\pi}{3}\big) & \text{and} & \displaystyle\lim_{z\to z_1}u(z)=p_1\\
		u(e^{i\alpha})\in L_2 & \text{if }\alpha\in\big(\frac{\pi}{3},\pi\big)& \text{and} & \displaystyle\lim_{z\to z_2}u(z)=p_2\\ 
	\end{array}\right.\;,
\end{equation}
and finite energy $E(u)<\infty$. These are solutions of a suitably formulated Cauchy--Riemann problem (see equation \eqref{generalCR} below), and the so obtained homotopy classes $[u]\in\pi_2(M,L_0\cup L_1\cup L_2)$ define Moduli spaces $\mathcal{M}(p_1,p_2,q;[u],J)$ which --- thanks to our assumptions regarding compactness and transversality --- is a smooth manifold of finite dimension. Similarly to Floer's differential, the desired product operation involves a weighted count of all such $J$-holomorphic maps, which we may also call \textbf{pseudoholomorphic triangles}\index{pseudoholomorphic triangle}, as Figure \ref{holomorphicstrip3} suggests. 
  
\begin{figure}[htp]
	\centering
	\includegraphics[width=0.9\textwidth]{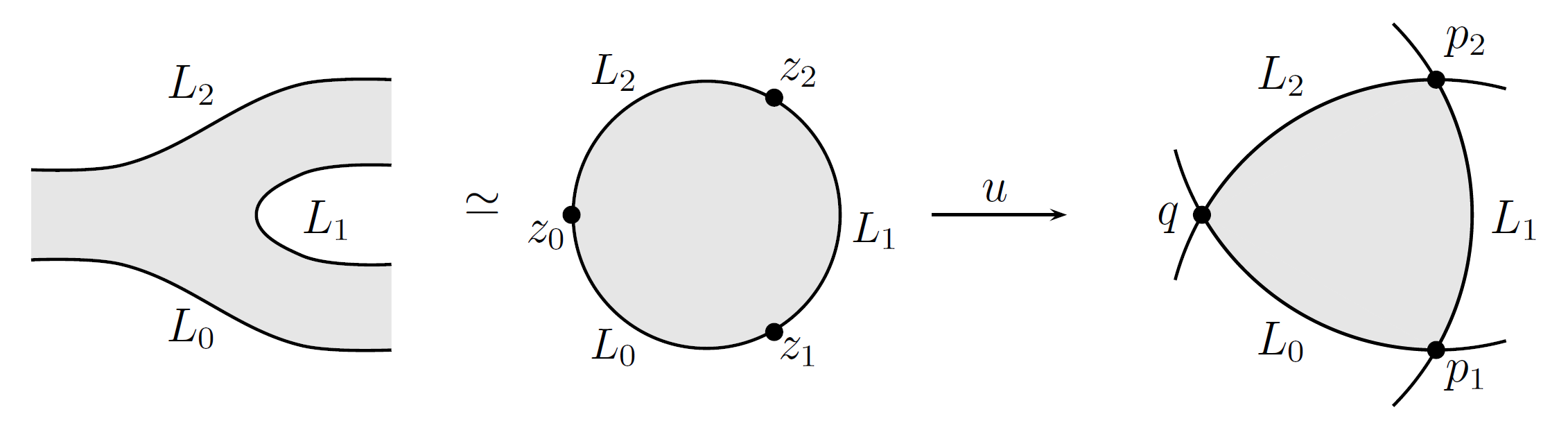}	
	\caption{A three-punctured $J$-holomorphic disk $u$ belonging to $\mathcal{M}(p_1,p_2,q;[u],J)$. [Source:\cite{[Aur13]}]}
	\label{holomorphicstrip3}	
\end{figure}

\begin{Rem}\label{striplike}					
	The similarities with \eqref{CRequation}--\eqref{CRenergy} are apparent. Indeed, we observe that any small enough open neighbourhood of each ``puncture'' $z_i$ in $\mathbb{D}^2$ is biholomorphic to $\mathbb{R}\times[0,1]\subset \mathbb{R}^2\cong\mathbb{C}$, as highlighted by Figure \ref{holomorphicstrip3}, which exhibits a biholomorphism between the three-punctured disk $\mathbb{D}^2\setminus\{z_0,z_1,z_2\}$ and the ``three-ends strip'' on the left-hand side. 
	
	On the other hand, Figure \ref{holomorphicstrip2} already alluded to the identification of $\mathbb{D}^2\setminus\{\pm 1\}$ with $\mathbb{R}\times[0,1]$. Thus one could --- and should, as we will shortly see --- interpret the problem \eqref{CRboundarycond3} as a ``degree 3 analogue'' of what we discussed throughout Chapter 5. 
\end{Rem}

Notice that, in contrast to Definition \ref{modulispace}, we did not require to quotient $\mathcal{M}(p_1,p_2,q;[u],J)$ by $\mathbb{R}$, as there is no global free action by translation on the space of solutions. Actually, since the automorphism group $\text{Aut}(\mathbb{D}^2)$ acts transitively on cyclically ordered triples of marked boundary points, even the choice of $z_j$'s is irrelevant. Then the dimension of $\mathcal{M}(p_1,p_2,q;[u],J)$ directly coincides with the Maslov index $\text{ind}([u])$:    

\begin{Lem}\label{moduli3dim}							
	Let $M^{2n}=(M,\omega,J)$ be an almost complex symplectic manifold with $2c_1(TM)=0$, and $L_0,L_1,L_2\subset M$ Lagrangian submanifolds as above, in particular graded with $[\phi_{L_i}]=0$. Then it holds:
	\[
	\textup{dim}(\mathcal{M}(p_1,p_2,q;[u],J))=\textup{ind}([u])=\textup{deg}(q)-\textup{deg}(p_1)-\textup{deg}(p_2)\,.
	\]
	\textup(Again, this is equal to the Fredholm index defined through $D_{\bar{\partial}_J,u}$.\textup) 
\end{Lem}

\begin{proof}(\textit{Sketch})
	One needs only adapt the definition of Maslov index given in Section \ref{ch5.4}. Renaming $q\coloneqq p_0$ provisionally, we move counterclockwise along $\partial(\text{im}(u))$: after suitably trivializing $u^*TM$, consider paths $l_i$ from $T_{p_j}L_j$ to $T_{p_k}L_j$ as described in Definition \ref{Maslovforstrips}, as well as canonical short paths $\lambda_{p_i}$ from $T_{p_i}L_k$ to $T_{p_i}L_i$, where $(i,j,k)$ runs over cyclic permutations of $(0,1,2)$; then we study the loop $\gamma_u\coloneqq l_2\ast\lambda_{p_1}\ast l_0\ast\lambda_{p_2}\ast l_1\ast\lambda_{p_0}$ in $LGr(n)$ based at $T_{p_0}L_0$. Definition \ref{degree} about the degree of a generator is still valid in this context (and makes sense thanks to our assumptions). Then the claimed equality comes from additivity under path concatenation. 
\end{proof}

The proof of Lemma \ref{moduli3dim} additionally shows that $CF(L_1,L_2)\otimes CF(L_0,L_1)$ is equipped with the natural $\mathbb{Z}$-grading induced by the two individual Floer complexes (hence, in particular, $\text{deg}(p_2\otimes p_1)=\text{deg}(p_2)+\text{deg}(p_1)$). 

\begin{Def}										
	Let $M=(M,\omega,J)$ be an almost complex symplectic manifold with $2c_1(TM)=0$, and $L_0,L_1,L_2\subset M$ Lagrangian submanifolds with the usual decorations, in particular such that $\langle\omega\rangle\cdot\pi_2(M,L_i)=0$. The \textbf{Floer product}\index{Floer product} $\boldsymbol{\cdot}:\,CF(L_1,L_2)\otimes CF(L_0,L_1)\rightarrow CF(L_0,L_2)$ is the $\Lambda_\mathbb{K}$-linear extension of
	\begin{equation}\label{Floerprod}
		p_2\otimes p_1\longmapsto p_2\boldsymbol{\cdot} p_1 \coloneqq \mkern-18mu\sum_{\substack{q\in\mathcal{X}(L_0,L_2) \\ [u]:\text{ind}([u])=0}}\mkern-18mu (\#\mathcal{M}(p_1,p_2,q;[u],J))T^{\omega([u])}q\;,
	\end{equation}
	where $\mathcal{M}(p_1,p_2,q;[u],J)$ is the moduli space of solutions to the Cauchy--Riemann problem \eqref{CRboundarycond3} representing $[u]\in\pi_2(M,L_0\cup L_1\cup L_2)$. The sum in \eqref{Floerprod} is well defined (argue as in the proof of Theorem \ref{HFwelldef}), and Lemma \ref{moduli3dim} tells us that $\boldsymbol{\cdot}$ is of degree 0. 
\end{Def}

Now we investigate the key feature of the Floer product in relation to Floer's differential: it fulfills a graded Leibniz rule. We use the geometric reasoning behind this to fine-tune the choice of perturbing, time-dependent almost complex structures and Hamiltonians.

\begin{Pro}\label{Leibniz}							
	The Floer product satisfies the graded Leibniz rule
	\begin{equation}\label{Leibnizrule}
	\partial(p_2\boldsymbol{\cdot} p_1)=(-1)^{|p_1|}\partial(p_2)\boldsymbol{\cdot} p_1-p_2\boldsymbol{\cdot}\partial(p_1)\,,
	\end{equation}
	where $|p_1|\coloneqq \textup{deg}(p_1)$ and each $\partial$ is a suitable Floer differential as given in Definition \ref{Floerdifferential}. Therefore, the product $HF(L_1,L_2)\otimes HF(L_0,L_1)\rightarrow HF(L_0,L_2)$ induced on cohomology is well defined, and actually independent from the choice of perturbation data involved \textup(cf. the proof\textup).
\end{Pro}

\begin{proof}(\textit{Sketch})
	We explain qualitatively the geometric idea involved, keeping an eye on what we did in the proof of Theorem \ref{HFwelldef}. Consider generators $p_1\in\mathcal{X}(L_0,L_1)$, $p_2\in\mathcal{X}(L_1,L_2)$, $q\in\mathcal{X}(L_0,L_2)$ as in Figure \ref{holomorphicstrip3}, and a class $[u]\in\pi_2(M,L_0\cup L_1\cup L_2)$ with $\text{ind}([u])=1$. Assuming transversality holds, $\mathcal{M}(p_1,p_2,q;[u],J)$ is a smooth 1-dimensional manifold which, by Gromov's Compactness Theorem \ref{Gromov}, can be compactified to some $\overline{\mathcal{M}}(p_1,p_2,q;[u],J)$, by adding broken strips (disk/sphere bubbling is excluded, by assumption).
	
	But now we have three ``strip-like'' ends where energy can concentrate. By additivity of the Maslov index under strip-breaking, its non-negativity for disks (by transversality), and because non-constant strips have index at least 1, we have only three possible partitions $\text{ind}([u])=\text{ind}([u'])+\text{ind}([u''])$: up to relabeling, $u'$ must represent an index 1 strip and $u''$ an index 0 disk with boundary on $L_0,L_1,L_2$.\footnote{This is different from the disk bubbles we are excluding by assumption --- the latters have boundary on a \textit{single} submanifold, and are the only disks causing energy blow-ups.} Bluntly said, we can just ``pinch'' the pseudoholomorphic triangle $u$ at one corner; see Figure \ref{3bubbling}. Specifically, the pictured configurations contribute to the coefficient of $T^{\omega([u])}q$ in $\partial(p_2\boldsymbol{\cdot} p_1)$, $p_2\boldsymbol{\cdot}\partial(p_1)$ and $\partial(p_2)\boldsymbol{\cdot} p_1$ respectively.
	\begin{figure}[htp]
		\centering
		\includegraphics[width=0.9\textwidth]{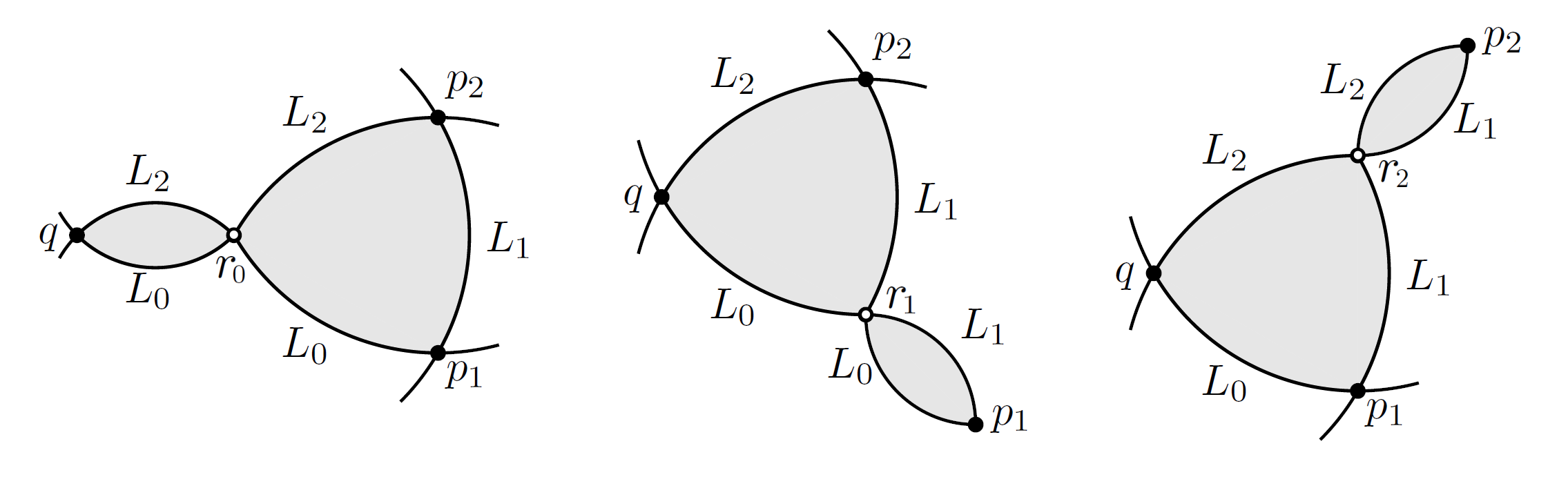}
		\caption{The three possible limit configurations of a three-punctured $J$-holomorphic disk in $\mathcal{M}(p_1,p_2,q;[u],J)$. [Source:\cite{[Aur13]}]}
		\label{3bubbling}	
	\end{figure}

	This means that $\partial\overline{\mathcal{M}}(p_1,p_2,q;[u],J)$ is contained in the disjoint union of products of moduli spaces like $\mathcal{M}(p_i,r_i;[u'],J)\times\mathcal{M}(r_i,p_j,p_k;[u''],J)$, where $r_j\in\mathcal{X}(L_i,L_j)$ are the ``pinched'' points (and we set again $q\coloneqq p_0$). The converse inclusion is provided by another \textit{gluing theorem}, so that $\overline{\mathcal{M}}(p_1,p_2,q;[u],J)$ is indeed a compact 1-dimensional smooth manifold with (coherently oriented) boundary as just described. Lemma \ref{boundarypoints} tells us that the latter's (un)signed cardinality is zero, thus summing over all homotopy classes --- and carefully accounting degrees --- we obtain the Leibniz rule \eqref{Leibnizrule}.
	
	The fact that \eqref{Floerprod} induces a well-defined product on cohomology is now an easy argument: if $\partial(p_1)=\partial(p_2)=0$, then $\partial(p_2\boldsymbol{\cdot} p_1)=0$ by \eqref{Leibnizrule}; if moreover $p_1=\partial(p_1')$, then $p_2\boldsymbol{\cdot} p_1=\partial(p_2\boldsymbol{\cdot} p_1')$. Thus cocycles and coboundaries are preserved, as desired.
	
	Finally, we discuss the general case when $L_i\,\cancel{\pitchfork}\,L_j$. To achieve transversality, we reintroduce time-dependent almost complex structures and Hamiltonians, specifically a pair $(J_{ij},H_{ij})\in C^\infty([0,1],\mathfrak{J}(M,\omega))\times C^\infty([0,1]\times M)$ for \textit{each} strip-like end, with labels $0\leq i< j\leq 2$ according to the bounding Lagrangian submanifolds (see Figure \ref{holomorphicstrip3} left). Notice this does make sense by Remark \ref{striplike}. This time, however, such data cannot be extended on the whole punctured disk: we must choose some global $J\in C^\infty(\mathbb{D}^2,\mathfrak{J}(M,\omega))$ and $H\in C^\infty(\mathbb{D}^2\times M)$ (that is, families smoothly varying in $\mathbb{D}^2$) which in a fixed neighbourhood of each $z_j$ correspond to the chosen pair $(J_{ij},H_{ij})$. 
	
	So we recover the perturbed PDE \eqref{perturbedCR} only locally about the punctures, where we can use holomorphic local complex coordinates $z=s+it$ for $(s,t)\in\mathbb{R}\times[0,1]$. To see this, fix some $\beta\in\Omega^1(\mathbb{D}^2)$ such that $\beta|_{\mathbb{S}^1}=0$ and $\beta=dt$ about the $z_j$'s. Perturbing the Cauchy--Riemann equation \eqref{CRequation}, we obtain
	\begin{equation}\label{generalCR}
	(du-X_H\otimes\beta)_J^{0,1}=0\;,
	\end{equation}
	which about the three punctures coincides exactly with Floer's equation: $0=(du-X_H\otimes dt)_J^{0,1}=\partial_su + J_t(u)(\partial_tu-X_{H_t}(u))$, where $J$ and $H$ are merely $t$-dependent. Analogously to what found in Section \ref{ch5.5}, it holds $\mathcal{X}(L_i,L_j)=L_i\cap(\phi_{H_{ij}}^1)^{-1}(L_j)$, and the proof of a few paragraphs above carries over to the perturbed product map $\boldsymbol{\cdot}:CF(L_1,L_2;J_{12},H_{12})\otimes CF(L_0,L_1;J_{01},H_{01})\rightarrow CF(L_0,L_2;J_{02},H_{02})$, then fulfilling the Leibniz rule.	
\end{proof}

\subsection{The higher order composition maps}\label{ch6.2}

Let us now consider higher order products involving $(d+1)$-many Lagrangian submanifolds $L_0,...,L_d\subset M$, with the usual decorations.

Given $d+1$ generators $p_1\in\mathcal{X}(L_0,L_1)$, $p_2\in\mathcal{X}(L_1,L_2)$, ..., $p_d\in\mathcal{X}(L_{d-1},L_d)$, $q\in\mathcal{X}(L_d,L_0)$ and boundary punctures $z_j=e^{i\pi(1+2j/(d+1))}$, $z_0=-1\in\mathbb{S}^1$ for $j=1,...,d$, we seek $J$-holomorphic maps $u:\mathbb{D}^2\setminus\{z_0,z_1,...,z_d\}\rightarrow M$ satisfying a degree-$(d+1)$ analogue of the perturbed Cauchy--Riemann equation \eqref{generalCR} with boundary conditions
\begin{equation}\label{CRboundarycondd+1}
	\left\{\begin{array}{llll}
		u(e^{i\alpha})\in L_0 & \text{if }\alpha\in\big(\pi,\pi+\frac{2\pi}{d+1}\big) & \text{and} & \displaystyle\lim_{z\to z_0}u(z)=q\\
		u(e^{i\alpha})\in L_j & \text{if }\alpha\in\big(\pi+\frac{2j\pi}{d+1},\pi+\frac{2(j+1)\pi}{d+1}\big) & \text{and} & \displaystyle\lim_{z\to z_j}u(z)=p_j\\ 
	\end{array}\right.,
\end{equation}
and finite energy $E(u)<\infty$. We call this the \textbf{generalized Cauchy--Riemann problem}\index{generalized Cauchy--Riemann problem}. Its solutions representing a fixed class $[u]\in\pi_2(M,\bigcup_{i=0}^d L_i)$ form up the moduli space $\mathcal{M}(p_1,...,p_d,q;[u],J,H)$. The situation is depicted in Figure \ref{holomorphicstripd}, which also makes clear why the images $\text{im}(u)$ of $J$-holomorphic solutions $u$ (and by extension, the $u$'s themselves) are called \textbf{pseudoholomorphic polygons}\index{pseudoholomorphic polygon}.

\begin{figure}[htp]
	\centering
	\includegraphics[width=0.95\textwidth]{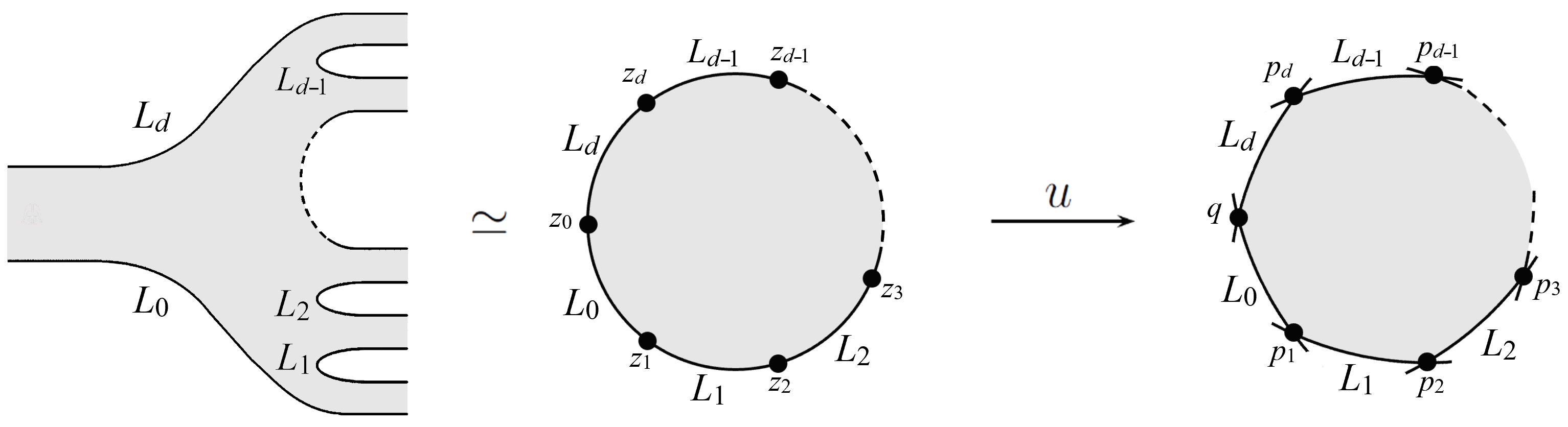}
	\caption{A $(d+1)$-punctured $J$-holomorphic disk $u$ in $\mathcal{M}(p_1,...,p_d,q;[u],J,H)$.}
	\label{holomorphicstripd}	
\end{figure}

Here we already chose suitable domain-dependent $J\in C^\infty(\mathbb{D}^2,\mathfrak{J}(M,\omega))$ and $H\in C^\infty(\mathbb{D}^2\times M)$, paying attention (see Definition \ref{perturbationdata} below) to the following fact: in contrast to the case $d=2$ of previous section, here there are many choices of cyclically ordered tuples $(z_0,z_1,...,z_d)\subset\mathbb{S}^1$, differing in the exact reciprocal position of each puncture along its arc.

\begin{Def}\label{conformal}							
	Let $\text{Punct}_{d+1}$ be the space of all possible ordered configurations $(z_0,z_1,...,z_d)\subset\mathbb{S}^1$ of boundary punctures, for $d\geq 2$, and $\text{Aut}(\mathbb{D}^2)$ the group of biholomorphisms of $\mathbb{D}^2$. Then $\mathcal{M}_{0,d+1}\coloneqq \text{Punct}_{d+1}/\text{Aut}(\mathbb{D}^2)$ is the \textbf{moduli space of conformal structures on $\mathbb{D}^2$}\index{moduli space!of conformal structures on $\mathbb{D}^2$}, a contractible $(d-2)$-dimensional smooth manifold.
	
	The space $\mathcal{M}_{0,d+1}$ has a natural compactification to a $(d-2)$-dimensional polytope $\overline{\mathcal{M}}_{0,d+1}$ --- which is (homeomorphically) none other than the associahedron $K_d$ we encountered back in Section \ref{ch1.0}! Its boundary $\partial\overline{\mathcal{M}}_{0,d+1}$ parametrizes all possible degenerations of $\mathbb{D}^2$ (\textit{nodal degenerations}) into subdisks each carrying at least one puncture $z_j$. Specifically, each $(d-3)$-skeleton component corresponds to a (homotopy class of a) degeneration into a pair of disks each retaining at least two punctures; likewise, to the $(d-4)$-skeleton we can assign configurations with three disks, and so on. 
\end{Def}

\begin{Ex}\label{K3}										
	If $d=3$, we know the compactified 1-dimensional polytope $\overline{\mathcal{M}}_{0,4}=K_3$ to be homeomorphic to $[0,1]\subset\mathbb{R}$ (see Figure \ref{associahedra}, bottom left). The two extremal points forming its 0-skeleton label the two possible degenerations of $\mathbb{D}^2$ for fixed punctures $z_1,z_2,z_3\in\mathbb{S}^1$ (without loss of generality): either $z_0\in\mathbb{S}^1$ can ``slide'' along $L_0$ towards $z_1$ and end up with it in a subdisk of the degenerate configuration corresponding to 0, or it can slide along $L_3$ towards $z_3$ to produce that associated to 1 (see Figure \ref{associahedron}). 
	\vspace*{0.1cm}
	\begin{figure}[htp]
		\centering
		\includegraphics[width=0.85\textwidth]{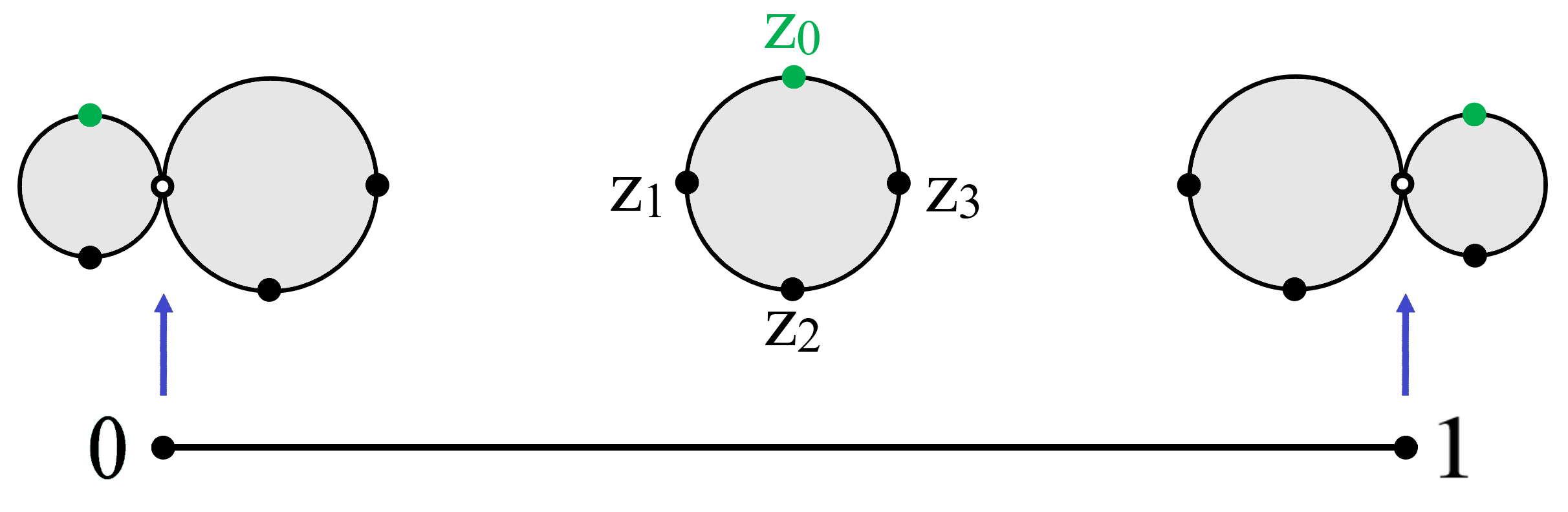}
		\caption{The polytope $\overline{\mathcal{M}}_{0,4}=K_3$ and its corresponding degenerate limit configurations (white dots denote newly generated nodes).}
		\label{associahedron}	
	\end{figure}
	
	Alternatively, we can fix $z_0,z_1,z_2$ and let $z_3$ vary. This ``sliding operation'' --- formalized by the homotopy $m_3$ of Section \ref{ch1.0} --- contributes to some higher order composition maps on Floer complexes, as we will show below. \hfill $\blacklozenge$
\end{Ex}

\begin{Rem}\label{hit}							
	Building on last example, we can justify the definition of $\mathcal{M}_{0,d+1}$ as follows. The group $\text{Aut}(\mathbb{D}^2)\cong\text{PSL}(2,\mathbb{R})$ is a 3-dimensional Lie group whose action can uniquely fix at most three boundary punctures of $\mathbb{D}^2$ (indeed $\mathcal{M}_{0,3}=\text{Punct}_{3}/\text{Aut}(\mathbb{D}^2)$ is 0-dimensional, thus negligible, as discussed right after Remark \ref{striplike}). If $d>2$, that is, if we work with more than three punctures, then these can slide freely along the arcs mapping to the Lagrangian submanifolds to produce non-trivial distinct classes in $\mathcal{M}_{0,d+1}$. Modulo the three punctures which can be fixed, it follows that $\mathcal{M}_{0,d+1}$ has structure of a smooth manifold of dimension $(d+1)-3=d-2$.
	
	In analogy to Example \ref{K3}, nodal degenerations of $\mathbb{D}^2$ arise when two or more punctures hit each other, causing a bubbling effect which resolves into subdisks carrying the points involved. The compactification of $\mathcal{M}_{0,d+1}$ amounts to including these limit configurations, then labelled by $\partial\overline{\mathcal{M}}_{0,d+1}$.
\end{Rem}

\begin{Ex}
	The intuition for $\overline{\mathcal{M}}_{0,5}=K_4$ is then the following. Given five punctures $z_0,z_1,z_2,z_3,z_4\in\mathbb{S}^1$, fix a (consecutive!) triple, say $z_1,z_2,z_3$, and let $z_0,z_4$ vary. We have ten ways to ``crash'' this pair. For example, $z_0$ can hit $z_1$, which we may suggestively write as $(z_0\ast z_1)$. The other combinations are (moving counterclockwise in Figure \ref{K4}): $(z_4\ast(z_0\ast z_1))$, $(z_4\ast z_0\ast z_1)$, $((z_4\ast z_0)\ast z_1)$, $(z_4\ast z_0)$, $(z_3\ast(z_4\ast z_0))$, $(z_3\ast z_4\ast z_0)$, $((z_3\ast z_4)\ast z_0)$, $(z_3\ast z_4)$ and $(z_0\ast z_1)(z_3\ast z_4)$. 
	
	Parentheses matter, as they regulate the collision order! Indeed, $(z_4\ast z_0\ast~z_1)$ corresponds to the nodal degeneration into two subdisks with one carrying $z_0, z_1, z_4$, while $(z_4\ast(z_0\ast z_1))$ encodes a configuration of three disks whose medial one carries $z_4$ and whose terminal one collects $z_0$ and $z_1$. Clearly, five among the listed combinations label the edges of the pentagon $K_4$ (that is, its 1-skeleton), while the remaining five its vertices (the 0-skeleton). 
	
	\begin{figure}[htp]
		\centering
		\includegraphics[width=0.80\textwidth]{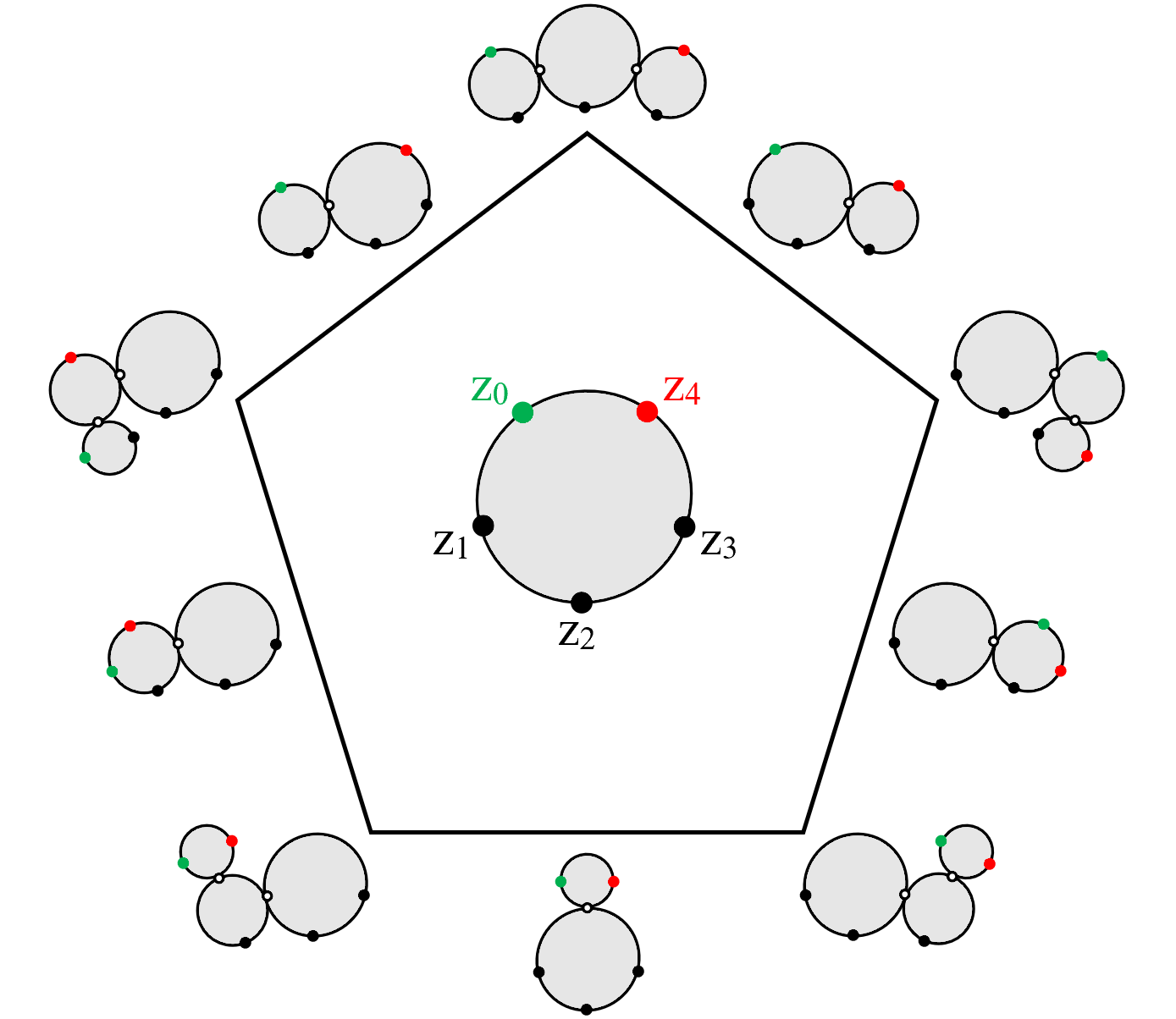}
		\caption{The polytope $\overline{\mathcal{M}}_{0,5}=K_4$ and its corresponding degenerate limit configurations (white dots denote newly generated nodes).}
		\label{K4}	
	\end{figure}
	
	We purposefully kept the variables in cyclic order to highlight yet another interpretation: a component of the $(d-r)$-skeleton of $K_d$ corresponds to a suitable insertion of $r-2$ pairs of parentheses in the string $z_0z_1...z_{d-1}z_d$. For example, $(z_4\ast(z_0\ast z_1))\equiv (z_0z_1))z_2z_3(z_4$. This makes the study of higher dimensional associahedra far easier. The reader may wish to investigate $K_5$ (the rightmost in Figure \ref{associahedra}), whose nine faces, twenty-one edges and fourteen vertices  parametrize the possible nodal degenerations into respectively two, three and four subdisks. \hfill $\blacklozenge$   	
\end{Ex}

We finally describe how to choose compatible perturbation data for the general case of $d+1$ Lagrangian submanifolds. The proof of its existence is addressed in \cite[Lemma 9.5]{[Sei08]}.

\begin{Def}\label{perturbationdata}						
	Given a conformal structure $(z_0,z_1,...,z_d)\in\mathcal{M}_{0,d+1}$, consider the associated generalized Cauchy--Riemann problem of the form \eqref{generalCR} with boundary conditions \eqref{CRboundarycondd+1} and finite energy constraint. Consider also the time-dependent $(J_{ij},H_{ij})\in C^\infty([0,1],\mathfrak{J}(M,\omega))\times C^\infty([0,1]\times M)$ determining the $CF(L_i,L_j)$'s, for each strip-like end (that is, defined on fixed neighbourhoods of the $p_j$'s and $q$ where holomorphic local coordinates are available), with $0\leq i< j\leq d$ according to the bounding Lagrangian submanifolds. To achieve transversality, we choose $J\in C^\infty(\mathbb{D}^2,\mathfrak{J}(M,\omega))$ and $H\in C^\infty(\mathbb{D}^2\times M)$ such that:
	\begin{itemize}[leftmargin=0.5cm]
		\item $(J,H)$ agrees with each $(J_{ij},H_{ij})$ in the respective strip-like end,
		\item upon nodal degeneration of $\mathbb{D}^2$ as in Definition \ref{conformal}, $J$ and $H$ are translation-invariant in the (strip-like) regions about the nodal points, and indeed coincide with the local Hamiltonians and almost complex structures there.
	\end{itemize}
	We call such a pair $(J,H)$ \textbf{Floer datum}\index{Floer datum}; by construction, it depends on $(z_0,z_1,...,z_d)$ as well as all $J_{ij}$'s and $H_{ij}$'s. Then the arising Floer cohomology is said to be ``unobstructed''.	
\end{Def}

The involvement of $\mathcal{M}_{0,d+1}$ in our discussion also implies we must adapt Lemma \ref{moduli3dim} and its proof accordingly.

\begin{Lem}\label{generatorddegree}				
	Fix a conformal structure $(z_0,z_1,...,z_d)\in\mathcal{M}_{0,d+1}$ and a compatible Floer datum $(J,H)\in C^\infty(\mathbb{D}^2,\mathfrak{J}(M,\omega))\times C^\infty(\mathbb{D}^2\times M)$ as specified above. Consider the associated generalized Cauchy--Riemann problem, yielding moduli spaces $\mathcal{M}(p_1,...,p_d,q;[u],J,H)$ for each solution $[u]\in\pi_2(M,\bigcup_{i=0}^d L_i)$. Then:
	\begin{align}\label{moduliddim}
		\textup{dim}(\mathcal{M}(p_1,...,p_d,q;[u],J,H)) &= \textup{dim}(\mathcal{M}_{0,d+1})+\textup{ind}([u]) \nonumber \\
		&= d-2+\textup{deg}(q)-\sum_{i=1}^d\textup{deg}(p_i)\;,
	\end{align}
	where $\textup{dim}(\mathcal{M}_{0,d+1})$ accounts for possible deformations of $(z_0,z_1,...,z_d)$.
\end{Lem}

\begin{Def}\label{Floercompositions}			
	Let $(M,\omega,J)$ be an almost complex symplectic manifold with $2c_1(TM)=0$, and $L_0,...,L_d\subset M$ Lagrangian submanifolds with the usual decorations, in particular such that $\langle\omega\rangle\cdot\pi_2(M,L_i)=0$. Choose a Floer datum $(J,H)\in C^\infty(\mathbb{D}^2,\mathfrak{J}(M,\omega))\times C^\infty(\mathbb{D}^2\times M)$ to achieve transversality. The \textbf{$d$-th Floer composition map}\index{Floer composition map} $\mu^d:\,CF(L_{d-1},L_d)\otimes...\otimes CF(L_0,L_1)\rightarrow CF(L_0,L_d)$ is the $\Lambda_\mathbb{K}$-linear extension of
	\begin{equation}\label{Floercompo}
		\mu^d(p_d,...,p_1) \coloneqq \mkern-18mu\sum_{\substack{q\in\mathcal{X}(L_0,L_d) \\ [u]:\text{ind}([u])=2-d}}\mkern-18mu (\#\mathcal{M}(p_1,...,p_d,q;[u],J,H))T^{\omega([u])}q\;,
	\end{equation}
	where $\mathcal{M}(p_1,...,p_d,q;[u],J,H)$ is the moduli space of solutions to the generalized Cauchy--Riemann problem \eqref{generalCR} representing $[u]\in\pi_2(M,\bigcup_{i=0}^d L_i)$. The sum is well defined, by the usual argument, and we note that Lemma \ref{generatorddegree} makes $\mu^d$ a graded map of degree $2-d$.  
	
	In particular, $\mu^1=\partial_{J,H}$ is the (degree 1) Floer differential \eqref{Floersdiff} and $\mu^2=\boldsymbol{\cdot}$ the (degree 0) Floer product \eqref{Floerprod}. 
\end{Def}

\begin{Pro}\label{AinfasseqsFloer}				
	The Floer composition maps $\mu^d$ fulfill the $A_\infty$-associativity equations \eqref{Ainfcatasseqs}\textup:
	\begin{equation}\label{Floerasseqs}
	\mkern-30mu\sum_{\substack{1\leq m\leq d \\ \quad 0\leq n\leq d-m}}\mkern-24mu(-1)^{\dagger_n}\mu^{d-m+1}(p_d,...,p_{n+m+1},\mu^m(p_{n+m},...,p_{n+1}),p_n,...,p_1) = 0\;,
	\end{equation}
	in $CF(L_0,L_d)$, where $p_i\in CF(L_{i-1},L_i)$ and\footnote{In light of equation \eqref{Floerasseqs}, we definitely swap $\text{deg}(p_i)$ for $|p_i|$.} $\dagger_n \coloneqq (\sum_{i=1}^{n}|p_i|)-n$. 
	
	For $d=1,2$, the first two such equations read $\partial(\partial(p_1)) = 0$ and $\partial(p_2\boldsymbol{\cdot} p_1) = (-1)^{|p_1|}\partial(p_2)\boldsymbol{\cdot} p_1 -p_2\boldsymbol{\cdot} \partial(p_1)$ \textup(the graded Leibniz rule \eqref{Leibnizrule}\textup). In particular, the case $d=3$ implies that Floer's product $\mu^2$ is associative up to homotopy given by $\mu^3$, hence associative on cohomology.
\end{Pro}

\begin{proof}(\textit{Sketch})
	The cases $d=1,2$ were addressed in Theorem \ref{HFwelldef} respectively Proposition \ref{Leibniz}. Proof of the latter also showed associativity of $\mu^2$ on cohomology. 
	
	To show that the general identity \eqref{Floerasseqs} holds, we follow the ``usual'' steps: in the framework of Definition \ref{Floercompositions}, we look at classes $[u]\in\pi_2(M,\bigcup_{i=0}^d L_i)$ of $\text{ind}([u])=3-d$ and their moduli spaces $\mathcal{M}(p_1,...,p_d,q;[u],J,H)$, which, by equation \eqref{moduliddim}, transversality and Proposition \ref{Gromov}, are then 1-dimensional smooth manifolds compactifying to some $\overline{\mathcal{M}}(p_1,...,p_d,q;[u],J,H)$ whose boundary points either correspond to
	\begin{itemize}[leftmargin=0.5cm]
		\item index 1 $J$-holomorphic strips breaking off at one among the $d+1$ generators, which will then be separated from the pseudoholomorphic polygon (as in Figure \ref{3bubbling}), or
		\item nodal degenerations of $\mathbb{D}^2$ described by the $(d-3)$-skeleton of $\overline{\mathcal{M}}_{0,d+1}$, which are pairs of disks each carrying \textit{at least two} punctures (as in Figure \ref{associahedron}).
	\end{itemize}
	The former case is caused by energy build-ups at one of the generators and produces all terms involving $\mu^1$ in \eqref{Floerasseqs}, while the second originates when punctures ``hit'' each other (as expalined in Remark \ref{hit}), contributing with all the remaining summands of \eqref{Floerasseqs}. The desired equation follows by arguing once again that the space $\overline{\mathcal{M}}(p_1,...,p_d,q;[u],J,H)$ is a compact 1-dimensional manifold with (coherently oriented) boundary whose signed cardinality must then vanish, and finally summing over all homotopy classes.\footnote{Consult \cite[section 3]{[Fuk02a]} for a slightly different approach, from the standpoint of Morse--Bott Floer theory.} 
	
	Let us discuss the case $d=3$. Boundary points in $\overline{\mathcal{M}}(p_1,p_2,p_3,q;[u],J,H)$ of the first type are index 1 strips breaking off from the index $-1$ pseudoholomorphic square at one of its four punctures; they are responsible for the terms $\partial(\mu^3(p_3,p_2,p_1))$, $\mu^3(p_3,p_2,\partial p_1)$, $\mu^3(p_3,\partial p_2, p_1)$ and $\mu^3(\partial p_3,p_2,p_1)$. Boundary points of the second type correspond instead to the two nodal configurations described in Example \ref{K3}, pairs of twice punctured index 0 disks; they yield $p_3\boldsymbol{\cdot}(p_2\boldsymbol{\cdot} p_1)$ and $(p_3\boldsymbol{\cdot} p_2)\boldsymbol{\cdot} p_1$. With coefficients dictated by the degrees of the generators, we see that the associated $A_\infty$-identity is indeed fulfilled. 
\end{proof}

\subsection{The $A_\infty$-structure of the Fukaya category}\label{ch6.3}

We are finally ready to start reaping the harvest of Part I.

\begin{Def}\label{Fukaya}						
	Let $M=(M,\omega)$ be a symplectic manifold with $2c_1(TM)=0$. Its (compact) \textbf{Fukaya category}\index{Fukaya category} $\mathscr{F}(M)$ is the $A_\infty$-category characterized by the following data:
	\begin{itemize}[leftmargin=0.5cm]
		\renewcommand{\labelitemi}{\textendash}
		\item A set of objects $\textup{obj}(\mathscr{F}(M))$ consisting of compact oriented spin graded Lagrangian submanifolds $L_i\subset M$ of vanishing Maslov class $\mu_{L_i}\equiv[\phi_{L_i}]=0\in H^1(L;\mathbb{Z})$ and fulfilling the no-bubbling condition $\langle\omega\rangle\cdot\pi_2(M,L_i)=0$, together with an implicit choice of graded lift $\tilde{\phi}_{L_i}:L_i\rightarrow\mathbb{R}$ and spin structure. 
		\item Perturbation data $(J_{L,L'},H_{L,L'})\in C^\infty([0,1],\mathfrak{J}(M,\omega))\times C^\infty([0,1]\times M)$ for every pair $(L,L')\subset\textup{obj}(\mathscr{F}(M))$, and Floer data $(J,H)\in C^\infty(\mathbb{D}^2,\mathfrak{J}(M,\omega))\times C^\infty(\mathbb{D}^2\times M)$ as in Definition \ref{perturbationdata} for every tuple $(L_0,...,L_d)\subset\textup{obj}(\mathscr{F}(M))$, thus compatible with all $(J_{L_i,L_j},H_{L_i,L_j})$'s in the strip-like regions of the $J$-holomorphic curves solving the underlying generalized Cauchy--Riemann problem \eqref{CRboundarycondd+1}. 
		\item For any two $L_0, L_1\in\textup{obj}(\mathscr{F}(M))$ (possibly equal), the $\mathbb{Z}$-graded free $\Lambda_\mathbb{K}$-module 
		\begin{equation}
		\textup{hom}_{\mathscr{F}(M)}(L_0,L_1) \coloneqq  CF(L_0,L_1;H_{L_0,L_1})\,,
		\end{equation}
		resulting from the Hamiltonian perturbation by $H_{L_0,L_1}$ of \eqref{Floercmplx}.
		\item (Multi)linear $\mathbb{Z}$-graded composition maps $\mu_{\mathscr{F}(M)}^{d}\coloneqq \mu^d$ as in \eqref{Floercompo}, for every $d\geq 1$ and $L_0,...,L_d\in\textup{obj}(\mathscr{F}(M))$ (then fulfilling the $A_\infty$-associativity equations, by Proposition \ref{AinfasseqsFloer}).
	\end{itemize} 
\end{Def}

\begin{Rem}\label{Fukayarem}						
	\ \vspace*{0.0cm}
	\begin{itemize}[leftmargin=0.5cm]
		\item We already hinted that there are many different definitions of Fukaya category, depending on the precise assumptions for both $M$ and its Lagrangian submanifolds. For example, dropping the requirements $2c_1(TM)$ $=0$ and $[\phi_{L_i}]=0$ for each Lagrangian submanifold reduces the $\mathbb{Z}$-grading to a $\mathbb{Z}_2$-grading, and choosing $\mathbb{K}$ of $\text{char}(\mathbb{K})=2$ allows us to ignore spin structures. Most prominently, the \textit{wrapped Fukaya category} considers only exact manifolds, building on what seen in Section \ref{ch5.8}; we briefly discuss this variant in Section \ref{ch6.5} below. The relative and projective versions of Fukaya categories are also discussed in \cite[section 8]{[Sei13]}.
		\item We can even tinker with the perturbation data; however, all the so obtained $A_\infty$-categories turn out to be \textit{quasi-isomorphic} (in the sense of Definition \ref{cohomfunctor}). This means that, up to quasi-isomorphism, $\mathscr{F}(M)$ is a ``symplectic invariant'' of $M$ (more on this in \cite[chapter 10]{[Sei08]}).
		\item If we drop instead the no-bubbling condition $\langle\omega\rangle\cdot\pi_2(M,L_i)=0$ and reshape our definitions in Morse--Bott Floer theoretic language, we ultimately obtain the notion of \textit{curved $A_\infty$-category}, where disk bubbles are encoded by images of a degree 0 composition map $\mu^0:L_i\subset M\mapsto\mu_{L_i}^0\in CF(L_i,L_i)$ which fits into the $A_\infty$-associativity equations (more details in \cite[Remark 2.11]{[Aur13]} and \cite[section 3]{[Fuk02a]}). In this scenario, one restricts the set of objects to those whose $\mu_L^0$ is a multiple of the c-unit $e_L$ (see equation \eqref{Fukayacunit} below), which in mirror symmetry corresponds to the so-called ``superpotential'' or ``central charge''.   
		\item In low dimension, the existence and well-definiteness of Fukaya categories might be guaranteed by topological constraints: for example, any 4-dimen-\break sional symplectic manifold $M^4$ (such as classical Calabi-Yau manifolds) with $2c_1(TM)=0$ admits no bubbling phenomena (see \cite[Lemma 3.3]{[AS09]}), hence its associated Fukaya category is properly defined (even though invariance under Hamiltonian isotopies of its objects --- Lagrangian surfaces $L^2\subset M$ of vanishing Maslov class --- is not a priori achieved; cf. \cite[Lemma 3.4]{[AS09]}).
	\end{itemize}
\end{Rem}

\begin{Def}										
	Let $M$ be a symplectic manifold with $2c_1(TM)=0$, and $\mathscr{F}(M)$ its Fukaya category. The induced cohomological category $H(\mathscr{F}(M))$ (as described in Definition \ref{cohomcategory}) is called the \textbf{Donaldson--Fukaya category}\index{Donaldson--Fukaya category}, or \textit{quantum category}. Specifically, its objects are still Lagrangian submanifolds and its Hom-sets are the total Floer cohomology rings \eqref{Floercohomring} induced by the Floer differentials.
	
	The subcategory $H^0(\mathscr{F}(M))\subset H(\mathscr{F}(M))$ only considers the zeroth cohomology of each morphism space, $\textup{Hom}_{H^0(\mathscr{F}(M))}(L_0,L_1)\coloneqq HF^0(L_0,L_1)$.
\end{Def}

Observe that, since Hamiltonian isotopic Lagrangian submanifolds $L_1, L_1'\subset M$ induce isomorphisms $\textup{Hom}_{H(\mathscr{F}(M))}(L_0,L_1)=HF(L_0,L_1)\cong HF(L_0,L_1')=\textup{Hom}_{H(\mathscr{F}(M))}(L_0,L_1')$ for any $L_0\subset M$ (cf. Definition \ref{isotopy} and Remark \ref{CRreform}), setting $L_0\coloneqq L_1$ proves that they are isomorphic as objects of $H(\mathscr{F}(M))$. This ticks off the only restriction we imposed back in Section \ref{ch5.2}.

Now that Fukaya categories are ready to go, we may wonder about their unitality as $A_\infty$-categories. The next few observations are discussed, for example, in \cite[section 3]{[AS09]}.

\begin{Lem}										
	Given any Fukaya category $\mathscr{F}(M)$, each Lagrangian submanifold $L\in\textup{obj}(\mathscr{F}(M))$ admits a cohomological unit $e_L\in\textup{hom}_{\mathscr{F}(M)}^0(L,L)$, which we recall is a degree $0$ cocycle \textup(thus $\mu_{\mathscr{F}(M)}^1(e_L)=0$\textup) such that $\langle e_L\rangle=\textup{id}_L\in HF(L,L)$ in the Donaldson--Fukaya category. Therefore, $H(\mathscr{F}(M))$ is an actual category.
\end{Lem}

\begin{proof}(\textit{Sketch})
	Recall that
	\[
	CF^0(L,L)=\mkern-18mu\bigoplus_{p\in \mathcal{X}(L,L): |p|=0}\mkern-18mu\Lambda_\mathbb{K}\cdot p\;,
	\]
	where $\mathcal{X}(L,L)=L\cap(\phi_H^1)^{-1}(L)$ is the transverse intersection of $L$ with a copy perturbed by a suitable Hamiltonian function $H\in C^\infty(M\times [0,1])$. Among all these perturbed degree 0 generators $p$, the cohomological unit $e_L$ is obtained by considering the degree 1 (so $d=0$) analogue of the Cauchy--Riemann problem \eqref{perturbedCR}, where we choose $L_0=L_1\coloneqq L$ and start with a single boundary puncture of $\mathbb{D}^2$ (then mapping to $e_L$). The cocycle we seek turns out to be
	\begin{equation}\label{Fukayacunit}
	e_L\coloneqq \mkern-18mu\sum_{q\in\mathcal{X}(L,L):|q|=0}\mkern-18mu (\#\mathcal{M}(q;[u],J,H))T^{\omega([u])}q\;.
	\end{equation}
	One can show with help of Proposition \ref{CFLL} that this does indeed represent the identity morphism in $HF(L,L)\cong H(L;\Lambda_\mathbb{K})$. The precise proof involves a gluing argument from TQFT --- far exceeding our scope.  
\end{proof}

\begin{Lem}										
	The perturbation data defining $\mathscr{F}(M)$ can be chosen in such a way that the cohomological unit $e_L\in\textup{hom}_{\mathscr{F}(M)}^0(L,L)$ of any Lagrangian submanifold $L\subset M$ fulfills $\mu_{\mathscr{F}(M)}^2(e_L,e_L)=e_L$ and $\mu_{\mathscr{F}(M)}^d(e_L,...,e_L)=0$ if $d>2$.
\end{Lem}

\begin{proof}(\textit{Sketch})
	Consider the Morse function $f\in C^\infty(L)$ we used in the proof of Proposition \ref{CFLL}, obtaining $HF(L,L)\cong HM(f;\Lambda_\mathbb{K})$. Morse homology is isomorphic to singular homology, which, in the case of a cell complex (such as a smooth manifold), is in turn isomorphic to cellular homology (see \cite[Theorem 2.35]{[Hat09]}). Since $L^n\subset M^{2n}$ is $n$-dimensional, it follows that $CF(L,L)$ is concentrated in the degrees between 0 and $n$. Therefore, as $\mu_{\mathscr{F}(M)}^d(e_L,...,e_L)\in CF^{2-d}(L,L)$ ($\mu_{\mathscr{F}(M)}^d$ being of degree $2-d$), $d>2$ forces it to vanish.
	
	To prove $\mu_{\mathscr{F}(M)}^2(e_L,e_L)=e_L$, we just take $f$ to have a unique maximum and minimum on each connected component of $L$.
\end{proof}

The last lemma tells us that $\mathscr{F}(M)$ is \textit{not} strictly unital, since back in equation \eqref{unitalrequirm} we imposed the stronger requirement that the two identities should hold whenever \textit{at least one} of the inputs is a unit $e_L$.

However, we can resort to some formal diffeomorphism $\Phi:\mathscr{F}(M)\rightarrow \widetilde{\mathscr{F}}(M)$ with target a strictly unital $A_\infty$-category $\widetilde{\mathscr{F}}(M)$, even though this will be geometrically meaningless: the composition maps $\mu_{\widetilde{\mathscr{F}}(M)}^d$ will no longer be counts of pseudoholomorphic polygons!
\vspace*{0.5cm}

\noindent Therefore, when dealing with Fukaya categories, we can work with cohomological units at best --- this is the reason why we focused so intensely on c-unital $A_\infty$-categories, after all! With the fundamental structure of $\mathscr{F}(M)$ clarified, let us now summarize the main results we know from Part I:
\begin{itemize}[leftmargin=0.5cm]				
	\item First, we can study $\mathscr{F}(M)$ through the $A_\infty$-functors mapping to it: for any non-unital $A_\infty$-category $\mathcal{C}$, $\mathcal{Q}\coloneqq nu\text{-}\!fun(\mathcal{C},\mathscr{F}(M))$ is c-unital, and if $E_\mathcal{G}\in\textup{hom}_\mathcal{Q}^0(\mathcal{G},\mathcal{G})$ is a c-unit, then $E_\mathcal{G}^0(X)\in CF^0(\mathcal{G}(X),\mathcal{G}(X))$ is the c-unit \eqref{Fukayacunit} for any $X\in\textup{obj}(\mathcal{C})$ (by Proposition \ref{funccatunitality}). Moreover, homotopic $A_\infty$-functors $\mathcal{G}_0\simeq\mathcal{G}_1\in nu\text{-}\!fun(\mathcal{C},\mathscr{F}(M))$ are isomorphic as objects of $H^0\big(nu\text{-}\!fun(\mathcal{C},\mathscr{F}(M))\big)$ (by Lemma \ref{homotopicareiso}).
	\newline On the other hand, for any c-unital $A_\infty$-category $\mathcal{B}$ and quasi-equivalence $\mathcal{G}\in fun(\mathcal{B},\mathscr{F}(M))$, there exists an inverse quasi-equivalence $\mathcal{F}\in fun(\mathscr{F}(M),\mathcal{B})$ so that the compositions are identity functors in the respective zeroth cohomological categories of c-unital $A_\infty$-functors (by Theorem \ref{aboutquasiequiv}, particularly useful if one wishes to relate Fukaya categories of two different manifolds).
	
	\item We can consider the $A_\infty$-category \[\mathscr{Q}(M)\!\coloneqq \!mod(\mathscr{F}(M))\] of c-unital $A_\infty$-modules over $\mathscr{F}(M)$, whose objects are c-unital $A_\infty$-functors $\mathcal{M}\!:\!\mathscr{F}(M)^{opp}\!\!\rightarrow\mathsf{Ch}$, yielding $\{\mathcal{M}(L)\!=\!\bigoplus_{n\in\mathbb{Z}}\mathcal{M}(L)^n \mid  L\in\textup{obj}(\mathscr{F}(M))\}$, with composition maps
	\[
	\mu_\mathcal{M}^d:\mathcal{M}(L_{d-1})\otimes CF(L_{d-2},L_{d-1})\otimes...\otimes CF(L_0,L_1)\rightarrow\mathcal{M}(L_0)[2-d]
	\]
	for all $d\geq 1$ and $L_i\in\textup{obj}(\mathscr{F}(M))$, and pre-module homomorphisms as in \eqref{tshit} (cf. Definition \ref{Ainfmod}). 
	\newline Then we can define a Yoneda embedding \[\Upsilon\equiv\Upsilon_{\mathscr{F}(M)}:\mathscr{F}(M)\rightarrow\mathscr{Q}(M)\,,\]cohomologically full and faithful, given by $\Upsilon(L)(\tilde{L})\coloneqq CF(\tilde{L},L)$ and $\mu_{\Upsilon(L)}^d\coloneqq\mu^d: CF(L_{d-1},L)\otimes...\otimes CF(L_0,L_1)\rightarrow CF(L_0,L)$ as in \eqref{Floercompo} (by Definition \ref{Yonedaemb}). For any $\mathcal{M}\in\textup{obj}(\mathscr{Q}(M))$ and $L\in\textup{obj}(\mathscr{F}(M))$, the Yoneda Lemma \ref{Yonedalemma} makes the extension $\lambda_\mathcal{M}:\mathcal{M}(L)\rightarrow\textup{hom}_{\mathscr{Q}(M)}(\Upsilon(L),\mathcal{M}(L))$ of
	\begin{align*}
	\Upsilon^1(p)^d:\; & CF(L_{d-1},\tilde{L})\otimes CF(L_{d-2},L_{d-1})\otimes...\otimes CF(L_0,L_1)\rightarrow CF(L_0,L),\\
	&(\tilde{p},p_{d-1},...,p_1)\mapsto\mu^{d+1}(p,\tilde{p},p_{d-1},...,p_1)
	\end{align*}
	a quasi-isomorphism (where $p\in CF(\tilde{L},L)$).
	
	\item The map $\lambda$ comes into play with Lemma \ref{quasi-repr}, which produces the following identifications between Floer cohomology rings:\footnote{The same warning we gave at the beginning of Section \ref{ch3.2} applies: direct sums, tensor products and shifts of Lagrangian submanifolds are to be understood as formal algebraic operations between objects of $\mathscr{F}(M)$, not geometrically meaningful --- and in any case, only when corresponding quasi-representatives exist!}
	\begin{align*}
		HF(L,L_0\oplus L_1)&\cong HF(L,L_0)\oplus HF(L,L_1)\;, \\ 
		HF(L_0,Z\otimes L_1)&\cong H(Z)\otimes HF(L_0,L_1)\;, \\
		HF(L_0,S^\sigma L_1)&\cong HF(L_0,L_1)[\sigma]\;, \\
		HF(S^\sigma L_0,L_1)&\cong HF(L_0,L_1)[-\sigma]\;, \\
		HF(S^\sigma L_0,S^\sigma L_1)&\cong HF(L_0,L_1)\;,	
	\end{align*}
	for all $L,L_0,L_1\in\textup{obj}(\mathscr{F}(M))$ and $Z\in\textup{obj}(\mathsf{Ch})$. 
	\newline Most significantly, any cocycle $p\in CF^0(L_0,L_1)$ identifies an abstract mapping cone $\mathcal{C}one(p)=\mathcal{C}\in\textup{obj}(\mathscr{Q}(M))$ (as in Definition \ref{abstractmappingcone}), which assigns to any $L\in\textup{obj}(\mathscr{F}(M))$ the cochain complex
	\[
	\mathcal{C}one(p)(L)=(\mathcal{C}(L),\mu_\mathcal{C}^1)\coloneqq\Big(CF(L,L_0)[1]\oplus CF(L,L_1),\begin{pmatrix}\mu^1 & 0 \\ \mu^2(p,\cdot) & \mu^1 \end{pmatrix}\Big)\;,
	\]
	and fits into the diagram \eqref{Qtriangle} formed by the ``abtract'' inclusion and projection, which in turn detects exact triangles \eqref{Atriangle} in the Donaldson--Fukaya category $H(\mathscr{F}(M))$ (all those whose Yoneda image is isomorphic to \eqref{Qtriangle}).
	
	\item Constructing an additive enlargement $\varSigma\mathscr{F}(M)$ for the Fukaya category as in Definition \ref{addenl}, we can ultimately define the $A_\infty$-category $Tw\mathscr{F}(M)$ of twisted complexes over $\mathscr{F}(M)$: objects are pairs $(L,\delta_L)$ where $L=(I,\{L^i\},\{V^i\})\in\textup{obj}(\varSigma\mathscr{F}(M))$ and $\delta_L=(\delta_{L}^{j,i})\in\textup{hom}_{\varSigma\mathscr{F}(M)}^1(L,L)$ is strictly lower-triangular, morphism spaces are as in equation \eqref{addenlmorphspaces}, and composition maps are given by \eqref{twistedcompo}. 
	\newline Our usual Lagrangian submanifolds can be seen as trivial twisted complexes thanks to the full inclusion $\mathscr{F}(M)\hookrightarrow Tw\mathscr{F}(M)$. In particular, $Tw\mathscr{F}(M)$ is c-unital (by Lemma \ref{TwAcunital}), and Corollary \ref{TwFquasiequiv} relates it through quasi-equivalences to other $A_\infty$-categories of twisted complexes.
	
	\item In $Tw\mathscr{F}(M)$ the direct sum, tensor product and shift operations listed above are valid also at the level of Floer complexes (Definition \ref{TwAops}). Similarly, any $p\in\textup{hom}_{Tw\mathscr{F}(M)}(L_0,L_1)$ yields a twisted mapping cone
	\[
	Cone(p)=(C,\delta_C)\coloneqq \Big(SL_0\oplus L_1,\begin{pmatrix}\delta_{L_0} & 0 \\ p & \delta_{L_1} \end{pmatrix}\Big)\in\textup{obj}(Tw\mathscr{F}(M))
	\]
	such that $\Upsilon(Cone(p))\cong\mathcal{C}one(p)$ (cf. Definition \ref{mappingcone}, which involves the $A_\infty$-shift functor $S:Tw\mathscr{F}(M)\rightarrow Tw\mathscr{F}(M)$), and fitting into diagram \eqref{Aexacttriangle} in $H^0(Tw\mathscr{F}(M))$.
	\newline As discussed, the advantage of twisted complexes is that $Tw\mathscr{F}(M)$ is a \textit{triangulated} $A_\infty$-category (by Corollary \ref{TwAtriang}): we can associate to any cocycle $p\in CF^0(L_0,L_1)$ the aforementioned canonical exact triangle
	\begin{equation}\label{Ftriangle}
		\begin{tikzcd}
			L_0\arrow[rr,"\langle p\rangle"] & & L_1\arrow[dl, "\langle i\rangle"]\\
			& Cone(p)\arrow[ul, "{\langle \pi\rangle[1]}"]
		\end{tikzcd}
	\end{equation}
	in $H^0(Tw\mathscr{F}(M))$ (where we denoted the canonical projection by $\pi$, so to avoid confusion). Proposition \ref{H0TwAtriangulated} shows that $(H^0(Tw\mathscr{F}(M)),H^0(S))$ is a standard triangulated category, then fulfilling the axioms of Section \ref{ch3.1}. For example, we can rotate diagram \eqref{Ftriangle} to produce the two exact triangles of \eqref{rotatedtriang}, or form long exact sequences of vector spaces in $H(Tw\mathscr{F}(M))$: 
	\begin{align*}
		& ...\rightarrow HF^n(L,L_0)\rightarrow HF^n(L,L_1)\rightarrow HF^n(L,C)\rightarrow HF^{n+1}(L,L_0)\rightarrow...\;, \\
		& ...\rightarrow HF^n(L_1,L)\rightarrow HF^n(L_0,L)\rightarrow HF^n(C,L)\rightarrow HF^{n+1}(L_1,L)\rightarrow...\;, 
	\end{align*}
	for any $L\in\textup{obj}(Tw\mathscr{F}(M))$.
	\newline Then we can use these results to show the criteria of Proposition \ref{exactcriterium} testing exactness of triangles in $H(\mathscr{F}(M))$ and of Lemma \ref{ifembquasiequiv} for triangularity of $\mathscr{F}(M)$.
	
	\item Moreover, $Tw\mathscr{F}(M)$ is generated by $\mathscr{F}(M)$ (by Lemma \ref{TwAgeneratedbyA}), and vice versa $(Tw\mathscr{F}(M),\imath)$ is a triangulated envelope for $\mathscr{F}(M)$ (by Proposition \ref{envelopeexists}), so that we can associate to the Fukaya category its bounded derived category $\mathsf{D^b}(\mathscr{F}(M))\coloneqq H^0(Tw\mathscr{F}(M))$.
	\newline We can include split objects by taking the split-closure $(\Pi\mathscr{F}(M),\!\Upsilon_{\!\mathscr{F}(M)})$ of $\mathscr{F}(M)$ (Proposition \ref{uniqueclosure}). For twisted complexes, $\Pi (Tw\mathscr{F}(M))$ has the advantage of being still $A_\infty$-triangulated (by Lemma \ref{splitclosuretriang}), so that the split-closed derived category \[
	\mathsf{D}^\pi(\mathscr{F}(M))\coloneqq H^0\big(\Pi (Tw\mathscr{F}(M))\big)\] is a standard triangulated category. Specifically, $\mathscr{F}(M)$ split-generates\break $\Pi(Tw\mathscr{F}(M))$, meaning that any object of the latter can be constructed from Lagrangian submanifolds by forming all possible mapping cones and shifts, iterating and finally taking images of all the resulting idempotent endomorphisms.
	\newline A useful result is that split-closed derived categories of quasi-equivalent Fukaya categories are themselves equivalent, which can be refined to quasi-equivalent split-generating full subcategories (this is Corollary \ref{splitequivgen}).
	
	\item About twisting: as per Definition \ref{abstracttwistfunc}, we can define an abstract twist functor $\mathcal{T}_{L_0}:\mathscr{Q}(M)\rightarrow\mathscr{Q}(M)$ along any $L_0\in\textup{obj}(\mathscr{F}(M))$, which on Yoneda modules $\Upsilon(L_1)\in\textup{obj}(\mathscr{Q}(M))$ produces graded vector spaces 
	\[
	\mathcal{T}_{L_0}(\Upsilon(L_1))(L)\coloneqq(CF(L_0,L_1)\otimes CF(L,L_0))[1]\oplus CF(L,L_1),
	\]
	for each $L\in\mathscr{F}(M)$. Any quasi-representative $T_{L_0}(L_1)\in\textup{obj}(\mathscr{F}(M))$ thereof is called twist of $L_1$ along $L_0$, and sits into the exact triangle
	\begin{equation}\label{Ftwisttriangle}
	\begin{tikzcd}
		HF(L_0,L_1)\otimes L_0\arrow[rr, "\langle \textup{ev}\rangle"] & & L_1\arrow[dl, "\langle i\rangle"] \\
		& T_{L_0}(L_1)\arrow[ul, "{\langle \pi\rangle[1]}"]
	\end{tikzcd}
	\end{equation}
	formed by the morphism $\textup{ev}: CF(L_0,L_1)\otimes L_0\rightarrow L_1$ mapping under $\Upsilon$ to the abstract evaluation $\epsilon\in CF^0(CF(L_0,L_1)\otimes\Upsilon(L_0),\Upsilon(L_1))$ of Definition \ref{evalmorph}. If we work with twisted complexes, then we obtain $T_{L_0}(L_1)= S(HF(L_0,L_1)\otimes L_0)\oplus L_1 = Cone(\textup{ev})\in\textup{obj}(Tw\mathscr{F}(M))$, allowing us to legitimately define a twist functor $T_{L_0}: Tw\mathscr{F}(M)\rightarrow Tw\mathscr{F}(M)$.     
\end{itemize} 

Roughly speaking, $\mathsf{D}^\pi(\mathscr{F}(M))$ maximizes the load of nice ``algebraic features'' we can ever hope to extract from $\mathscr{F}(M)$, whereas taking cohomology immediately, thus dealing with the Donaldson--Fukaya category, would for example result in the loss of the triangulation properties.

But why do we care so much about exact triangles? And what is even the point of studying Fukaya categories if all we can infer about them is just a matter of notational reformulation of Part I? Well, it turns out that the notion of mapping cones is geometrically meaningful in this particular context. We discuss this in the coming section.

\subsection{Geometric features of Fukaya categories}\label{ch6.4}

We begin with a brief aside regarding the geometrical relevance of Hochschild cohomology (discussed in \cite[section 3]{[AS09]}).

\begin{Rem}\label{quantumHochschild}			
	Consider a Hamiltonian function $H\in C^\infty(\mathbb{R}/\mathbb{Z}\times M)$ and the \textit{Hamiltonian} Floer complex it defines,
	\[
	CF(H)\coloneqq \mkern-12mu\bigoplus_{\gamma\in\mathcal{P}_1^\circ(H)}\mkern-12mu\mathbb{Z}_2\cdot\langle\gamma\rangle\;,
	\]
	where $\mathcal{P}_1^\circ(H)$ is the set of contractible 1-periodic orbits of the associated Hamiltonian vector field $X_H\in\mathfrak{X}(\mathbb{R}/\mathbb{Z}\times M)$ (see for example \cite[chapter 2]{[Mer14]}). Pick one such generator $\gamma_x\coloneqq \phi_H^1(x)\in\mathcal{P}_1^\circ(H)$ ($x\in M$ must be a fixed point of the flow $\phi_H^t$ of $H$). Then we can define a map
	\[
	\Psi: CF(H)\rightarrow CC(\mathscr{F}(M))
	\]
	with target the Hochschild complex of Definition \ref{Hochcohom} which sends $\langle\gamma_x\rangle$ to a sequence of (multi)linear graded maps $\Psi(\langle\gamma_x\rangle)=(h^0,h^1,...)$ whose $d$-th term $h^d$ counts pseudoholomorphic disks with an interior puncture converging to $\gamma_x$ and having $d$-many boundary punctures (the inputs of $h^d$, as usual). 
	
	The map $\Psi$ is actually a chain map, then inducing a homomorphism $H(\Psi)$ from the quantum cohomology $QH(M)\coloneqq H(M)\otimes\Lambda_\mathbb{K}$ of $M$ to the Hochschild cohomology $HH(\mathscr{F}(M))$ of its Fukaya category. On the other hand, the assignment $\Xi: CC(\mathscr{F}(M))\rightarrow CF(L,L),\,(h^d)_{d\geq 0}\mapsto h_L^0$ is itself a chain map, thus inducing a homomorphism $H(\Xi): HH(\mathscr{F}(M))\rightarrow HF(L,L)$, for any $L\in\mathscr{F}(M)$. In particular, as $L$ bounds no pseudoholomorphic disk (by definition of objects of $\mathscr{F}(M)$), the composition
	\[
	H(\Xi)\circ H(\Psi): QH(M)\rightarrow HH(\mathscr{F}(M)) \rightarrow HF(L,L)
	\]
	is just the restriction on cohomology (\cite[Lemma 3.11]{[AS09]}).     
\end{Rem}

Let us discuss something of more significance to our journey: Dehn twists. The relevant theory can be found in \cite[chapter 1]{[Sei02]}.

\begin{Def}											
	Consider the cotangent bundle $T\coloneqq T^*\mathbb{S}^n\cong\mathbb{S}^n\times\mathbb{R}^{n+1}$ equipped with standard symplectic form $\omega_T\in\Omega^2(T)$ which, in global coordinates $(q,p)\in\mathbb{S}^n\times\mathbb{R}^{n+1}$ such that $\langle q, p\rangle_\text{std} = 0$, reads $\omega_T=\sum_{i=1}^n dp_i\wedge dq_i$. Then for each $\lambda\geq 0$ we define the subbundle 
	\[
	T(\lambda)\coloneqq \{(q,p)\in T\mid |\!|p|\!|\leq\lambda\}\subset T\,,
	\]
	a tubular neighbourhood of the zero section $\mathbb{S}^n\equiv T(0)\subset T$, and the fibre-length functional $\varrho:T\rightarrow\mathbb{R},\,\varrho(q,p)\coloneqq |\!|p|\!|$, which generates the normalized geodesic flow on $\mathbb{S}^n$:
	\[
	\Theta:\;(t,(q,p))\in(0,\pi)\times T\longmapsto\Big(\cos(t)p-\sin(t)|\!|p|\!|q,\,\cos(t)q+\sin(t)\frac{p}{|\!|p|\!|}\Big)\in T\;.
	\]
	(recall that $T\cong T\mathbb{S}^n$). Note $\Theta_\pi$ would be the antipodal map $a(q,p)\coloneqq(-q,-p)$, which continuously extends over $T(0)$.
\end{Def}

The following is the key technical step, proved in \cite[Lemma 1.8]{[Sei02]}.

\begin{Lem}\label{technical}						
	Let $h\in C^\infty(\mathbb{R})$ be a smooth function such that $h(t)=0$ if $t\gg 0$, and $h(-t)=h(t)-kt$ for $|t|$ small and some $k\in\mathbb{Z}$. Define the Hamiltonian function $H\coloneqq h\circ\varrho|_{T\setminus T(0)}\in C^\infty(T\setminus T(0))$. Then the associated Hamiltonin diffeomorphism $\phi_H^{2\pi}:T\setminus T(0)\rightarrow T\setminus T(0)$ extends smoothly over $T(0)$ to a compactly supported $\phi\in\textup{Symp}(T,\omega_T)$.
	
	Moreover, if $\textup{supp}(h)\subset (-\infty,\lambda)$ for some $\lambda>0$, then $\phi|_{T(\lambda)}\!\in\!\textup{Symp}(T(\lambda),\omega_T)$ and is the identity near the boundary.	
\end{Lem}

\begin{Def}												
	If $k=1$ and $\text{supp}(h)\subset (-\infty,\lambda)$ in the last lemma, then we call the resulting symplectomorphism $\tau_{\,\mathbb{S}^n}\coloneqq \phi\in\text{Symp}(T,\omega_T)$ a (model) \textbf{Dehn twist}\index{Dehn twist}. Explicitly:
	\begin{equation}\label{dehntwist}
		\tau_{\,\mathbb{S}^n}(q,p) \coloneqq
		\begin{cases}
			\Theta_{2\pi h'(|\!|p|\!|)}(q,p) & $if $ (q,p)\in T(\lambda)\setminus T(0) \\ 
			a(q,p) & $if $ (q,p)\in T(0)
		\end{cases}\;. 
	\end{equation}
	An easy manipulation, along with the assumptions on $h$, shows that the angle of rotation $2\pi h'(|\!|p|\!|)$ continuously decreases from $2\pi h'(0)=\pi$, where $\Theta_\pi\equiv a$ is the antipodal map sending $\mathbb{S}^n$ onto itself, to $2\pi h'(\lambda)=0$, where $\Theta_0\equiv \textup{id}$. Any other Dehn twist in $T(\lambda)$ is symplectically isotopic to \eqref{dehntwist}. 
\end{Def}

\begin{figure}[htp]
	\centering
	\includegraphics[width=0.6\textwidth]{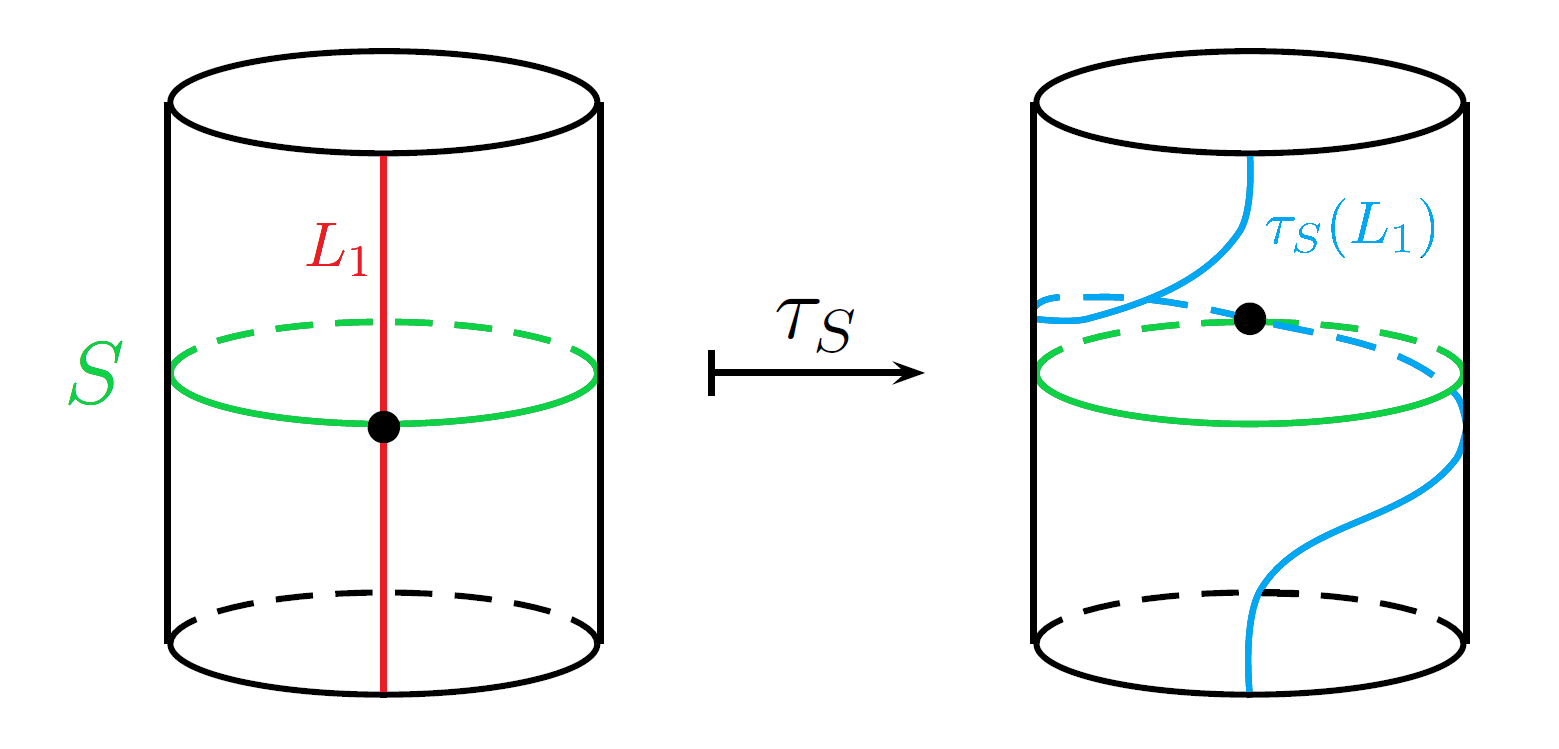}
	\caption{The Dehn twist $\tau_S$ along the sphere $S\coloneqq\mathbb{S}^n$ at work: the chosen fibre $L_1$ spirals around $S$ once and is asymptotically preserved. [Source:\cite{[Aur13]}]}
	\label{Dehn}	
\end{figure}

The action of the Dehn twist $\tau_{\,\mathbb{S}^n}$ on a fibre of $T(\lambda)$ is depicted in Figure \ref{Dehn}. It is sometimes convenient to control the ``wobbliness'' of $\tau_{\,\mathbb{S}^n}$ by imposing that $h'|_{\mathbb{R}_{\geq 0}}\geq 0$ and $h''<0$ whenever the ``twisting speed'' $h'$ exceeds some tolerance $0<\varepsilon<1/2$. Then one can deduce where the twisted fibre will intersect (transversely) another non-twisted one; this will happen exactly at the zero section, precisely as shown in Figure \ref{Dehn}, if the two fibers have antipodal footpoints (see \cite[Lemma 1.9]{[Sei02]}).

We transport the theory to the manifold world.

\begin{Def}											
	Let $M^{2n}=(M,\omega)$ be a symplectic manifold. A compact Lagrangian submanifold $S^n\subset M$ is called a \textbf{Lagrangian sphere}\index{Lagrangian sphere} if there exists a diffeomorphism $f:\mathbb{S}^n\rightarrow S$. It is \textit{framed} if there is a class $[f]$ of such diffeomorphisms, where $f_1\sim f_2$ if and only if $f_2^{-1}\circ f_1$ can be continuously deformed inside $\text{Diff}(\mathbb{S}^n)$ to some orthogonal matrix of $\text{O}(n+1)$.
	
	We can associate a Dehn twist $\tau_S\in\text{Symp}(M,\omega)$ to any framed Lagrangian sphere $S=(S,[f])$. Given some $\lambda>0$, choose any representative map\break $f:\mathbb{S}^n\equiv T(0)\rightarrow S$ and extend it to $T(\lambda)$, which embeds symplectically into $M$ via $\iota: T(\lambda)\hookrightarrow M$, and construct the Dehn twist \eqref{dehntwist} supported inside $T(\lambda)$. The \textbf{Dehn twist along $S$}\index{Dehn twist!along a Lagrangian sphere} is\footnote{The reader should not be fooled by the notation adopted: it is currently unknown whether $[\tau_S]\in\pi_0(\text{Symp}(M,\omega))$ depends on the specific class $[f]$ or not, although all other choices made in the construction are in fact irrelevant. For practicity, we avoid mention of the specific framing $[f]$, and implicitly consider all Lagrangian spheres to be framed.} then:
	\begin{equation}
		\tau_S(x) \coloneqq 
		\begin{cases}
			\iota\circ\tau_{\,\mathbb{S}^n}\circ\iota^{-1}(x) & $if $ x\in \text{im}(\iota) \\ 
			x & $else $
		\end{cases}\;.
	\end{equation}
	(Concretely, $\tau_S$ is the unique possible extension of $\phi_H^{2\pi}:M\setminus S\rightarrow M\setminus S$ over $S$, where $H\coloneqq h\circ\varrho\circ\iota^{-1}|_{\text{im}(\iota)\setminus S}$ is zero on $M\setminus\text{im}(\iota)$.) In particular, $\tau_S$ maps $S$ to itself antipodally.   
\end{Def}

The last definition made implicit use of the Weinstein's Lagrangian Neighbourhood Theorem (\cite[Theorem 8.4]{[Sil12]}): any neighbourhood of $S\subset M$ is symplectomorphic to a neighbourhood $T(\lambda)$ of $\mathbb{S}^n\equiv T(0)\subset T^*\mathbb{S}^n$ (we exploited this also back in Proposition \ref{CFLL}).

Being a symplectomorphism, $\tau_S$ maps Lagrangian submanifolds to Lagrangian submanifolds, preserving all their decorations (non-trivial!). In other words, $L\in\textup{obj}(\mathscr{F}(M))$ implies $\tau_S(L)\in\textup{obj}(\mathscr{F}(M))$. The main result we are interested in is:

\begin{Thm}\label{Dehncone}							
	Let $M$ be a symplectic manifold with associated Fukaya category $\mathscr{F}(M)$. Let $S\in\textup{obj}(\mathscr{F}(M))$ be a Lagrangian sphere and $L\in\textup{obj}(\mathscr{F}(M))$ any Lagrangian submanifold. Then 
	\begin{equation}
	\tau_S(L)\cong T_S(L)=Cone(\textup{ev})\in\textup{obj}\big(\mathsf{D^b}(\mathscr{F}(M))\big)\,,
	\end{equation}
	where $T_S(L)$ is the algebraic twist of $L$ along $S$, coinciding with the mapping cone of $\textup{ev}: CF(S,L)\otimes S\rightarrow L$ as in Definition \ref{evalmorph}. Therefore, there is an exact triangle
	\begin{equation}\label{dehnexacttriangle}
	\begin{tikzcd}
		HF(S,L)\otimes S\arrow[rr, "\langle \textup{ev}\rangle"] & & L\arrow[dl, "\langle c\rangle"] \\
		& \tau_S(L)\arrow[ul, "{\langle \pi\rangle[1]}"]
	\end{tikzcd}
	\end{equation}
	in $\mathsf{D^b}(\mathscr{F}(M))$. Consequently, application of $HF(\tilde{L},\square)$ respectively $HF(\square,\tilde{L})$ yields long exact sequences
	\begin{align}
	&...\rightarrow (HF(S,L)\otimes HF(\tilde{L},S))^n\xrightarrow{\langle\mu^2\rangle} HF^n(\tilde{L},L)\rightarrow HF^n(\tilde{L},\tau_S(L))\xrightarrow{[1]}... \nonumber \\
	&... \rightarrow HF^n(L,\tilde{L})\rightarrow (HF(S,L)\otimes HF(S,\tilde{L}))^n\xrightarrow{\langle d\rangle[1]} HF^{n+1}(\tau_S(L),\tilde{L})\rightarrow...\nonumber
	\end{align}
	for any $\tilde{L}\in\textup{obj}(\mathscr{F}(M))$, where $\mu^2$ is the Floer product and $\langle d\rangle$ is a composition involving the canonical isomorphism $HF(L,S)\cong HF(\tau_S(L),\tau_S(S))=HF(\tau_S(L),S)$ and the Donaldson ``pair-of-pants'' product.
\end{Thm}

\begin{proof}(\textit{Sketch}) 
	We just give a very vague outline of how to proceed --- a precise proof is the scope of the whole paper \cite{[Sei02]} by Seidel. Goal is to work first at the $A_\infty$-module level in order to apply Lemma \ref{detecttwist}. Relabeling $L_0\coloneqq S$, $L_1\coloneqq L$ for consistency and choosing $L_2\in\textup{obj}(\mathscr{F}(M))$ to be isotopic to $\tau_{L_0}(L_1)$, one defines a Floer cocycle $c\in CF(L_1,L_2)$ (thus $\mu^1(c)=0$) and a chain homotopy $k:CF(L_0,L_1)\rightarrow CF(L_0,L_2)[-1]$ between $\mu^2(c,\cdot):CF(L_0,L_1)\rightarrow CF(L_0,L_2)$ and the zero map: $\mu^1\circ k+ k\circ\mu^1 + \mu^2(c,\cdot)=0$. 
	
	After producing the suitable additional data required by \cite[Lemma 5.3]{[Sei08]} (this is done throughout sections (17d)--(17f) of \cite{[Sei08]}), Lemma \ref{detecttwist} yields a well-defined module isomorphism $t:\mathcal{T}_{L_0}(\Upsilon(L_1))\rightarrow \Upsilon(L_2)$.	One can improve this result to take into account signs, gradings and brane structures \cite[Theorem 17.16]{[Sei08]}. Therefore, $\mathcal{T}_{L_0}(\Upsilon(L_1))\cong \Upsilon(L_2)\in\textup{obj}\big(H^0(\mathscr{Q}(M))\big)$.
	
	Now, we move to twisted complexes by applying the isomorphisms $\mathcal{T}_{Y_0}(\tilde{\Upsilon}(Y_1))$ $\cong\tilde{\Upsilon}(T_{Y_0}(Y_1))$ of the second bullet point in Remark \ref{abouttwists}. Letting without loss of generality $L_2=\tau_{L_0}(L_1)\in\textup{obj}(Tw\mathscr{F}(M))$, we obtain: 
	\[
	\tilde{\Upsilon}(T_S(L))=\tilde{\Upsilon}(T_{L_0}(L_1))\cong\mathcal{T}_{L_0}(\tilde{\Upsilon}(L_1))\cong\tilde{\Upsilon}(\tau_{L_0}(L_1))=\tilde{\Upsilon}(\tau_S(L))\,,
	\]
	whence $T_S(L)\cong\tau_S(L)\in\textup{obj}\big(\mathsf{D^b}(\mathscr{F}(M))\big)$ (recall that $\tilde{\Upsilon}$ is a quasi-equivalence).
	
	It follows immediately from diagram \eqref{Ftwisttriangle} that $\tau_S(L)$ fits into the exact triangle \eqref{dehnexacttriangle}, where we recognize $c:L\rightarrow\tau_S(L)$ to be precisely the Floer cocycle constructed above, while $\langle\pi\rangle\in\textup{Hom}_{\mathsf{D^b}(\mathscr{F}(M))}\big(\tau_S(L),(HF(S,L)\otimes S)[1]\big)\cong\big(HF(S,L)\otimes HF(\tau_S(L),S)\big)^1\cong\big(HF(S,L)\otimes HF(L,S)\big)^n$. Moreover, one easily checks that indeed $HF(\tilde{L},\langle \textup{ev}\rangle)=\langle\mu^2\rangle$. For the exact definition of $d$, we point again the reader to \cite{[Sei02]}.
\end{proof}

Recalling that twisted mapping cones quasi-represent their abstract counterpart through the Yoneda embedding, from the proof of Theorem \ref{Dehncone} we also deduce:

\begin{Cor}										
	Let $M$ be a symplectic manifold with associated Fukaya category $\mathscr{F}(M)$. Let $S\in\textup{obj}(\mathscr{F}(M))$ be a Lagrangian sphere and $L\in\textup{obj}(\mathscr{F}(M))$ any Lagrangian submanifold. Then the Dehn twist $\tau_S(L)\in\textup{obj}\big(\mathsf{D^b}(\mathscr{F}(M))\big)$ of $L$ along $S$ quasi-represents the abstract algebraic twist $\mathcal{T}_S(\Upsilon(L))\in\textup{obj}(\mathscr{Q}(M))$ of the Yoneda module $\Upsilon(L)\in\textup{obj}(\mathscr{Q}(M))$ along $S$.  
\end{Cor}

As the reader may have suspected by now, the underlying theory is \textit{way} deeper than what we touched upon --- it revolves around Lefschetz fibrations, which are among the fundamental objects of interest of Picard--Lefschetz theory. Unfortunately, a thorough study falls out of our reach; we refer the hungry reader to the third part of \cite{[Sei08]}.
\vspace*{0.5cm}

\noindent A second geometrically meaningful aspect of Fukaya categories is the so-called \textit{Lagrangian surgery}, whose construction is outlined in \cite[section 10.54]{[FOOO07]}. First, another piece of technicality.

\begin{Def}\label{technical2}							
	For $\varepsilon\in\mathbb{R}$ with $|\varepsilon|>0$ small enough, consider the smooth function $f_\varepsilon:\mathbb{R}^n\setminus\{0\}\rightarrow\mathbb{R},\,f_\varepsilon(x)\coloneqq \varepsilon \log(|\!|x|\!|)$ and the 1-form it defines:
	\[
	df_\varepsilon(x)=\varepsilon\frac{x\cdot dx}{|\!|x|\!|^2}=\sum_{j=1}^ny_jdx_j\in T_x^*(\mathbb{R}^n\setminus\{0\})\,,\quad\text{with }y_j\coloneqq \varepsilon\frac{x_j}{|\!|x|\!|^2}\in\mathbb{R}\;. 
	\]
	Then $H_\varepsilon\coloneqq \text{graph}(df_\varepsilon)=\{(x_1,y_1,...,x_n,y_n)\mid x\in\mathbb{R}^n\setminus\{0\}\}\subset T^*(\mathbb{R}^n\setminus\{0\})\cong\mathbb{C}^n\setminus(i\mathbb{R})^n$ is a Lagrangian submanifold which tends asymptotically to $(i\mathbb{R})^n$ (the cotangent fibre over $0\in\mathbb{R}^n$) as $|\!|x|\!|\rightarrow 0$, and to the zero section $\mathbb{R}^n$ as $|\!|x|\!|\rightarrow\infty$, where the identification with the complex space follows by setting $z_j\coloneqq x_j+iy_j\in\mathbb{C}$.
	
	Observe that, denoting by $r:\mathbb{C}^n\rightarrow\mathbb{C}^n$ the reflection along the diagonal $(x_j+iy_j\mapsto y_j+ix_j)$ for all $j=1,...,n$, it holds $r(H_\varepsilon)=H_\varepsilon$ (indeed, symmetry of the construction implies that $H_\varepsilon$ is also the mirrored graph consisting of all $(y_1,x_1,...,y_n,x_n)\in T^*(\mathbb{R}^n\setminus\{0\})$ such that $y\in\mathbb{R}^n\setminus\{0\}$ and $x_j=\varepsilon y_j/|\!|y|\!|^2$).
	
	Now take a corrective smooth function $g\in C^\infty(\mathbb{R}_{>0})$ such that $g(x)=\log(x)-|\varepsilon|$ if $0<x\leq x_0$, where $x_0\equiv x_0(|\varepsilon|)$ is small enough, and is constant if $x\geq 2x_0$, with $g'\geq 0$ and $g''\leq 0$ (compare this with the map $h$ of Lemma \ref{technical}). We define the adjusted smooth function 
	\[
	\tilde{f}_\varepsilon:\mathbb{R}^n\setminus\{0\}\rightarrow\mathbb{R},\,\tilde{f}_\varepsilon(x)\coloneqq \varepsilon g(|\!|x|\!|) = 
	\begin{cases}
		f_\varepsilon(x)-\varepsilon|\varepsilon| & \text{if } 0<x\leq x_0 \\ 
		\text{constant} & \text{else}
	\end{cases}\;,
	\]
	and look at its graph $\tilde{H}_\varepsilon\coloneqq \text{graph}(d\tilde{f}_\varepsilon)\subset T^*(\mathbb{R}^n\setminus\{0\})$, again a Lagrangian submanifold invariant under reflection by $r$. In particular, $\tilde{H}_\varepsilon=\mathbb{R}^n\cup(i\mathbb{R})^n$ outside the $n$-dimensional complex ball $B_{2x_0}^n(0)\subset\mathbb{C}^n$ at 0 of radius $2x_0$.  
\end{Def}

\begin{Def}											
	Let $M^{2n}=(M,\omega)$ be a symplectic manifold with associated Fukaya category $\mathscr{F}(M)$. Let $L_1, L_2\in\textup{obj}(\mathscr{F}(M))$ be transversely intersecting at some $p\in M$. With a suitable Darboux chart\footnote{This is a chart of local coordinates $(x^j,y^j)_{j=1}^n$ such that $\omega|_U=\sum_{j=1}^ndx^j\wedge dy^j$, always existing by \cite[Theorem 8.1]{[Sil12]}.} $\psi:U\subset M\rightarrow V\subset\mathbb{R}^{2n}\cong \mathbb{C}^n$ on a neighbourhood $U$ of $p$, centered at $p$ (so $\psi(p)=0$), we can arrange the two Lagrangian submanifolds in such a way that $\psi(L_1\cap U)=\mathbb{R}^n\cap V$ and $\psi(L_2\cap U)= (i\mathbb{R})^n\cap V$.
	
	Now construct $\tilde{H}_\varepsilon$ as in Definition \ref{technical2}, adapting $\varepsilon$ so that $B_{2x_0}^n(0)\subset\psi(U)$. Then
	\begin{equation}
		L_1\#_\varepsilon L_2 \coloneqq 
		\begin{cases}
			(L_1\cup L_2)\setminus U & \text{on } M\setminus U \\ 
			\psi^{-1}(\tilde{H}_\varepsilon)\cap U & \text{on } U
		\end{cases}\;,
	\end{equation}
	is the \textbf{Lagrangian connected sum}\index{Lagrangian connected sum} (or \textbf{Lagrangian surgery}\index{Lagrangian surgery}) of $L_1$ and $L_2$ at $p$ (sometimes just shortened $L_1\# L_2$). It is a Lagrangian submanifold of $M$ inheriting all the desired decorations (non-trivial!), meaning that $L_1\#_\varepsilon L_2\in\textup{obj}(\mathscr{F}(M))$. The region $(L_1\#_\varepsilon L_2)\cap U=\psi^{-1}(\tilde{H}_\varepsilon)\cap U$ is called \textit{Lagrangian handle}.	
\end{Def}

The 2-dimensional case is shown in Figure \ref{surgery}, while the 4-dimensional is discussed for example in \cite[section 3]{[AS09]}.

From the above constructions, it follows that $L_1\#_\varepsilon L_2 = L_2\#_{-\varepsilon} L_1$ (\textit{not} $L_2\#_{\varepsilon} L_1$, which is not even isotopic to $L_1\#_\varepsilon L_2$!). Moreover, $L_1\#_\varepsilon L_2$ is Hamiltonian isotopic to some $L_1\#_\delta L_2$ as soon as $\text{sgn}(\varepsilon)=\text{sgn}(\delta)$ (this can be seen by working in a suitable tubular neighbourhood of the Lagrangian handle). However, if we perform surgeries at multiple distinct points, the end result loses isotopy invariance. 

\begin{figure}[htp]
	\centering
	\includegraphics[width=0.7\textwidth]{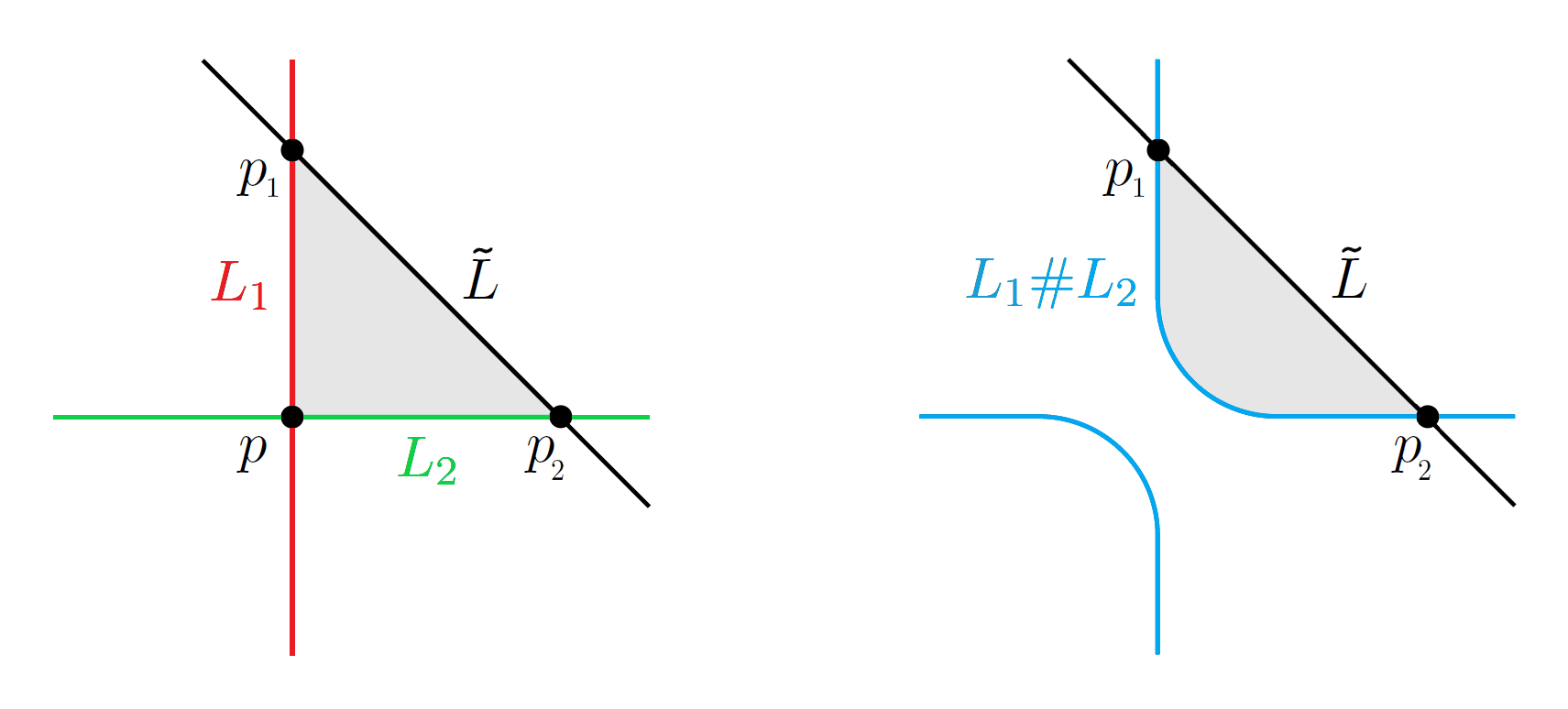}
	\caption{The Lagrangian connected sum $L_1\# L_2$ of two 1-dimensional Lagrangian submanifolds $L_1,L_2\subset M^2$ transversely intersecting at $p\in M$, which are mapped by a suitable Darboux chart to the real respectively imaginary axis in a neighbourhood of 0 in $\mathbb{C}$. Here the ``curved region'' of $L_1\# L_2$ is what we call Lagrangian handle. [Source:\cite{[Aur13]}]}
	\label{surgery}	
\end{figure}

\begin{Rem}\label{twistequivsurg}				
	Take $L_2=S\in\textup{obj}(\mathscr{F}(M))$ to be a Lagrangian sphere transversely intersecting $L_1$ at a single point $p$. Then $L_1\# S$ is Hamiltonian isotopic to the Dehn twist $\tau_S(L_1)$, 
	\begin{equation}
		L_1\# S\simeq\tau_S(L_1)
	\end{equation}
	(compare to Figure \ref{Dehn}). Therefore, $S$ regulates the twisting, while $L_1$ governs the asymptotical behaviour. In the 2-dimensional case, the local Darboux picture of $L_1$ corresponds to the imaginary line $i\mathbb{R}$ and that of $S$ to $\mathbb{S}^1$ centered at $1\in\mathbb{C}$. Then the Lagrangian surgery amounts to smoothing the angular regions squeezed between $\mathbb{S}^1$ and $i\mathbb{R}$ about 0. This can be taken informally as the general recipe to produce Lagrangian connected sums.   
\end{Rem}

Similarly to Dehn twists, Lagrangian connected sums are closely related to twisted mapping cones:

\begin{Thm}\label{surgerycone}				
	Let $L_1,L_2\in\textup{obj}(\mathscr{F}(M))$ intersect transversely at a single $p\in M$, thus forcedly\footnote{This is immediate by definition of Floer's complex and differential: $\mu^1$ is of degree 1, and therefore cannot preserve the only generator $p\in\mathcal{X}(L_2,L_1)$.} a Floer cocycle of $CF(L_2,L_1)$. Consider a third $\tilde{L}\in\textup{obj}(\mathscr{F}(M))$ with intersections $\tilde{L}\cap(L_1\#_\varepsilon L_2) = (\tilde{L}\cap L_1) \cup (\tilde{L}\cap L_2)$ \textup(possible for $|\varepsilon|$ small enough; cf. Figure \ref{surgery}\textup). Write $p_i$ for the transverse intersection point in $\tilde{L}\cap L_i$. Then for a suitable almost complex structure $J\in\mathfrak{J}(M,\omega)$ on the region bounded by $\tilde{L}$ and $L_1\#_\varepsilon L_2$, we have the bijective correspondence
	\[
	\left\{
	\begin{aligned}
		& J\text{-holomorphic strips from} \\
		& \text{$p_2$ to $p_1$ with boundary} \\
		& \text{on $\tilde{L}$ and $L_1\#_\varepsilon L_2$}
	\end{aligned} 
	\right\} \longleftrightarrow
	\left\{
	\begin{aligned}
		& J\text{-holomorphic triangles from} \\
		& \text{$p_2$ to $p_1$ to $p$ with boundary} \\
		& \text{on $\tilde{L}$, $L_1$ and $L_2$}
	\end{aligned}
	\right\}\;.
	\] 
	As a consequence, $CF(\tilde{L},L_1\#_\varepsilon L_2)$ is the standard mapping cone of $\mu^2(p,\cdot):CF(\tilde{L},L_2)\rightarrow CF(\tilde{L},L_1)$ \textup(in the chain complex sense\textup), thus fitting into the short exact sequence
	\[
	\{0\}\rightarrow CF(\tilde{L},L_1)\rightarrow CF(\tilde{L},L_1\#_\varepsilon L_2)\rightarrow CF(\tilde{L},L_2)\rightarrow \{0\}
	\] 
	inducing the long exact sequence
	\begin{equation}\label{surgexactseq}
		...\rightarrow HF^n(\tilde{L},L_1)\rightarrow HF^n(\tilde{L},L_1\#_\varepsilon L_2)\rightarrow HF^n(\tilde{L},L_2)\xrightarrow{\langle\mu^2(p,\cdot)\rangle[1]}...\quad. 
	\end{equation}
	In fact, there is an exact triangle
	\begin{equation}\label{surgexacttriangle}
		\begin{tikzcd}
			L_2\arrow[rr, "\langle p\rangle"] & & L_1\arrow[dl, "\langle i\rangle"] \\
			& L_1\#_\varepsilon L_2\arrow[ul, "{\langle \pi\rangle[1]}"]
		\end{tikzcd}
	\end{equation}
	in $\mathsf{D^b}(\mathscr{F}(M))$, meaning that 
	\begin{equation}
	L_1\#_\varepsilon L_2\cong Cone(p)\in\textup{obj}\big(\mathsf{D^b}(\mathscr{F}(M))\big)\,.
	\end{equation}
\end{Thm}

A proof of this result is developed throughout \cite[sections 10.59--10.62]{[FOOO07]} (where it is referred to as ``Theorem Z''). In order to study the moduli space $\mathcal{M}(p_2,p_1;[u],J,H)$ associated to $J$-holomorphic strips bounded by $L_1\#_\varepsilon L_2$ and $\tilde{L}$, this proof heavily relies on the analytical tools we occasionally mentioned in Chapter 5, as well as gluing techniques, exploiting the local construction in the complex space $\mathbb{C}^n$.

\begin{Rem}											
	Take again $L_2=S\in\textup{obj}(\mathscr{F}(M))$ to be a Lagrangian sphere so that $L_1\# S\simeq\tau_S(L_1)$. The claims of Theorem \ref{surgerycone} are then compatible with those of Theorem \ref{Dehncone}. Indeed, the long exact sequence induced by \eqref{dehnexacttriangle} is just \eqref{surgexactseq} after fixing the input $\langle p\rangle\in HF(S,L_1)$ for the connecting homomorphism $\langle\mu^2\rangle:HF(S,L_1)\otimes HF(\tilde{L},S)\rightarrow HF(\tilde{L},L_1)$ (and renaming $L\coloneqq L_1$). Similarly, diagram \eqref{dehnexacttriangle} takes the shape of \eqref{surgexacttriangle}.
\end{Rem}

To conclude, we warn that the identifiability of Dehn twists and Lagrangian connected sums with twisted mapping cones is two-edged: the advantage of such a dictionary between algebraic and geometrical objects is partially compensated by the necessity to deal with higher order Floer composition maps $\mu^d$ --- think about the local picture of Figure \ref{surgery} when several surgeries are performed. Consequently, studying the Lagrangian constituents unavoidably requires us to enhance the analytical tools at disposal. This is beyond our scope; we venture no further.
\vspace*{0.5cm}

\noindent Finally, some heuristics about the importance of the split-closure $\mathsf{D}^\pi(\mathscr{F}(M))$. Given finitely many Lagrangian submanifolds $L_i\in\textup{obj}(\mathscr{F}(M))$ which generate an $A_\infty$-subcategory $\mathcal{A}\subset\mathscr{F}(M)$, the $A_\infty$-category $Tw\mathcal{A}$ contains all possible mapping cones which can be iteratively constructed from them, in particular Dehn twists and Lagrangian connected sums. Surgery of all starting objects may produce a Lagrangian submanifold which is disconnected. Though an element of $\mathcal{A}$, its components may no longer be as much. Therefore, it is conjectured that taking the idempotent completion of $Tw\mathcal{A}$ could include the individual constituents as objects of $\mathsf{D}^\pi(\mathcal{A})$, solving this deficiency. 
This is particularly true in the case of the 2-torus:

\begin{Ex}\label{2torus}						
	Consider the \textbf{symplectic 2-torus}\index{symplectic 2-torus} $\mathbb{T}^2=\mathbb{S}^1\times\mathbb{S}^1$ equipped with the standard area form $\omega \coloneqq d\theta_1\wedge d\theta_2$, where $\theta_1$ and $\theta_2$ are the angle coordinates\footnote{Observe that $\theta_j(e^{2\pi it})\coloneqq t\in(0,1)$ is well defined only on $\mathbb{S}^1\setminus\{1\}$, not globally, so that $d\theta_j$ does \textit{not} denote the 1-form obtained as exterior derivative of $\theta_j$!} on the respective factors $\mathbb{S}^1\subset\mathbb{C}$. Objects of $\mathscr{F}(\mathbb{T}^2)$, our decorated Lagrangian submanifolds, are all simple loops (that is, non-self-intersecting closed paths) on $\mathbb{T}^2$, identifiable as class representatives of $\mathbb{Z}\oplus\mathbb{Z}=\pi_1(\mathbb{T}^2)$ and determined (up to homotopy) by their rational slope in the quotient space representation of $\mathbb{T}^2$. 
	
	Fix a preferred longitude $L_1\coloneqq \alpha\in\textup{obj}(\mathscr{F}(\mathbb{T}^2))$ and meridian $L_2\coloneqq \beta\in\textup{obj}(\mathscr{F}(\mathbb{T}^2))$ along the two factors of the identification polygon of $\mathbb{T}^2$ (Figure \ref{firsttorus} top left); these belong to the classes $(1,0)$ respectively $(0,1)$ of $\pi_1(\mathbb{T}^2)$. Call $p\in CF(L_2,L_1)$ the unique intersection point of $\alpha$ and $\beta$. Since $\beta$ is clearly diffeomorphic to $\mathbb{S}^1$, it is a Lagrangian sphere, so that we can construct the Dehn twist of $\alpha$ along it. According to Remark \ref{twistequivsurg} and Theorem \ref{surgerycone}, this yields a new simple loop $\tau_\beta(\alpha)\simeq\alpha\#\beta\cong Cone(p)\in\textup{obj}\big(\mathsf{D^b}(\mathscr{F}(\mathbb{T}^2))\big)$, depicted in Figure \ref{firsttorus} top right. The Dehn twisting procedure can be iterated indefinitely (this amounts to repeatedly take algebraic twists as discussed in Section \ref{ch4.4}), applying also $\tau_\alpha$. 
	
	\begin{figure}[htp]
		\centering
		\includegraphics[width=0.45\textwidth]{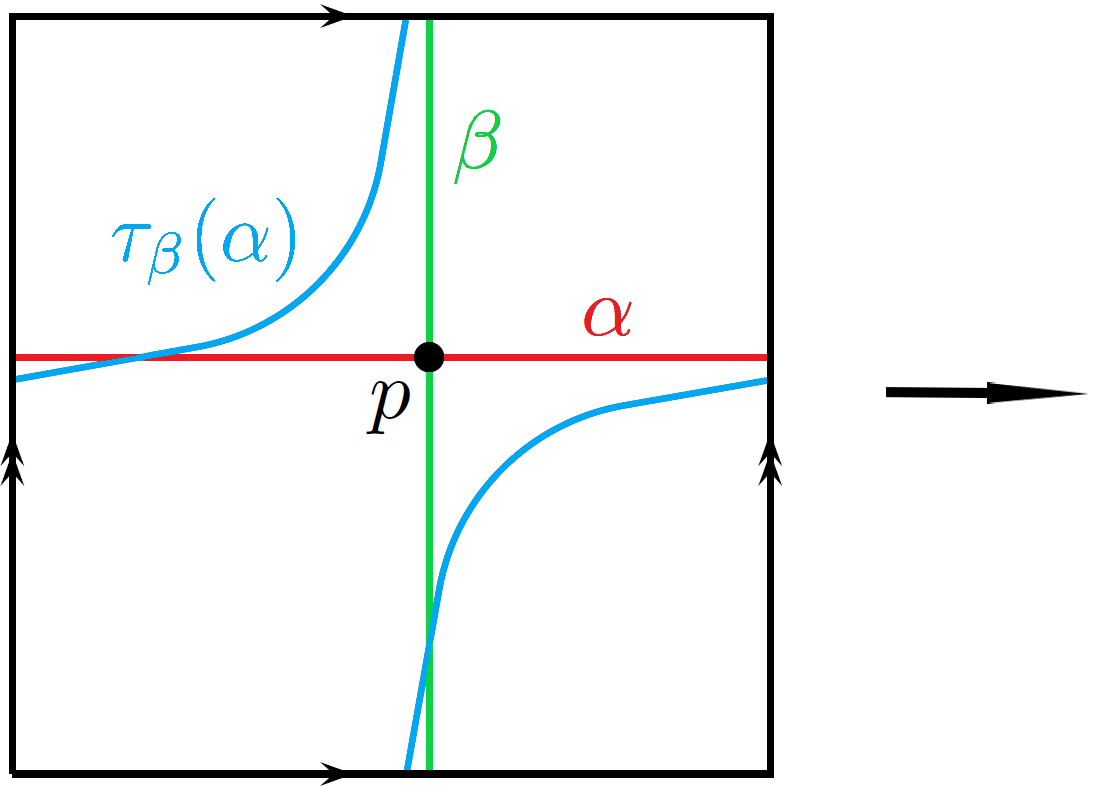}\quad
		\includegraphics[width=0.45\textwidth]{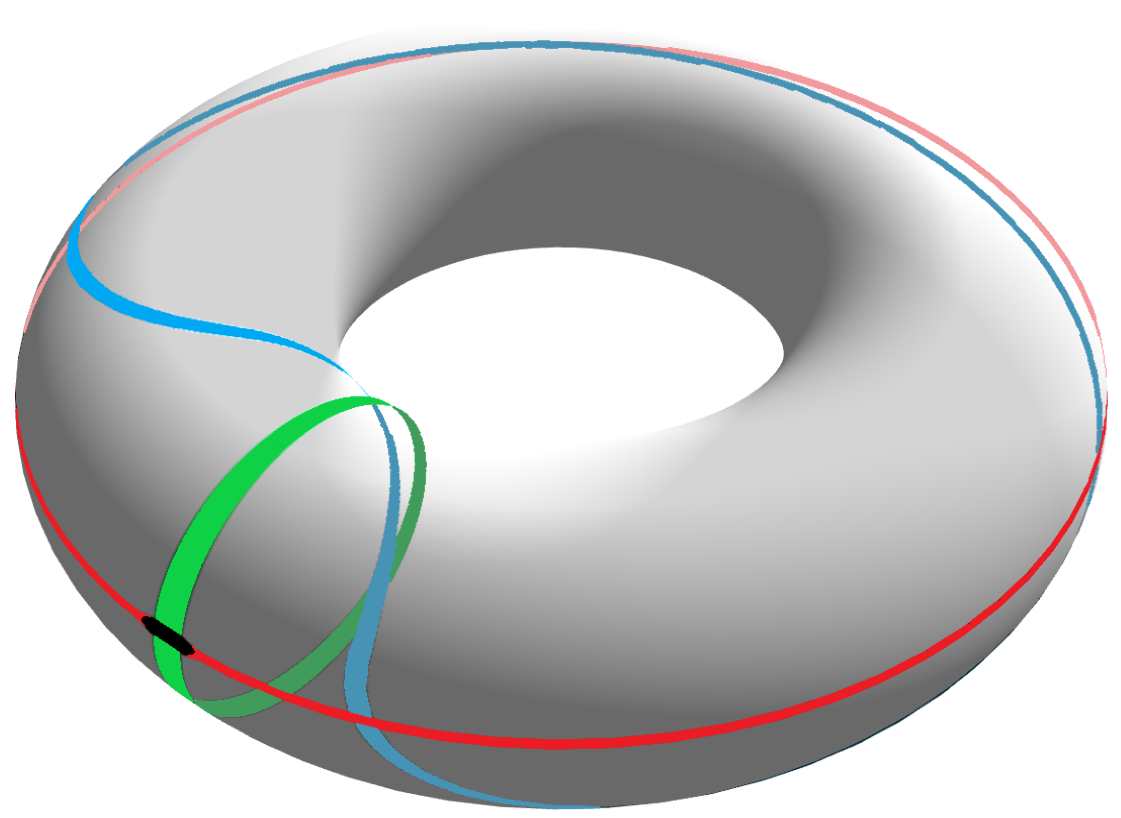}
		
		\vspace{0.5cm}
		
		\includegraphics[width=0.45\textwidth]{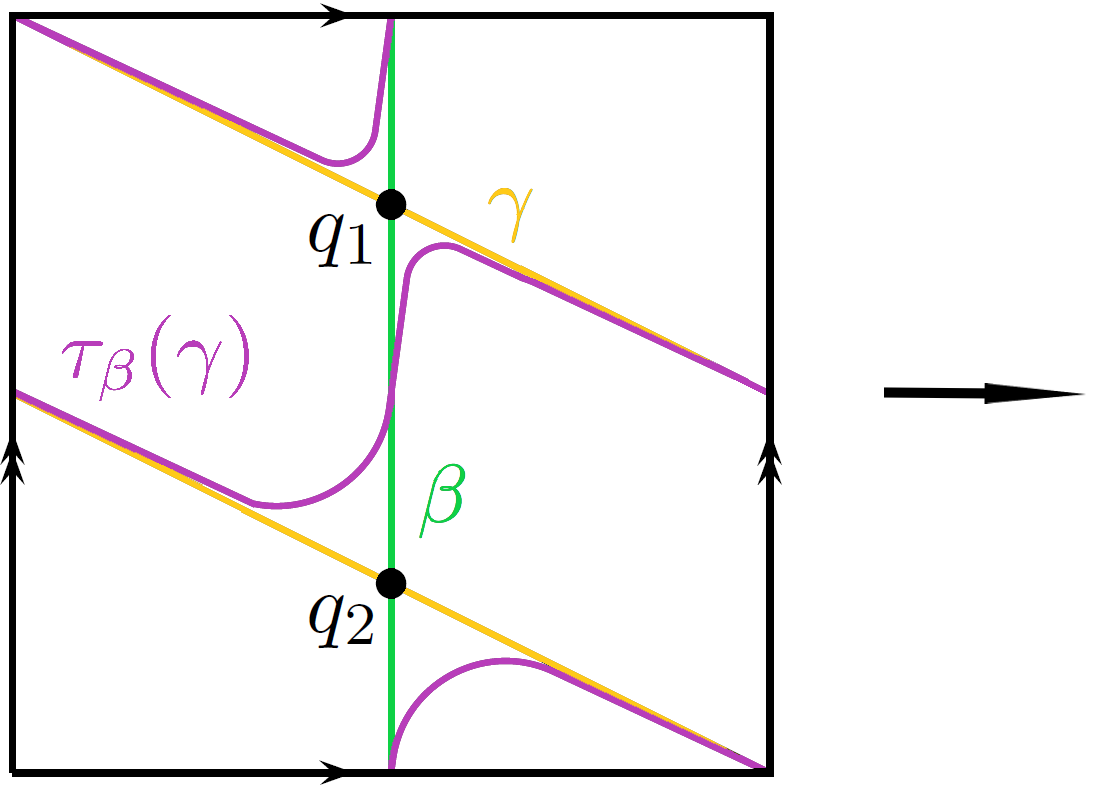}\quad
		\includegraphics[width=0.45\textwidth]{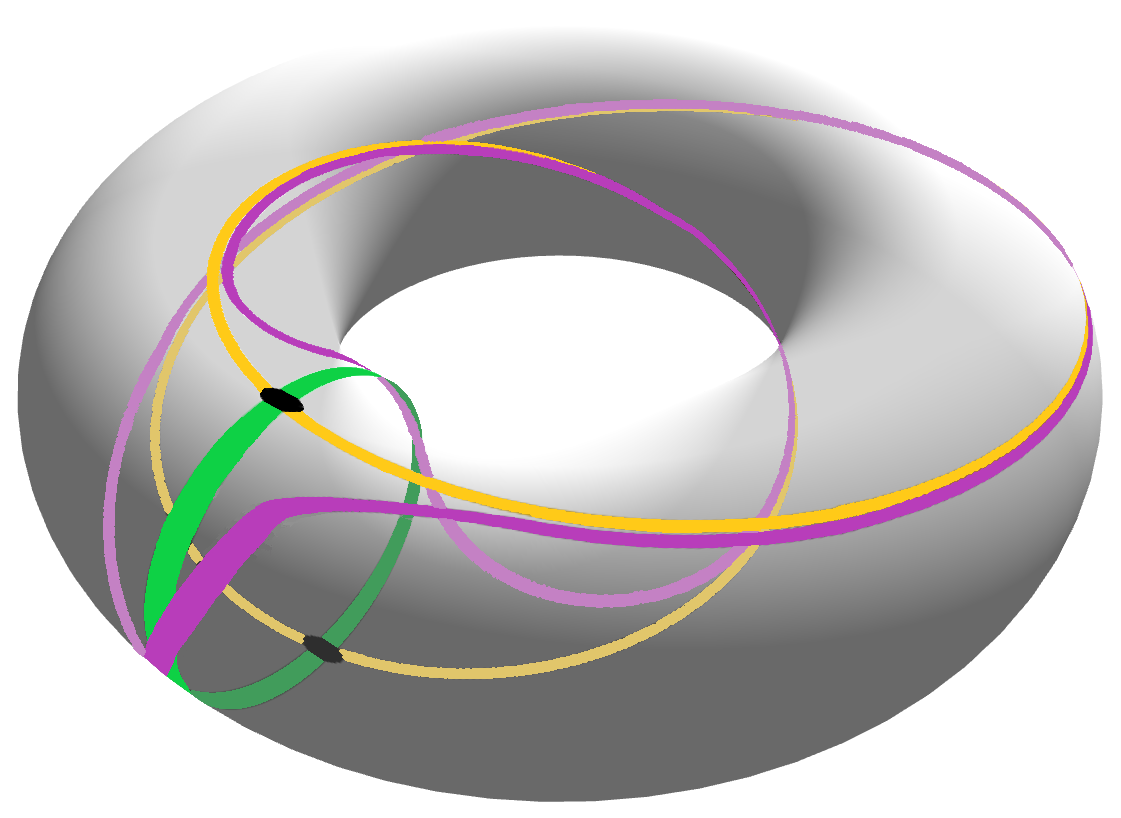}
		\caption{The Lagrangian connected sum of two simple loops possibly yields a  direct sum of unbalanced simple loops (bottom row). These are genuine objects of the split-closed derived category $\mathsf{D}^\pi(\mathscr{F}(\mathbb{T}^2))$.
		[Source: \cite{[Aur13]} (identification polygons); special thanks to S. Bone for creating these gorgeous tori!]}
		\label{firsttorus}	
	\end{figure}
	
	Since we will come across all homotopy classes of $\pi_1(\mathbb{T}^2)$ --- indeed, $\tau_\alpha$, $\tau_\beta\in\text{Diff}(\mathbb{T}^2)$ generate the mapping class group of $\mathbb{T}^2$, consisting of ambient isotopy classes of its diffeomorphisms --- one would legitimately expect to obtain all objects of $\mathscr{F}(\mathbb{T}^2)$; however, this is far from true: let $\nu\in\Omega^1(\mathbb{T}^2\setminus\{p\})$ be such that $d\nu=\omega$ (this is possible on the punctured torus!) and $\int_\alpha \nu =\int_\beta\nu = 0$, which means $\alpha,\beta\subset\mathbb{T}^2$ are \textit{balanced} (or exact) Lagrangian submanifolds with $\nu|_\alpha = d\theta_\alpha$ and $\nu|_\beta = d\theta_\beta$ (so that indeed $\int_\alpha \nu =\int_\alpha d\theta_\alpha=\theta_\alpha(e^{2\pi i})-\theta_\alpha(e^0)=0$, and similarly on $\beta$). Then also all the successive mapping cones constructed above will be exact: looking at $\tau_\beta(\alpha)$ in Figure \ref{firsttorus} top left, we perceive that its bijectivity on $[0,1]\times[0,1]$ causes the Dehn twisted Lagrangian submanifold it describes to be exact --- put differently, if we follow along the corresponding simple loop, we see it revolves around each factor exactly once (indeed, $\tau_\beta(\alpha)$ represents $(1,1)\in\pi_1(\mathbb{T}^2)$). This is manifestly a constraint we need to remove, as not all simple loops are exact!
	
	For example, consider a third simple loop $L_3\coloneqq \gamma\in\textup{obj}(\mathscr{F}(\mathbb{T}^2))$ of rational slope $-1/2$ (Figure \ref{firsttorus} bottom left), representing the class $(2,1)\in\pi_1(\mathbb{T}^2)$ (that of $\tau_\alpha(\tau_\alpha(\beta))$), which intersects $\beta$ twice, say at $q_1,q_2\in CF(L_2,L_3)$. Constructing the (double) connected sum of $\gamma$ with $\beta$ and carefully applying Theorem \ref{surgerycone}, it turns out that $\tau_\beta(\gamma)\simeq\gamma\#\beta\cong Cone(T^{w_1}q_1+T^{w_2}q_2)$ in $\mathsf{D^b}(\mathscr{F}(\mathbb{T}^2))$, for suitable weights $w_1,w_2\in\mathbb{R}$. The geometrical outcome is precisely the direct sum of the two purple, non-exact simple loops of Figure \ref{firsttorus} bottom right, which both lie in the homotopy class $(1,0)$ of $\alpha$. The mindful reader already knows the next move: we take the split-closure of $\mathscr{F}(\mathbb{T}^2)$, which enables us to include both summands of the mapping cone as objects of $\mathsf{D}^\pi(\mathscr{F}(\mathbb{T}^2))$.
	
	To summarize, we can obtain all simple loops of $\mathbb{T}^2$ (up to Hamiltonian isotopy) by iterated twist along $\alpha$ and $\beta$ and eventual inclusion of the arising direct summands. A far more systematic discussion is given in \cite{[LP11]}, with special focus on Lagrangian surgery and the arising exact triangles. Our take aims to provide a geometrical intuition behind the next result. \hfill $\blacklozenge$
\end{Ex}

\begin{Pro}\label{splittorus}					
	Any pair of longitude and meridian of the symplectic\break $2$-torus $\mathbb{T}^2=\mathbb{S}^1\times\mathbb{S}^1$ split-generates its Fukaya category $\mathscr{F}(\mathbb{T}^2)$. Consequently, any compact $2$-dimensional symplectic manifold $M^2$ diffeomorphic to $\mathbb{T}^2$ has Fukaya category $\mathscr{F}(M)$ split-generated isomorphically by any pair of longitude and meridian.
\end{Pro}

\begin{proof}(\textit{Sketch})
	We work directly on $M$. Let $S_1, S_2\in\textup{obj}(\mathscr{F}(M))$ be diffeomorphic to some fixed longitude respectively meridian as in previous example. Choose any homotopically essential closed curve (the manifold analogue of simple loops in $\mathbb{T}^2$). We want to show it lies in the split-closure of $\{S_1,S_2\}$. Then so will the whole Fukaya category $\mathscr{F}(M)$.
	
	Firstly, one can show that for all $k\geq 1$ holds $(\tau_{S_1}\circ\tau_{S_2})^{6k}=\mathds{1}\in SL(2,\mathbb{Z})$. We thus have a canonical morphism $i: L\rightarrow(\tau_{S_1}\circ\tau_{S_2})^{6k}(L)\cong L$ in $\mathsf{D^b}(\mathscr{F}(M))$ fitting into the diagram
	\[
	\begin{tikzcd}
		Y\arrow[rr] & & L\arrow[dl, "\langle i\rangle"] \\
		& (\tau_{S_1}\circ\tau_{S_2})^{6k}(L)\arrow[ul]
	\end{tikzcd}\quad,
	\]
	where by axiom (T1) for triangulated categories we can find some $Y\cong Cone(i)$ $\in\textup{obj}\big(\mathsf{D^b}(\mathscr{F}(M))\big)$ completing the exact triangle. Note that $\langle i\rangle\in HF^0(L,L)\cong H^0(L)$ is just a singular cohomology class, by Proposition \ref{CFLL}. If it vanishes, it follows immediately from the discussion of Remark \ref{t2rem} that $Cone(i)\cong L\oplus L$, making $L$ a direct summand which will therefore be included in the split-closure of $\{S_1,S_2\}$, as claimed.
	
	One can show that if $k=2$, then $\langle i\rangle = 0$. However the argument behind this is too convoluted for the limited set of tools we developed. We refer the reader to \cite[Proposition 6.1]{[AS09]} for an outline of the main strategy.   
\end{proof}

\subsection{The wrapped Fukaya category}\label{ch6.5}

To conclude this chapter, we discuss briefly, and quite vaguely, the wrapped version of Fukaya categories, which deals with \textit{exact} symplectic manifolds (recall the setting of Section \ref{ch5.8}). The main advantage is that we no longer need to impose compactness of the exact Lagrangian submanifolds, only a certain behaviour at infinity. We follow yet again the outline of \cite[section 4.1]{[Aur13]}, with an eye on \cite[section 2]{[Abo10]}. First, some more working tools (see \cite{[Sil12]}).

\begin{Def}										
	Let $M^n$ be a smooth manifold. A point $p\in M$ is called a \textit{contact point} if it has a tangent hyperplane $H_p\subset T_pM$, that is, a subspace with $\text{dim}(H_p)=n-1$. 
	
	A smooth section $H\in\Gamma(TM)$ assigning $p\mapsto H_p$ (a \textit{contact distribution}) can be locally written as $H=\text{ker}(\alpha)$ for some not unique $\alpha\in\Omega^1(M)$ such that $d\alpha|_H$ is a symplectic form. Then each $H_p$ is the kernel of a linear map $0\neq\alpha_p: T_p\rightarrow \mathbb{R}$. We call $\alpha$ a \textbf{contact form}\index{contact form} on $M$. Its non-degeneracy forces the dimension $n$ of $M$ to be odd.  
\end{Def}

\begin{Def}										
	An exact symplectic manifold $M^{2n}=(M,\omega)$ with $\omega=d\lambda\in\Omega^2(M)$ is called a \textbf{Liouville manifold}\index{Liouville manifold} if equipped with some complete $Z\in\mathfrak{X}(M)$ which is outward-pointing at infinity (see next paragraph) and fulfills $\iota_Z(\omega)=\lambda$, or equivalently, $\mathcal{L}_Z(\omega)=\omega$. Then $Z$ is the \textbf{Liouville vector field}\index{Liouville vector field} associated to the \textbf{Liouville form}\index{Liouville form} $\lambda\in\Omega^1(M)$.
\end{Def}

We are particularly interested in Liouville manifolds $(M^{2n},\omega=d\lambda)$ possessing some compact domain $M^{in}\subset M$ with the following properties (see Figure \ref{wrappedman}):
\begin{itemize}[leftmargin=0.5cm]
	\item the $(2n-1)$-dimensional boundary $\partial M^{in}$ is a smooth hypersurface endowed with contact form $\alpha\coloneqq \lambda|_{\partial M^{in}}\in\Omega^1(\partial M^{in})$,
	\item $\text{zeros}(Z)\subset M^{in}$ and $Z$ is positively transverse to $\partial M^{in}$, which means $Z(x)\notin T_x(\partial M^{in})$ and $Z(x)(x^1)>0$ (on a suitable half-space chart about $x$) for all $x\in\partial M^{in}$.  
\end{itemize}

The flow $\phi_Z^t:M\rightarrow M$ of $Z$ allows us to identify $M\setminus M^{in}$ with the symplectization $[1,\infty)\times\partial M^{in}$, where, given the radial coordinate function $r\in C^\infty([1,\infty))$, we have $Z=r\frac{\partial}{\partial r}$ and $\lambda = r\alpha$. By abuse of notation, we can informally write our Liouville manifold as the adjunction space 
\begin{equation}
M=M^{in}\sqcup_{\partial M^{in}} ([1,\infty) \times\partial M^{in})\,.
\end{equation}
For later use, we let $\psi_Z^{\rho}\coloneqq\phi_Z^{\log\rho}$ denote the time-$\log\rho$ flow of $Z$ for $\rho>1$ (observe that $\phi_Z^t$ is conformally symplectic, implying that $\psi_Z^\rho\in\text{Symp}(M,\omega)$).

\begin{figure}[htp]
	\centering
	\includegraphics[width=0.8\textwidth]{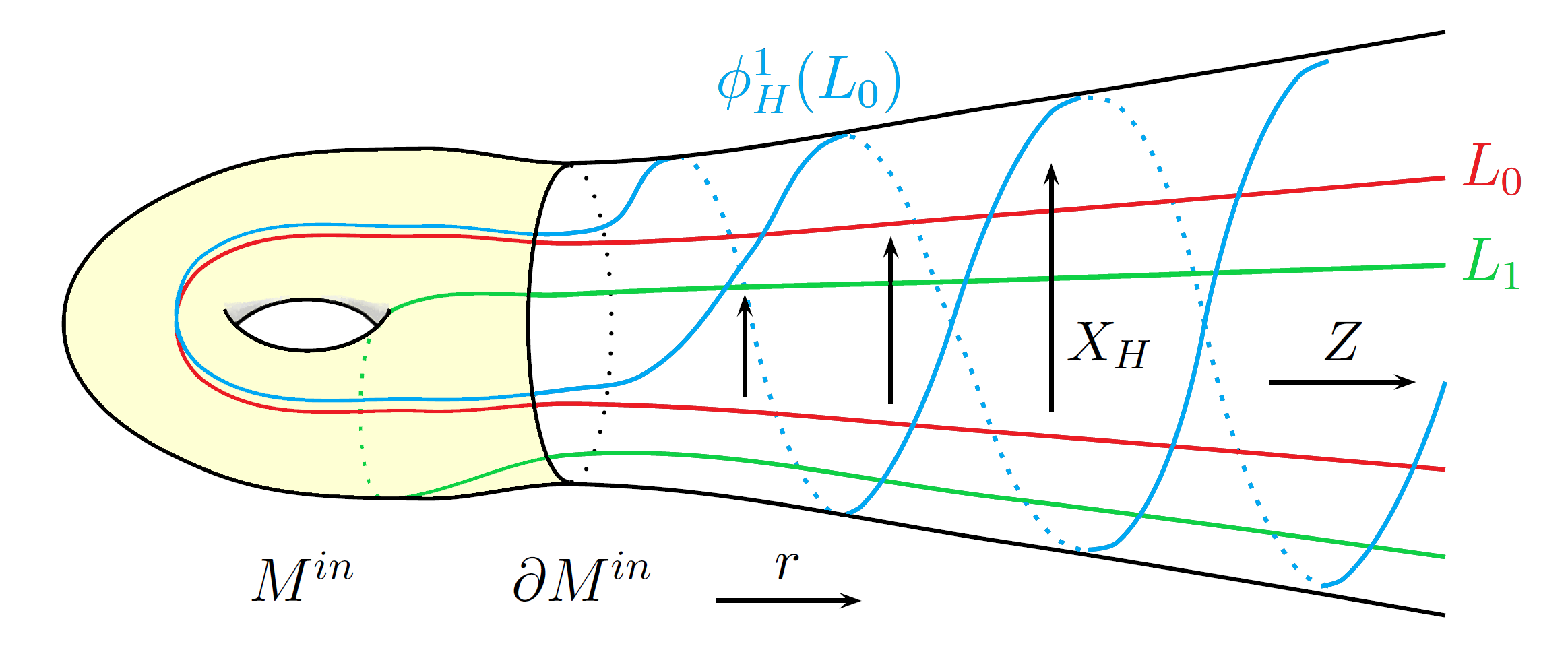}
	\caption{Exact Lagrangian submanifolds $L_0,L_1\in\textup{obj}(\mathscr{W}\!(M))$ of a Liouville 2-dimensional manifold $M$. The flow of the ``wrapping Hamiltonian'' $H\in C^\infty(M)$, quadratic outside of some compact subset of $M$, causes $L_0$ to spiral outwards in radial direction. [Source:\cite{[Aur13]}]}
	\label{wrappedman}	
\end{figure}

\noindent Given a Liouville manifold $(M^{2n},\omega=d\lambda)$ as just described, with $2c_1(TM)=0$, its \textbf{wrapped Fukaya category}\index{wrapped Fukaya category} $\mathscr{W}\!(M)$ is the $A_\infty$-category constructed as follows (refer to \cite{[Abo10]} for a more leisurely investigation):
\begin{itemize}[leftmargin=0.5cm]				
	\renewcommand{\labelitemi}{\textendash}
	\item Objects in $\textup{obj}(\mathscr{W}\!(M))$ are opportunely decorated (in the sense of Chapter 5) connected exact Lagrangian submanifolds 
	\begin{equation}\label{wrappedlagr}
		L=L^{in}\sqcup_{\partial L^{in}} ([1,\infty) \times\partial L^{in})\subset M
	\end{equation}
	(``conical at infinity'') with $L^{in}\subset M^{in}$, such that $\lambda$ vanishes on the region $[1,\infty)\times\partial L^{in}$ (whence $\partial L^{in}=L\cap\partial M^{in}$ is a Legendrian submanifold). Consequently, any such $L$ must be invariant under $\phi_Z^t$ outside of $L^{in}$.
	
	\item To ensure transversality of the relevant moduli spaces, Hamiltonian perturbations $H\in C^\infty(M)$ must be quadratic outside of some compact subset $C\subset M$, that is, $H|_{M\setminus C}=r^2$; there, the associated Hamiltonian vector field $X_H\in\mathfrak{X}(M)$ reads $X_H|_{M\setminus C}=2rR_\alpha$, where $R_\alpha$ is the \textit{Reeb vector field} of $\alpha=\lambda|_{\partial M^{in}}$, the unique vector field such that $d\alpha(R_\alpha,\cdot)=0$ and $\alpha(R_\alpha)=1$. Moreover, we look at almost complex structures $J\in C^\infty([0,1],\mathfrak{J}(M,\omega))$ which are \textit{of contact type} on $\partial M^{in}$: $\lambda\circ J|_{\partial M^{in}}= r\alpha\circ J|_{\partial M^{in}}=dr$ for all $t\in[0,1]$.
	
	\item The \textbf{wrapped Floer complex}\index{wrapped Floer complex} $CW(L_0,L_1)\equiv CW(L_0,L_1;H)$ is then the free module (over $\mathbb{K}$!) generated by the perturbed intersection points of $\mathcal{X}(L_0,L_1)\coloneqq \phi_H^1(L_0)\cap L_1$, where $\phi_H^1(L_0)\pitchfork L_1$ (cf. Figure \ref{wrappedman}). According to Corollary \ref{flowlines}, the latter are just time-1 Hamiltonian chords from $L_0$ to $L_1$ inside of $M^{in}$, while they are non-degenerate \textit{Reeb chords} from $\partial L_0$ to $\partial L_1$ in its complement.
	
	\item Like for standard Fukaya categories, the first composition map is the Floer differential $\mu_{\mathscr{W}\!(M)}^1\coloneqq\partial_{J,H}:CW(L_0,L_1)\rightarrow CW(L_0,L_1)[1]$ as defined in \eqref{Floersdiff}, counting perturbed $J$-holomorphic strips of Maslov index 1 solving Floer's equation \eqref{perturbedCR}. Its well-definiteness and independence from $(J,H)$ in cohomology is justified along the same lines of Theorem \ref{HFwelldef} (refer to \cite[section 2]{[Abo10]} for an adaptation of the proof). We just emphasize that the choices made above guarantee that any solving strip remains within a bounded subset of $M$. Therefore, we can talk about \textbf{wrapped Floer cohomology}\index{wrapped Floer cohomology} rings $HW(L_0,L_1)$.
	
	\item Here is however the first caveat: in order $d\geq 2$, the defining equation \eqref{generalCR} of the generalized Cauchy--Riemann problem is not well behaved and does not remedy the non-compactness of $M$. Instead, we must impose the 1-form $\beta\in\Omega^1(\mathbb{D}^2)$ to additionally satisfy $d\beta\leq 0$, which can be achieved through a rescaling trick  found by Abouzaid (\cite[section 3.2]{[Abo10]}): for $\rho>1$ is $\psi_Z^\rho(L_i)$ exact Lagrangian isotopic to $L_i\in\textup{obj}(\mathscr{W}\!(M))$, and $\frac{1}{\rho}H\circ\psi_Z^\rho=\rho H$ ($=2r^2$ for $\rho=2$) at infinity, so that one can define a Floer product
	\begin{align*}
		\mu_{\mathscr{W}\!(M)}^2:\;\;&CW(L_1,L_2;J,H)\otimes CW(L_0,L_1;J,H)\\
		&\rightarrow CW(\psi_Z^2(L_0),\psi_Z^2(L_2);(\psi_Z^2)_*J,\begin{tiny}\tfrac{1}{2}\end{tiny}H\circ\psi_Z^2)\cong CW(L_0,L_2;J,H)\;.
	\end{align*}
	This results from the count of index 0 $\tilde{J}$-holomorphic triangles of finite energy solving equation \eqref{generalCR}, for a suitably rescaled Floer datum $(\tilde{J},\tilde{H})\in C^\infty([0,1],\mathfrak{J}(M,\omega))\times C^\infty([0,1]\times M)$ coinciding with the chosen $(J,H)$ about the ``input punctures'' $z_1, z_2$ (where $\beta=dt$) and with $(\frac{1}{4}H\circ\psi_Z^2,(\psi_Z^2)_*J)$ near the ``output puncture'' $z_0$ (where $\beta=2dt$). Finally, we applied the natural isomorphisms $CW(L_0,L_1;J,H)\cong CW(\psi_Z^\rho(L_0),\psi_Z^\rho(L_1);(\psi_Z^\rho)_*J,\frac{1}{\rho}H\circ\psi_Z^\rho)$ induced by exact Lagrangian isotopy.
	\item The higher order composition maps $\mu_{\mathscr{W}\!(M)}^d$ are constructed in a similar fashion (see \cite[section 4.1]{[Abo10]}). 
\end{itemize}

\begin{Rem}									
	Notice that in the first step we did \textit{not} prevent the exact Lagrangian submanifolds $L\in\textup{obj}(\mathscr{W}\!(M))$ to be compact, which is the case for those lying in $M^{in}$. From the construction outlined, it then follows that $\mathscr{F}(M^{in})\subset\mathscr{W}\!(M)$ is a full $A_\infty$-subcategory, as defined in Section \ref{ch5.8}. 
\end{Rem}

The concept of wrapped Fukaya category is especially useful in the study of cotangent bundles $M=T^*N$ of compact manifolds $N$, which are naturally equipped with the Liouville form $\lambda\coloneqq p dq\in\Omega^1(M)$. Here the ``wrapping Hamiltonian'' one resorts to is $H\coloneqq |\!|p|\!|^2$. The prototypical example is $T^*\mathbb{S}^1$, for which $H$ is once again the quadratic function $r^2$ (this is discussed in \cite[section 4.2]{[Aur13]}). One can deduce several important results, both for wrapped Floer cohomology and for $\mathscr{W}\!(M)$ itself; some are highlighted in \cite[section 4.3]{[Aur13]}.

\newpage

\section{Towards Homological Mirror Symmetry}
\thispagestyle{plain}

\subsection{Physical origins of mirror symmetry}\label{ch7.1}

Mirror Symmetry is a phenomenon arising from \textit{string theory}, a branch of high-energy theoretical physics attempting to unify gravitation with standard particles physics, imposing itself as a ``theory of everything'' encompassing quantum gravity (\cite{[BBS07]}). While pre-string physics historically relied on classical geometry, string theory, first studied around the late 1960s/early 1970s, is built upon --- and actively helped building --- an enriched version known as \textit{quantum geometry}, whose new features become relevant only at Planck's scale, $\hbar\sim 10^{-35}$ m (see \cite{[Gre97]} for a broader discussion). 

The core idea of quantum geometry is to substitute the quantum field theoretic description of particles as point-like, 0-dimensional entities with \textit{strings}. These are very small, 1-dimensional objects, either closed loops or open segments, whose different vibrational states determine fundamental properties such as mass, charge and spin, effectively replicating characteristics and interacting behaviours of all elementary particles and force carriers (\cite{[BBS07]}).

However, for this model to be physically and mathematically consistent, one needs to add extra dimensions to the ambient spacetime inhabited by strings, along which their oscillations can occur. Therefore, Minkowskian spacetime $\mathbb{R}^4$ must be replaced with some higher dimensional manifold $M=\mathbb{R}^4\times X$. Later iterations of string theory successfully integrated the concept of \textit{supersymmetry}, which establishes a physical duality between bosons and fermions, giving birth to five major \textit{superstring theories} (consult \cite{[Zwi09]}), all requiring $X$ to be some specific 6-dimensional compact Riemannian manifold. 

To explain consistency with the experiments conducted in our 4-dimensional Minkowskian limit, the six additional (spatial) dimensions must be macroscopically undetectable, that is, they must ``close up'' onto themselves within the scale of $\text{diam}(X)\sim \hbar$. Such a model is called a \textit{compactification}. Accounting for supersymmetry, one must necessarily take $X$ to be some 3-dimensional complex \textit{Calabi--Yau manifold/variety} (see \cite{[GHJ03]} and \cite{[YN10]}; we will introduce such an object in next section). Inside them, strings form compact 2-dimensional surfaces called \textit{worldsheets} (as opposed to worldlines spanned by point particles), onto which physicists construct 2-dimensional \textit{superconformal quantum field theories} (SCFT), whose properties depend on invariants of $X$ (\cite{[GHJ03]}).

Now, in the late 1980s it was observed that several non-trivial physical dualities existed between the various superstring theories, most prominently a \textit{T-duality} connecting string theories of type IIA and IIB, which implies the equivalence of their SCFTs, hence of the described physics (cf. \cite{[BBS07]}). This was found to be reflected in the correspondence between some invariants of their underlying Calabi--Yau manifolds, thus said to form a ``mirror pair''. The ensuing studies soon took life of their own, making \textit{Mirror Symmetry} a legitimate mathematical branch (\cite{[ABC+09]}). Enumerative geometry, a field concerned with counting solutions to geometric problems, benefited significantly from this fresh addition (as discussed in \cite{[HKK+03]}): for example, it was found that the number of lines which can be embedded in a quintic Calabi--Yau manifold (such as the one pictured in Figure \ref{calabiyau} at the end of this chapter) is 2875, while the count of conics is 609250 (refer to \cite{[YN10]}).

By the end of the decade, many longstanding such puzzles had been solved. Then, in 1991, the research \cite{[CdGP91]} of physicists Candelas, de la Ossa, Green and Parkes brought renewed interest to the subject: they showed that mirror symmetry can reduce hard problems on a model Calabi--Yau manifold $X$ to significantly easier questions on its mirror manifold $X^\dagger$, such as predicting the number of embedded rational curves of \textit{any} dimension! Currently, one of the most important and best formalized statement about mirror symmetry is Conjecture \ref{MS} (cf. \cite{[Aur09]}), which connects a pair $(X,X^\dagger)$ as follows:
\begin{quote}
	The A-model on $X$ corresponds to the B-model on $X^\dagger$, and the B-model on $X$ corresponds to the A-model on $X^\dagger$.
\end{quote}
Roughly, the (\textit{topologically twisted nonlinear sigma}) \textit{A-model} selects only the symplectic structure of the manifold, while the (\textit{topologically twisted nonlinear sigma}) \textit{B-model} captures just its complex, algebro-geometric features; one often talks about A-side and B-side (such terminology coming from the aforementioned type IIA and IIB superstring theories). Understanding this conjecture is the aim of Section \ref{ch7.2}. \\

By the early 1990s, most of the mathematical machinery was built upon ``closed'' string theories, whose objects are closed strings forming tiny loops spanning tubular worldsheets. However, new research impetus was found when considering open strings with endpoints lying on \textit{D-branes}. A \textit{$p$-brane} is a generalized string of dimension $p$, sweeping a $(p+1)$-dimensional spacetime volume (\cite{[Moo05]}). $D$-branes are a particular class thereof, imposing Dirichlet boundary conditions to any worldsheet surface of open strings with endpoints on them (whence the ``D''). Figure \ref{strings} illustrates some possible configurations.

\begin{figure}[htp]
	\centering
	\includegraphics[width=0.8\textwidth]{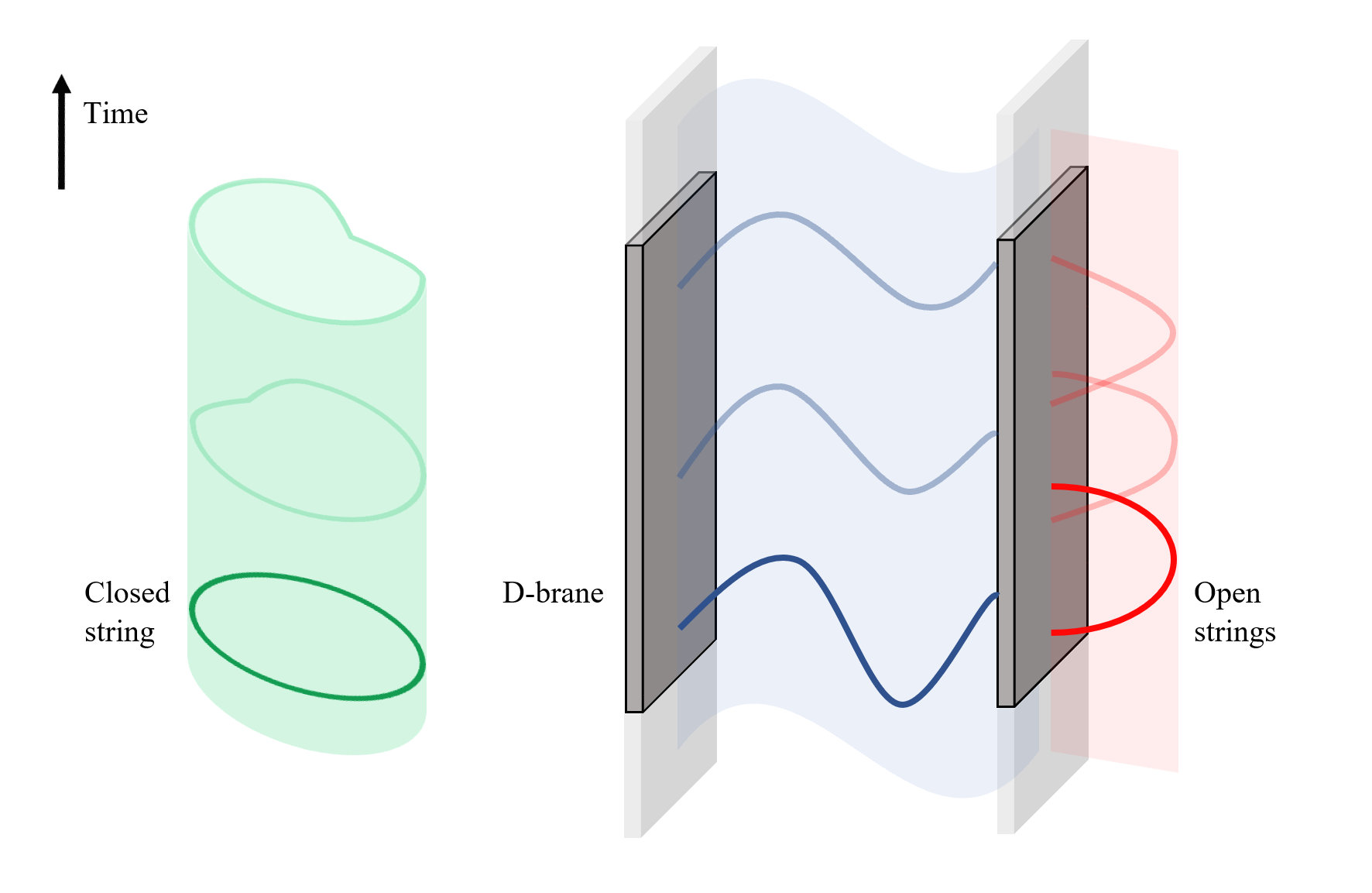}
	\caption{Some possible string configurations and their worldsheets. The blue, open string with boundary on distinct D-branes should be reminiscent of a \protect\hyperlink{certainprob}{certain problem} we discussed a lot, back in Chapter 5...}
	\label{strings}	
\end{figure}

The axioms of quantum field theory already suggested that $D$-branes should be organized in mathematical categories (cf. \cite{[ABC+09]}). Specifically, in the A-model the relevant objects (\textit{A-branes}) are Lagrangian submanifolds with flat line bundles, while in the B-model they are complex submanifolds with a choice of holomorphic vector bundle (\textit{B-branes}). We will investigate both types in Section \ref{ch7.3}. 
  
In 1994, boosted by this open string viewpoint, the mathematician M. Kontsevich entered the frame. During the ICM held in Zürich, he gave a homological refinement of the Mirror Symmetry Conjecture (published in \cite{[Kon94]}), where he claimed that, given a $2n$-dimensional symplectic manifold $(X,\omega)$ with $c_1(X)=0$ and dual $n$-dimensional complex algebraic manifold $X^\dagger$, 
\begin{quote}
	``the derived category constructed from the Fukaya category $\mathscr{F}(X)$ (or a suitably enlarged one) is equivalent to the derived category of coherent sheaves on $X^\dagger$'',
\end{quote}
drawing a parallel between open string theory and triangulated categories. Half of this argument we already have confidence with. Hence, goal of Section \ref{ch7.3} is to plausibly explain what that mirror side exactly is and how it fits in Conjecture \ref{HMS}, an (overreaching) version of Kontsevich's assertion (outlined in \cite{[Fuk02a]}). At the very end, we also mention a conjecture of the trio Strominger, Yau and Zaslow, whose work \cite{[SYZ96]} from 1996 explains qualitatively how to ``concretely'' construct a mirror pair of Calabi--Yau manifolds.

Before moving on, we emphasize that a rigorous mathematical formulation of string theory and its subdisciplines, mirror symmetry included, is still very much work in progress, whence their conjectural nature. On top of this fact, we will also opt for a more direct and less dispersive approach, discussing all necessary ingredients just at a definitory level. Apart from common sense, the reason for this is twofold: on the one hand, we aim at a self-contained survey allowing us to see where the theory developed so far embeds within the ``larger picture''; on the other, this chapter serves as a tentative roadmap to orientate potential future work.

\subsection{Mirror symmetry}\label{ch7.2}

The protagonists of the Mirror Symmetry Conjecture are Calabi--Yau manifolds. The definition we provide for them requires a few preliminary steps in the realm of Kähler geometry. First, we give an alternative characterization for complex manifolds, bypassing the holomorphic atlas (a proof is easily found throughout literature).

\begin{Lem}[Newlander--Nirenberg Theorem]\label{maybecomplex}	
	Let $M^n=(M,J)$ be an almost complex manifold of dimension $n$ over $\mathbb{C}$. Then it is a \textbf{complex manifold}\index{complex manifold} \textup(as described in Section \ref{ch5.1}\textup) if and only if its almost complex structure $J\in\mathcal{T}^{1,1}(M)$ makes the \textup{Nijenhuis tensor} 
	\[
	N_J(X,Y)\coloneqq [X,Y]+J([JX,Y]+[X,JY])-[JX,JY]
	\]
	vanish for all $X,Y\in\mathfrak{X}(M)$ \textup(then $J$ is called an \textup{integrable complex structure} for $M$\textup).   
\end{Lem}

Compatibility with the usual definition of complex manifold is due to the Nijenhuis tensor $N_J$ being the obstruction to the existence of ($J$-)\textbf{holomorphic functions}\index{holomorphic function}: these are smooth functions $f\in C^\infty(M,\mathbb{C})$ such that $(df)_x\circ J=i (df)_x$ for all $x\in M$ (where $df\coloneqq d(\mathfrak{Re}(f))+id(\mathfrak{Im}(f))$; compare this with Definition \ref{Jholocurve} of $J$-holomorphic curves). 

Since, fixed some Riemannian metric (always existing), symplectic manifolds are naturally almost complex, we sometimes capture their being additionally complex through the notation $(M^{2n},J,\omega)$, where the dimension $2n$ is over $\mathbb{R}$. 

\begin{Rem}								
	Here are some fundamental notions of complex differential geometry (we refer the reader to \cite[chapters 14, 15]{[Sil12]} for a more digestible introduction).
	\begin{itemize}[leftmargin=0.5cm]
		\item An almost complex structure $J\in\mathcal{T}^{1,1}(M)\cong\Gamma(\text{End}(TM))$ on $M^{2n}$ can be extended to the \textit{complexified tangent bundle} $TM^\mathbb{C}\coloneqq TM\otimes\mathbb{C}$ of fibers $T_xM^\mathbb{C}=T_xM\otimes\mathbb{C}$ simply by imposing $J(X\otimes z)\coloneqq J(X)\otimes z$ for all $X\in\mathfrak{X}(M)$ and $z\in\mathbb{C}$. This yields a tensor field $J\in\mathcal{T}^{1,1}(TM^\mathbb{C})$ of eigenvalues $\pm i$ (by definition), so that we may decompose $TM^\mathbb{C}$ into the respective eigenvalue subbundles, $TM^\mathbb{C}=TM^{1,0}\oplus TM^{0,1}$. Similarly, we can form the \textit{complexified cotangent bundle} $T^*M^\mathbb{C}\coloneqq T^*M\otimes\mathbb{C}=T^*M^{1,0}\oplus T^*M^{0,1}$.
		
		\item In particular, it makes sense to talk about complex-valued forms on $(M,J)$: letting $\Lambda^{p,q}(T^*M^\mathbb{C})\coloneqq \Lambda^p(T^*M^{1,0})\wedge\Lambda^q(T^*M^{0,1})$, a \textbf{$(p,q)$-form}\index{pq@$(p,q)$-form} is an element of $\Omega^{p,q}(M)\coloneqq \Gamma(\Lambda^{p,q}(T^*M^\mathbb{C}))$. Then we define 
		\[
		\Omega^r(M)\coloneqq \bigoplus_{p+q=r}\Omega^{p,q}(M)\,.
		\]
		Assuming $M$ to be complex, we can think for example of $\Omega^{1,1}(M)$ as being locally spanned by $dz\wedge d\bar{z}$, for some (holomorphic) coordinate function $z=x+iy\in C^\infty(M,\mathbb{C})$ (thus fulfilling $dz\circ J=i dz=idx-dy$).	
		
		\item Let $\pi^{p,q}:\Lambda^r(T^*M^\mathbb{C})\coloneqq\bigoplus_{p+q=r}\Lambda^{p,q}(T^*M^\mathbb{C})\rightarrow \Lambda^{p,q}(T^*M^\mathbb{C})$ denote the projection onto the $(p,q)$-th summand. Then the (extended) exterior differential operator $d:\Omega^\bullet(M)\rightarrow\Omega^{\bullet+1}(M)$ yields two differentials $\partial, \bar{\partial}$ defined by 
		\[
		\partial_p\coloneqq \pi^{p+1,q}\circ d:\Omega^{p,q}(M)\rightarrow\Omega^{p+1,q}(M) \,,
		\]
		\[
		\newline \bar{\partial}_q\coloneqq \pi^{p,q+1}\circ d:\Omega^{p,q}(M)\rightarrow\Omega^{p,q+1}(M)\,.
		\]
		For functions $f\in C^\infty(M,\mathbb{C})$, it holds $d_0f=\partial_0 f+\bar{\partial}_0f$, so that in particular $f$ is holomorphic if and only if $\bar{\partial}_0f=0$.
		
		\item In fact, on complex(!) manifolds $M$, the equality $d=\partial+\bar{\partial}$ holds for any form $\nu\in\Omega^r(M)$, and we can deduce from $0=d^2\nu=\partial^2\nu + (\partial\bar{\partial}+\bar{\partial}\partial)\nu+\bar{\partial}^2\nu$ that $\partial,\bar{\partial}$ are actual (anticommuting) coboundary operators for the cochain complexes $\Omega^{\bullet,q}$ respectively $\Omega^{p,\bullet}$. Thus, the \textbf{Dolbeault cohomology groups}\index{Dolbeault cohomology group} 
		\begin{equation}
			H_\text{Db}^{p,q}(M)\coloneqq \frac{\text{ker}(\bar{\partial}_q)}{\text{im}(\bar{\partial}_{q-1})}
		\end{equation}
		are well defined for all $p,q\in\mathbb{Z}$.
		
		\item Then the \textbf{Dolbeault Theorem}\index{Dolbeault Theorem} gives isomorphisms 
		\begin{equation}
			H_\text{Db}^{p,q}(M)\cong H^q(M;\Omega_M^{p,0})\,,
		\end{equation}
		where the right-hand side is the \textit{$q$-th sheaf cohomology group} associated to the sheaf $\Omega_M^{p,0}$ of $(p,0)$-forms on $M$ (sheaf cohomology is discussed, for example, in \cite[chapter III]{[Har77]}). This result can be regarded as the complex analogue of de Rham's Theorem.
	\end{itemize}
\end{Rem}

\begin{Def}\label{Kahlerman}						
	Let $M^n=(M,J)$ be a complex manifold equipped with a Riemannian metric $h\in\Gamma(T^*M\otimes T^*M)$. Then $h$ is a \textit{Hermitian metric} if $h(X,Y)=h(JX,JY)$ for all $X,Y\in\mathfrak{X}(M)$ (this means we have a smoothly varying positive definite hermitian inner product on $TM$). We call $(M,J,h)$ a \textbf{Hermitian manifold}\index{Hermitian manifold}. Note that these data identify a real $(1,1)$-form 
	\[
	\kappa\coloneqq-\mathfrak{Im}(h)\in\Omega^2(M)\quad\text{fulfilling}\quad \kappa(J\cdot,\cdot)=h(\cdot,\cdot)\,.
	\]
	  
	Moreover, $h$ is also a \textit{Kähler metric} if additionally $\nabla J=0$ (for $\nabla$ the Levi--Civita connection of $h$), or equivalently, if $d\kappa=0$ or $\nabla\kappa=0$. When true, we call $M=(M,J,h,\kappa)$ a \textbf{Kähler manifold}\index{Kähler manifold} and $\kappa$ its \textbf{Kähler form}\index{Kähler form}, on a local chart $(U,(z_j)_{j=1}^n)$ given by 
	\begin{equation}
	\kappa=\frac{i}{2}\sum_{j,k=1}^n G_{jk}dz_j\wedge d\bar{z}_k\,,
	\end{equation}
	where $(G_{jk}(x))$ is a positive definite hermitian matrix at all $x\in M$ (as shown in \cite[section 16.1]{[Sil12]}).
\end{Def}

From the symplectic viewpoint, a Kähler manifold is just a symplectic manifold $(M^{2n},\kappa)$ equipped with an integrable complex structure $J\in\mathfrak{J}(M,\kappa)$, and subtended Kähler metric $\kappa^\mathbb{C}\coloneqq B-i\kappa$ --- also called \textbf{complexified Kähler form}\index{Kähler form!complexified} --- where $B=\mathfrak{Re}(\kappa^\mathbb{C})$ (the ``B-field'') is the well-defined Riemannian metric on $M$ given by $B(\cdot,\cdot)\coloneqq \kappa(J\cdot,\cdot)$ (cf. Definition \ref{almostcomplex}). In light of this, we drop the redundancy and simply denote Kähler manifolds by $(M,J,\kappa^\mathbb{C})$, with $\kappa=-\mathfrak{Im}(\kappa^\mathbb{C})$.

Moreover, observe that $J\in\text{Symp}(M,\kappa)$, the reason why $\kappa\in\Omega^2(M)$ is actually a ($\partial/\bar{\partial}$-closed) $(1,1)$-form, then yielding a well-defined Dolbeault cohomology class $\langle\kappa\rangle\in H_{Db}^{1,1}(M)$.

For compact Kähler manifolds there is a direct relation between de Rham and Dolbeault cohomologies:

\begin{Thm}[Hodge]\label{Hodge}						
	Let $M^n=(M,J,\kappa^\mathbb{C})$ be a compact Kähler manifold. Then we have the following \textbf{Hodge decomposition}\index{Hodge decomposition}:
	\[
	H_{dR}^r(M;\mathbb{C})\cong\bigoplus_{p+q=r}H_{Db}^{p,q}(M)\,,
	\]
	with $H_{Db}^{p,q}(M)\cong \overline{H_{Db}^{q,p}(M)}$ and all $H_{Db}^{p,q}(M)$ finite-dimensional. Denoting by $b^r(M)\coloneqq \textup{dim}(H_{dR}^r(M;\mathbb{C}))$ the Betti numbers and by $h^{p,q}(M)\coloneqq \textup{dim}(H_{Db}^{p,q}(M))$ the \textbf{Hodge numbers}\index{Hodge numbers} of $M$, we therefore have $b^r(M)=\sum_{p+q=r}h^{p,q}(M)$ and $h^{p,q}(M)=h^{q,p}(M)$. 
\end{Thm}	

Information about Hodge numbers is encoded in the so-called \textbf{Hodge diamond}\index{Hodge diamond}:
\begin{equation}
	\begin{tabular}{c c c c c}
				& 				& $h^{n,n}$		&				&			 	\\
				& $h^{n,n-1}$ 	& 				& $h^{n-1,n}$ 	& 				\\
	$h^{n,n-2}$	&				& $h^{n-1,n-1}$	&				& $h^{n-2,n}$ 	\\
	\vdots		&				& \vdots		&				& \vdots		\\
	$h^{2,0}$ 	&				& $h^{1,1}$ 	&				& $h^{0,2}$		\\
				& $h^{1,0}$		& 				& $h^{0,1}$		& 				\\
				& 				& $h^{0,0}$		&				&			 	
	\end{tabular}
\end{equation}
The central column is made of all $h^{k,k}(M)$, which are shown to be positive integers. Complex conjugation of the underlying Dolbeault cohomology groups corresponds to the mirroring with respect to this central column, while symmetry about the diamond center is given by the \textit{Hodge $\ast$-operator}, which is the map $\ast:\Lambda(T_xM)\rightarrow\Lambda(T_xM),\, \ast(v_1\wedge...\wedge v_k)\coloneqq v_{k+1}\wedge...\wedge v_n$ (with $\ast(1)=v_1\wedge...\wedge v_n$ and $\ast(v_1\wedge...\wedge v_n)=1$), yielding isomorphisms $H_{Db}^{p,q}(M)\cong H_{Db}^{n-p.n-q}(M)$. 

One last technical bit to go.

\begin{Def}											
	Let $M=(M,g)$ be a Riemannian manifold equipped with Levi--Civita connection $\nabla$. The \textbf{Riemannian holonomy group}\index{Riemannian holonomy group} of $g$ at $x\in M$ is \begin{align*}
		\text{Hol}_x(g)\coloneqq \text{Hol}^\nabla(x)=\{\mathbb{P}_\gamma:T_xM\rightarrow T_xM \mid\;
		& \mathbb{P}_\gamma\text{ parallel transport map},\\
		& \gamma \text{ piecewise smooth loop at }x\}\,,
	\end{align*}
	which can be regarded as a subgroup\footnote{Recall that, up to conjugation, $\text{Hol}^\nabla(x)<\text{GL}(n,\mathbb{R})$ as subgroup for any $x\in M$, assuming $M$ is connected (which we implicitly do).} $\text{Hol}(g)<\text{O}(n)$, as explained in \cite[section 2.5]{[GHJ03]}.
	In particular, given a Kähler manifold $(M^n,J,\kappa^\mathbb{C})$, one can show that $\text{Hol}(\kappa^\mathbb{C})$ $\coloneqq\text{Hol}(\mathfrak{Re}(\kappa^\mathbb{C}))<\text{U}(n)$ (itself a subgroup of $\text{O}(2n)$; see \cite[section 4.2]{[GHJ03]}).  
\end{Def}

We are finally ready to define Calabi--Yau manifolds (with reference to \cite[Definition 4.3]{[GHJ03]}). 

\begin{Def}											
	A \textbf{Calabi--Yau manifold}\index{Calabi--Yau manifold} $X^n=(X,\varpi)$ is a compact Kähler manifold $(X,J,\kappa^\mathbb{C})$ with $\text{Hol}(\kappa^\mathbb{C})=\text{SU}(n)$ and equipped with a holomorphic volume form $\varpi\in\Omega^{n,0}(X)$ fulfilling $\kappa^n/n!=(-1)^{n(n-1)/2}(i/2)^n\varpi\wedge\overline{\varpi}$, which we call its \textbf{Calabi--Yau form}\index{Calabi--Yau form}.
\end{Def}

\begin{Rem}(\textit{Warning})							
	There are actually several, possibly inequivalent definitions of Calabi--Yau manifold adopted in literature. Some drop the compactness requirement or the identity for global holonomy, while others do not mention the holomorphic volume form $\varpi$ (this in fact being uniquely set, up to phase; see \cite[Lemma 4.4]{[GHJ03]}). 
	
	Many authors just impose triviality of the \textit{canonical line bundle} $K_X\coloneqq \Lambda^n(T^*X)\cong \Omega_X^{n,0}$, that is, $K_X\cong\mathcal{O}_X$ (the sheaf of holomorphic functions on $X$, defined in Example \ref{holosheaf} below), of which $\varpi$ is a constant section. 
	In all cases, it follows that $c_1(TX)=0$ and that $X$ has vanishing Ricci curvature (also consequence of the \textit{Calabi conjecture} proved by Yau).
\end{Rem}

Let us summarize the chain of implications between the possible structures of manifolds we have seen so far:
\[
\begin{tikzpicture}[mybox/.style={draw, inner sep=1.5pt}]
	\node[mybox] (box){
\begin{tikzcd}[column sep = small, row sep = 0cm]
	& & \mkern-12mu\text{\begin{small}complex\end{small}}\arrow[dr, Rightarrow, shorten <=-0.0cm, shorten >=0.0cm, start anchor={[yshift=-1ex]east}, end anchor={[yshift=0.5ex]west}] & & \\
	\text{\begin{small}Calabi--Yau\end{small}}\arrow[r, Rightarrow] & \text{\begin{small}Kähler \end{small}} 
	\arrow[ur, Rightarrow, shorten <=0.2cm, shorten >=0.2cm, start anchor={[xshift=-2ex]east}, end anchor={[yshift=-0.2ex]west}] 
	\arrow[dr, Rightarrow, shorten <=0.2cm, shorten >=0.1cm, start anchor={[xshift=-2ex]east}, end anchor={[yshift=-0.0ex]west}] 
	& & \text{\begin{small}almost complex\end{small}}\arrow[r, Rightarrow] & \begin{small}\text{smooth even-dim.}\end{small} \\
	& & \mkern-6mu\text{\begin{small}symplectic\end{small}}\arrow[ur, Rightarrow, shorten <=-0.1cm, shorten >=0.0cm, start anchor={[yshift=0.8ex]east}, end anchor={[yshift=-1ex]west}] & & \\
\end{tikzcd}
};
\end{tikzpicture}
\]

Finally, the following statement is a very dried up version of the Mirror Symmetry Conjecture. A (more) complete formulation brings into play \textit{Gromov--Witten theory} and the so-called \textit{Yukawa couplings}; the reader is referred to \cite[chapters 15--17]{[GHJ03]}. 

\begin{Con}[\textbf{Mirror Symmetry Conjecture}]\index{Mirror Symmetry Conjecture}\label{MS}		
	Given a Calabi--Yau manifold $X^n=(X,J,\kappa^\mathbb{C},\varpi)$, there exists a ``mirror'' Calabi--Yau manifold $X^\dagger=(X^\dagger,J^\dagger,\kappa^{\mathbb{C}\dagger},\varpi^\dagger)$ of same dimension such that there are natural isomorphisms
	\begin{equation}\label{MSconj}
		H^q(X,\Omega^p(TX))\cong H^q(X^\dagger,\Omega_{X^\dagger}^p),\;\;\;
		H^q(X^\dagger,\Omega^p(TX^\dagger))\cong H^q(X,\Omega_X^p).
	\end{equation}
	Since $\Lambda^{p,0}(TX^\mathbb{C})\rightarrow\Omega^{n-p,0}(X),\, v_1\wedge...\wedge v_p\mapsto (i_{v_1}\circ...\circ i_{v_p})(\varpi)$ is a natural isomorphism, whence $\Omega^p(TX)\cong\Omega_X^{n-p,0}$ \textup(and similarly $\Omega^p(TX^\dagger)\cong\Omega_{X^\dagger}^{n-p,0}$\textup), we can simply rephrase \eqref{MSconj} through Dolbeault cohomology:
	\begin{equation}
		H_{Db}^{n-p,q}(X)\cong H_{Db}^{p,q}(X^\dagger)\quad\text{and}\quad H_{Db}^{n-p,q}(X^\dagger)\cong H_{Db}^{p,q}(X)\,. 
	\end{equation}
	In particular, defining the moduli spaces $\mathcal{M}_{symp}(X)\coloneqq \{\kappa\in\Omega^{1,1}(X)\mid \kappa\text{ Kähler}\}$ and $\mathcal{M}_{cplx}(X)\coloneqq \{J\in\textup{Symp}(X,\kappa)\mid J\; \kappa\text{-compatible}\}/\textup{Diff}(X)$, one expects two ``mirror maps'' $\mathcal{M}_{symp}(X)\rightarrow\mathcal{M}_{cplx}(X^\dagger)$ and $\mathcal{M}_{symp}(X^\dagger)\rightarrow\mathcal{M}_{cplx}(X)$ relating $H_{Db}^{1,1}(X)\cong H^1(X^\dagger,\Omega_{X^\dagger}^{n-1})\cong H^1(X^\dagger, TX^\dagger)$ and vice versa\footnote{In the words of Kontsevich from \cite{[Kon94]}, it is conjectured that: ``Complex variations of pure Hodge structures constructed using Gromov--Witten invariants of the symplectic [Calabi--Yau] manifold $X$ are locally equivalent to algebro-geometric variations [on $X^\dagger$]''.} \textup(the group $H^1(X,TX)$ governs ``deformations'' of $(X,J)$\textup).	
\end{Con}

Physically speaking, last statement claims the existence of the correspondences  ``A-model on $(X,\kappa^\mathbb{C})\leftrightarrow$ B-model on $(X^\dagger,J^\dagger)$'' and ``B-model on $(X,J)\leftrightarrow$ A-model on $(X^\dagger,\kappa^\mathbb{C\dagger})$'' we anticipated back in Section \ref{ch7.1}. Furthermore, the conjecture predicts that the SCFTs on $X$ and $X^\dagger$ (thus ``their physics'') are equivalent.

\begin{Rem}												
	For Calabi--Yau manifolds $X^n$, an additional symmetry on Dolbeault cohomology comes from the Hodge $\ast$-operator and the wedge product with the Calabi--Yau form, namely $H_{Db}^{p,0}(X)\cong H_{Db}^{n,n-p}(X)\cong H^{n-p}(X,\mathcal{O}_X)\cong H_{Db}^{0,n-p}(X)$, whence $h^{p,0}(X)=h^{n-p,0}(X)$, in addition to $h^{p,q}(X)=h^{q,p}(X)$ and $h^{p,q}(X)=h^{n-p,n-q}(X)$ from Theorem \ref{Hodge}. 
	
	Observe further that $H_{Db}^{n,0}(X)\cong H^0(X,\mathcal{O}_X)\cong \mathcal{O}_X(X)$ (last isomorphism is a standard result for sheaf cohomology), whence $h^{n,0}(X)=h^{0,n}(X)=1$. Moreover, one can prove that $H_{Db}^{0,0}(X)\cong\mathbb{C}$, thus $b^0(X)=h^{0,0}(X)=1$, and $h^{k,0}(X)=0$ for $0<k<n$ (see \cite[Proposition 5.6]{[GHJ03]}). 
\end{Rem}

\begin{Ex}\label{mirrorsymmex}							
	Let us consider a 3-dimensional Calabi--Yau manifold $X$. From previous remark we know that $h^{2,0}(X)=h^{0,2}(X)=0$, thus $h^{1,3}(X)=h^{3,1}(X)=0$, and $h^{1,0}(X)=h^{0,1}(X)=0$, thus $h^{2,3}(X)=h^{3,2}(X)=0$, along with $h^{3,3}(X)=h^{0,0}(X)=1$ and $h^{3,0}(X)=h^{0,3}(X)=1$. Therefore, we obtain the following reduced Hodge diamond (also exploiting $h^{2,2}(X)=h^{1,1}(X)$ and $h^{2,1}(X)=h^{1,2}(X)$):\vspace*{-0.2cm}
	\begin{equation}
		\begin{tabular}{c c c c c c c c}
				&		& 			& 	1		&			&		&		\\
				&		& 	0 		& 			& 	0 		& 		&		\\
				&	0	&			& $h^{1,1}$	&			& 0 	&		\\
			1	&		& $h^{1,2}$ & 			& $h^{1,2}$	& 		&	1	\\
			 	&	0	&			& $h^{1,1}$ &			& 0		&		\\
				& 		& 	0		&			&	0		& 		& 		\\
				& 		& 			&	1		&			&		&			 	
		\end{tabular}
	\end{equation}
	Then the Mirror Symmetry Conjecture is ``simply'' the statement that there exists a dual 3-dimensional Calabi--Yau manifold $X^\dagger$ whose Hodge diamond results from a $90^\circ$ rotation of the above one around its center. In particular, $h^{1,1}(X)=h^{1,2}(X^\dagger)$ and $h^{1,2}(X)=h^{1,1}(X^\dagger)$.
	
	In a similar vein, in complex dimension 1 respectively 2 we obtain Hodge diamonds:
	\begin{equation}
		\begin{tabular}{c c c c c c c c c}
			&		&	& $\mkern+72mu$	&		& 			& 	1		&			&		\\
			& 	1 	&	& 				&		& 	0 		& 			& 	0 		& 		\\
			1 &	  & 1 	& 				&	1	&			& 	20		&			& 1 	\\
			& 	1 	&	& 				&		& 	0		&			&	0		& 		\\
			&		&   & 				&		& 			&	1		&	  &					 	
		\end{tabular}
	\end{equation}
	Indeed, a 1-dimensional Calabi--Yau manifold is necessarily a complex elliptic curve $\mathbb{C}/\Lambda$ for $\Lambda$ some $\mathbb{Z}_2$-lattice (in particular, it is a symplectic 2-torus $\mathbb{T}^2$ as that studied in Example \ref{2torus}). As hinted by the uniqueness of their Hodge diamond, Conjecture \ref{MS} claims that such manifolds are self-dual, producing a direct correspondence between symplectic and complex structures on them. 
	
	On the other hand, 2-dimensional Calabi--Yau manifolds are \textit{$K3$-surfaces}, that is, compact complex surfaces $X$ with $h^{1,0}(X)=0$ and trivial canonical bundle, which can be realized as algebraic surfaces in $\mathbb{C}P^3$ (such as the projective variety $X=V_p(x_0^4+x_1^4+x_2^4+x_3^4)$, the \textit{Fermat quartic}). Conversely, any such surface can be given a Calabi--Yau structure. Since all $K3$-surfaces are diffeomorphic, with $b^2(X)=22$, we obtain the Hodge diamond pictured above. In dimension 2, mirror symmetry thus translates into an intimate relation between $K3$-surfaces.
	
	However, starting from $n=3$, we see that the conjecture acquires a non-trivial flavour, especially since it is currently unknown whether there exists a finite number of topological families of Calabi--Yau 3-folds or not. The case of non-degenerate quintic 3-folds --- algebraic surfaces such as the \textit{Fermat quintic} $V_p(x_0^5+x_1^4+x_2^4+x_3^4+x_4^5)\subset\mathbb{C}P^4$ --- was the one originally studied in \cite{[CdGP91]}, where the respective Hodge diamonds were shown to have $h^{1,1}(X)=1=h^{1,2}(X^\dagger)$ and $h^{1,2}(X)=101=h^{1,1}(X^\dagger)$. \hfill $\blacklozenge$      
\end{Ex}

\subsection{Homological mirror symmetry}\label{ch7.3}

We start by collecting the relevant definitions from algebraic geometry, the predominant theory governing the B-side of the Homological Mirror Symmetry Conjecture (more extensive accounts are provided in \cite{[Har77]} or \cite{[Gat20]}).

\begin{Def}									
	Let $X$ be a topological space, $\mathsf{C}$ a category. A \textbf{presheaf $\mathcal{F}$ on $X$ with values in $\mathsf{C}$}\index{presheaf with values in a category} consists of:
	\begin{itemize}[leftmargin=0.5cm]
		\item an object $\mathcal{F}(U)\in\textup{obj}(\mathsf{C})$ for all open subsets $U\subset X$, whose elements $s\in\mathcal{F}(U)$ we call \textit{sections} of $\mathcal{F}(U)$,
		\item \textit{restriction maps} $\rho_U^V\in\textup{Hom}_\mathsf{C}(\mathcal{F}(V),\mathcal{F}(U))$ for all pairs of open subsets $U\subset V\subset X$,
	\end{itemize}
	fulfilling: $\mathcal{F}(\emptyset)=\{0\}$, $\rho_U^U=\textup{id}_{\mathcal{F}(U)}$ for all open $U\subset X$, and $\rho_U^V\circ\rho_V^W=\rho_U^W$ for all triples of open subsets $U\subset V\subset W\subset X$.
	
	The presheaf $\mathcal{F}$ is a \textbf{sheaf on $X$ with values in $\mathsf{C}$}\index{sheaf!with values in a category} if the ``gluing property'' holds: for any open subset $U\subset X$ with open cover $\{U_i\}_{i\in I}$ so that all sections $s_i\in\mathcal{F}(U_i)$ are compatible, $s_i|_{U_i\cap U_j} = s_j|_{U_i\cap U_j}$, there exists a unique $s\in\mathcal{F}(U)$ such that $s|_{U_i}=s_i$ for all $i\in I$.
	
	When $\mathsf{C}=\mathsf{Sets}$ and, specifically, $\mathcal{F}(U)\subset\textup{Hom}_\mathsf{Sets}(U,\mathbb{K})$ for all open subsets $U\subset X$, $\mathcal{F}$ is a \textit{sheaf of $\mathbb{K}$-valued functions} ($\mathbb{K}$ being any ground field). 
\end{Def}

Direct sums and restrictions of (pre)sheaves are themselves (pre)sheaves under the obvious assignments. Here we are most interested in the case $\mathbb{K}=\mathbb{C}$ and $\mathsf{C}=\mathbb{C}$-$\mathsf{Alg}$ (the category of $\mathbb{C}$-algebras):

\begin{Ex}\label{holosheaf}							
	Let $(X,J)$ be a complex manifold. We define the \textbf{sheaf of holomorphic functions}\index{sheaf!of holomorphic functions} $\mathcal{O}_X$ on $X$ through the assignment 
	\[
	U\mapsto \mathcal{O}_X(U)\coloneqq \{f\in C^\infty(U,\mathbb{C})\mid f\text{ $J$-holomorphic}\}
	\]
	(we defined $J$-holomorphic functions right at the beginning of Section \ref{ch7.2}). This is a sheaf of $\mathbb{C}$-valued functions with values in $\mathbb{C}$-$\mathsf{Alg}$. \hfill $\blacklozenge$
\end{Ex}

\begin{Def}											
	Let $X$ be a topological space endowed with sheaves $\mathcal{F},\mathcal{G}$ with values in some category $\mathsf{C}$. A \textbf{morphism of sheaves}\index{morphism!of sheaves} $\chi:\mathcal{F}\rightarrow\mathcal{G}$ over $X$ is a collection of maps $\chi_U\in\textup{Hom}_\mathsf{C}(\mathcal{F}(U),\mathcal{G}(U))$ on all open subsets $U\subset X$ fulfilling $\rho_U^V\circ\chi_V=\chi_U\circ\rho_U^V\in\textup{Hom}_\mathsf{C}(\mathcal{F}(V),\mathcal{G}(U))$ for any pair of open subsets $U\subset V\subset X$ (where $\rho$ denotes the restriction maps of both $\mathcal{F}$ and $\mathcal{G}$). 
	
	Given another morphism of sheaves $\tilde{\chi}:\mathcal{G}\rightarrow\mathcal{H}$ over $X$, we simply define $\tilde{\chi}\circ\chi$ ``pointwise'' by $(\tilde{\chi}\circ\chi)_U\coloneqq\tilde{\chi}_U\circ\chi_U\in\textup{Hom}_\mathsf{C}(\mathcal{F}(U),\mathcal{H}(U))$. These operations yield a well-defined category $\mathsf{Sh}(X;\mathsf{C})$, the \textbf{category of sheaves on $X$ with values in $\mathsf{C}$}\index{category!of sheaves with values in a category}.
	
	Choosing for example the category of rings $\mathsf{C=Rings}$ and $\mathcal{F}\in\mathsf{Sh}(X;\mathsf{Rings})$ makes $X=(X,\mathcal{F})$ a \textbf{ringed space}\index{ringed space} with \textit{structure sheaf} $\mathcal{F}$.
\end{Def}

\begin{Ex}										
	Let $X$ be a complex manifold, $\mathcal{O}_X$ its sheaf of holomorphic functions. Then $(X,\mathcal{O}_X)$ is a well-defined ringed space (for $\mathbb{C}$-algebras are just enriched rings). \hfill $\blacklozenge$
\end{Ex}	
	
\noindent\minibox[frame]{Henceforth, when talking about complex manifolds $X$, we will implicitly re- \;\\ gard them as ringed spaces with structure sheaf $\mathcal{O}_X$.}
\vspace*{0.2cm}

\begin{Rem}										
	In algebraic geometry, one usually studies the algebraically richer notion of locally ringed space. Let $X$ be a topological space, $\mathcal{F}$ a presheaf on $X$ with values in $\mathsf{C}$. The \textit{stalk} of $\mathcal{F}$ at $x\in X$ is $\mathcal{F}_x\coloneqq \{(U,s)\mid x\in U,\, U\subset X \text{ open}, \,s\in\mathcal{F}(U)\}/\!\sim$, where $(U,s)\sim(U',s')$ if and only if there exists some open subset $V\subset U\cap U'$ such that $x\in V$ and $s|_V=s'|_V$. We call a class $[(U,s)]\in\mathcal{F}_x$ a \textit{germ} of $\mathcal{F}$ at $x$. Observe that stalks inherit the structure of an object of $\mathsf{C}$.
	
	Then $(X,\mathcal{F})$ for $\mathcal{F}\in\mathsf{Sh}(X;\mathsf{Rings})$ is a \textit{locally ringed space} if it is a ringed space where additionally each stalk $\mathcal{F}_x$ is a local ring (meaning it has a unique maximal ideal).  

\end{Rem}

\begin{Def}											
	Let $(X,\mathcal{F})$ be a ringed space. Then a \textbf{sheaf of $\mathcal{F}$-modules}\index{sheaf!of $\mathcal{F}$-modules} over $X$ (also abbreviated to \textit{$\mathcal{F}$-module}) is a sheaf $\mathcal{Q}$ on $X$ such that $\mathcal{Q}(V)$ is an $\mathcal{F}(V)$-module for every open subset $V\subset X$, and the associated restrictions are $\mathcal{F}(V)$-module homomorphisms: $\rho_U^V(s_1+s_2)=\rho_U^V(s_1)+\rho_U^V(s_2)$ and $\rho_U^V(\lambda\cdot s)=\rho_U^V(\lambda)\cdot\rho_U^V(s)$, for all open $U\subset V\subset X$ and $s,s_1,s_2\in\mathcal{Q}(V)$, $\lambda\in\mathcal{F}(V)$.
	
	Given another sheaf $\tilde{\mathcal{Q}}$ of $\mathcal{F}$-modules on $X$, $\chi:\mathcal{Q}\rightarrow\tilde{\mathcal{Q}}$ is a \textbf{morphism of $\mathcal{F}$-modules}\index{morphism!of $\mathcal{F}$-modules} if each ring morphism $\chi_U:\mathcal{Q}(U)\rightarrow\tilde{\mathcal{Q}}(U)$ is $\mathcal{F}(U)$-linear. With usual composition between any two such morphisms, we obtain again the structure of a category, and $\mathcal{F}$-modules can be restricted or direct summed to form other $\mathcal{F}$-modules within it.
	
	\noindent Now, $\mathcal{Q}$ is also a \textbf{coherent sheaf of $\mathcal{F}$-modules}\index{coherent sheaf!of $\mathcal{F}$-modules} over $X$ if moreover:
	\begin{itemize}[leftmargin=0.5cm]
		\item $\mathcal{Q}$ is \textit{finitely generated}, that is, for each point $x\in X$ there exists some open neighbourhood $U\subset X$ and a surjective map of (restricted) sheaves $\psi:\mathcal{F}|_U^{\oplus n}\rightarrow \mathcal{Q}|_U$, for some $n\in\mathbb{N}$.
		\item For all open $U\subset X$, $n\in\mathbb{N}$ and morphisms of $\mathcal{F}$-modules $\chi:\mathcal{F}|_U^{\oplus n}\rightarrow \mathcal{Q}|_U$, the \textit{kernel sheaf} $\mathcal{K}er(\chi)$ of $\chi$ is finitely generated (this is the sheaf set by $V\mapsto\text{ker}(\chi_V)$).
	\end{itemize}
	Furthermore, if the surjective map $\psi$ in the first condition is actually an isomorphism of sheaves (if and only if each $\psi_V$ is bijective), then $\mathcal{Q}$ is \textbf{locally free of rank $n$}\index{coherent sheaf!locally free of rank $n$}.
	
	Defining morphisms between coherent sheaves of $\mathcal{F}$-modules over $X$ to be exactly the underlying morphisms of $\mathcal{F}$-modules, we obtain $\mathsf{Coh}(X,\mathcal{F})$, the \textbf{category of coherent sheaves of $\mathcal{F}$-modules over $X$}\index{category!of coherent sheaves of $\mathcal{F}$-modules}.   
\end{Def}

\begin{Rem}\label{bundlesheaf}					
	Let $X=(X,J)$ be a complex manifold, then there exists a bijective correspondence between complex vector bundles $\pi: E\rightarrow X$ of rank $k$ and locally free sheaves $\mathcal{E}_E\cong\bigoplus_{i=1}^k\mathcal{O}_X$ of $\mathcal{O}_X$-modules of rank $k$ (this is explained for example in \cite[Remark 14.18]{[Gat20]} or \cite[Exercise 5.18]{[Har77]}). In particular, if $k=1$, then we talk about line bundles; we will therefore interchange notation accordingly: $E\leftrightarrow\mathcal{E}$ or $TL\leftrightarrow\mathcal{L}$.
	
	A similar statement identifies holomorphic vector bundles on $X$ --- which are complex vector bundles with holomorphic transition functions --- with coherent sheaves of $\mathcal{O}_X$-modules over $X$. Quite reductively, we call any complex submanifold $Y\subset X$ equipped with a holomorphic vector bundle $\pi:E\rightarrow Y$ a \textbf{(B-)brane}\index{B-brane} of the B-model compactified by $(X,J)$. We denote by $\mathcal{V}(X,J)$ the corresponding \textbf{moduli space of (stable) B-branes on $X$}\index{moduli space!of (stable) B-branes}. This plays an active role in the Homological Mirror Symmetry Conjecture. 
\end{Rem}

\begin{Rem}									
	Any coherent sheaf $\mathcal{Q}$ in $\mathsf{Coh}(X,\mathcal{F})$ is automatically quasi-coherent, meaning that there is an open cover $\{U_i\}_{i\in I}$ of $X$ and exact sequences
	\[
	\bigoplus_{j\in I_i}\mathcal{F}|_{U_i}\xrightarrow{\psi_i}\bigoplus_{j\in J_i}\mathcal{F}|_{U_i}\xrightarrow{\chi_i}\mathcal{Q}|_{U_i}\rightarrow\{0\}
	\]
	for (possibly infinite) index sets $I_i,J_i$, for all $i\in I$. Then observe that $\mathcal{Q}$ is locally the cokernel of a morphism of $\mathcal{F}$-modules: $\mathcal{Q}|_{U_i}=\text{im}(\chi_i)\cong \big(\mathcal{F}|_{U_i}^{\oplus |J_i|}\big)/\text{ker}(\chi_i)=\text{coker}(\psi_i)$, by exactness and the isomorphism theorem for sheaves.
	
	Indeed, $\mathsf{Coh}(X,\mathcal{F})$ is an abelian category, thus containing kernels and cokernels of morphisms. 
\end{Rem}
\vspace*{0.3cm}

\noindent By last remark, we can legitimately talk about the bounded derived category $\mathsf{D^b}(\mathsf{Coh}(X,\mathcal{F}))$. Since this turns out to be the counterpart of the bounded derived category of $\mathscr{F}(M)$ in the homological mirror conjecture, we aim at a closer look to its structure. In the spirit of Definition \ref{derivedcat}, we just give an outline and refer the reader to \cite[sections III.5--III.6]{[GM03]} for a more rigorous derivation.

\begin{Def}\label{derivedfunc}					
	Let $\mathsf{A},\mathsf{B}$ be abelian categories, $\mathsf{F}:\mathsf{A}\rightarrow\mathsf{B}$ an additive left (respectively right) exact functor. A class $\mathcal{R}\subset\mathsf{A}$ is \textit{adapted to $\mathsf{F}$} if: 
	\begin{itemize}[leftmargin=0.5cm]
		\item $\mathcal{R}$ is stable under direct sums,
		\item for $C^\bullet\in\textup{obj}(\mathsf{Kom}^+(\mathcal{R}))$ any acyclic complex in $\mathsf{A}$ \textit{bounded from below} (respectively \textit{from above} if in $\mathsf{Kom}^-(\mathcal{R})$), $\mathsf{F}(C^\bullet)$ is an acyclic complex in $\mathsf{B}$,
		\item for all $A\in\textup{obj}(\mathsf{A})$ there exist some $R\in\textup{obj}(\mathcal{R})$ and $i\in\textup{Hom}_\mathsf{A}(A,R)$ which is injective ($\mathcal{R}$ is said to have ``enough injectives''; conversely, $\mathcal{R}$ has ``enough projectives'' if there are surjective morphisms $p\in\textup{Hom}_\mathsf{A}(R,A)$). 
	\end{itemize}
	If such a class exists, we can define a functor $\mathsf{RF}$ from $\mathsf{Kom}^+(\mathcal{R})$ (suitably localized) to $\mathsf{D}^+(\mathsf{B})$ (the corresponding derived category bounded from below) by setting $\mathsf{RF}(C^\bullet)^i\coloneqq\mathsf{F}(C^i)$. With some additional choices, this can be extended to a well-defined triangulated functor $\mathsf{RF}:\mathsf{D}^+(\mathsf{A})\rightarrow\mathsf{D}^+(\mathsf{B})$, called the \textbf{right derived functor}\index{functor!right derived} of $\mathsf{F}$ (if $\mathsf{F}$ is right exact, we obtain the left derived functor $\mathsf{LF}:\mathsf{D}^-(\mathsf{A})\rightarrow\mathsf{D}^-(\mathsf{B})$, coinciding with $\mathsf{RF}$ in case $\mathsf{F}$ is exact).
	
	We let $\mathsf{R}^k\mathsf{F}\coloneqq H^k(\mathsf{RF}):\mathsf{D}^+(\mathsf{A})\rightarrow\mathsf{Ab}$, thus given by $\mathsf{R}^k\mathsf{F}(C^\bullet)=H^k(\mathsf{F}(C^\bullet))$ (and similarly $\mathsf{L}^k\mathsf{F}\coloneqq H^k(\mathsf{LF}):\mathsf{D}^-(\mathsf{A})\rightarrow\mathsf{Ab}$). Then, by triangularity of $\mathsf{RF}$, a short exact sequence $\{0\}\rightarrow A^\bullet\rightarrow B^\bullet\rightarrow C^\bullet\rightarrow\{0\}$ induces a long exact sequence
	\[
	...\rightarrow\mathsf{R}^k\mathsf{F}(A^\bullet)\rightarrow\mathsf{R}^k\mathsf{F}(B^\bullet)\rightarrow\mathsf{R}^k\mathsf{F}(C^\bullet)\rightarrow\mathsf{R}^{k+1}\mathsf{F}(A^\bullet)\rightarrow...
	\]
	(and similarly for $\mathsf{LF}$). From this construction, it is also plain to see that $\mathsf{R}^k\mathsf{F}=0$ for $k<0$ and $\mathsf{R}^0\mathsf{F}=\mathsf{F}$ (respectively, $\mathsf{L}^k\mathsf{F}=0$ for $k>0$ and $\mathsf{L}^0\mathsf{F}=\mathsf{F}$).   
\end{Def}

\begin{Def}\label{extgroups}					
	Let $\mathsf{A}$ be an abelian category and $A,B\in\textup{obj}(\mathsf{A})$, identifying single object complexes $A^\bullet\coloneqq(...\rightarrow\{0\}\rightarrow A\rightarrow\{0\}\rightarrow...)$ and $B^\bullet\coloneqq(...\rightarrow\{0\}\rightarrow B\rightarrow\{0\}\rightarrow...)$ concentrated in degree 0 (whereas $B^\bullet[k]\coloneqq B^{\bullet+k}$ means that $B$ is concentrated in degree $-k$, as usual). Then for $k\in\mathbb{Z}$ we define the \textbf{$k$-th Ext-abelian group}\index{Ext-abelian group} 
	\begin{equation}\label{Ext}
		\text{Ext}_\mathsf{A}^k(A,B)\coloneqq \textup{Hom}_{\mathsf{D^b}(\mathsf{A})}(A^\bullet,B^\bullet[k])
	\end{equation} ($=\textup{Hom}_{\mathsf{D^b}(\mathsf{A})}(A^\bullet[l],B^\bullet[k+l])$ for any $l\in\mathbb{Z}$, under $l$-fold application of the translation functor of $\mathsf{D^b}(\mathsf{A})$). 
	
	This yields functors $\text{Ext}_\mathsf{A}^k:\mathsf{A}^\textup{opp}\times\mathsf{A}\rightarrow\mathsf{Ab}$, with $\text{Ext}_\mathsf{A}^k(A,\square)$ covariant and $\text{Ext}_\mathsf{A}^k(\square,A)$ contravariant for any $A\in\textup{obj}(\mathsf{A})$. Consequently, the composition $A^\bullet\rightarrow B^\bullet[k]\rightarrow C^\bullet[k+l]$ in $\mathsf{D^b}(\mathsf{A})$ yields a bilinear product
	\begin{equation}\label{Yonedaprod}
		\diamond:\text{Ext}_\mathsf{A}^k(A,B)\otimes\text{Ext}_\mathsf{A}^l(B,C)\rightarrow\text{Ext}_\mathsf{A}^{k+l}(A,C)\,,
	\end{equation}
	the \textbf{Yoneda product}\index{Yoneda product}.
\end{Def}

\begin{Pro}									
	Let $\mathsf{A}$ be an abelian category, $A,B\in\textup{obj}(\mathsf{A})$ fixed. The \textup{Ext}-groups have the following properties:
	\begin{itemize}[leftmargin=0.5cm]
		\item Short exact sequences $\{0\}\rightarrow A\rightarrow A'\rightarrow A''\rightarrow\{0\}$ and $\{0\}\rightarrow B\rightarrow B'\rightarrow B''\rightarrow\{0\}$ in $\mathsf{A}$ \textup(seen as distinguished triangles in $\mathsf{D^b}(\mathsf{A})$\textup) induce respectively long exact sequences
		\[
		...\rightarrow \textup{Ext}_\mathsf{A}^k(A'',B)\rightarrow \textup{Ext}_\mathsf{A}^k(A',B)\rightarrow \textup{Ext}_\mathsf{A}^k(A,B)\rightarrow \textup{Ext}_\mathsf{A}^{k+1}(A'',B)\rightarrow...\;,
		\]
		\[
		...\rightarrow \textup{Ext}_\mathsf{A}^k(A,B)\rightarrow \textup{Ext}_\mathsf{A}^k(A,B')\rightarrow \textup{Ext}_\mathsf{A}^k(A,B'')\rightarrow \textup{Ext}_\mathsf{A}^{k+1}(A,B)\rightarrow...\;.
		\]
		\item $\textup{Ext}_\mathsf{A}^k(A,B)=\{0\}$ if $k<0$, while $\textup{Ext}_\mathsf{A}^0(A,B)=\textup{Hom}_\mathsf{A}(A,B)$.
		\item If $\mathsf{A}$ has ``enough injectives'', the left exact functor $\textup{Hom}_\mathsf{A}(A,\square):\mathsf{A}\rightarrow~\!\!\mathsf{Ab}$ has right derived functor $\mathsf{R}\textup{Hom}_\mathsf{A}(A,\square):\mathsf{D^+(A)}\rightarrow\mathsf{D^+(Ab)}$ which fulfills\break $\mathsf{R}^k\textup{Hom}_\mathsf{A}(A,\square)\cong\textup{Ext}_\mathsf{A}^k(A,\square)$ \textup(compatible with \eqref{Ext}\textup).
		\item If $\mathsf{A}$ has ``enough projectives'', the left exact functor  $\textup{Hom}_\mathsf{A}(\square,B)\!:\mathsf{A}^\textup{opp}\!\rightarrow~\!\mathsf{Ab}$ has right derived functor $\mathsf{R}\textup{Hom}_\mathsf{A}(\square,B):\mathsf{D^+}(\mathsf{A}^\textup{opp})\rightarrow\mathsf{D^+(Ab)}$ which fulfills\break $\mathsf{R}^k\textup{Hom}_\mathsf{A}(\square,B)\cong\textup{Ext}_\mathsf{A}^k(\square,B)$ \textup(compatible with \eqref{Ext}\textup). 
	\end{itemize}
	\textup(Note: the last two bullet points imply the first two together with Definition \ref{derivedfunc}, and can in fact be used as an alternative definition for the \textup{Ext}-groups.\textup)
\end{Pro}

\begin{Def}									
	Consider the abelian category $\mathsf{Coh}(X)$ of a complex manifold $X=(X,\mathcal{O}_X)$. Then, according to Definitions \ref{derivedcat} and \ref{extgroups}, we can describe its \textbf{bounded derived category of coherent sheaves}\index{bounded derived category!of coherent sheaves} $\mathsf{D^b}(\mathsf{Coh}(X))$ as follows:
	\begin{itemize}[leftmargin=0.5cm]
		\renewcommand{\labelitemi}{\textendash}
		\item Objects in $\textup{obj}\big(\mathsf{D^b}(\mathsf{Coh}(X))\big)=\textup{obj}(\mathsf{Kom^b(\mathsf{Coh}(X))})$ are (homotopy classes of) bounded cochain complexes of coherent sheaves over $X$, such as $\mathcal{Q}^\bullet\coloneqq(...\rightarrow\mathcal{Q}^{i-1}\rightarrow\mathcal{Q}^i\rightarrow\mathcal{Q}^{i+1}\rightarrow...)$ with each $\mathcal{Q}^i\in\textup{obj}(\mathsf{Coh}(X))$.
		\item For $\mathcal{Q}^\bullet,\tilde{\mathcal{Q}}^\bullet\in\textup{obj}\big(\mathsf{D^b}(\mathsf{Coh}(X))\big)$ one-object complexes, the Hom-groups\break $\textup{Hom}_{\mathsf{D^b}(\mathsf{Coh}(X))}(\mathcal{Q}^\bullet,\tilde{\mathcal{Q}}^\bullet)$ can be identified with the Ext-groups $\text{Ext}_{\mathsf{Coh}(X)}(\mathcal{Q},\tilde{\mathcal{Q}})$ through equation \eqref{Ext}, and this extends to generic complexes of coherent sheaves over $X$. 
		\newline Alternatively, by definition, morphisms $\mathcal{Q}^\bullet\rightarrow\tilde{\mathcal{Q}}^\bullet$ in $\mathsf{D^b}(\mathsf{Coh}(X))$ are (classes of) ``roofs''
		\begin{equation}\label{roof}
			\begin{tikzcd}
				& \mathcal{Q}^\bullet\arrow[dl, "s"']\arrow[dr, "t"] & \\
				\mathcal{U}^\bullet & & \tilde{\mathcal{Q}}^\bullet
			\end{tikzcd}
		\end{equation}
		in the homotopy category of cochain complexes $\mathsf{H}(\mathsf{Coh}(X))$, where $s\in\mathsf{S}$ \big(the localizing class of all quasi-isomorphisms in $\mathsf{H}(\mathsf{Coh}(X))$\big) and $t\in\textup{Hom}_{\mathsf{H}(\mathsf{Coh}(X))}(\mathcal{Q}^\bullet,\tilde{\mathcal{Q}}^\bullet)$.
		\item The composition map $\textup{Hom}_{\mathsf{D^b}(\mathsf{Coh}(X))}(\mathcal{Q}_1^\bullet,\mathcal{Q}_2^\bullet)\times\textup{Hom}_{\mathsf{D^b}(\mathsf{Coh}(X))}(\mathcal{Q}_0^\bullet,\mathcal{Q}_1^\bullet)\rightarrow\textup{Hom}_{\mathsf{D^b}(\mathsf{Coh}(X))}(\mathcal{Q}_0^\bullet,\mathcal{Q}_2^\bullet)$ is the extension of the corresponding Yoneda product \eqref{Yonedaprod}.
	\end{itemize}
	Now, Theorem \ref{DAtriangulated} explains how $\mathsf{D^b}(\mathsf{Coh}(X))$ can be regarded as a triangulated category. 
\end{Def}
\vspace*{0.3cm}

\noindent To properly state the Homological Mirror Symmetry Conjecture, we must also enrich our definition of Fukaya category.

\begin{Def}\label{localsystem}						
	Let $X$ be a complex symplectic manifold, $L\subset X$ a Lagrangian submanifold. A \textbf{local system}\index{local system} for $L$ is\footnote{More generally, a (linear) local system over a topological space is a locally constant sheaf on it whose stalks have structure of finite-dimensional vector spaces. Local systems thus serve as coefficients for sheaf cohomology.} a complex vector bundle $E\rightarrow L$ equipped with a connection $\nabla$ which is flat (hence the corresponding distribution on $E$ is integrable) and has unitary holonomy groups $\text{Hol}^\nabla(x)<\text{U}(n)$ for all $x\in L$. It is \textit{hermitian} if $E$ is equipped with a compatible hermitian metric $h\in\Gamma(E^*\otimes E^*)$ (in the sense of Definition \ref{Kahlerman}). We write succintly $\mathcal{L}\coloneqq (L,E,\nabla)$.
\end{Def}

\begin{Rem}\label{branes}							
	Let $X=(X,J,\kappa^\mathbb{C},\varpi)$ be Calabi--Yau, with complexified Kähler form $\kappa^\mathbb{C}=B-i\kappa$. In the context of the A-model compactified by $(X,\kappa^\mathbb{C})$, we call $\mathcal{L}=(L,E,\nabla)$ an \textbf{(A-)brane}\index{aab@A-brane} if the curvature of $\nabla$ is $R^\nabla=iB|_L$ (for $E$ possibly a flat $U(1)$-line bundle).
	
	Then we call two branes $\mathcal{L}_0=(L_0,E_0,\nabla_0)$ and $\mathcal{L}_1=(L_1,E_1,\nabla_1)$ \textit{Hamiltonian equivalent} if $L_0, L_1$ are Hamiltonian isotopic through some $\phi\in\text{Ham}(X,\kappa)$ (cf. Definition \ref{isotopy}) such that $\phi^*E_1\cong E_0$, and there exists a connection $\tilde{\nabla}$ on $L_0\times[0,1]$ satisfying $R^{\tilde{\nabla}}=i\phi^*(B)$ and $\tilde{\nabla}|_{L_0\times\{t\}}=\phi_t^*(\nabla_t)$, for $t=0,1$. We denote the induced quotient set by $\mathcal{L}ag(X,\kappa^\mathbb{C})$. 
	
	Some additional work (see \cite[chapter 2]{[Fuk02a]}) allows us to regard $\mathcal{L}ag(X,\kappa^\mathbb{C})$ as the \textbf{moduli space of (stable) (A-)branes on $X$}\index{moduli space!of (stable) A-branes} of the A-model compactified by $(X,\kappa^\mathbb{C})$. Remarkably, it has a complex structure with tangent spaces $T_\mathcal{L}\mathcal{L}ag(X,\kappa^\mathbb{C})\cong H^1(L;\Lambda_\mathbb{C})$. 
\end{Rem}

\begin{Def}										
	Let $X=(X,\omega)$ be a complex symplectic manifold with $2c_1(TX)=0$. The \textbf{``enriched'' Fukaya category}\index{enriched@``enriched'' Fukaya category} $\widetilde{\mathscr{F}}(X)$ of $X$ has the following structure:
	\begin{itemize}[leftmargin=0.5cm]
	 	\renewcommand{\labelitemi}{\textendash}
	 	\item Objects are (opportunely decorated) compact Lagrangian submanifolds $\mathcal{L}=(L,E,\nabla)$ equipped with a hermitian local system.
	 	\item The morphism space of $\mathcal{L}_0,\mathcal{L}_1\in\textup{obj}(\widetilde{\mathscr{F}}(X))$ is
	 	\begin{equation}
	 	\textup{hom}_{\widetilde{\mathscr{F}}(X)}(\mathcal{L}_0,\mathcal{L}_1) = CF(\mathcal{L}_0,\mathcal{L}_1) \coloneqq \mkern-18mu\bigoplus_{p\in\mathcal{X}(L_0,L_1)}\mkern-18mu\textup{Hom}(E_0|_p,E_1|_p)\otimes\Lambda_\mathbb{C}\;.
	 	\end{equation}
	 	\item Given $\mathcal{L}_0,...,\mathcal{L}_d\in\textup{obj}(\widetilde{\mathscr{F}}(X))$ so that $L_d\cap L_0=\{q\coloneqq p_{d+1}\equiv p_0\}$ and\break $L_i \cap$ $L_{i+1} = \{p_{i+1}\}$ for $i=0,...,d-1$, parallel transport (with respect to the corresponding connections $\nabla_i$) along the Lagrangian arcs bounding any solution $u\in\mathcal{M}(p_1,...,p_d,q;[u],J,H)$ of \eqref{generalCR} yields linear isomorphisms $\mathbb{P}_{\gamma_i}\in\textup{Hom}(E_i|_{p_i},E_i|_{p_{i+1}})$. Then we define the composition maps $\mu^d\equiv\mu_{\widetilde{\mathscr{F}}(X)}^d$ to be the $\Lambda_\mathbb{K}$-linear extension of
	 	\begin{equation}
	 	\mu^d(f_d,...,f_1)\coloneqq \mkern-30mu\sum_{\substack{q\in\mathcal{X}(L_0,L_d) \\ [u]:\text{ind}([u])=2-d}}\mkern-30mu (\#\mathcal{M}(p_1,...,p_d,q;[u],J,H))T^{\omega([u])}\eta_{[u],f_d,...,f_1},
	 	\end{equation}
	 	where $f_i\in\textup{hom}_{\widetilde{\mathscr{F}}(X)}(\mathcal{L}_{i-1},\mathcal{L}_i)$ and $\eta_{[u],f_d,...,f_1}\coloneqq \mathbb{P}_{\gamma_d}\circ f_d\circ...\circ\mathbb{P}_{\gamma_1}\circ f_1\circ\mathbb{P}_{\gamma_0}\in\textup{Hom}(E_0|_q,E_d|_q)$.
	\end{itemize}
 	Taking cohomology, we obtain the ``enriched'' Fukaya--Donaldson category\break $H(\widetilde{\mathscr{F}}(X))$, whose Hom-groups are Floer cohomology rings $\textup{Hom}_{H(\widetilde{\mathscr{F}}(X))}(\mathcal{L}_0,\mathcal{L}_1)$ $=HF(\mathcal{L}_0,\mathcal{L}_1)$ coming from the cochain complexes $(CF(\mathcal{L}_0,\mathcal{L}_1),\mu^1)$. 
\end{Def}

We will not discuss the underlying Floer theoretic details upon which this variant is built (such an analysis is partially carried in \cite[chapter 2]{[Fuk02a]}), though we shall preserve our notation for Floer complexes and homologies. We also just point out that the definition of wrapped Fukaya category can be adapted to fit this enriched version. 

As we know by now, immediately taking cohomology to produce $H(\widetilde{\mathscr{F}}(X))$ kills too much information. We rather consider the split-closed derived category of $\widetilde{\mathscr{F}}(X)$, $\mathsf{D^\pi}(\widetilde{\mathscr{F}}(X))=H^0\big(\Pi(Tw\widetilde{\mathscr{F}}(X))\big)$, which retains triangularity and split summands. For an explicit enunciation of the latter's structure, we invite the reader to revise Sections \ref{ch3.5} and \ref{ch3.6}.

The homological version of Conjecture \ref{MS} can then be phrased as follows (cf. \cite[chapter 2]{[Fuk02a]}).

\begin{Con}[\textbf{Homological Mirror Symmetry Conjecture}]\index{Homological Mirror Symmetry Conjecture}\label{HMS}		
	Given a Calabi--Yau manifold $X^n=(X,J,\kappa^\mathbb{C},\varpi)$, there exist a ``mirror'' Calabi--Yau manifold $X^\dagger=(X^\dagger, J^\dagger,\kappa^{\mathbb{C}\dagger},\varpi^\dagger)$ of same dimension and equivalences of triangulated categories
	\begin{equation}\label{HMSconj}
		\mathsf{D}^\pi(\widetilde{\mathscr{F}}(X))\cong \mathsf{D^b}(\mathsf{Coh}(X^\dagger)) \quad\,\text{and}\quad\, \mathsf{D}^\pi(\widetilde{\mathscr{F}}(X^\dagger))\cong \mathsf{D^b}(\mathsf{Coh}(X))\,.
	\end{equation}
	Therefore, the following correspondences hold\textup:
	\begin{itemize}[leftmargin=0.5cm]
		\renewcommand{\labelitemi}{\textendash}
		\item \textup(Split-closed twisted complexes built out of\textup) Lagrangian submanifolds $\mathcal{L}\in\textup{obj}\big(\mathsf{D}^\pi(\widetilde{\mathscr{F}}(X))\big)$ correspond to \textup(homotopy classes of\textup) bounded cochain complexes of coherent sheaves $\mathcal{E}^\dagger\in\textup{obj}\big(\mathsf{D^b}(\mathsf{Coh}(X^\dagger))\big)$, and specularly $\mathcal{L}^\dagger\in\textup{obj}\big(\mathsf{D}^\pi(\widetilde{\mathscr{F}}(X^\dagger))\big)$ to $\mathcal{E}\in\textup{obj}\big(\mathsf{D^b}(\mathsf{Coh}(X))\big)$.
		\item \textup{Hom}-groups are isomorphic to \textup{Ext}-groups,  \begin{align*}
			& \textup{Hom}_{\mathsf{D}^\pi(\widetilde{\mathscr{F}}(X))}(\mathcal{L}_0,\mathcal{L}_1)\equiv HF(\mathcal{L}_0,\mathcal{L}_1)\cong\textup{Ext}(\mathcal{E}_0^\dagger,\mathcal{E}_1^\dagger)\equiv\textup{Ext}_{\mathsf{Coh}(X^\dagger)}(\mathcal{E}_0^\dagger,\mathcal{E}_1^\dagger)\,, \\
			&\textup{Hom}_{\mathsf{D}^\pi(\widetilde{\mathscr{F}}(X^\dagger))}(\mathcal{L}_0^\dagger,\mathcal{L}_1^\dagger)\equiv\break HF(\mathcal{L}_0^\dagger,\mathcal{L}_1^\dagger)\cong\textup{Ext}(\mathcal{E}_0,\mathcal{E}_1)\equiv\textup{Ext}_{\mathsf{Coh}(X)}(\mathcal{E}_0,\mathcal{E}_1)\,.
		\end{align*} 
		Remarkably, $\textup{Ext}^k(\mathcal{E}^\dagger,\mathcal{E}^\dagger)\cong H^k(L;\Lambda_\mathbb{C})$ \textup(by an adaptation of Proposition \ref{CFLL}\textup), and vice versa.
		\item Under these isomorphisms, the Floer products 
		\begin{align*}
		& \langle \mu^2\rangle: HF(\mathcal{L}_1,\mathcal{L}_2)\otimes HF(\mathcal{L}_0,\mathcal{L}_1)\rightarrow HF(\mathcal{L}_0,\mathcal{L}_2)\quad\text{and} \\
		& \langle \mu^2\rangle: HF(\mathcal{L}_1^\dagger,\mathcal{L}_2^\dagger) \otimes HF(\mathcal{L}_0^\dagger,\mathcal{L}_1^\dagger)\rightarrow HF(\mathcal{L}_0^\dagger,\mathcal{L}_2^\dagger)
		\end{align*}
		turn respectively into the Yoneda products 
		\begin{align*}
		& \diamond:\textup{Ext}(\mathcal{E}_1^\dagger,\mathcal{E}_2^\dagger)\otimes\textup{Ext}(\mathcal{E}_0^\dagger,\mathcal{E}_1^\dagger) \rightarrow\textup{Ext}(\mathcal{E}_0^\dagger,\mathcal{E}_2^\dagger)\quad\text{and} \\
		& \diamond:\textup{Ext}(\mathcal{E}_1,\mathcal{E}_2)\otimes\textup{Ext}(\mathcal{E}_0,\mathcal{E}_1) \rightarrow\textup{Ext}(\mathcal{E}_0,\mathcal{E}_2)\,.
		\end{align*}  
	\end{itemize}
	Physically speaking, we obtain a bijective correspondence\footnote{More precisely, there is an isomorphism between a neighbourhood $\mathcal{U}\subset\mathcal{V}(X^\dagger,J^\dagger)$ of $\mathcal{E}^\dagger$ and a scheme $Y\subset\text{Ext}^1(\mathcal{E}^\dagger,\mathcal{E}^\dagger)\cong H^1(L;\Lambda_\mathbb{C})\cong T_\mathcal{L}\mathcal{L}ag(X,\kappa^\mathbb{C})$.} between the mod-\break uli space $\mathcal{L}ag(X,\kappa^\mathbb{C})$ of \textup(stable\textup) A-branes $\mathcal{L}$ on $X$ of the A-model compacti-\break fied by $(X,\kappa^\mathbb{C})$ and the moduli space $\mathcal{V}(X^\dagger,J^\dagger)$ of \textup(stable\textup) B-branes $\mathcal{E}^\dagger$ on $X^\dagger$ of the B-model compactified by $(X^\dagger,J^\dagger)$ \textup(cf. Remarks \ref{bundlesheaf} and \ref{branes}\textup), and specularly.
\end{Con}

\begin{Rem}											
	Conjecture \ref{HMS} is an inflated version of that given by Kontsevich in \cite{[Kon94]}, to the extent of being too general to hold true. For example, the bijection on objects is broken due to $\textup{obj}\big(\mathsf{D^b}(\mathsf{Coh}(X^\dagger))\big)$ being usually bigger than $\textup{obj}\big(\mathsf{D}^\pi(\widetilde{\mathscr{F}}(X))\big)$. Also, $\mathcal{L}ag(X,\kappa^\mathbb{C})$ is a priori unobstructed, but $\mathcal{V}(X^\dagger,J^\dagger)$ might be still, thus jeopardizing the correspondences among morphism groups and composition maps.\footnote{A possible solution, resorting to \textit{bounding chains} within curved $A_\infty$-categories (as described in Remark \ref{Fukayarem}), is suggested by Fukaya in \cite[Conjecture 3.21]{[Fuk02a]}. Otherwise, $HF(\mathcal{L},\mathcal{L})$ and $H(L;\Lambda_\mathbb{C})$ are in general merely related through a spectral sequence.} 
	
	Consequently, part of the Calabi--Yau decorations are sometimes dropped in favour of just a symplectic structure on $X=(X,\kappa)$ and a complex structure on its mirror $X^\dagger=(X^\dagger,J^\dagger)$, this choice being also dictated by the absence --- so far, to the author's knowledge --- of examples exhibiting the correspondence $(X,J)\leftrightarrow(X^\dagger,\kappa^\dagger)$.
	
	In all cases, we stress that both existence and uniqueness of a mirror manifold are not guaranteed --- we are talking a conjecture, after all.
\end{Rem}

\begin{Ex}											
	Let us briefly mention a few examples and point towards the relevant literature.\footnote{More exhaustive lists of verified cases can be easily found online, for example at \texttt{https://ncatlab.org/nlab/show/mirror+symmetry}.}
	\begin{itemize}[leftmargin=0.5cm]
		\item The first non-trivial case is that of complex elliptic curves (whose mirror symmetry we already discussed in Example \ref{mirrorsymmex}), proved by Polishchuk and Zaslow in \cite{[PZ00]}. The advantage here is that the construction of $\mathsf{D}^\pi(\widetilde{\mathscr{F}}(X))$ involving twisted complexes can be bypassed. Instead, the authors work directly with the enriched Donaldson--Fukaya category $\mathscr{F}^0(X)\coloneqq H(\widetilde{\mathscr{F}}(X))$ and prove
		\begin{equation}
		\mathscr{F}^0(E^\rho)\cong\mathsf{D^b}(\mathsf{Coh}(E_\tau))\,,
		\end{equation}
		where $X\coloneqq E^\rho$ is the symplectic 2-torus $\mathbb{T}^2$ equipped with complexified Kähler form $\rho=B+i\kappa$ (recall Example \ref{2torus}) and $X^\dagger\coloneqq E_\tau$ is the elliptic curve $\mathbb{C}/\Lambda$ with $\Lambda\coloneqq\langle 1, \tau\rangle$ (for $\mathfrak{Im}(\tau)>0$). The duality translates to $\rho\leftrightarrow\tau$.
		
		\item In \cite{[PZ00]}, elliptic curves are treated over a generic ground field $\mathbb{K}$, while the symplectic side of the 2-torus is (without loss of generality) discussed over $\mathbb{C}$. Abouzaid and Smith specialize the analysis to $\mathbb{K}=\Lambda_\mathbb{R}$, proving \cite[Theorem 6.4]{[AS09]}:
		\begin{equation}
		\mathsf{D}^\pi(\mathscr{F}(\mathbb{T}^2))\cong\mathsf{D^b}(\mathsf{Coh}(E_A))\,,
		\end{equation}
		where $\mathbb{T}^2$ is equipped with the standard area form $\kappa=Ad\theta_1\wedge d\theta_2$, while $E_A\coloneqq\Lambda_\mathbb{R}^*/\langle T^A\rangle$ is the (analytification of the) \textit{Tate elliptic curve}, defined algebraically by its coordinate ring $\Lambda_\mathbb{R}[x,y]/\{y^2+xy=x^3+a_4(T)x+ a_6(T)\}$ (see \cite[Definition 6.3]{[AS09]} for the definition of the infinite series $a_4, a_6$ in the formal variable $T\in\Lambda_\mathbb{R}$). This result is later used to show that $QH(\mathbb{T}^2;\Lambda_\mathbb{R})\cong HH(\mathscr{T}^2)$; cf. Remark \ref{quantumHochschild}.
		
		\item In \cite{[AS09]}, Abouzaid and Smith then move along to their main goal, namely an equivalence over $\Lambda_\mathbb{R}$ for the 4-torus\footnote{A similar analysis to that hinted in Example \ref{2torus} and Proposition \ref{splittorus} reveals $\mathscr{F}(\mathbb{T}^4)$ to be split-generated by the four 2-tori obtained taking pairwise products of longitudes and meridians.}  $\mathbb{T}^4=\mathbb{R}^4/\mathbb{Z}^4$:
		\begin{equation}
		\mathsf{D}^\pi(\mathscr{F}(\mathbb{T}^4))\cong\mathsf{D^b}(\mathsf{Coh}(E_A\times E_A))\,,
		\end{equation}
		which extends to $\mathsf{D}^\pi(\mathscr{F}(\mathbb{T}^{2k}))\cong\mathsf{D^b}(\mathsf{Coh}(E_A^k))$ for any $k\in\mathbb{Z}$ (and an unobstructed version of Fukaya category). 
		
		\item In \cite{[Sei13]}, Seidel proves the conjecture over $\Lambda_\mathbb{Q}$ for any quartic surface $X_0\subset\mathbb{C}P^3$ (that is, any projective variety given by the vanishing locus of a degree 4 homogeneous polynomial in 4 variables).
		
		\item Furthermore, parts of the Homological Mirror Symmetry Conjecture are discussed by Fukaya for abelian varieties (in \cite{[Fuk00]}) and by Kontsevich and Soibelman for torus fibrations (in \cite{[KS01]}).
		
		\item There are also several extensions to non-Calabi--Yau manifolds, such as surfaces of genus $g\geq 2$ (see for example \cite{[Sei11]}) or Fano varieties whose mirrors are Landau--Ginzburg models (\cite{[Aur07]}). \hfill $\blacklozenge$        
	\end{itemize} 
\end{Ex}
\vspace*{0.3cm}

\noindent Now, the reader may be legitimately wondering: ``How can mirror pairs of Calabi--Yau manifolds be constructed, in practice?''. The answer is provided by the work \cite{[SYZ96]} of Strominger, Yau and Zaslow, published two years after Kontsevich's conjecture. The motivation comes from the following observation.

\begin{Rem}\label{specialrem}					
	Points $p\in X^\dagger$ in the mirror Calabi--Yau manifold are in bijection with \textbf{skyscraper sheaves}\index{skyscraper sheaf} $\mathcal{O}_p\in\textup{obj}(\mathsf{Coh}(X^\dagger))\subset\textup{obj}\big(\mathsf{D^b}(\mathsf{Coh}(X^\dagger))\big)$, which are defined by the assignment
	\vspace*{-0.2cm}
	\begin{equation}
	U\subset X^\dagger\text{ open}\longmapsto\mathcal{O}_p(U)\coloneqq 
	\begin{cases}
	\mathbb{C} & $if $ p\in U \\ 
	\{0\} & $else $
	\end{cases}\,. 
	\end{equation}
	(The main features of skyscraper sheaves, especially within the realm of algebraic geometry, are discussed in \cite[chapters 13, 14]{[Gat20]}.) Therefore, we can identify $X^\dagger$ with the moduli space of skyscraper sheaves over $X^\dagger$, which in turn, by Conjecture \ref{HMS}, should correspond to a moduli space of certain objects $\mathcal{L}_p\in\mathsf{D}^\pi(\widetilde{\mathscr{F}}(X))$.
	
	Now, it is an exercise in sheaf cohomology to prove that $\text{Ext}_{\mathsf{Coh}(X^\dagger)}^k(\mathcal{O}_p,\mathcal{O}_p)\cong\Lambda^kT_pX^\dagger$ (see for example \cite[Lecture 16]{[Aur09]}). Since moreover $\text{Ext}(\mathcal{O}_p,\mathcal{O}_p)\cong H(\mathbb{T}^n;\mathbb{C})\cong HF(\mathbb{T}^n,\mathbb{T}^n)$ (as graded vector spaces, again exploiting Proposition \ref{CFLL}), we are tempted to say that the $\mathcal{L}_p$ are none other than \textbf{Lagrangian tori}\index{Lagrangian torus} $T^n$, $2n$-dimensional symplectic manifolds diffeomorphic to the standard torus $\mathbb{T}^n\subset\mathbb{R}^{2n}\cong\mathbb{C}^n$. 
\end{Rem} 
	
This is close enough to what we are interested in. We only require the following notions.

\begin{Def}										
	Let $X=(X,J,\kappa^\mathbb{C})$ be a Kähler manifold, $\varpi\in\Omega^{n,0}(X)$ a holomorphic volume form. Then $X$ is \textit{almost Calabi--Yau} if $|\varpi|_\kappa\coloneqq \sqrt{|\kappa(\varpi,\varpi)|}$ $=g\in C^\infty(X,\mathbb{R}_{>0})$ and $\varpi\wedge\overline{\varpi}=const(n)g^2\kappa^n$. 
	
	A Lagrangian torus $T\subset X$ fulfills $\varpi|_T=e^{i\phi}g\cdot\text{vol}_\kappa|_T$ for some phase function $e^{i\phi}:T\rightarrow\mathbb{S}^1$ (recall Definition \ref{phase}). We call it \textit{special} if $e^{i\phi}$ is actually constant (eventually implying $\mathfrak{Im}(\varpi|_T)=0$, the condition for a general Lagrangian submanifold to be special). 
\end{Def}

We finally state the ``geometrical'' Mirror Symmetry Conjecture, named after Strominger, Yau and Zaslow. It builds on the intuition from Remark \ref{specialrem} that generic points of the mirror Calabi--Yau manifold $X^\dagger$ correspond to isomorphism classes of some $(T,\nabla)$, where $T\subset X$ is a Lagrangian torus and $\nabla$ a flat $U(1)$-connection on $\underline{\mathbb{C}}\rightarrow T$ \big(thus they correspond to elements of $\textup{Hom}(\pi_1(T),U(1))$\big). 

Originally, in \cite{[SYZ96]}, the conjecture was not cast in a formal mathematical statement --- on which there is currently no agreement yet. Here we favoured the formulation of \cite[section 12.2]{[GHJ03]} and \cite[Lecture 21]{[Aur09]}.  

\begin{Con}[\textbf{SYZ Conjecture}]\index{SYZ Conjecture}\label{SYZ}				
	Let $X$ be an $n$-dimensional Calabi--Yau manifold \textup(originally $n=3$\textup). Then there exists a mirror $n$-dimensional Calabi--Yau manifold $X^\dagger$ and some compact $n$-dimensional \textup(over $\mathbb{R}$\textup) topological base manifold $B$ such that $X$ and $X^\dagger$ carry dual fibrations by special Lagrangian tori,
	\begin{equation}
	\begin{tikzcd}
		\mathbb{T}^n \arrow[r] & X \arrow[d, "\pi"'] \\
		& B
	\end{tikzcd}\quad\Longleftrightarrow\quad
	\begin{tikzcd}
		X^\dagger \arrow[d, "\pi^\dagger"] & (\mathbb{T}^n)^\vee\arrow[l]  \\
		B &
	\end{tikzcd},	
	\end{equation}
	for $(\mathbb{T}^n)^\vee\coloneqq \textup{Hom}(\pi_1(\mathbb{T}^n),U(1))$, meaning that
	\begin{align}
		X^\dagger=\{(T,\nabla)\mid T\!=\!(\pi^\dagger)^{-1}(b)\!\subset\! X^\dagger\text{ special Lagr. \!torus},\,\nabla\!\in\!(\mathbb{T}^n)^\vee\}\,,
	\end{align} 
	and that the fibres $\pi^{-1}(b)$, $(\pi^\dagger)^{-1}(b)$ should be dual to each other for all $b\in B$ \textup(some fibres may be singular tori, so that $B$ should actually be restricted to some subset $B_{reg}\subset B$\textup). The specular statement holds as well.
\end{Con}

The situation is inspirationally depicted in Figure \ref{calabiyau}. We don't indulge on the details, but limit ourselves to a few remarks:

\begin{Rem}										
	Like for Conjecture \ref{HMS}, the SYZ conjecture we gave is affected by its high level of generality. A more sound statement \textit{presupposes} the existence of a dual pair $(X,X^\dagger)$, only in dimension 3, and claims that the fibres $\pi^{-1}(b)$, $(\pi^\dagger)^{-1}(b)$ are non-singular special Lagrangian tori $T^3$ in $X$ respectively $T^{\dagger3}$ in $X^\dagger$. Topologically, this can be interpreted as $H^1(T)\cong H_1(T^\dagger)$ in singular (co)homology. Geometrically, this makes sense for tori $T=V/\Lambda$ with \textit{flat} Riemannian metrics (for $\Lambda$ a lattice in some vector space $V$), whose duals then take the form $T^\dagger=V^*/\Lambda^*$ (the dual vector space quotient the dual lattice).
	
	Latter point is the reason why Conjecture \ref{SYZ} should really be understood as a statement about the ``large complex structure limit'' $t\rightarrow 0$ of a family of dual 3-dimensional Calabi--Yau manifolds $\{(X_t,X_t^\dagger)\}_t$. In this limit, the fibers diameters are small compared to that of the base manifold $B$, and their metrics are approximately flat. There are two main approaches (discussed in \cite[Sections 12.3, 12.4]{[GHJ03]}):
	\begin{itemize}[leftmargin=0.5cm]
		\item A symplectic topological one, which treats $X$ and $X^\dagger$ as merely 6-dimensional real symplectic manifolds and $\pi$, $\pi^\dagger$ as non-special \textit{smooth} Lagrangian fibrations over a \textit{smooth} 3-dimensional real manifold $B$. The resulting duality is well understood, especially for quintic 3-folds and Calabi--Yau hypersurfaces in 4-tori.
		\item A local geometric one, stressing singular behaviours of special Lagrangian fibrations for \textit{generic} 3-dimensional (almost) Calabi--Yau manifolds. This perspective shows that $\pi$, $\pi^\dagger$ need not be smooth, but only piecewise smooth, and this determines the nature of the singular base locus $B\setminus B_{reg}$: either of codimension 1 or 2 in $B$, with respective fibres singular along an $\mathbb{S}^1$ or at finitely many points.    
	\end{itemize}     
\end{Rem}     

\begin{figure}[htp]
	\centering
	\includegraphics[width=0.95\textwidth]{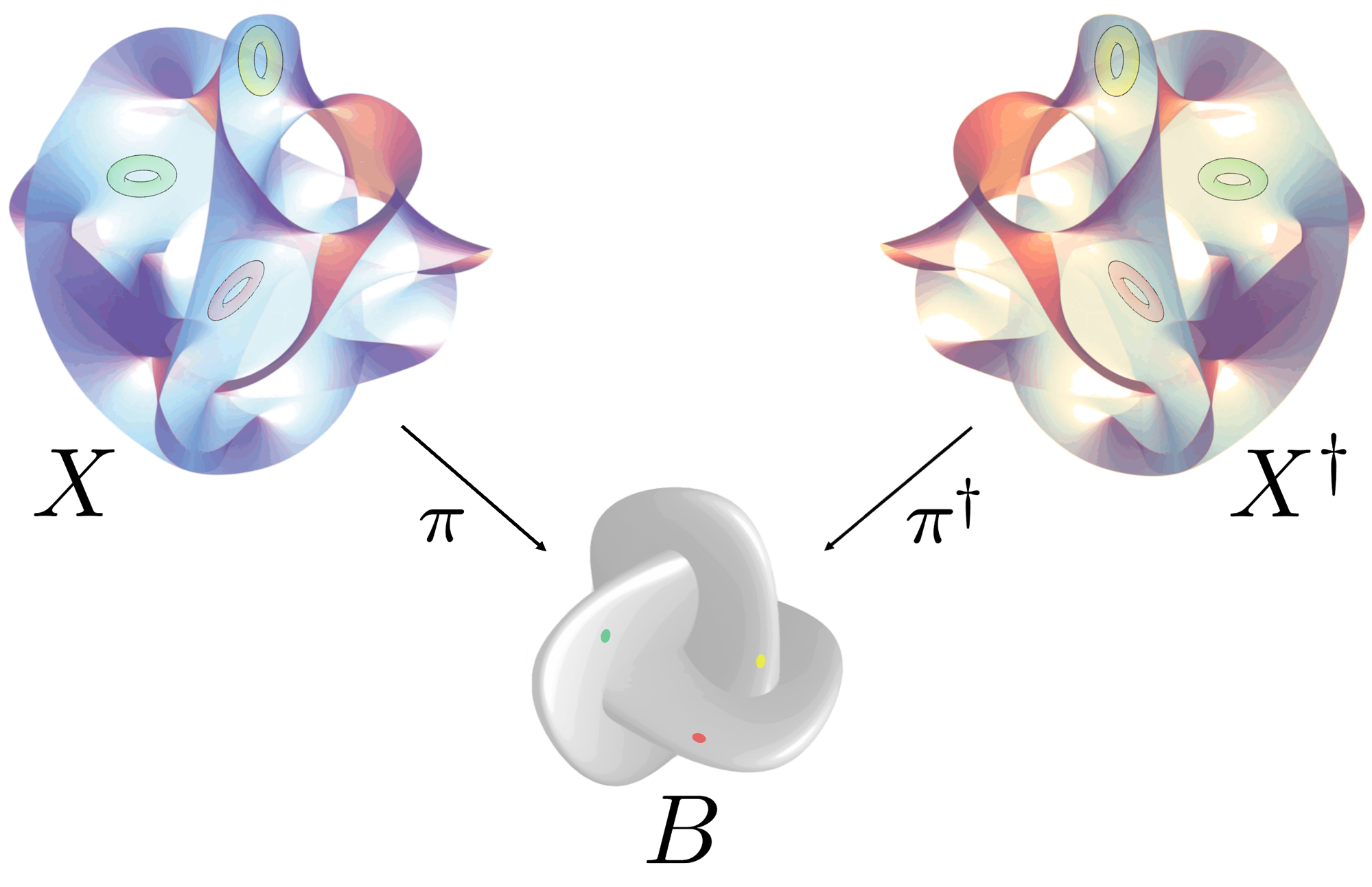}
	\caption{The SYZ conjecture relates the fibres, special Lagrangian tori, of a mirror pair $(X,X^\dagger)$ of Calabi--Yau manifolds.\newline [Sources: \begin{footnotesize}\texttt{http://hci.stanford.edu/~cagatay/projects/boys/Demiralp-2009-CLF.pdf}, \texttt{https://plus.maths.org/content/hidden-dimensions}\end{footnotesize}; special thanks to KoMa, who suggested the best way to edit this and many other figures!]}
	\label{calabiyau}	
\end{figure}

\newpage
\thispagestyle{plain}

\clearpage
\section*{Outlook}
\addcontentsline{toc}{section}{Outlook}

Before wrapping up the discussion, a little retrospection is in order. Homological mirror symmetry brought us on a journey which, starting from a good amount of ``abstract nonsense'' (read category theory and homological algebra), led us to the geometrical realm described by Floer theory, to ultimately arch back to Kontsevich's algebraic conjecture, our declared goal. This parabola required us to borrow notions from many fundamental branches of pure mathematics: symplectic and complex geometry, hints of functional analysis, some ever-present algebraic topology and, ultimately, glimpses of algebraic geometry, all firmly accompanied by categorical notation. 
Useless to say, we barely scratched what this topic has to offer. And in this seemingly never-ending potential lies perhaps its undeniable charm.  

After all, it seems only fair that a tentative physical ``theory of everything'' would gather so much mathematical richness and point out its unexpected connections, promoting the vision of a more subtly interwoven theoretical landscape and, by reflection, reality. Of course, part of what presented here is --- to the author's best guess --- already quite obsolete, and the homological mirror conjecture remains, as of now, still unproven in its most general form. But the lessons remain out there, and the symphony is one worth listening to.

\newpage
\thispagestyle{plain}

\clearpage
\section*{}

\newpage
\thispagestyle{plain}

\printindex


\addcontentsline{toc}{section}{References}



\begin{thebibliography}{20000000}
	\vspace*{0.5cm}
	
	\bibitem[ABC+09]{[ABC+09]} P.S. Aspinwall, T. Bridgeland, A. Craw, M.R. Douglas, M. Gross, A. Kapustin, G.W. Moore, G. Segal, B. Szendröi, P.M.H. Wilson. \textit{Dirichlet branes and mirror symmetry}. Clay Mathematics Monographs, vol. 4, AMS, (2009).
	\begin{footnotesize}ISBN\end{footnotesize}: 978-0-8218-3848-8. \begin{footnotesize}URL\end{footnotesize}: \url{http://www.claymath.org/library/monographs/cmim04.pdf}.
	
	\bibitem[Abo10]{[Abo10]} M. Abouzaid. \textit{A geometric criterion for generating the Fukaya category}. Publications mathématiques de l'IHES, Tome 112 (2010), 191--240, (September 29, 2010.) \begin{footnotesize}URL\end{footnotesize}: \url{https://arxiv.org/abs/1001.4593}.
	
	\bibitem[AS09]{[AS09]} M. Abouzaid, I. Smith. \textit{Homological mirror symmetry for the four-torus}. Duke Math. J. 152, no. 3 (2010), 373--440, (July 31, 2009). \begin{footnotesize}URL\end{footnotesize}: \url{https://arxiv.org/abs/0903.3065v2}.
	
	\bibitem[Aur07]{[Aur07]} D. Auroux. \textit{Mirror symmetry and T-duality in the complement of an anticanonical divisor}, (July 10, 2007). \begin{footnotesize}URL\end{footnotesize}: \url{https://arxiv.org/abs/0706.3207v2}.
	
	\bibitem[Aur09]{[Aur09]} D. Auroux. \textit{Mirror symmetry}. Lecture notes, MIT, (Spring 2009). \begin{footnotesize}URL\end{footnotesize}: \url{https://math.berkeley.edu/~auroux/18.969-S09/}. 
	
	\bibitem[Aur13]{[Aur13]} D. Auroux. \textit{A beginner's introduction to Fukaya categories}, (January 29, 2013). \begin{footnotesize}URL\end{footnotesize}: \url{https://arxiv.org/abs/1301.7056v1}.
	
	\bibitem[BBS07]{[BBS07]} K. Becker, M. Becker, J. Schwarz. \textit{String theory and M-theory: A modern introduction}. Cambridge University Press, (2007). \begin{footnotesize}ISBN\end{footnotesize}: 978-0-521-86069-7.
	
	\bibitem[CdGP91]{[CdGP91]} P. Candelas, X. de la Ossa, P. Green, L. Parkes. \textit{A pair of Calabi--Yau manifolds as an exactly soluble superconformal theory}. Nuclear Physics B, Vol. 359, Issue 1, 21--74, (February 1, 1991). 
	\begin{footnotesize}URL\end{footnotesize}: \url{https://webspace.science.uu.nl/~beuke106/HypergeometricFunctions/COGP.pdf}.
	
	\bibitem[CGK02]{[CGK02]} K. Cieliebak, V.L. Ginzburg, E. Kerman. \textit{Symplectic homology and periodic orbits near symplectic submanifolds}, (November 12, 2002). \begin{footnotesize}URL\end{footnotesize}: \url{https://arxiv.org/abs/math/0210468v2}.
	
	\bibitem[Flo88]{[Flo88]} A. Floer. \textit{Morse theory for Lagrangian intersections}. J. Differential Geom. 28, no. 3, 513--547, (1988). \begin{footnotesize}URL\end{footnotesize}: \url{https://projecteuclid.org/journals/journal-of-differential-geometry/volume-28/issue-3/Morse-theory-for-Lagrangian-intersections/10.4310/jdg/1214442477.full}. 
	
	\bibitem[FOOO07]{[FOOO07]} K. Fukaya, Y.-G. Oh, H. Ohta, K. Ono. \textit{Lagrangian surgery and holomorphic discs}. Chapter 10, (November 17, 2007). \begin{footnotesize}URL\end{footnotesize}: \url{https://www.math.kyoto-u.ac.jp/~fukaya/Chapter10071117.pdf}.
	
	\bibitem[FOOO09]{[FOOO09]} K. Fukaya, Y.-G. Oh, H. Ohta, K. Ono. \textit{Lagrangian intersection Floer theory: anomaly and obstruction}. AMS/IP Studies in Advanced Mathematics, Vol. 46, (2009). \begin{footnotesize}ISBN\end{footnotesize}: 978-0-8218-5253-8.
	
	\bibitem[Fuk93]{[Fuk93]} K. Fukaya. \textit{Morse homotopy, $A_\infty$-category, and Floer homologies}. In \textit{Proceedings of GARC workshop on Geometry and Topology} (H.J. Kim ed.), Seoul National University, 1--102, (1993). \begin{footnotesize}URL\end{footnotesize}: \url{https://www.math.kyoto-u.ac.jp/~fukaya/mfikki.pdf}.
	
	\bibitem[Fuk99]{[Fuk99]} K. Fukaya. \textit{Floer homology for 3-manifold with boundary I}. Kyoto University, (1999). \begin{footnotesize}URL\end{footnotesize}: \url{https://www.math.kyoto-u.ac.jp/~fukaya/bdrt1.pdf}.
	
	\bibitem[Fuk00]{[Fuk00]} K. Fukaya. \textit{Mirror symmetry of abelian varieties and multi-theta functions}. J. Algebraic Geom. 11, 393--512, (2000). \begin{footnotesize}URL\end{footnotesize}: \url{https://www.math.kyoto-u.ac.jp/~fukaya/abelrev.pdf}.
	
	\bibitem[Fuk02a]{[Fuk02a]} K. Fukaya. \textit{Floer homology and mirror symmetry I}. (2002) \begin{footnotesize}URL\end{footnotesize}: \url{https://www.math.kyoto-u.ac.jp/~fukaya/mirror1.pdf}.
	
	\bibitem[Fuk02b]{[Fuk02b]} K. Fukaya. \textit{Floer homology and mirror symmetry II}. (2002) \begin{footnotesize}URL\end{footnotesize}: \url{https://projecteuclid.org/download/pdf\_1/euclid.aspm/1546230292}.

	\bibitem[Gat20]{[Gat20]} A. Gathmann. \textit{Algebraic geometry}. Lecture notes, TU Kaiserslautern, (2019--2020). \begin{footnotesize}URL\end{footnotesize}: \url{https://www.mathematik.uni-kl.de/~gathmann/de/alggeom-2019.php}.
	
	\bibitem[GHJ03]{[GHJ03]} M. Gross, D. Huybrechts, D. Joyce. \textit{Calabi--Yau manifolds and related geometries}. Lectures at a Summer School in Nordfjordeid (June 2001), Springer Universitext, (2003). \begin{footnotesize}ISBN\end{footnotesize}: 978-3-540-44059-8.
	
	\bibitem[GJ90]{[GJ90]} E. Getzler, J.D.S. Jones. \textit{$A_\infty$-algebras and the cyclic bar complex}. Illinois J. Math. 34, no. 2, 256--283, (1990). \begin{footnotesize}URL\end{footnotesize}: \url{https://cpb-us-e1.wpmucdn.com/sites.northwestern.edu/dist/c/2278/files/2019/08/cyclic2.pdf}. 
	
	\bibitem[GM03]{[GM03]} S.I. Gelfand, Yu.I. Manin. \textit{Methods of homological algebra}. Second edition, Springer Monographs in Mathematics, (2003). \begin{footnotesize}ISBN\end{footnotesize}: 978-3-642-07813-2.
	
	\bibitem[Gre97]{[Gre97]} B.R. Greene. \textit{String theory on Calabi--Yau manifolds}. Columbia University, New York, (February 23, 1997). \begin{footnotesize}URL\end{footnotesize}: \url{https://arxiv.org/abs/hep-th/9702155v1}.
	
	\bibitem[Har77]{[Har77]} R. Hartshorne. \textit{Algebraic Geometry}. Springer Graduate Texts in Mathematics 52, (1977). \begin{footnotesize}ISBN\end{footnotesize}: 978-1-4757-3849-0.
	
	\bibitem[Hat09]{[Hat09]} A. Hatcher. \textit{Algebraic Topology}. Cambridge University Press, (2009). \begin{footnotesize}ISBN\end{footnotesize}: 978-0-521-79540-1. \begin{footnotesize}URL\end{footnotesize}: \url{http://pi.math.cornell.edu/~hatcher/AT/ATpage.html}.
	
	\bibitem[HKK+03]{[HKK+03]} K. Hori, S. Katz, A. Klemm, R. Pandharipande, R. Thomas, C. Vafa, R. Vakil, E. Zaslow. \textit{Mirror symmetry}. Clay Mathematics Monographs, vol. 1, AMS, (2003).
	\begin{footnotesize}ISBN\end{footnotesize}: 978-0-8218-2955-4.  \begin{footnotesize}URL\end{footnotesize}: \url{https://web.archive.org/web/20060919020706/http://math.stanford.edu/~vakil/files/mirrorfinal.pdf}.
	
	\bibitem[Kad80]{[Kad80]} T.V. Kadeishvili. \textit{On the theory of homology of fiber spaces}. (Russian) International Topology Conference (Moscow State Univ., Moscow, 1979). Uspekhi Mat. Nauk 35 (1980), no. 3(213), 183--188. \begin{footnotesize}URL\end{footnotesize}: \url{https://arxiv.org/pdf/math/0504437.pdf}.
	
	\bibitem[Kel01]{[Kel01]} B. Keller. \textit{Introduction to A-infinity algebras and modules}. Homology Homotopy Appl. 3, (January 12, 2001). \begin{footnotesize}URL\end{footnotesize}: \url{https://arxiv.org/abs/math/9910179v2}.
	
	\bibitem[Kel06]{[Kel06]} B. Keller. \textit{A-infinity algebras, modules and functor categories}. (February 12, 2006) \begin{footnotesize}URL\end{footnotesize}: \url{https://arxiv.org/abs/math/0510508v3}.  
	
	\bibitem[Kon94]{[Kon94]} M. Kontsevich. \textit{Homological algebra of mirror symmetry}. In \textit{Proceedings of the International Congress of Mathematicians}, Birkhäuser, 120--139, (November 30, 1994). \begin{footnotesize}URL\end{footnotesize}: \url{https://arxiv.org/abs/alg-geom/9411018v1}.
	
	\bibitem[KS01]{[KS01]} M. Kontsevich, Y. Soibelman. \textit{Homological mirror symmetry and torus fibrations}. (June 4, 2001). \begin{footnotesize}URL\end{footnotesize}: \url{https://arxiv.org/abs/math/0011041v2}.
	
	\bibitem[Lef03]{[Lef03]} K. Lefèvre-Hasegawa. \textit{Sur les $A_\infty$-catégories}. PhD thesis, Université Paris 7, (October 21, 2003).
	\begin{footnotesize}URL\end{footnotesize}: \url{https://arxiv.org/abs/math/0310337v1}.
	
	\bibitem[LP11]{[LP11]} Y. Lekili, T. Perutz. \textit{Fukaya categories of the torus and Dehn surgery}. (March 4, 2011)
	\begin{footnotesize}URL\end{footnotesize}: \url{https://arxiv.org/abs/1102.3160}.
	
	\bibitem[May01]{[May01]} J.P. May. \textit{The axioms for triangulated categories}, edited extract from \textit{The additivity of traces in triangulated categories}, Advances in Mathematics 163, 34–-73, (2001). \begin{footnotesize}URL\end{footnotesize}: \url{http://www.math.uchicago.edu/~may/MISC/Triangulate.pdf}.
	
	\bibitem[McS94]{[McS94]} D. McDuff, D. Salamon. \textit{J-holomorphic curves and quantum cohomology}. University Lecture Series, Vol. 6, AMS, (1994). \begin{footnotesize}ISBN\end{footnotesize}: 0-8218-0332-8. \begin{footnotesize}URL\end{footnotesize}: \url{https://people.math.ethz.ch/~salamon/PREPRINTS/jholsm.pdf} (May 1995).
	
	\bibitem[Mer14]{[Mer14]} W.J. Merry. \textit{Floer theory}. Lecture notes, ETH Zürich, (2013--2014). 
	\begin{footnotesize}URL\end{footnotesize}: \url{https://www.merry.io/other-notes/floer-theory/}.
	
	\bibitem[Mer18]{[Mer18]} W.J. Merry. \textit{Algebraic topology}. Lecture notes, ETH Zürich, (2017--2018). \begin{footnotesize}URL\end{footnotesize}: \url{https://www.merry.io/courses/algebraic-topology/}.
	
	\bibitem[Mer19]{[Mer19]} W.J. Merry. \textit{Differential geometry}. Lecture notes, ETH Zürich, (2018--2019). \begin{footnotesize}URL\end{footnotesize}: \url{https://www.merry.io/courses/differential-geometry/}.
	
	\bibitem[Moo05]{[Moo05]} G. Moore. \textit{What is... a brane?}. Notices of the AMS, vol. 52, no. 2, 214--215, (February 2005). \begin{footnotesize}URL\end{footnotesize}: \url{http://www.ams.org/notices/200502/what-is.pdf}. 
		
	\bibitem[MS74]{[MS74]} J.W. Milnor, J.D. Stasheff. \textit{Characteristic classes}. Annals of Mathematics Studies, no. 76, Princeton University Press; University of Tokyo Press, (1974). \begin{footnotesize}ISBN\end{footnotesize}: 978-0-691-08122-9. \begin{footnotesize}URL\end{footnotesize}: \url{https://www.maths.ed.ac.uk/~v1ranick/papers/milnstas.pdf}.
	
	\bibitem[Oh91]{[Oh91]} Y.-G. Oh. \textit{Symplectic topology and Floer homology}. IBS Center for Geometry and Physics, and Pohang University of Science and Technology, Republic of Korea. Cambridge University Press, (1991). \begin{footnotesize}URL\end{footnotesize}: \url{https://cgp.ibs.re.kr/~yongoh/all.pdf}.
	
	\bibitem[PZ00]{[PZ00]} A. Polishchuk, E. Zaslow. \textit{Categorical mirror symmetry: the elliptic curve}. Adv. Theor. Math. Phys. 4, 1187--1207, (April 7, 2000). \begin{footnotesize}URL\end{footnotesize}: \url{https://arxiv.org/abs/math/9801119v3}.
	
	\bibitem[Sal97]{[Sal97]} D. Salamon. \textit{Lectures on Floer homology}. IAS/Park City Graduate Summer School on Symplectic Geometry and Topology, (December 1, 1997). \begin{footnotesize}URL\end{footnotesize}: \url{https://people.math.ethz.ch/~salamon/PREPRINTS/floer.pdf}.
	
	\bibitem[Sei99]{[Sei99]} P. Seidel. \textit{Lagrangian two-spheres can be symplectically knotted}. J. Differential Geom. 52, 145--171, (May 31, 2000). \begin{footnotesize}URL\end{footnotesize}: \url{https://arxiv.org/abs/math/9803083}.
	 
	\bibitem[Sei02]{[Sei02]} P. Seidel. \textit{A long exact sequence for symplectic Floer cohomology}. Revised version, Topology 42, 1003--1063, (July 20, 2002). \begin{footnotesize}URL\end{footnotesize}: \url{https://arxiv.org/abs/math/0105186v2}.
	
	\bibitem[Sei08]{[Sei08]} P. Seidel. \textit{Fukaya categories and Picard-Lefschetz theory}. Zurich Lectures in Advanced Mathematics, European Mathematical Society, (2008). \begin{footnotesize}ISBN\end{footnotesize}: 978-3-03719-063-0.
	
	\bibitem[Sei11]{[Sei11]} P. Seidel. \textit{Homological mirror symmetry for the genus two curve}. (August 21, 2011)	\begin{footnotesize}URL\end{footnotesize}: \url{https://arxiv.org/abs/0812.1171v5}.
	
	\bibitem[Sei13]{[Sei13]} P. Seidel. \textit{Homological mirror symmetry for the quartic surface}. (August 6, 2013) \begin{footnotesize}URL\end{footnotesize}: \url{https://arxiv.org/abs/math/0310414v4}.
	
	\bibitem[Sil12]{[Sil12]} A.C. da Silva. \textit{Lectures on symplectic geometry}. Lecture Notes in Mathematics, Springer, (2008). \begin{footnotesize}ISBN\end{footnotesize}: 978-3-540-45330-7. \begin{footnotesize}URL\end{footnotesize}: \url{https://people.math.ethz.ch/~acannas/Papers/lsg.pdf} (January 2006).
	
	\bibitem[Sta63]{[Sta63]} J.D. Stasheff. \textit{Homotopy associativity of H-spaces, I-II}. Trans. of AMS, Vol. 108, 275--292 resp. 293--312, (August 1963). \begin{footnotesize}URL\end{footnotesize}: \url{https://www.ams.org/journals/tran/1963-108-02/S0002-9947-1963-99936-3/S0002-9947-1963-99936-3.pdf} resp. \url{https://www.ams.org/journals/tran/1963-108-02/S0002-9947-1963-0158400-5/S0002-9947-1963-0158400-5.pdf}.
	
	\bibitem[SYZ96]{[SYZ96]} A. Strominger, S.-T. Yau, E. Zaslow. \textit{Mirror symmetry is T-duality}. Nucl. Phys. B479, 243--259, (June 14, 1996). \begin{footnotesize}URL\end{footnotesize}: \url{https://arxiv.org/abs/hep-th/9606040v2}.
	
	\bibitem[Ver96]{[Ver96]} J.-L. Verdier. \textit{Des catégories dérivées des catégories abéliennes}. Astérisque 239, Société Mathématique de France, (1996). \begin{footnotesize}URL\end{footnotesize}: \url{http://www.numdam.org/issue/AST_1996__239__R1_0.pdf}.
	
	\bibitem[Wen14]{[Wen14]} C. Wendl. \textit{Lectures on holomorphic curves in symplectic and contact geometry}. University College London, (May 27, 2014). \begin{footnotesize}URL\end{footnotesize}: \url{https://arxiv.org/abs/1011.1690v2}.
	
	\bibitem[YN10]{[YN10]} S.-T. Yau, S. Nadis. \textit{The shape of inner space: string theory and the geometry of the universe's hidden dimensions}. Basic Books, (2010). \begin{footnotesize}ISBN\end{footnotesize}: 978-0-465-02023-2.
	
	\bibitem[Zwi09]{[Zwi09]} B. Zwiebach. \textit{A first course in string theory}. Cambridge University Press, (2009). \begin{footnotesize}ISBN\end{footnotesize}: 978-0-521-88032-9.
	
\end{thebibliography}
\end{document}